\documentclass[11pt]{imsart}

\usepackage{bbm}

\RequirePackage[OT1]{fontenc}
\RequirePackage{amsthm,amsfonts,amsmath,amssymb,amsbsy,array,calrsfs,enumerate,enumitem,bbold,hyperref,mathrsfs,color,graphicx,caption,mwe,natbib,soul,titlesec,xr}
\usepackage{fullpage}
\usepackage[colorinlistoftodos]{todonotes}
\arxiv{arXiv:2010.14715}

\startlocaldefs
\numberwithin{equation}{section}
\theoremstyle{plain}
\newtheorem{theorem}{Theorem}[section]
\newtheorem{definition}[theorem]{Definition}
\newtheorem{lemma}[theorem]{Lemma}
\newtheorem{corollary}[theorem]{Corollary}
\newtheorem{proposition}[theorem]{Proposition}

\theoremstyle{remark}
\newtheorem{remark}{Remark}[section]
\newtheorem{example}{Example}[section]

\endlocaldefs

\def\A{{\cal A}}
\def\B{{\cal B}}
\def\ii{{\mathbbm{i}}}
\def\P{\mathbb P}
\def\dist{{d_{\mathbb V}}}  %% The metric on V
\def\Z{\mathbb Z}
\def\C{\mathbb C}
\def\bbN{\mathbb N}
\newcommand{\ol}{\overline}
\def\H{\mathcal H} %% The self-similarity exponent -- was \alpha -- I replaced it throughout by \H, which we can change to what we want.
\def\V{\mathbb V}
\def\I{\mathrm{I}}
\newcommand\comment[1]{{\color{red} \fbox{Stuff commented out -- see the latex source.}}}

\newcommand{\bal}{\begin{align*}}
\newcommand{\eal}{\end{align*}}
\newcommand{\baln}{\begin{align}}
\newcommand{\enalign}{\end{align}}

\newcommand{\bn}{\begin{enumerate}}
	\newcommand{\en}{\end{enumerate}}

\newcommand{\bc}{\begin{cases}}
	\newcommand{\ec}{\end{cases}}

\newcommand{\tr}{{\rm tr}}

\newcommand{\bbR}{{\mathbb R}}

\newcommand{\bb}{{\boldsymbol{b}}}

\newcommand{\cid}{\stackrel{d}{\to}}
\newcommand{\Cid}{\stackrel{d}{\longrightarrow}}
\newcommand{\cas}{\buildrel a.s. \over \longrightarrow}

\newcommand{\eqid}{\buildrel d \over =}
\newcommand{\wh}{\widehat}
\newcommand{\wt}{\widetilde}
\newcommand{\E}{\mathbb{E}}
\newcommand{\pr}{\mathbb{P}}
\newcommand{\R}{\mathbb{R}}
\newcommand{\N}{\mathbb{N}}
\newcommand{\bbT}{\mathbb{T}}
\newcommand{\T}{\mathbb{T}}
\newcommand{\CT}{\mathcal{T}}
\newcommand{\CB}{\mathcal{B}}
\newcommand{\CK}{\mathcal{K}}
\newcommand{\bbL}{\mathbb{L}}
\newcommand{\CI}{\mathcal{I}}
\newcommand{\CJ}{\mathcal{J}}
\newcommand{\CF}{\mathcal{F}}
\newcommand{\bbC}{\mathbb{C}}
\newcommand{\bbS}{\mathbb{S}}

\newcommand{\ep}{\noindent\begin{flushright}${{\hfill\llap{$\sqcup\!\!\!\!\sqcap$}}}$\end{flushright}}
 \newenvironment{proof2}[2] {\paragraph{\bf{Proof of {#1} {#2}:}}}{\hfill$\square$\vskip.25cm}

\def\clb{}
\usepackage{cancel}
\newcommand\replace[2]{{#2}}

\newcommand{\be}{\begin{eqnarray*}}
\newcommand{\ee}{\end{eqnarray*}}
\newcommand{\ben}{\begin{eqnarray}}
\newcommand{\een}{\end{eqnarray}}
\def\CC{\mathcal C}
\def\bbH{\mathbb H}
\def\la{\langle}
\def\ra{\rangle}
\def\wtilde{\widetilde}

\begin{document}

\begin{frontmatter}
\title{Tangent fields, intrinsic stationarity, and self-similarity\\ (with a supplement on Matheron Theory)}
\runtitle{Tangent fields, intrinsic stationarity and self-similarity}

\begin{aug}
\author{\fnms{Jinqi} \snm{Shen,}\ead[label=e1]{jqshen@umich.edu}}
\author{\fnms{Stilian} \snm{Stoev,}\ead[label=e2]{sstoev@umich.edu}}
\and
%\author{\fnms{Tailen} \snm{Hsing}\thanksref{e3}\ead[label=e3,mark]{thsing@umich.edu}%
\author{\fnms{Tailen} \snm{Hsing}\ead[label=e3]{thsing@umich.edu}
%\ead[label=u1,url]{www.foo.com}
}

\address[a]{Department of Statistics\\
The University of Michigan, Ann Arbor\\
\printead{e1,e2,e3}}

%\address[b]{Airbnb
%\printead{e3},
%\printead{u1}}
\runauthor{J.\ Shen et al.}

\affiliation{University of Michigan}
 
\end{aug}

\begin{center}

\bigskip
{\em Dedicated to the Memory of Mark Marvin Meerschaert (1955--2020)}

\medskip
\end{center}

\begin{abstract}
This paper studies the local structure of continuous random fields on $\mathbb R^d$ taking values in a complete 
separable linear metric space ${\mathbb V}$.  Extending seminal work of Falconer, %  \cite{falconer:2002}, 
we show that the generalized  $(1+k)$-th order increment tangent fields are self-similar and almost everywhere intrinsically stationary in the sense of  Matheron. % \cite{matheron:1973}.  
These results motivate the further study of the structure of ${\mathbb V}$-valued 
intrinsic random functions of order $k$ (IRF$_k$,\ $k=0,1,\cdots$).   To this end, we focus on the special case where ${\mathbb V}$ is a 
Hilbert space. Building on the work of  Sasvari and Berschneider, %\cite{Sasvari:2009} and \cite{Berschneider:2012de}, 
we establish the spectral  characterization of all second order ${\mathbb V}$-valued IRF$_k$'s, extending the classical Matheron theory.
Using these results, we further characterize the class of Gaussian, {\em operator self-similar} ${\mathbb V}$-valued IRF$_k$'s,
generalizing results of Dobrushin and Didier, Meerschaert and Pipiras, % \cite{Dobrushin:1979vi} and \cite{didier:meerschaert:pipiras:2017}, 
among others.  These processes are the Hilbert-space-valued versions of the general $k$-th order operator fractional Brownian fields and are characterized by their 
self-similarity operator exponent as well as a finite trace class operator valued spectral measure.   We conclude with several 
examples motivating future applications to probability and statistics.

In a technical \hyperlink{supp}{Supplement} of independent interest, we provide a unified treatment of the Matheron spectral theory for 
second-order stationary and intrinsically stationary processes taking values in a separable Hilbert space.  We give the proofs of the 
Bochner-Neeb and Bochner-Schwartz theorems.
\end{abstract}


\end{frontmatter}

\section{Introduction.} 

The tangent process of a random field is the stochastic process obtained in the limit of the suitably normalized increments of the random field at a fixed location. A pair of papers, \cite{falconer:2002,falconer:2003}, discovered a remarkable property about the 
structure of the tangent process. Briefly speaking, Falconer proved that the tangent processes 
must be self-similar and have stationary increments (a.e.).  The self-similarity of the tangent field is not surprising, which 
is a consequence of scaling. It is akin to what is shown in many limit theorems in the literature, such as those in the seminal work of Lamperti \citep[cf.][]{lamperti:1962}; see also \cite{davydov:paulauskas:2017} and a host of results on (univariate and multivariate) regular variation in \cite{gnedenko:1943,meerschaert:1984,hult:lindskog:2006}. 
The property of stationary increments  for the tangent process, however,  is less expected.  The proof of this property in Falconer's works involves a remarkable 
Lebesgue-density argument and ideas from geometric measure theory \citep[cf.][]{falconer:1986,preiss:1987}.   

The starting point of our paper is extending Falconer's results in two directions.  Firstly, we consider generalized tangent processes obtained by taking local $(1+k)$-th order increments, $k\ge 0$.  This is 
necessary if one wants to study the local behavior of many models arising in spatial statistics.  Secondly, we consider random fields taking values in a linear complete separable metric space $\V$ such 
as but not limited to a separable Banach space.  The resulting limit processes, will be referred to as $k$-th order tangent processes.
In \cite{falconer:2002,falconer:2003}, $k=0$ and $\V=\bbR$.
The self-similarity property continues to hold for  $k$-th order tangent processes, where self-similarity is in the sense of a general class of scaling actions, including operator 
scaling \citep[cf.][]{meerschaert:scheffler:2001book}, that commensurate with the generality of the state space $\V$.
To establish the generalized stationary-increment property, we introduce a new proof strategy based on the Lusin and Egorov theorems as well as some core ideas in \cite{falconer:2002,falconer:2003}.

 Interestingly, the stationarity of the higher-order increments of the $k$-th order tangent processes is related to the notion of {\em intrinsic random functions} of order $k$ (IRF$_k$) introduced
 by  \cite{matheron:1973}.  In the special case of real-valued processes ($\V=\R$), the classic results of Matheron as well as \cite{gelfand:vilenkin:1964d} lead to a concrete formula 
 for {\em all possible} covariance structures of the $k$-th order tangent fields.  It involves the local self-similarity exponent $H \in (0,k+1]$ and the local {\em spectral  measure} $\sigma$. Such results 
 have been established by \cite{Dobrushin:1979vi} in the setting where the paths of the stochastic processes are generalized functions, i.e., random elements in $\mathcal{S}'(\R^d)$ -- the topological 
 dual of the Schwartz space $\mathcal{S}(\R^d)$.  Our study, motivated by applications to spatial statistics and functional data analysis, considers random fields taking values in a separable Hilbert space $\V$.
 We follow the approach of Matheron rather than Dobrushin and realize the notion of a higher-order increment by integrating the process against signed measures with finite supports.

To this end, in Section \ref{sec:c2s5}, we develop an extension of Matheron's theory to the case of processes taking values in a separable Hilbert space $\V$.
Our theoretical development for Hilbert-space-valued IRF$_k$'s is of independent interest and builds on a large body of existing although somewhat scattered work.
With no intention to provide a complete list, we refer to \cite{bochner1948vorlesungen} and \cite{khintchine1934korrelationstheorie} for Bochner's theorem and \cite{Neeb1998} for extensions  to general spaces; \cite{cramer1942} for the spectral representation of stationary random fields; \cite{matheron:1973}, \cite{Sasvari:2009} for the existence of general covariance of IRF$_k$ and  its integral representation; \cite{Berschneider:2012de} for the integral representation of IRF$_k$ in an abstract space. A more comprehensive summary on this line of literature can be found in \cite{Berschneider2018}. Our proofs in 
this regard are contained in the \hyperlink{supp}{Supplement}, %extended version of this paper \citep[][]{shen:stoev:hsing:2020_extended}, 
which aims to be self-contained and only uses arguments that are common in probability and statistics.

The developed theory is then utilized in Section \ref{s:self-similar IRFk} to characterize the covariance structure of self-similar intrinsic random functions taking values in a
separable Hilbert space.  General linear operator scaling actions are considered as well as the classic multiplication by a scalar.   In particular, our results provide characterizations of 
Gaussian operator self-similar IRF$_k$'s, which can be viewed as infinite-dimensional versions of the $k$-th order fractional Brownian fields.  
The $0$-th order operator fractional Brownian fields taking values in $\V=\R^m$ have been the subject of active investigation and numerous applications 
\citep[see e.g.][and the references therein]{mason:xiao2001,amblard:coeurjolly:2011,kechagias:pipiras:2015,abry:didier:2018,abry:wendt:jaffard:didier:2019,duker:2020,bierme:lacaux:2020}. 
Most if not all of the existing work, however, focuses on random fields taking values in $\R^m$.   In this paper, we provide a first comprehensive treatment of Hilbert space valued 
operator fractional Brownian fields and their higher order stationary increment counterparts -- the Gaussian IRF$_k$'s.  This leads to infinite-dimensional extensions of seminal results due to \cite{didier:pipiras:2011,didier:meerschaert:pipiras:2017,perrin2001nth} among others.

This paper also contributes to statistical research in several ways.  Matheron's work on $\R$-valued IRF$_k$'s has had a substantial impact on the field of spatial statistics
 \citep[see e.g.\ the monographs of][] {MR1697409, chiles:delfner:2012}.  Our extension of the Matheron theory to the case of Hilbert-space-valued random fields provides novel tools and 
 framework for spatially dependent functional data analysis -- an active area in statistics \citep[see, e.g., the monographs][and references therein]{ramsay:silverman:2005,horvath:kokoszka:2012,Hsing2015}. Unfortunately, the details of Matheron's 
 theory have been elusive to the broader community.  The \hyperlink{supp}{Supplement} 
 %Our work and its extended version \cite{shen:stoev:hsing:2020_extended}
 will be a useful resource for those who are interested in 
 learning those details and their novel generalizations. The self-similar IRF$_k$ is itself a flexible model for spatial statistics. The self-similarity exponent operator $\H$
 and spectral measure $\sigma$, which now takes values in the space of positive trace-class operators on $\V$ characterize the covariance structure.  The pair $(\H,\sigma)$ may be object 
 of further modeling and inference in the context of in-fill asymptotics \citep[cf.][]{MR1697409}.  The tangent process connection also provides guidance for building flexible random field models 
 with desired local properties.

\medskip
\noindent{\em The paper is organized as follows.} Section~\ref{sec:section2} introduces a suitable topology on the path space and scaling actions needed to define and
study higher-order tangent fields.  In Section~\ref{s:tangent_fields}, we establish the main results on the structure of higher-order tangent fields, namely their
self-similarity and almost everywhere intrinsic stationarity.  Section~\ref{sec:c2s5} develops the spectral theory for second-order stationary and 
intrinsically stationary random fields taking values in a separable Hilbert space $\V$.  This treatment unifies and extends 
results of Bochner, Cram\'er, Gelfand-Vilenkin, Matheron, Neeb, Sasv\'ari, and Berschneider.  The covariance structure of the self-similar 
$\V$-valued IRF$_k$ is characterized in Section \ref{s:self-similar IRFk}.  General linear operator-scaling actions (Section~\ref{s:operator-ss}) and the classic
scalar scaling actions (Section \ref{ss:T_rrx}) are studied.  Open problems, examples and connections to the existing literature are presented in 
Section \ref{sec:related-work-and-examples}.  Some technical proofs are relegated to the Appendix. Further background and details are given in the \hyperlink{supp}{Supplement}. %\cite{shen:stoev:hsing:2020_extended}.

%----------- Based on those two properties,  some spectral theories for mean-square continuous tangent fields are provided in Section~\ref{sec:c2s5} with different 
%scaling operators. Finally, Section~\ref{sec:c2s6} provides some examples in Gaussian related to the multi-fractional Brownian motion and Mat\'ern class. 
% Additional proofs are available in the \tcr{Supplement} of this paper. 
%In Section~\ref{sec:c2s2}, we will introduce notations, topologies and a formal definition of tangent fields. Self-similarity and intrinsic stationarity of tangent field will be proved in Section~\ref{sec:c2s3} and Section~\ref{sec:c2s4} respectively. Section~\ref{sec:c2s5} will contain a refined spectral theory for tangent field with its unique representation in Gaussian. Two examples will be given in Section~\ref{sec:c2s6} to illustrate the connection of our theory with common spatial models. Finally, two preliminary inference methods are provided in Section~\ref{sec:c2s7} based on the our spectral theory and  the periodogram kernel method in \cite{Panaretos:2013hs}. 

	\section{Preliminaries.} \label{sec:section2}
	This section develops some tools that will be useful for the study of tangent fields. 
	We commence by defining some key spaces and operations.
	
	For $d=1,2,\ldots$, let $\Lambda$ denote the collection of complex-valued measures 
	on $\mathbb{R}^d$ supported on finitely many points, i.e.,
	\begin{align}\label{e:lambda-rep}
	\lambda(d u) = \sum_{i=1}^n c_i \delta_{t_i}(d u),
	\end{align}
	where $n=1,2,\dots$, $c_i\in \mathbb{C}, t_i \in \mathbb R^d$ for all $i$, and $\delta_a$ is the
	Dirac measure at $a$. Without loss of generality, we always assume the $t_i$'s in the representation
	\eqref{e:lambda-rep} to be distinct. In this paper we will be concerned with functions defined on $\Lambda$
	or a subspace of $\Lambda$ and take values in some vector space $\V$.
	A special case of that is 
	$$ g(\lambda) := \int g d\lambda \equiv \sum_{i=1}^n c_i g(t_i),$$
	for any $g: \mathbb{R}^d\mapsto \V$.
	
	For any $k\in \mathbb{N}:=\{0,1,2,\ldots\}$, a monomial of degree $k$ on $\bbR^d$ is any function of the form 
	$u =(u_1,\dots, u_d) \mapsto u_1^{j_1} \cdots u_d^{j_d},$ where $j_1,\ldots,j_d$ are non-negative integer powers 
	such that $j_1+\cdots+j_d=k$. More generally, a polynomial of degree $k$ is any complex 
	linear combination of monomials of degree less than or equal to $k$ with at least one non-zero degree $k$ term.

	\begin{definition} For any $k=0,1,2,\ldots$, let ${\Lambda}_{k}$ be the class of 
		$\lambda\in\Lambda$ such that $\int_{\bbR^d} f d\lambda=0$ for polynomials $f$ with degree ${\rm deg}(f)\le k$. 
		Thus, measures in $\Lambda_k$ ``annihilate'' all polynomials of degree up to $k$. We also let
		$\Lambda_{-1} := \Lambda$.
	\end{definition}
	
	Next, we define two operations pertinent to the definition of tangent fields. 
	As usual, for any set $B\subset\bbR^d$, $c\in\bbR$ and $s\in\bbR^d$, 
	let $c\cdot B=\{c t\, :\,  t \in B\}$ and $s+B=B+s=\{s+t\, :\, t \in B\}$. Also, define the 
	scaling and translation operations on $\Lambda$: 
	\begin{align}\label{e:sca_trans}
	r\cdot\lambda :=\lambda(r^{-1}\cdot), \ r \not = 0, \quad\hbox{and}
	\quad s+\lambda:=\lambda(\cdot-s), \ \ s\in \mathbb{R}^d.
	\end{align}
	Clearly, ${\Lambda}_{k}$ is closed with respect to both of these operations.

Assume that the random elements considered in the paper take values in a 
complete and separable metric linear space $(\V, \dist)$ over $\bbC$. Recall that
$(\V, \dist)$ is said to be a metric linear space \citep[cf.][]{rolewicz1985} if scalar multiplication 
and addition are continuous with respect to $\dist$.  Namely, for all $x_n, y_n, x, y \in \V$ and 
$c_n, c\in\bbC$, such that $ |c_n-c| + \dist(x_n,x) + \dist(y_n,y) \to 0$,  we have
\begin{align*}
\dist(c_n x_n, cx) \to 0 \quad\hbox{and}\quad \dist(x_n+y_n, x+y) \to 0.
\end{align*}
%Typical examples for $(E,\dist)$ are the Euclidean space 
%$\mathbb R^d$, or a Banach space, like the space ${ C}([0,1])$ of continuous functions equipped with the 
%sup-norm, or $L^2(\mathbb R)$. For any function $f$ from $\bbR^d$ to ${\mathbb V}$, let $f(\lambda) = \int f d\lambda$ whenever there is no ambiguity. 
By the Birkhoff-Kakutani Theorem \citep[cf. Theorem 1.1.1 of][]{rolewicz1985}, 
	without loss of generality, we can and do assume that the metric $d_{\V}$ is translation
	invariant, that is,
	\begin{align}\label{e:translation}
	\dist(x,y) = \dist(x-y,0) \ \hbox{ for any $x,y\in \V$}. 
	\end{align}
A typical example of $\V$ in our applications is a separable Banach or even Hilbert space. 
However, we do not restrict to only Banach spaces for now.
	%By continuity w.r.t. scalar multiplication, we have $\lim_{n\to\infty}\dist(cx_n,cx)=0$ for any $c\in\mathbb{C}$ and $x_n,x\in \V$ with $\lim_{n\to\infty} \dist(x_n,x)=0$. However, 
	We also assume throughout that the following continuity condition holds: For any $K>0$, 
		\begin{align}\label{eq:extra_condition}
		\lim_{\delta\to 0}\sup_{|c|\leq K,\dist(x,0)<\delta}\dist(cx,0)\to 0.
		\end{align}
Note that \eqref{eq:extra_condition} readily holds if $\V$ is a normed space
and $d_\V$ is induced by the norm.
%but it is not true in general. 

\subsection{The spaces $S(\Lambda_{k},\V)$ and $\breve S(\Lambda_k,\V)$.} 
\label{s:2.1}

A function $f$ from $\Lambda$ to ${\mathbb V}$ is linear if $f(c_1{\clb \lambda_1}+c_2\lambda_2) 
= c_1f(\lambda_1)+c_2 f(\lambda_2), c_1,c_2\in\bbC, \lambda_1,\lambda_2\in\Lambda$.
Denote by $S(\Lambda_{k},\V)$ the set of all linear functions from 
$\Lambda_k$ to ${\mathbb V}$.
% and $S_c(\Lambda_{k},\V)$ the subset of all continuous linear functions from $\Lambda_k$ to ${\mathbb V}$.
In this section, we focus on obtaining a representation of functions in 
$S(\Lambda_{k},\V)$ in terms of functions from $\bbR^d$ to ${\mathbb V}$ as well as a topological 
structure for a subspace of $S(\Lambda_{k},\V)$.

Next we discuss the important notion of {\em representation} introduced by \cite{matheron:1973}.
	A function $\breve f : \bbR^d \mapsto \V$ is said to be a representation of $f\in S(\Lambda_{k},\V)$ if
	\begin{align} \label{e:rep_def}
	f(\lambda) = \int \breve f d\lambda, \ \lambda\in\Lambda_k.
	\end{align}
Consider the following construction of a representation.
Denote by $m_i,\ i=1,\ldots,M_k:=\binom{k+d}{k}$ all the monomials of 
degree less than or equal to $k$ on $\bbR^d$, where the ordering is arbitrary. 
Define the $\bbR^{M_k}-$valued function 
	\begin{align}
	 \label{e:bt}
	  \bb(t) := (m_1(t),\ldots,m_{M_k}(t))^\top,\quad  t\in\bbR^d.
	\end{align}
	Pick the points $t_i\in\bbR^d, i=1,\ldots,M_k$, such that 
	the ${M_k\times M_k}$ matrix $B :=(\bb(t_1),\ldots,\bb(t_{M_k}))$ has full rank.  
	Such points always exist. For instance, if $d=1$ then $B$ has full rank 
	for arbitrary distinct $t_i$'s. However, for $d>1$ some care is needed in selecting 
	the $t_i$ to ensure that $B$ has full rank. 
	For $t\in\bbR^d$, define the measure
	\begin{align}\label{e:eval1}
	\lambda_t = \delta_t-(\delta_{t_1},\ldots,\delta_{t_{M_k}}) B^{-1}\bb(t),
	\end{align}
	where, for any $\boldsymbol{c}=(c_1,\ldots,c_{M_k})^\top\in\bbC^{M_k}$, 
	$(\delta_{t_1},\ldots,\delta_{t_{M_k}})\boldsymbol{c}$ denotes
	the measure $\sum_{j=1}^{M_k} c_j\delta_{t_j}$. 
	Below, for convenience, we adopt such matrix notation when there 
	is no ambiguity. 
	It follows from \eqref{e:bt} and \eqref{e:eval1} that
	\begin{align*}
	\left(\int m_1 d \lambda_t, \ldots, \int m_{M_k}d\lambda_t\right)^\top
	= (I-B B^{-1})\bb(t)  = \bf{0},
	\end{align*}
	and so $\lambda_t\in\Lambda_k$ for all $t$. 
	For any $f\in S(\Lambda_{k},\V)$, consider
	\begin{align}\label{e:tilde_f1}
	\breve f(t):=f(\lambda_t).
	%= f(\delta_t) - \left(f(\delta_{t_1}),\ldots,f(\delta_{t_{M_k}})\right) B^{-1}\bb(t).
	\end{align}
	Note that $\lambda_{t_i} = 0$, the null measure, for all $i$. Thus,
	\begin{align}\label{e:tilde_f2}
	\left(\breve f(t_1),\ldots,\breve f(t_{M_k})\right) = \bf{0}.
	\end{align}
	Moreover, for $\lambda\in\Lambda_k$, by linearity, 
	\begin{align} \label{e:tilde_f3}
	 \int \breve f d\lambda = f\left(\lambda - (\delta_{t_1},\ldots,\delta_{t_{M_k}}) B^{-1}\bb(\lambda)\right) = f(\lambda),
	\end{align}
	since $\bb(\lambda) = \bf{0}$, and therefore \eqref{e:rep_def} holds showing that $\breve f$ is
	a representation of $f$. 
	
	Clearly, the function $\breve f$ defined by \eqref{e:tilde_f1} is not the only possible representation of $f$. 
	However, any two representations $g_1,g_2$ of $f$ differ by
	a polynomial of degree $k$, since, for all $t$,
	\begin{align} \label{e:equiv_rep}
	0 = \int (g_1-g_2) d \lambda_t  = g_1(t) - g_2(t) 
	- (g_1(t_1)-g_2(t_1), \ldots, g_1(t_{M_k})-g_2(t_{M_k})) 	B^{-1}\bb(t),
	\end{align}
	The difference will not affect any of the results in this paper. Thus, from now on, we will adhere to the 
	representation $\breve f$ defined by \eqref{e:tilde_f1}.
	
	It follows from \eqref{e:tilde_f2} and \eqref{e:tilde_f3} that there is a one-to-one
	correspondence between $S(\Lambda_{k},\V)$ and 
	$\breve{S}(\Lambda_k,\V)$, where $\breve{S}(\Lambda_k,\V)$ denotes the set of functions from
	$\bbR^d$ to ${\mathbb V}$ that are equal to zero at each $t_i, i = 1,\ldots, M_k$,
 	%\begin{align*}
	%\breve{S}(\Lambda_k,\V) := \{\breve f\,:\, \breve f 
	%\hbox{ is continuous and such that }
	%\breve f(t_i)=0, i = 1,\ldots, M_k\},
	%\end{align*}
	where the isomorphism is determined by the bijection
	\begin{align}\label{e:iso}
	\begin{split}
	%\mathcal{J}&: f \mapsto f(\lambda_t), \ S(\Lambda_k,\V) \mapsto 
	%\breve{S}(\Lambda_k,\V), \\
	%\mathcal{K}&: \breve f \mapsto \int \breve f d\lambda, \ \breve{S}(\Lambda_k,\V) \mapsto
	%S(\Lambda_k,\V).
	(\mathcal{J}f) (t) &= {f(\lambda_t)}, \ f\in S(\Lambda_k,\V), \\  
	(\mathcal{J}^{-1}\breve f)(\lambda) & = \int \breve f d\lambda, \ \breve f\in 
	\breve{S}(\Lambda_k,\V).
	\end{split}
	\end{align}
			
	A linear function $f\in S(\Lambda_{k},\V)$ is said to be continuous if its 
	representation $\breve f$ is a continuous function from $\bbR^d$ to ${\mathbb V}$. 
	By \eqref{e:equiv_rep}, this property is ``intrinsic'' to $f$ and does not 
	depend on the representation. Let $S_c(\Lambda_{k},\V)$ denote the subset of
	continuous linear functions from $\Lambda_k$ to ${\mathbb V}$, and $\breve{S}_c(\Lambda_k,\V)$ 
	the corresponding set of representations. The mapping $\mathcal{J}$ in \eqref{e:iso} continues to
	be an isomorphism between $S_c(\Lambda_k,\V)$ and
	$\breve{S}_c(\Lambda_k,\V)$.
	
	We now proceed to define a topology on $S_c(\Lambda_k,\V)$ and 
	$\breve{S}_c(\Lambda_k,\V)$.
	A convenient metric on $\breve{S}_c(\Lambda_k,\V)$ is
	\begin{align}\label{eq:metric}
	\rho(\breve f,\breve g) = \sum_{n\geq 1} 2^{-n}\Big(1-\exp\Big\{-\sup_{\|t\|\leq n}\dist(\breve f(t) , 
	\breve g(t))\Big\}\Big),
	\end{align}
	where, as before, $\|\cdot\|$ denotes the Euclidean norm on $\bbR^d$
	and $\dist$ denotes the metric on ${\mathbb V}$.  
	Clearly, $\rho$ metrizes the local uniform 
	convergence topology, i.e., uniform convergence on compact sets.  
	We also have the following simple but important fact, proved in Section \ref{sec:proofs_sec:section2} in Appendix.
	
	\begin{lemma}\label{l:S_is_csms} The metric space $(\breve{S}_c(\Lambda_k,\V),\rho)$ is complete and separable.
	\end{lemma}
		 
	Using the canonical bijection $\mathcal{J}$, we define the corresponding metric on $S_c(\Lambda_k,\V)$ 
	as $\rho(f,g) := \rho(\mathcal{J} f, \mathcal{J} g)$, for $f,g\in S_c(\Lambda_k,\V)$,
	where the same symbol $\rho$ is adopted for convenience.
	Again, by \eqref{e:translation}, \eqref{eq:extra_condition} and \eqref{e:equiv_rep}, the 
	topology so defined does not depend on the particular representation used to define 
	$(\breve{S}_c(\Lambda_k,\V),\rho)$. 
	It follows that $\mathcal{J}$ is an isometry and both $(S_c(\Lambda_k,\V),\rho)$ and 
	$(\breve{S}_c(\Lambda_k,\V),\rho)$ are separable and complete.  
	Thus, weak convergence of probability measures on 
	these spaces can be defined in the usual manner
	\citep[see, e.g.,][]{billingsley:1999}. 
	Specifically, by Prokhorov's theorem, convergence in distribution in $(S_c(\Lambda_k,\V),\rho)$ or equivalently 
	$(\breve{S}_c(\Lambda_k,\V),\rho)$ is equivalent to the convergence of the finite-dimensional distributions and tightness. 
	The following result provides a general criterion \citep[see also Proposition 2.1 in][]{falconer:2002}.
		
	\begin{proposition}\label{l:tightness}
		Let $X_n, X$ be random elements in
		$(S_c(\Lambda_k,\V),\rho)$ and let $\breve X_n = \mathcal{J}(X_n)$ and 
		$\breve X = \mathcal{J}(X)$. Then $X_n\cid X$ in $(S_c(\Lambda_k,\V),\rho)$, 
		or equivalently $\breve X_n\cid \breve X$ in 
		$(\breve{S}_c(\Lambda_k,\V),\rho)$,  if and only if the following two conditions hold:
		\begin{enumerate}
			\item For all $m>0$ and $s_1,\ldots,s_m\in\bbR^d$, 
			\begin{align}\label{l:tightness-fdd}
			(\breve X_n(s_1), \ldots, \breve X_n(s_m)) \cid (\breve X(s_1), \ldots, \breve X(s_m))
			\end{align}
			\item\label{item:condition2} For every compact set $K\subset \mathbb{R}^d$, $\breve X_n$ 
			is strongly stochastically equicontinuous on $K$, namely, for all $\eta,\epsilon>0$, there exists 
			$\delta>0$ such that 
			\begin{align*}
				\limsup_{n\to \infty}\mathbb{P}\left(\sup_{\substack{\|s-t\|<\delta,\ s,t\in K}}\dist(\breve X_n(s) , \breve X_n(t))>\eta\right)<\epsilon.
			\end{align*}
		\end{enumerate}
	\end{proposition}
	
	\begin{proof} Since $(\breve{S}_c(\Lambda_k,\V),\rho)$ is separable and complete, the result is a direct consequence 
	of Theorem 14.5 and Proposition 14.6 in \cite{kallenberg:1997}. 
	\end{proof}

The convergence of the finite-dimensional distributions \eqref{l:tightness-fdd} is often easier to establish, while the challenge 
is to prove tightness.  The following result provides a simple sufficient condition, which also implies the H\"older continuity of the
limit. It is a restatement of Corollary 14.9 in \cite{kallenberg:1997}.

\begin{proposition} \label{p:tightness-via-moments}  Suppose that $\breve X_n, n\in\N$ take values in $(\breve{S}_c(\Lambda_k,\V),\rho)$ 
and let the sequence of random variables $\{\breve X_n(s_0),\ n\in\N\}$ be tight, for some $s_0\in\R^d$. 

\begin{enumerate}
\item
If for some $p>0$ and $\alpha>0$,
and all $M>0$, there exist $C_M<\infty$, such that 
\begin{align*}
 \sup_{n\in\N} \E  [\dist(\breve X_n(s),\breve X_n(t))^p] \le C_M \|s-t\|^{d+\alpha},\ \ \mbox{ for all }\|s\|,\|t\| \le M,\ s,t\in\R^d,\
\end{align*}
then the laws of the processes $\{\breve X_n, n\in\N\}$ are tight in $(\breve{S}_c(\Lambda_k,\V),\rho)$.  \\

\item If, moreover, $\breve X_n \stackrel{d}{\to} \breve X$, in $(\breve{S}_c(\Lambda_k,\V),\rho)$, then with probability one, 
$\breve X$ has $\gamma$-H\"older continuous paths, for all $\gamma \in (0,\alpha/p)$.  That is,  there is an event $\Omega^*$, 
with $\P(\Omega^*) =1$, such that for all $M>0$, we have
\begin{align*}
\dist (\breve X(s,\omega), \breve X(t,\omega)) \le C_M(\omega) \|s-t\|^{\gamma},\ \ \mbox{ for all }\|s\|, \|t\| \le M,\ \omega\in \Omega^*,
\end{align*}
with some $C_M(\omega)<\infty$.
\end{enumerate}
\end{proposition} 

\begin{remark} By taking $\breve X_n \equiv \breve X$ in Proposition \ref{p:tightness-via-moments}, we recover an 
extension of the well-known Kolmogorov-Chentsov path-regularity criterion.  See also Theorem 2.23 in \cite{kallenberg:1997}.
\end{remark}

 \begin{remark}\label{rem:tightness_at-a-point} If $k\ge 0$, recall that by \eqref{e:tilde_f2}, we have $\breve X_n(t_i) = 0,\ i=1,\cdots,M_k$. Thus,
in Proposition \ref{p:tightness-via-moments}, one can trivially take $s_0=t_1$ and the required tightness of the random variables $\{\breve X_n(s_0),\ n\in\N\}$ is immediate.  
This condition is non-trivial only when $k=-1$ (the case of stationary processes) where by convention ${\cal J}$ is the identity and $\breve{S}_c(\Lambda_{-1},\V) \equiv {S}_c(\Lambda_{-1},\V)$. 
\end{remark}

The above moment-based criterion is used in Section \ref{sec:related-work-and-examples} to furnish examples of tangent processes.

\subsection{Scaling actions on ${\mathbb V}$.}

 When considering limit theorems for ${\mathbb V}$-valued processes, one may need to rescale the process using an operator different from 
 the usual scalar multiplication. This is particularly relevant for the case where ${\mathbb V}$ is an infinite dimensional space of functions.
 The next definition introduces the natural conditions that such rescaling operators should possess.  It is similar to the one considered 
 in \cite{hult:lindskog:2006} in their abstract treatment of regular variation.

	\begin{definition}\label{def:scaling_action} 
	A family of (possibly non-linear) operators $T_a: \V \to \V$, 
	indexed by the multiplicative group $\mathbb R_+:= (0,\infty)$ is said to be a 
	scaling action on $\V$ if the following conditions hold:
		\begin{enumerate}
		\item for all $a_1>0$ and $a_2>0$, we have $T_{a_1} \circ T_{a_2} = T_{a_1 a_2}$,
		\vskip.2cm
		\item $T_1$ is the identity, and $T_a(0) = 0$ for all $a>0$, 
		\vskip.2cm
		\item $\{T_a\}$ is continuous, i.e., $\dist(T_{a_n}(x_n), T_{a} (x)) \to 0$, whenever $a_n\to a>0$ 
		and $\dist(x_n,x)\to 0$, 
		\vskip.2cm
		\item $\{T_a\}$ is radially monotone, i.e., $\dist(T_{a_1}(x), 0) < \dist(T_{a_2}(x),0)$, for all 
		$0\not=x\in \V$ and $0<a_1<a_2$, and
		\vskip.2cm
		\item $\dist(T_a(x), 0)\to 0$ as $a\downarrow 0$, for all $x\in \V$.
		\end{enumerate}
	\end{definition} 
	
	The above definition readily implies that $T_a, a>0$ are bijections and in particular
	$T_a(x)\not =0$ for all $x\not=0$.

	\begin{remark} \label{rem:property-v-continuity}
	Property (v) in Definition \ref{def:scaling_action} can be replaced by the equivalent 
	condition of	
	\begin{align}\label{e:def:scaling_action-v}
	\bigcup_{n=1}^\infty T_{n}(B_r) = \V \hbox{ for all $r\in\bbR_+$},
	\end{align}
	where $B_r := \{x\in \V\, :\, \dist(x,0)<r\}$ is the open ball centered at the origin with radius $r$.
	To see the equivalence, first assume that (\ref{e:def:scaling_action-v}) holds and, by (iv),  
	verifying (v) then amounts to showing that $\dist(T_{1/n}(x),0) \to 0$ for all $x\in \V$. 
	For every $r:=\epsilon>0$, however, \eqref{e:def:scaling_action-v} entails that $x\in T_{n}(B_\epsilon)$ for all 
	sufficiently large $n$. By (i) and (ii), this implies that $T_{1/n}(x)\in B_\epsilon$, or
	$\dist(T_{1/n}(x),0)<\epsilon$, for all sufficiently large $n$.  
	The converse argument showing (v) implies 
	(\ref{e:def:scaling_action-v}) is similar.
	\end{remark}
	
	{\clb \begin{remark}
	 Definition \ref{def:scaling_action} does not require the space $\mathbb V$ to be linear.  The linearity of $\mathbb V$, however, 
	 is needed in the definition of the path-spaces $S(\Lambda_k,\mathbb V)$ and $S_c(\Lambda_k,\mathbb V)$.
	 \end{remark}
	}
	
	Many common examples of scaling operations readily satisfy the above conditions.  
	For instance, if $\V$ is a linear normed space, a natural scaling action is scalar multiplication itself:
	\begin{align} \label{e:scalar_scaling}
	T_a(x) := a \cdot x.
	\end{align}
	More generally, the scalar multiplication is a scaling action if the underlying metric 
	is homogeneous, e.g., $\dist(a\cdot x, a \cdot y) = a^\gamma \dist(x,y),\ \gamma>0$.  
	
	Observe that since $T_1 = \I$ by property (ii), where $\I$ stands for the identity operator, 
	we have $T_a^{-1} = T_{1/a},\ a>0$ {\clb by property (i)}. Hence the 
	mappings $T_a$ are homeomorphisms and map open (closed) sets to open (closed) sets.  The radial 
	monotonicity property {\clb (iv)} implies that $T_{a_2}^{-1}(B_r) \subset T_{a_1}^{-1}(B_r)$, for all $0<a_1<a_2$.  This,
	since $T_{a}^{-1} = T_{1/a}$, entails
	\begin{align}\label{e:action-nested-balls}
	 T_{a_1} (B_r) \subset T_{a_2} (B_r),\quad \mbox{ for all } 0<a_1<a_2.
	 \end{align}
	
	Note that the metric $\dist$ and the action need not be ``compatible'', 
	that is, $\dist(T_a (x), 0)$ is in general not equal to $a \cdot \dist(x,0)$ and therefore, 
	$T_a (B_r)$ is in general not $B_{ar}$.

	\begin{remark}\label{rem:operator-scaling-actions}
	In the case where $\V$ is a Hilbert space and ${\cal H}:\V\to \V$ is a fixed bounded linear operator, one can consider the 
	action $T_c(x):= c^{\cal H} x,\ x\in \V$ for $c>0$, where $c^{\cal H}:= e^{\log(c){\cal H}}$ (see \eqref{e:exp}).  
	Lemma \ref{l:operator-scaling} shows that $c^{\cal H}$ is a scaling action in the sense of Definition \ref{def:scaling_action},
	under certain natural conditions on the operator ${\cal H}$.
	
	Limit theorems under linear operator scaling on $\V:=\R^m$ have been studied extensively in the literature 
	\citep[see e.g.][and the references therein and thereof]{meerschaert:scheffler:2001book}. 
	Such actions for a general separable Hilbert space $\V$ will be considered in Section \ref{s:operator-ss}.
	\end{remark}
		
	Given a scaling action $\{T_a,\ a\in \R_+\}$ on ${\mathbb V}$, it is natural to consider its 
	coordinate-wise extension 
	on the space of ${\mathbb V}$-valued functions ${S}( \Lambda_k, \V)$. Namely, the action 
	$\wt T_a : {S}_c( \Lambda_k, \V) \to {S}_c( \Lambda_k, \V)$ is defined such that for all 
	$f \in {S}_c( \Lambda_k, \V)$, and any $\lambda \in \Lambda_k$
	\begin{align}\label{e:T-a-widetilde}
	\wt T_a (f)(\lambda) = T_a(f(\lambda)).
	\end{align}
	
	The following result shows that the coordinate-wise action is in fact a scaling action on $S_c( \Lambda_k, \V)$.
	Its proof is given in Section~\ref{sec:proofs_sec:section2}, below.  
	
	\begin{lemma}\label{l:T-a-widetilde}
	For any scaling action $\{T_a,\ a\in \bbR_+\}$ on $(\V,\dist)$, the coordinate-wise action $\{\widetilde{T}_a ,\ a\in \bbR_+\}$ 
	in \eqref{e:T-a-widetilde} is a scaling action on the linear space ${S}_c( \Lambda_k, \V)$ equipped with the 
	metric $\rho$ in \eqref{eq:metric}.
       \end{lemma}

In view of Lemma \ref{l:T-a-widetilde}, from now on we will use the same notation 
$\{T_a\}$ for the scaling action on ${\mathbb V}$ 
and its coordinate-wise extensions on $S_c( \Lambda_k, \V)$.

\section{Tangent fields and their properties.} \label{s:tangent_fields}
	
	Throughout this section, suppose that $X=\{X(\lambda)$, $\lambda\in\Lambda_k\}$ is a random 
	element in $(S_c(\Lambda_k,\V),\rho)$. {\clb That is, for some probability space $(\Omega,{\cal F},\mathbb P)$, we have
	that $X:\Omega\to S_c(\Lambda_k,\V)$ is an ${\cal F}|{\cal B}_{(S_c(\Lambda_k,\V),\rho)}$-measurable map, where ${\cal B}_{(S_c(\Lambda_k,\V),\rho)}$ stands for 
	the Borel $\sigma$-field on $S_c(\Lambda_k)$.  }	
	For $s\in\bbR^d, \lambda\in\Lambda_k$ and $r>0$, define
	\begin{align}\label{e:X_s_lambda}
	X(s,r\cdot \lambda)=X(s + r\cdot \lambda),
	\end{align}
	where $s+r\cdot\lambda$ is as defined in \eqref{e:sca_trans}.
	With some abuse of notation, an example of $X(\lambda)$ is $X(\lambda) := \int_{\bbR^d} X(u)\lambda(du)$
	for some random field $X=\{X(s),\ s\in\R^d\}$ with continuous sample paths.
	{\clb (Note that if $\breve X$ is the representation of $X$
	as defined in Section \ref{s:2.1}, then $X(s,r\cdot \lambda) = \int_{\R^d} \breve X(s+ru) \lambda(du) =   \int_{\bbR^d} X(s+ru)\lambda(du)$, 
	since $\lambda\in \Lambda_k$ annihilates all polynomials of degree up to $k$.)
	
	One can interpret $X(s,r\cdot \lambda)$ as a generalized $(1+k)$-th order 
	increment of $X$ at location $s$ and scale $r>0$, relative to $\lambda$. Indeed, consider for example the case $d=1$ and let $\lambda(du) 
	= \sum_{j=0}^{k+1} {k+1 \choose j} (-1)^j \delta_{\{j\}}(du)$. Then, for $s\in \R$ and $r>0$,
	$$
	X(s,r\cdot \lambda)  =  \sum_{j=0}^{k+1} {k+1 \choose j} (-1)^j X(s+r j) \equiv \Delta_r^{k+1}X(s),
	$$
	where $\Delta_r X(s):= X(s+r) - X(s)$ and $\Delta_r^{k+1}X(s):= \Delta_r ( \Delta_r^k X(\cdot) )(s),\ k=0,1,\cdots$ is the usual $(1+k)$-th order regular difference operator.
	
	Thus, considering the process $X$ as a function of $(s+r\cdot \lambda)$, for all (any) signed measures $\lambda\in \Lambda_k$ effectively amounts to zooming in on its 
	$(1+k)$th order increments at location $s$ and scale $r>0$.  By letting $r\downarrow 0$, one can examine the local behavior of $X$ and arrive at (generalized) tangent fields as detailed next.
         This indeed extends the setting of \cite{falconer:2002} who focused on $\V=\bbR, k=0$, and studied the increment 
	 process $X(s+rt) - X(s) \equiv X(s+r\cdot t\lambda),\ t\in \R$.}
	 		
	One of the goals of the paper is to study the asymptotic behavior of the 
	generalized increments $X(s,r\cdot\lambda)$ as $r\downarrow 0$
	for fixed $s\in\bbR^d$, where $r\cdot\lambda$ is as in (\ref{e:sca_trans}).
	The normalization of the asymptotics will be facilitated by scaling actions $T_s=\{T_{s,a},\ a>0\}$
	as described by the next definition.  {\clb In this context, we use that the process
	$\left\{T_{s,c(s,r)}(X(s, r\cdot\lambda)),\ \lambda\in\Lambda_k \right\}$ is a random element in $S_c(\Lambda_k,\V)$ for all $s, r$,
        which is easy to verify by \eqref{e:X_s_lambda} and the continuity of the scaling action. This remark applies to similar situations below and will 
        not be repeated.}

\begin{definition}\label{def:tangent_field} 
{\clb Let $s\in\bbR^d$. A random process $Y_s = \{Y_s(\lambda),\ \lambda\in\Lambda_k\}\in
S_c(\Lambda_k,\V)$ is said to be a $k$-th order tangent field (or tangent process) to $X$ at $s$ based on the scaling action $T_s = \{T_{s,a}, a > 0\}$, if 
it is non-zero and} for some normalizing function $c(s,r)>0$, we have
	\begin{align}
	\label{assumption}
	\left\{T_{s,c(s,r)}(X(s, r\cdot\lambda)),\ \lambda\in\Lambda_k \right\}\Cid\{ Y_s(\lambda),\ \lambda\in\Lambda_k\},
	\ \ \mbox{ as } r\downarrow 0,
	\end{align} 
	where the convergence in distribution takes place in $(S_c(\Lambda_k, \V),\rho)$.
\end{definition}

{\clb The role of the function $c(s,r)$ is to provide flexibility in the choice of normalization without having to change the scaling action.  For example, 
in the special setting $\V=\bbR, k=0$ with the simple scalar scaling action $T_{s,a}(x) \equiv T_a(x)=a\cdot x$, and 
$\lambda$ replaced by $\lambda_t(dx) := \delta_{\{t\}}(dx)-\delta_{\{0\}}(dx)$,  Relation \eqref{assumption} implies 
$$
\{ c(s,r) (X(s + r t) - X(s)),\ t\in\R\} \Cid \{Y_s(\lambda_t),\ t\in\R\},\ \ \mbox{ as } r\downarrow 0.
$$
This recovers the classic setting, of tangent processes, where $c(s,r)$ plays the role of a normalizing constant.} 
In this case, \cite{falconer:2002} showed that tangent fields must be 
self-similar and have stationary increments. In the following two subsections, 
we extend Falconer's results to the general setting of this paper. 
%Unless otherwise stated, we let $\{T_a,\ a\in \R_+\}$ be some fixed general
%scaling action on ${\mathbb V}$.

\subsection{Tangent fields are self-similar.}
%
 %\tcr{Say something about self-similarity and perhaps quote Lamperti.} 
 %\fbox{Here is an attempt.  Please edit at will!}
 %
Self-similarity is a distributional invariance phenomenon, which is ubiquitous in
the study of stochastic process limit theory.  Recall that a real-valued stochastic 
process $\xi = \{\xi(t),\ t\in \R^d\}$ is said to be self-similar with 
self-similarity exponent {\clb $H >0$}, if for all $r>0$, we have
$\{\xi(r t),\ t\in \R^d\} \stackrel{fdd}{=} \{ r^{H} \xi(t),\ t\in\R^d\},$
{\clb where $\stackrel{fdd}{=}$ means equality of all finite-dimensional distributions.}
For ${\mathbb V}$-valued processes, we have the following natural extension of the notion of self-similarity.

\begin{definition} \label{def:self-similar}
An ${\mathbb V}$-valued stochastic process $\xi=\{\xi(\lambda),\lambda\in\Lambda_k\}$ is said to be
self-similar relative to the scaling action $\{T_a\}$ if for some $\alpha\in \R$, we have 
\begin{align}\label{e:def:self-similar}
\{ \xi(r\cdot\lambda),\ \lambda\in\Lambda_k\} {\clb \stackrel{fdd}{=}}
\{ T_{r^\alpha} (\xi (\lambda)),\ \lambda\in\Lambda_k\},\quad \mbox{ for all $r>0$.}
\end{align}
\end{definition}

{\clb \begin{remark}\label{rem:measurability} The above definition views $\xi$ in the wide sense as a measurable map $\xi:\Omega \to \V^{\Lambda_k}$,
where $\V^{\Lambda_k}$ is equipped with the product $\sigma$-field ${\cal B}_{\V^{\Lambda_k}}$ generated by the class ${\cal C}$ all finite-dimensional cylinder sets 
$C:=\{ x \in \V^{\Lambda_k}\, :\, x(\lambda_i) \in B_i,\ i=1,\cdots,m \},\ B_i\in {\cal B}_{(\mathbb V,\dist)},\ \lambda_i \in \Lambda_k, i=1,\cdots,m,\ m\in\N$.  When the paths of
$\xi$ are continuous, i.e., $\xi$ is a random element in $(S_c(\Lambda_k,\mathbb V),\rho)$, it can be shown 
that the equality of the finite-dimensional distributions ``$\stackrel{fdd}{=}$'' in 
\eqref{e:def:self-similar} is equivalent to equality in distribution ``$\stackrel{d}{=}$'' between $S_c(\Lambda_k,\mathbb V)$-valued 
random elements.  Indeed, firstly, all finite-dimensional projections $\pi_{\lambda_1,\cdots,\lambda_m}: S_c(\Lambda_k,\V) \to \V^{m},\ m\in\N$ for $\lambda_i\in \Lambda_k,\ i=1,\cdots,m$ 
are continuous and hence measurable.  This shows that ${\cal B}_{\V^{\Lambda_k}}\subset {\cal B}_{(S_c(\Lambda_k,\mathbb V),\rho)}$ and hence 
``$\stackrel{d}{=}$''  implies ``$\stackrel{fdd}{=}$''. On the other hand, the fact that $(S_c(\Lambda_k,\V),\rho)$ is second-countable, entails that ${\cal B}_{(S_c(\Lambda_k,\V),\rho)}$ is generated 
by the class of all closed balls, for example.  Since each such ball is a countable intersection of cylinder sets \cite[e.g., as in the proof of Proposition 12.2.2 in ][]{dudley:1989} it follows
that the Borel $\sigma$-field ${\cal B}_{(S_c(\Lambda_k,\V),\rho)}$ is determined by the $\pi$-system ${\cal C}$ of all finite-dimensional cylinder sets. Thus, appealing to the 
$\pi$-$\lambda$ Theorem, we see that  ``$\stackrel{fdd}{=}$''  implies also ``$\stackrel{d}{=}$''. 
 \end{remark}
}

The seminal work of \cite{lamperti:1962} shows that  all non-trivial large-scale limits of stochastically 
continuous processes are self-similar.  From this perspective, it is expected that tangent fields (as small-scale limits) 
be self-similar. \cite{falconer:2002,falconer:2003} has shown that this is indeed the 
case for $k=0$.  The next result addresses the general case of $k$-th order tangent fields of 
$\V$-valued processes.
 
\begin{theorem}\label{pro:self_similar}

	Assume that, for some location $s\in\bbR^d$, 
	$\{Y_s(\lambda), \lambda\in\Lambda_k\}$ is a $k$-th order tangent field to $X$ at $s$ with respect to the scaling
	action $T_s$. That is, Relation \eqref{assumption} holds for some $c(s,r)$. 
	
	\begin{itemize} 
	\item [(i)] Then, for all $r>0$, we have
	\begin{align}\label{e:pro:self_similar}
	\{Y_s(r\cdot\lambda),\lambda\in \Lambda_k \}\overset{d}{=}\{T_{r^{\alpha(s)}}Y_s(\lambda),
	\lambda\in \Lambda_k \},
	\end{align}
	where $\alpha(s)>0$ is some positive constant. We have, moreover, that
	\begin{align}\label{e:c(z,r)_rep}
	c(s, r) = r^{-\alpha(s)}\ell_s(r),
	\end{align}
	where $\ell_s(r)$ is a slowly varying function at $0$, i.e., for every fixed $h>0$, $\ell_s(h r)/\ell_s(r) \to 1,\ r\downarrow 0$. 
	
        \item [(ii)]	{\clb The tangent process is unique up to rescaling.  That is, if \eqref{assumption} also holds with $c(s,r)$ and $Y_s=\{Y_s(\lambda)\}$ 
        replaced by $\wt c(s,r)$ and $\wt Y_s = \{\wt Y_s(\lambda)\}$, respectively, then we have
        \begin{equation}\label{e:tangent-uniqueness}
        \lim_{r\downarrow 0} \frac{\wt c(s,r)}{c(s,r)} = a \in (0,\infty)\quad \mbox{ and }\quad \wt Y_s \stackrel{d}{=} T_{s,a} (Y_s).
        \end{equation}}
        \end{itemize}
\end{theorem}

%\begin{remark} Relation \eqref{e:pro:self_similar} shows that the tangent fields are necessarily self-similar relative to the underlying scaling action and have a positive self-similarity exponent $\H(s)$.  At a fixed $s$, the self-similarity exponent can always be set to $1$ if one considers a new scaling action $a\mapsto T_{a^{\H(s)}}$.  To standardize $H(s)$ this way, one wold have to modify the scaling action for each $s$, which is inconvenient.  We prefer  to fix a scaling action and let the exponent $\H(s)$ depend on the location $s\in \R^d$. \end{remark}

\begin{remark} Relations \eqref{e:pro:self_similar} and \eqref{e:c(z,r)_rep} show that the normalization used to 
define a tangent field may differ from the scaling action that characterizes the self-similarity of the tangent field by a slowly varying factor,
which cannot be dropped in general.  This is akin to the fundamental role of slowly varying functions in the normalization of the partial sums in the non-Gaussian Central Limit Theorem.
\end{remark}

\begin{proof2}{Theorem}{\ref{pro:self_similar}}
%We now return to the proof of Theorem \ref{pro:self_similar}. 
	{\clb {\em Proof of part (i).}} For all fixed $h>0$, by \eqref{assumption}, as $r\downarrow 0$, we have
	\begin{align*}
	  \widetilde \xi_r:= \left\{T_{s,c(s,r)}(X(s,(hr)\cdot\lambda)), \lambda \in \Lambda_k \right\}\Cid 
	  \widetilde \xi := \left\{Y_s(h\cdot\lambda), \lambda \in \Lambda_k\right\}. 
	\end{align*}
	On the other hand, as $r\downarrow 0$,
	\begin{align*}
	 \xi_r:= \left\{T_{s,c(s,hr)}(X(s,(hr)\cdot\lambda)), \lambda \in \Lambda_k\right\}
	 \Cid \xi := \left\{Y_s(\lambda), \lambda \in \Lambda_k\right\}.
	\end{align*}
	By assumption both $\xi$ and $\widetilde \xi$ are  non-zero.  Observe that
	\begin{align*}
	T_{s,c(s,r)}(X(s,(hr)\cdot\lambda))=T_{s,\frac{c(s,r)}{c(s,hr)}}\circ T_{s,c(s,hr)}(X(s,(hr)\cdot\lambda)) 
	\end{align*}
	and hence
	\begin{align*}
	T_{s,\frac{c(s,r)}{c(s,hr)}} (\xi_r) = \widetilde \xi_r \cid \widetilde \xi.
	\end{align*}
	Applying Lemma \ref{l:Slutsky} (with $\mathbb X := S_c(\Lambda_k,\V)$ -- recall Lemma \ref{l:T-a-widetilde}) gives
	\begin{align}\label{eq:coef_relation}
	\frac{c(s,r)}{c(s,hr)}\to a(s,h),\ \ \mbox{ as } r\downarrow 0, 
	\end{align}
	for some positive $a(s,h)>0$.  We have, moreover, $\widetilde \xi \eqid T_{s,a(s,h)}(\xi)$, which reads 
	\begin{align}\label{e:pro:self_similar-2}
	\left\{Y_s(h\cdot\lambda), \lambda\in\Lambda_k\right\}\eqid\left\{T_{s,a(s,h)}(Y_s(\lambda)), \lambda\in\Lambda_k\right\}.
	\end{align}
	
	We will next show that $a(s,h) = h^{\alpha(s)}$, for some $\alpha(s)>0$. First, Relation \eqref{e:pro:self_similar-2} readily implies that  for all
	$ h_1>0$ and $h_2>0$
	\begin{align}\label{rel1}
	a(s,h_1h_2)=a(s,h_1)a(s,h_2). 
	\end{align}
	Indeed, by \eqref{e:pro:self_similar-2},
	\begin{align*}
	\{T_{s,a(s,h_1h_2)}(Y_s(\lambda))\}&\overset{d}{=}\{Y_s((h_1h_2)\cdot\lambda)\} 
	\overset{d}{=}\{T_{s,a(s,h_1)}(Y_s(h_2\cdot\lambda))\}\\
	& \overset{d}{=}\{T_{s,a(s,h_1)}\circ T_{s,a(s,h_2)}(Y_s(\lambda))\} = \{T_{s,a(s,h_1)a(s, h_2)}(Y_s(\lambda))\}.
	\end{align*}
	Since $Y_s$ is nonzero, the last relation implies \eqref{rel1} by (i) of Lemma \ref{l:Slutsky}.
	
	The function $a(s,h)$ is also continuous in $h \in (0,\infty)$.  Indeed, for any sequence 
	$h_n\to h$, $h_n,h \in (0,\infty)$, by Lemma~\ref{l:continuity} {\clb (applied with $X_n:=Y_s$, $v_n:=0$, and $r_n:=h_n$)}, we have 
	$\{Y_s(h_n\cdot\lambda)\}\overset{d}{\to} \{Y_s(h\cdot\lambda)\}$.  
	Therefore, by \eqref{e:pro:self_similar-2},
	\begin{align}\label{e:pro:self_similar-3}
	\{T_{s,a(s,h_n)}(Y_s(\lambda))\}\overset{d}{=}\{Y_s(h_n\cdot\lambda)\}\Cid \{Y_s(h\cdot\lambda)\}\overset{d}{=}\{T_{s,a(s,h)}(Y_s(\lambda))\}.
	\end{align}
	Since $Y_s$ is nonzero, applying (ii) of Lemma~\ref{l:Slutsky}, we obtain
	\begin{align}\label{rel2}
	a(s,h_n)\to a(s,h),
	\end{align}
	which shows the desired continuity.  
	
	Combining \eqref{rel1}, \eqref{rel2}, the continuity of $a(s,\cdot)$ and the
	fact that, trivially, $a(s,1)=1$, it is straightforward to conclude that $a(s,h)=h^{\alpha(s)},\ h>0$, 
	for some $\alpha(s)\in (-\infty,\infty)$, which is a special example of Cauchy's 
	functional equation \citep[cf. Theorem 5.2.1 of][]{Kuczma:2009cc}.
	
	We will show next that $a(s,h_n)\to 0$ as $h_n\downarrow 0$, which necessarily implies $\alpha(s)>0$.  Indeed, with $h=0$,
	\eqref{e:pro:self_similar-3} implies that $\{T_{s,a(s,h_n)}(Y_s(\lambda))\}\stackrel{d}{\to} 0 =\{Y_s(0\cdot \lambda)\}$, as $n\to\infty$. 
	This, by (iii) of Lemma \ref{l:Slutsky}, yields $a(s,h_n)\to 0$.
	
	To conclude the proof of part (i), letting $c(s, r) =: r^{-\alpha(s)}\ell_s (r)$, we see from Equation \eqref{eq:coef_relation} that, for all $h>0$,
	$\ell_{s}(r)/\ell_{s}(hr) \to 1$, as $r\downarrow 0$, which shows $\ell_s$ is a slowly varying function at $0$.
	
	\noindent {\clb {\em Proof of part (ii).} Assume now that in addition to \eqref{assumption}, we have
	$$
	\wt \eta_r:= \{ T_{s,\wt c(s,r)} (X(s,r\cdot\lambda)),\ \lambda \in\Lambda_k\} \stackrel{d}{\longrightarrow} \wt Y_s = \{ \wt Y_s(\lambda),\ \lambda \in \Lambda_k\},
	$$
	as $r\downarrow 0$.	By the properties of the scaling action, with $\eta_r:= \{T_{s,c(s,r)} (X(s,r\cdot\lambda)),\ \lambda \in\Lambda_k\}$, we have
	$$
	\wt \eta_r = T_{s,\frac{\wt c(s,r)}{c(s,r)}} (\eta_r)\stackrel{d}{\longrightarrow} \wt Y_s,\ \ \mbox{ as } r\downarrow 0.
	$$
	On the other hand, Relation \eqref{assumption} reads $\eta_r \stackrel{d}{\to} Y_s$, as $r\downarrow 0$.  Since both limits $\wt Y_s$ and $Y_s$ are non-zero, Lemma \ref{l:Slutsky}
	entails $\wt c(s,r)/ c(s,r) \to a>0$ and $T_{s,a}(Y_s) \stackrel{d}{=}\wt Y_s$, which proves  \eqref{e:tangent-uniqueness}, i.e., the essential uniqueness of the tangent process.
	} 
\end{proof2}

\subsection{Tangent fields are intrinsically stationary.}

One of the key results in \cite{falconer:2002} is that (almost all) tangent fields have 
stationary increments (cf.\! Theorem 3.6 therein). The proof of that is based 
on a delicate measure-theoretic argument. Below, we show that this phenomenon extends 
to higher order tangent fields to processes taking values in a linear separable metric space $\V$.

\medskip
Let \eqref{assumption} hold and let
\begin{align}\label{e:Fn}
F_{n}(s) :={\rm Law \ of} \Big\{ T_{s,c(s,1/n)} X\left(s , (1/n) \cdot\lambda \right),\ 
\lambda\in \Lambda_k\Big\}
\end{align}
be the probability distribution of the rescaled version of 
$\{X(s,\lambda),\lambda\in\Lambda_k\}$ in $(S_c(\Lambda_k,\V),\rho)$. Similarly, let
\begin{align}\label{e:G}
G(s):= {\rm Law \ of } \left\{ Y_s(\lambda),\ \lambda\in \Lambda_k\right\}.
\end{align}
In this notation, the convergence in \eqref{assumption}  (with $r:=1/n$) is simply
\begin{align}\label{eq:weak_convergence}
F_{n}(s) \stackrel{w}{\longrightarrow} G(s),\ \ n\to\infty,
\end{align}
{\clb where `$\stackrel{w}{\to}$' denotes the weak convergence of probability measures.}
An important result that will be utilized below is Proposition~\ref{P4} in the Appendix.  
In that regard, we first equip the space ${\cal P}(S_c(\Lambda_k,\V),\rho)$ of probability 
measures on $(S_c(\Lambda_k, \V),\rho)$ with a separable 
metric that metrizes the weak convergence \eqref{eq:weak_convergence}. 
Since $(S_c(\Lambda_k,\V),\rho)$ is complete and separable,
a suitable metric is $d_{\rm LP}$, the L\'evy-Prokhorov distance  
\citep[cf. Theorem 6.8 of][]{billingsley:1999}.
Thus, \eqref{eq:weak_convergence} can be re-expressed as 
\begin{align} \label{e:LP}
F_{n}(s) \to G(s) \hbox{ in $({\cal P}(S_c(\Lambda_k,\V),\rho),d_{\rm LP})$},
\end{align}
namely, $d_{\rm LP}(F_{n}(s),G(s))\to 0$.

%	\fbox{ Jinqi: I believe the following proof should work and is not complex. }

\begin{definition}
A process $Y=\{Y(\lambda), \lambda\in\Lambda_k\}$ is said to be {\clb strictly} intrinsically stationary if
\begin{align*} 
\{Y(w+\lambda),\ \lambda\in\Lambda_k \}\overset{fdd}{=}\{Y(\lambda),\ \lambda\in\Lambda_k\},\quad \mbox{ for all } w\in\R^d.
\end{align*}
\end{definition}
Note that this is different from the usual notion of {\clb weak or second-order} intrinsic stationarity
in the literature \citep[cf.][]{Sasvari:2009}. The latter is the topic of Section \ref{sec:c2s5}. 

{\clb\begin{remark} Observe that the notion of strict intrinsic stationarity like that of self-similarity in \eqref{e:def:self-similar}
is stated in greater generality using equality in the sense of finite-dimensional distributions.  As discussed in Remark \ref{rem:measurability}, 
when the processes therein take values in the path space $S_c(\R^d,\V)$, the equality of the finite-dimensional distributions is equivalent to that of the 
probability distributions of the processes.  In the next result, the processes are understood as random elements in $S_c(\R^d,\V)$.
\end{remark} }

\begin{theorem}\label{IRF}
	Let $B$ be a Borel set of $\R^d$. Assume that $X=\{X(\lambda),\ \lambda\in \Lambda_k\}$ {\clb is a random element in $S_c(\R^d,\V)$ and it}
	has a $k$-th order tangent field $Y_s=\{Y_s(\lambda), \lambda\in\Lambda_k\}$ at every $s\in B$ {\clb in the sense of Definition \ref{def:tangent_field}.}
	Also assume that for any $r$, the normalization $c(s,r)$ is Borel measurable in $s$,
	and for any $s$ and any sequence $w_n \to w\in\bbR^d$, 
	\begin{align} \label{e:T_regularity}
	T_{s_n, c(s_n,1/n)} \circ T_{s,c(s,1/n)}^{-1} \to \I,
	\end{align} 
	where $\I$ is the identity operator and $s_n:=s+n^{-1}w_n$.
	Then, there exists a set $U$ with zero Lebesgue measure 
	such that for all $s\in B\setminus U$, the tangent field $Y_s$ is {\clb strictly} intrinsically stationary.
	That is, at almost all locations $s$, tangent fields are {\clb strictly} intrinsically stationary.
	\end{theorem}
	
 The proof of this result uses the following proposition established in Section \ref{Proposition_proof} below.
 % A.4 of the extended version of the paper \cite{shen:stoev:hsing:2020_extended}.%
 
\begin{proposition}\label{P4}
	Let $B\subset \mathbb{R}^d$ be a Borel set with finite Lebesgue measure $\text{Leb}(B)<\infty$. 
	Suppose that $F_n:B\to E$ is a sequence of Borel measurable functions 
	into the separable metric space $(E,\rho_E)$ 
	such that
	\begin{align*}
	F_n(s)\mathop{\longrightarrow}_{n\to \infty} G(s), \text{   for almost all }s\in B.
	\end{align*}
	
	Then, for every $\epsilon>0$, there exists a compact set $K_\epsilon\subset B$, such that 
	$\mathrm{Leb}(B\setminus K_\epsilon)<\epsilon$, the function $G$ being continuous on $K_\epsilon$, and 
	\begin{align}\label{e:p:EL}
	F_n(s_n) \to G(s),\mbox{ whenever  $s_n\to s$, for $s_n,s\in K_\epsilon$.}
	\end{align}
\end{proposition}

\begin{proof2}{Theorem}{\ref{IRF}}
   {\clb By the $\sigma$-additivity of the Lebesgue measure on $\R^d$, it suffices to establish the result for the case
    ${\rm Leb}(B)<\infty.$ }
	
	The assumption implies that \eqref{e:LP} holds for all $s\in B$.
	The continuity of $X$ and the Borel-measurability of $s\mapsto c(s,1/n)$ entail that $s\mapsto F_{n}(s)$ 
	 is a sequence of Borel measurable functions in $(E,\rho_E):= ({\cal P}(S_c(\Lambda_k,\V),\rho),d_{\rm LP})$. 
	  Therefore, the assumptions of Proposition~\ref{P4} are fulfilled and
	 for any $\epsilon>0$, there is a compact set  $K_{\epsilon}\subset B$ 
	with 	$\text{Leb}(B\setminus K_\epsilon)<\epsilon$ such that $F_{n}(s_n) \to G(s)$ 
	so long as $s_n,s\in K_{\epsilon}$ and $s_n\to s$.  
	
	It follows from Lebesgue's density theorem that there is a subset $K'_{\epsilon}$ of $K_{\epsilon}$ 
	on which the Lebesgue density is equal to $1$ and $\text{Leb}(K_{\epsilon}\setminus K'_{\epsilon}) = 0$. 
	By Lemma 3.5 of \cite{falconer:2002}, for any $s\in K'_{\epsilon}$ and $w\in \R^d$, there exists a sequence 
	$w_n$ such that $w_n\to w$ and $s_n:=s+w_n/n\in K_{\epsilon}$.
	Thus, for any $s\in K'_\epsilon$, $F_{n}(s_n) \to G(s)$ as $n\to\infty$ or, equivalently,
	\begin{align}\label{eq:re1}
	  \xi_n:= \left\{T_{s_n,c(s_n,1/n)}(X(s_n,(1/n)\cdot\lambda)),\lambda\in\Lambda_k\right\}
	  \Cid \xi:= \{Y_s(\lambda),\lambda\in\Lambda_k\},
	\end{align}
	
	On the other hand, we have
	\begin{align}\label{eq:re2}
	\begin{split}
	\xi_n & = \left\{T_{s_n, c(s_n,1/n)}(X(s_n,(1/n)\cdot\lambda),\ \lambda\in\Lambda_k\right\}\\
	&= \left\{T_{s_n, c(s_n,1/n)} \circ T_{s, c(s,1/n)}^{-1} \circ T_{s, c(s,1/n)}
	(X(s_n,(1/n)\cdot\lambda)),\ \lambda\in\Lambda_k\right\}\\
	&=\left\{T_{s_n, c(s_n,1/n)} \circ T_{s, c(s,1/n)}^{-1} \circ T_{s, c(s,1/n)}
	(X(s,(1/n)\cdot (w_n+\lambda))),\ \lambda\in\Lambda_k\right\}.
	\end{split}
	\end{align}
	Since by \eqref{e:T_regularity}, we have $T_{s_n, c(s_n,1/n)} \circ T_{s, c(s,1/n)}^{-1}\to \I$, Lemma \ref{l:continuity} implies
	\begin{align}\label{eq:converge2}
	\begin{split}
	\xi_n \Cid \wt \xi := \{Y_s(w+\lambda),\ \lambda\in\Lambda_k \},
	  \end{split}
	\end{align}
	which implies that $\xi \eqid \wt \xi$. 

Finally, we take $U := \cap_{k=1}^\infty (B\setminus K'_{1/n})$, which is a set with measure $0$. This concludes the proof. 
\end{proof2}

We make next an important observation that Condition \eqref{e:T_regularity} holds automatically in the case when the scaling actions 
$T_{s,c}$ can be expressed through a single scaling action independent of the location $s$.  This is the case 
in particular for the usual scalar multiplication actions \eqref{e:scalar_scaling} considered for example 
in \citep[cf.][]{falconer:2002}.

\begin{corollary}\label{cor:IRF} Assume the conditions of Theorem \ref{IRF}.  If the scaling action does not depend on location, i.e.,
$T_{s,c(s,r) } = T_{c(s,r)},\ c>0$, then Condition \eqref{e:T_regularity} always holds.
\end{corollary}

The proof is given in Section \ref{ss:Supplementary_proof_section3}.  We conclude this section with several remarks. 

\begin{remark} The null set $U$ in Theorem \ref{IRF} cannot be dropped in general.
While all tangent fields are self-similar, not all of them are intrinsically stationary.   Indeed, one can consider the
simple example $X(t) = \|t\| ^{H} Z,\ t\in \R^d$, where $Z$ is a fixed random variable and $H>0$. Consider the usual scalar multiplication action and
observe that with $s:=0$, and $\lambda = \sum_i c_i \delta_{t_i} \in \Lambda_k$, for all $r>0$, we have
\begin{align*}
X(s+r\cdot\lambda) = \sum_{i} c_i X(rt_i) = r^{H} \sum_i c_i X(t_i) = r^{H} X(s+ \lambda).
\end{align*}
That is, $X$ is its own tangent field at $s=0$ for all $k\ge 0$.  Note that $X$ is not intrinsically stationary if $H \not\in\N$.
\end{remark}

\begin{remark} \label{rm:tangent_nesting} Notice that $\Lambda_{k_2} \subset \Lambda_{k_1}$ for all $0\le k_1 < k_2$.  Therefore, 
all $k_1$-order tangent fields are also $k_2$-order tangent fields.  Specifically, if \eqref{assumption} holds with $k=k_1$, then it also 
holds with $k_2$.  
\end{remark}

\begin{remark} As in  \cite{falconer:2002}, we focus here on random fields with continuous paths.  
	One can study the structure of generalized tangent fields for processes with discontinuous paths and
	potentially extend the results in \cite{falconer:2003} which focus on the space of {\em c\`adl\'ag} functions equipped with
	the Skorokhod $J_1$-topology.  The key challenge is coming up with a suitable topology on the 
	path-space in question which is separable and complete.  Provided that this is the case, we believe that
	versions of Theorems \ref{pro:self_similar} and \ref{IRF} will continue to hold.
	\end{remark}
	
\begin{remark}\label{rem:paths} In principle, in the definition of the tangent field \eqref{assumption} one could apply a general scaling action 
on both the domain $\R^d$ of the stochastic  process as well as on its range ${\mathbb V}$.  In this case, we anticipate that an analog of 
Theorem \ref{pro:self_similar} will hold, where the limits are scale-invariant processes similar to the 
ones studied in \cite{bierme:meerschaert:scheffler:2007,didier:meerschaert:pipiras:2017}.  Here, for simplicity, we chose to 
apply a general scaling action only on the range of the process and retain the usual rescaling by scalars in the domain $\R^d$.  
\end{remark}

%\tcr{Add summary/subsection to summarize, maybe add the linear fractional Brownian stable motion, motivating the study of the second-order (spectral theory) for the Gaussian case. Say the general non-Gaussian case is future work.}
%	\begin{definition}
%		Let ${Z}$ be an IRF$_k$. A symmetric continuous function $\mathcal{K}(h)$ defined on $\R^d$ is called a generalized covariance of $Z$ if 
%		\begin{align*}
%			\mathbb{E}[Z(\lambda)Z(\mu)] = \sum_{i,j}\lambda_i\mu_j \mathcal{K}(y_j-x_i)
%		\end{align*}
%		for any  of measures $\lambda,\mu\in\Lambda_k$. 
%	\end{definition}

\section{Spectral theory for Hilbert space valued IRF$_k$'s.}
\label{sec:c2s5}

In this section, we develop the general correlation theory for stationary and intrinsically stationary processes 
taking values in a separable Hilbert space $\V$ over $\C$ equipped with the inner product $\langle\cdot,\cdot\rangle$.
In the following section, we present a generalization of the celebrated Bochner Theorem and then in Section \ref{ss:spectral_notation_1}, we 
extend the Matheron spectral characterization to the class of $\V$-valued intrinsic random functions.
The applications of these results to the characterization of Gaussian $\V$-valued stationary and intrinsically stationary processes requires us to carefully 
consider {\em both} real and complex Hilbert spaces (discussed in Section \ref{sec:real-complex}).

Throughout this paper, a random element $X$ in ${\mathbb V}$ is said to have mean zero and finite variance, together 
referred to as {\em second order}, 
if $\E[ X] = 0$ and $\E[ \|X\|^2] < \infty$, where, for definiteness, 
all expectations here are defined in the Bochner sense 
%\citep[see Section S.2.1 in][]{shen:stoev:hsing:2020_extended}. %
(see Section \ref{supp:sec:Bochner}).
A process is said to be second order if each element is second order.

Denote by $\bbT$ the collection of trace-class operators on $\V$.
That is, linear operators $\CT: \V\to \V$, with finite {\em trace norm}:
\begin{align*}
\|\CT\|_{\tr} = \sum_{j=1}^\infty \langle (\CT^*\CT)^{1/2} e_j, e_j\rangle,
\end{align*}
where $\{e_j\}$ is an arbitrary complete orthonormal system (CONS) on ${\mathbb V}$, and where $\CT^*$ denotes
the adjoint operator of $\CT$.  One can show that the trace norm does not depend on the choice of the CONS
and the space $\bbT$ equipped with the trace norm is a Banach space 
\citep[cf.][]{simon2015comprehensive}. 

Recall that $\CT$ is self-adjoint if $\CT=\CT^*$.  Also $\CT$ is positive definite 
(or just positive), denoted $\CT\ge 0$, if $\CT$ is self-adjoint and $\langle f,\CT f\rangle\ge 0$, 
for all $f\in \V$.  The class of positive and trace-class operators will be denoted by $\bbT_+$.

\subsection{The Bochner Theorem.} \label{sec:Bochner}

The aim of this subsection is to review the basic properties of second order covariance-stationary 
processes on $\bbR^d$ taking values in the separable Hilbert space $\V$ over $\C$. 
We start with the important notion of positive definiteness. 

\begin{definition} \label{def:pos-def-function} A collection of operators $\{\mathcal{K}(t),t\in\bbR^d\}$ 
on the complex Hilbert space ${\mathbb V}$ is said to be positive definite in the weak sense if for all $c_j\in \mathbb C,\ t_j\in \bbR^d,\ j=1,\cdots,n$, 
we have
\begin{align}\label{eq:psd}
\sum_{j=1}^n\sum_{j'=1}^n c_j\overline c_{j'}\mathcal{K}(t_{j}-t_{j'}) \ge 0 \hbox{ (operator positivity)}.
\end{align}
\end{definition}

The classical Bochner's Theorem \citep[cf.][]{bochner1948vorlesungen, 
khintchine1934korrelationstheorie} which connects the space of positive-definite functions with
range in $\mathbb{C}$ and finite positive measures has provided a fundamental tool for  constructing useful 
models for stationary random fields.  Below we state an extension of that for the infinite-dimensional setting.
To do so we need the notion of integration with respect to a $\mathbb{T}_+$-valued measure. 

We say that $\mu: \mathcal{B}(\bbR^d) \mapsto \mathbb{T}_+$ is a $\mathbb{T}_+$-valued measure if $\mu$
is $\sigma$-additive, where ${\cal B}(\bbR^d)$ denotes the $\sigma$-field of Borel sets in $\R^d$.  Note that {\em a fortiori}
$\mu(\emptyset) = 0$ and $\mu$ is finite in the sense that $0\le \mu(B)\le \mu(\bbR^d) \in \bbT_+,\ B\in{\cal B}(\R^d)$ as positive 
operators.  Integration of a $\mathbb{C}$-valued measurable function with respect to such $\mu$ can be defined along the line 
of Lebesgue {\clb integration}; see Section \ref{sec:appdix_integral_def}.

\begin{theorem}\label{th:operator_Bochner} Let $\{\mathcal{K}(t),t\in\bbR^d\} \subset \bbT$ 
be a positive-definite set of trace-class operators in the sense of Definition \ref{def:pos-def-function}. 
If ${\cal K}$ is continuous at $0$ in the trace norm, i.e., 
$\|{\cal K}(t) - {\cal K}(0)\|_{\rm tr}\to 0$, as $t\to 0$,
then there exists a unique finite $\mathbb{T}_+$-valued measure $\mu$ such that 
\begin{align}\label{e:th:operator_Bochner}
\mathcal{K}(t) = \int_{\R^d} e^{\ii t^\top x }\mu(dx),\quad \mbox{ for all }t\in\bbR^d.
\end{align}
Conversely, for every finite $\bbT_+$-valued measure $\mu$, Relation \eqref{e:th:operator_Bochner} yields a 
positive-definite set of trace class operators.
\end{theorem} 

We note that Theorem \ref{th:operator_Bochner} or variations of it have been mentioned in
the literature.  See, for instance, 
\cite{kallianpur1971spectral}, \cite{holmes1979mathematical}, \cite{Neeb1998}, 
\cite{durand2020spectral} and \cite{van2020note}. 
In Section \ref{suppsec:thm4.2_proof} of Supplement,
%S.4.1 of \cite{shen:stoev:hsing:2020_extended}, %~\ref{suppsec:thm4.2_proof} of the
we provide a detailed proof that uses standard arguments familiar to the readers in the statistics and probability
community.

Both \cite{Neeb1998} and \cite{van2020note} present Bochner's Theorem in terms of
the following natural but stronger version of 
positivity. 

\begin{definition} \label{def:complely-pos-def-function} A collection of operators 
$\{\mathcal{K}(t),t\in\bbR^d\}$ on ${\mathbb V}$ is said to be completely positive definite, or just positive definite, if
\begin{align}\label{eq:complete-psd}
\sum_{j=1}^n\sum_{j'=1}^n \langle f_j, \mathcal{K}(t_{j}-t_{j'}) f_{j'}\rangle \ge 0,
\end{align}
for all $f_j\in \V, t_j\in \bbR^d,\ j=1,\cdots,n$ and $n\in \bbN$. 
\end{definition}

Definition \ref{def:complely-pos-def-function} simply means that the matrices $({\cal K}(t_j-t_{j'}))_{n\times n}$
with operator $\bbT$-valued entries are self-adjoint positive definite operators on the product Hilbert space 
$\V^{n}$. For more mathematical insight into this condition, see abstract literature on Hilbert 
$C^*$-modules, e.g., \cite{murphy:1997} and \cite{pellonpaa:yilinen:2011}. 
Clearly \eqref{eq:complete-psd} implies \eqref{eq:psd}. However, observe that for every finite $\bbT_+$-valued 
measure $\mu$, Relation \eqref{e:th:operator_Bochner} defines a  completely positive definite
kernel ${\cal K}(\cdot)$. This entails the following curious result, already noted in \cite{durand2020spectral}.

\begin{corollary}\label{c:weak-psd-implies-complete-psd} Let ${\cal K}:\bbR^d\to \bbT$ be continuous at 
$0$ in $\|\cdot\|_{\rm tr}$. The collection of operators $\{{\cal K}(t),\ t\in \bbR^d\}$ is positive definite in the sense of Definition \ref{def:pos-def-function} if and only 
if it is completely positive definite in the sense of Definition \ref{def:complely-pos-def-function}.
\end{corollary}

Let now $\{X(t),t\in\bbR^d\}$ be a ${\mathbb V}$-valued, second order random field.
The {\em cross covariance operator} for $X$ is then well-defined as 
\begin{align*}
\mathcal{C}_X(s,t) := \mathbb{E}[X(s)\otimes X(t)],\ \ s,t\in\R^d,
\end{align*}
and takes values in the space of  trace-class operators $\bbT$ equipped with the trace norm 
%\citep[see e.g., Lemma S.2.2 in][]{shen:stoev:hsing:2020_extended}. %
Lemma \ref{l:cross-cov} in Supplement.
The latter expectation is understood to be defined in the sense of Bochner integral in the separable 
Banach space $(\bbT,\|\cdot\|_{\rm tr})$, and for $f,g\in \V$, the outer product operator is by definition $(f\otimes g) h := \langle h, g\rangle f$.  
Observe that $\mathcal{C}_X(t,t)$ is positive definite, and
\begin{align*}
\mathcal{C}_X(s,t) = \mathcal{C}_X^*(t,s),\mbox{ for all }s,t\in\R^d. 
\end{align*}

The process $X$ is said to be mean-square or $L^2$-continuous if  
\begin{align} \label{e:meansqcontinuity}
\E \| X(s) -X(t)\|^2 \to 0 \hbox{ as $s\to t$, for all $t\in \R^d$}. 
\end{align}
It is easy to see
(Section \ref{supp:existence-cross-cov} or Proposition \ref{supp:p:continuity-cross-cov} in Supplement) that
%\citep[e.g., Section S.2.2 or Proposition S.2.3 in][]{shen:stoev:hsing:2020_extended} that
$X$ is mean-square continuous if and only if $\|\mathcal{C}_X(s',t') - \mathcal{C}_X(t,t)\|_{\rm tr}\to 0$ as $(s',t') \to (t,t)$.
 
\begin{definition} A process $X$ is said to be weakly or covariance-stationary if it is second order 
and its cross covariance is shift invariant, i.e., 
\begin{align*}
 {\cal K}_X(h):=\mathcal{C}_X(s,s+h),\quad \mbox{ for all }h\in \bbR^d,
\end{align*}
does not depend on $s\in\bbR^d$.  The function  ${\cal K}_X(h),\ h\in\mathbb{R}^d$, will be 
referred to as the stationary covariance function of $X$.
\end{definition}

Observe that every stationary covariance function ${\cal K}_X$ is positive definite in 
the sense of \eqref{eq:complete-psd} (and hence \eqref{eq:psd}).
  Also, the $L^2$-continuity
of a stationary process is equivalent to the continuity of its stationary
covariance function at $0$. Thus, the characterization
in Theorem \ref{th:operator_Bochner} readily holds for the stationary covariance
of a stationary process that is $L^2$-continuous. 

We conclude this section with a version of the classical Cram\'er 
stochastic representation of stationary Hilbert-space-valued random fields.
Recall that, for $\V=\bbC$, the well-known integral representation \citep[cf.][]{cramer1942}
of a covariance stationary random process $X$ on $\bbR^d$
states that
\begin{align}\label{e:Cramer-stoch-int-rep}
X(t) = \int_{\bbR^d} e^{\ii t^\top x} \xi(d x), \ \ \mbox{ almost surely, }\quad t\in \R^d, 
\end{align}
where $\xi$ is a second order random measure with orthogonal increments. 
To extend this result to a general ${\mathbb V}$, we first have to define integration with respect to
a random measure with orthogonal increments in that setting. This is done in Section 
\ref{sec:appdix_integral_def} of Appendix.  {\clb Here, we only give the main ideas.

 Let $\bbL^2(\Omega)$ be the $L^2$ space of all ${\mathbb V}$-valued random 
elements $\eta$ on the probability space $(\Omega,\CF,\pr)$ with
$\E[\|\eta\|^2] < \infty$, equipped with the inner product
\begin{align*}
\langle \eta_1,\eta_2\rangle_{\Omega} := \E\langle\eta_1,\eta_2\rangle.
\end{align*}

\begin{definition}\label{def:orthogonal-measure} Let $\mu$ be a 
$\bbT_+$-valued measure on ${\cal B}(\bbR^d)$.  A second order stochastic process 
$\xi = \{\xi(A), A\in\mathcal{B}(\bbR^d)\} \subset \bbL^2(\Omega)$ indexed by the Borel sets is said to 
be a ${\mathbb V}$-valued orthogonal random measure on $\R^d$ with control measure 
$\mu$ if the following conditions hold:
\begin{enumerate}% [label=(\roman*)]
	\item [(i)] $\E [\|\xi(A_n)\|^2] \to 0$ if $A_n\to \emptyset$, 
	\vskip.2cm
	 \item  [(ii)] $\mu (A\cap B) =\E [ \xi(A) \otimes \xi(B)],$ for all $A,\ B \in {\cal B}(\bbR^d)$,  where the expectation is in the sense of 
Bochner on $(\bbT, \|\cdot\|_{\rm tr})$. 
 \end{enumerate}
 \end{definition}
 
 It is straightforward to see that (ii) implies for disjoint $A$ and $B$, that $\xi(A\cup B) = \xi(A) + \xi(B)$, almost surely, so that $\xi$ is in fact 
 an additive set-function.  This, combined with the continuity property (i), yields the $\sigma$-additivity of $\xi$ (for more details, see 
 Section \ref{sec:appdix_integral_def}).

For a simple function $f(t) = \sum_{i=1}^n c_i I_{A_i}(t)$, with $c_i\in\C$ and pairwise disjoint $A_i$'s, we naturally define
 $
 {\cal I}_\xi(f):= \int_{\R^d} f(t)\xi(dt) := \sum_{i=1}^n c_i \xi(A_i).
 $
Letting $\|\mu\|_{\rm tr}(A):= \| \mu(A)\|_{\rm tr}$ be the trace measure of $\mu$, we see that 
 $$
  \E[ \|{\cal I}_\xi(f) \|^2 ] = \sum_{i=1}^n |c_i|^2 \E \|\xi(A_i)\|^2 = \int_{\R^d} |f(t)|^2 \|\mu\|_{\rm tr}(dt).
 $$
 That is, the linear operator ${\cal I}_\xi$ is an isometry from the space of simple functions in $L^2(\R^d, \|\mu\|_{\rm tr})$ into
 the Hilbert space $\bbL^2(\Omega)$.  Thus, one can extend the definition of ${\cal I}_\xi(f)$, by continuity, to all $f\in L^2(\R^d, \|\mu\|_{\rm tr})$.
 We have moreover that, for all $f,g\in L^2(\R^d,\|\mu\|_{\rm tr})$, 
 $$
 \E [ {\cal I}_\xi(f)\otimes {\cal I}_\xi(g)] = \int_{\R^d}  f(t) \overline {g(t)} \mu(dt),
 $$
 where the latter integral is in the sense of Bochner (cf Section 
\ref{sec:appdix_integral_def} of Appendix).
}

\begin{theorem}\label{th:integral_representation_stationary}
Let $X=\{X(t),\ t\in \R^d\}$ be an $L^2$-continuous, weakly stationary 
process taking values in the separable Hilbert space ${\mathbb V}$ and having stationary covariance function ${\cal K}$. 

Then, \eqref{e:th:operator_Bochner} holds and there exists an a.s unique orthogonal $\V$-valued random measure $\xi$ with control 
measure $\mu$, such that \eqref{e:Cramer-stoch-int-rep} holds.
%\begin{align*}
%X(t) = \int_{\R^d} e^{\ii t^\top u} \xi(du)\quad \mbox{ almost surely} 
%\end{align*}
%for all $t\in \R^d$.
\end{theorem}

\noindent 
The proof of this result can be found in Section \ref{supp:proof:Cramer-stationary}.
%Section S.4.3 of \cite{shen:stoev:hsing:2020_extended}.\\ %in  \ref{supp:proof:Cramer-stationary} 

\subsection{Spectral theory for general IRF$_k$.}\label{ss:spectral_notation_1}

\cite{gelfand:vilenkin:1964d} provide an illuminating treatment of the spectral theory of generalized 
stochastic processes, i.e., processes with paths in the space of generalized functions.  In this setting,
the paths of the stochastic process have derivatives of all orders and one can naturally study processes 
with stationary $(1+k)$-th order derivatives.  One drawback of this treatment is that it is difficult to
use generalized process models in practice.  Motivated by fundamental problems in spatial statistics,
\cite{matheron:1973} developed the framework of intrinsic stationary functions, which allows one to
study classical random field models with stationary increments.

In a series of works, Matheron developed the theory of intrinsic random functions, which has become 
the {\em lingua franca} of spatial statistics \citep[see e.g.][]{chiles:delfner:2012}.  
Our goal here is to extend the Matheron theory to the functional setting, where the underlying stochastic processes
take values in a separable Hilbert space $\V$.  This is not straightforward and new {\em covariance 
asymmetry} phenomena arise that reflect the potential {\em irreversibility} of multivariate 
IRF's (see Remark \ref{rem:real-imaginary}).

Following \cite{matheron:1973}, in this section we will focus on second order linear processes 
$Y=\{Y(\lambda), \lambda\in\Lambda_k\}$, {\clb  viewed in the weak sense as stochastic processes indexed by $\Lambda_k$. That is, $Y$ is a 
random element in $S(\Lambda_k,\V)$ equipped with the product $\sigma$-field ${\cal B}_{\mathbb V^{\Lambda_k}}$ generated by all finite-dimensional cylinder sets.
We emphasize that,} in contrast to Section \ref{s:tangent_fields}, here we no longer require that $Y$ has continuous paths.
The (cross) covariance operator of $Y$ is defined as
\begin{align*}
\mathcal{C}_Y(\lambda,\mu):= \E[Y(\lambda)\otimes Y(\mu)], \quad \lambda,\mu
\in\Lambda_k.
\end{align*}
Denote by $\breve Y=\{\breve Y(t),t\in\bbR^d\}$ the representation of $Y$ (cf. Section \ref{s:2.1}) in 
$\breve{S}(\Lambda_k,\V)$, i.e., $\breve Y(t) = Y(\lambda_t)$, 
so that %\fbox{\tcr{Why mention representer here?}}
\begin{align}\label{e:representer}
Y(\lambda) := \int \breve Y(t) \lambda(dt), \quad\lambda\in\Lambda_k.
\end{align}
%Thus, $Y(\lambda)$ is a generalized $(k+1)$-th order increment of $\breve Y(t)$. 
We say that $Y$ is mean-square continuous if $\breve Y$ is mean-square 
continuous in the sense of \eqref{e:meansqcontinuity}.

%We will be interested in processes with second-order stationary generalized increments in the following sense of \cite{matheron:1973}.

\begin{definition} \label{d:IRF}
A second order process $\{Y(\lambda),\lambda\in\Lambda_k\}\in S(\Lambda_k,\V)$ is 
said to be an intrinsic random function of order $k$ (IRF$_k$), $k=-1, 0,1,\ldots$, 
if 
\begin{align} \label{e:stationary_Y}
\mathcal{C}_Y(\lambda,\mu)\equiv \mathcal{C}_Y(w+\lambda,w+\mu), \quad w \in\bbR^d, \lambda,\mu\in\Lambda_k.
\end{align}
\end{definition}
Note that \eqref{e:stationary_Y} is equivalent to 
\begin{align} \label{e:stationary_Y_1}
\mathcal{C}_Y(\lambda,\lambda)\equiv \mathcal{C}_Y(w+\lambda,w+\lambda), \quad w \in\bbR^d, \lambda\in\Lambda_k,
\end{align}
by Lemma \ref{l:identification_operator}, and, in turn, to the weak stationarity of
$\{Y(t+\lambda), t\in\bbR^d\}$ in $t$ for all $\lambda\in\Lambda_k$.
Indeed, if $Y(t+\lambda)$ is stationary in $t$ for all $\lambda$ then \eqref{e:stationary_Y_1} holds, and if 
\eqref{e:stationary_Y} holds then $Y(t+\lambda)$ is stationary in $t$ for all $\lambda$.

%\begin{remark} \fbox{\tcr{Do we need this notion?}} Clearly, if $Y$ is IRF$_k$ for some $k$, then it is also IRF$_{k'}$ for all $k'>k$ since 
%$\Lambda_{k'}\subset \Lambda_k$. The smallest $k$ for which $Y$ is IRF$_k$ will be called the 
%{\em intrinsic rank} of $Y$.   When none exists, the intrinsic rank  is defined to be infinity.  
%Stationary processes and those with stationary increments (but not stationary) have intrinsic 
%ranks $-1$ and $0$, respectively.
%\end{remark}

%\fbox{Switching notation $x_i,y_i\to t_i,s_i$ in the supports of the measures.}

\begin{definition} \label{ed:cond_pos}
A  collection of trace-class operators $\{{\cal K}(h), h\in\bbR^d\}\subset\mathbb{T}$ is said to be 
conditionally positive definite of degree $k$, $k=-1, 0,1,\ldots$, if {\clb for all $n\ge 1$,}
\begin{align}\label{eq:cond-psd}
  \sum_{j=1}^n\sum_{j'=1}^n c_j\overline{c}_{j'}\mathcal{K}(t_j-t_{j'}) \ge 0 \hbox{ (operator positivity)}
\end{align}
for all $c_i\in\bbC,t_j\in\bbR^d, 1\le i\le n$, such that $\lambda(du)=\sum_j c_j \delta_{t_j}(du)\in \Lambda_k$. 
\end{definition}
Relation \eqref{eq:cond-psd} can be succinctly written as ${\cal K}(\lambda *\wt \lambda)\ge 0$, where
\begin{align*}
\wt \lambda (du) : = \overline \lambda(-du) =\sum_{j} \overline c_j \delta_{-t_j}(du),
\end{align*}
and $\lambda*\mu$ denotes the usual convolution.
More generally, with $\lambda = \sum_j c_j \delta_{t_j}$ and $\mu = \sum_{j'} d_{j'}\delta_{s_{j'}} \in \Lambda_k$, 
\begin{align}\label{e:K-lambda-mu-tilde}
{\cal K}(\lambda * \wt \mu) = \sum_{j}\sum_{j'} c_j \overline d_{j'} {\cal K}(t_j-s_{j'}).
\end{align}
Interestingly, since 
$(w+\lambda)*\wt{(w+\mu)} \equiv \lambda*\wt\mu$, for all $w\in \bbR^d$, the map $(\lambda,\mu)\mapsto {\cal K}(\lambda*\wt\mu)$ is automatically shift invariant. This motivates the following definition.

\begin{definition}\label{def:GC}
	%Let $Y$ be an ${\mathbb V}$-valued IRF$_k$.  A continuous function ${\cal K}:\bbR^d\to \bbT$
	A  collection of operators ${\cal K}:\bbR^d\to \bbT$ is said to be a generalized covariance of $Y$ with
	degree $k$ if 
	\begin{align}\label{e:def:GC}
		\mathcal{C}_Y(\lambda,\mu)  = {\cal K}(\lambda *\wt \mu),\ \ \mbox{ for all }\lambda,\mu \in \Lambda_k.
	\end{align}
	%\tcr{We can write the right hand side as $\mathcal{K}(\lambda*\mu)$ where $\lambda*\mu:=\sum_i\sum_j \lambda_i\overline{\mu_j} \delta_{x_i-y_j}$.}
\end{definition} 
Again, by Lemma \ref{l:identification_operator}, \eqref{e:def:GC} is equivalent to
\begin{align*}
		\mathcal{C}_Y(\lambda,\lambda)  = {\cal K}(\lambda *\wt \lambda),\ \ \mbox{ for all }\lambda \in \Lambda_k.
	\end{align*}
The following result describes the connections between the notions in 
Definitions \ref{d:IRF}-\ref{def:GC}, and gives a {\em spectral representation} of a conditionally positive definite
$\mathcal{K}$. 
As a terminology, a polynomial in $\bbT$ refers to a linear combination of $d$-dimensional monomials with coefficients in $\bbT$, where the degree is equal to the highest degree of the monomials in the linear combination. %\tcr{the terminology was used before in representation definition}

\begin{theorem}\label{th:IRF_k_spectral_operator_value_1}
	Let $k\ge -1$ and the process $Y=\{Y(\lambda), \lambda\in\Lambda_k\}\in S(\Lambda_k,\V)$ be second order.
	\begin{enumerate} 
	\item
	 If $Y$ has a generalized covariance $\mathcal{K}$ of degree $k$, then $Y$ is IRF$_k$ and 
	${\cal K}$ must be conditionally positive definite of degree $k$. Conversely, if $Y$ is a mean-square
	continuous IRF$_k$, then it has a continuous generalized covariance of degree $k$. 
	\vskip.2cm
%	A mean-square continuous process $Y=\{Y(\lambda), \lambda\in\Lambda_k\}$ is IRF$_k$ if and only if it has a generalized covariance that is conditionally positive definite of degree $k$.
        \item
	A  continuous function ${\cal K}:\bbR^d\to \bbT$ is 
	conditionally positive definite of degree $k$ if and only if it can be represented as
	\begin{align}\label{e:K-h-rep}
	\mathcal{K}(h)=\int_{\bbR^d} \frac{e^{\ii  u^\top h}-I_B(u)P(u^\top h)}{1\wedge\|u\|^{2k+2}}\chi(du) + \mathcal{Q}(h),
	\end{align}
	where $P(x)=\sum_{j=0}^{2k+1} (\ii x)^{j}/j!$, $B$ is some arbitrary bounded neighborhood 
	of $0$, $\mathcal{Q}(h)$ is a conditionally positive definite polynomial with degree 
	no more than $2k+2$ and $\chi$ is a finite  $\T_+$-valued measure with no point 
	mass at $0$. The measure $\chi$ in \eqref{e:K-h-rep} is unique and does not depend on the choice 
        of the set $B$.  The polynomial $\mathcal{Q}$ therein is unique modulo an additive 
        polynomial of degree $2k+1$.
        %where the coefficients of the monomials of degree $2k+2$ is independent of $B$. 
        \end{enumerate}
        	\end{theorem}
	
The detailed proof of this result can be found in Section \ref{suppsec:thm4.3_proof} of Supplement.
%S.4.2 of \cite{shen:stoev:hsing:2020_extended}. % \ref{suppsec:thm4.3_proof} of the .  
Note that the proof follows closely the general and elegant treatment of \cite{Sasvari:2009}.

\begin{remark} In the notation of \cite{Sasvari:2009}, our situation corresponds to having a single multiplicative 
function (character) $\gamma_1 \equiv 1$ and $y_1:= 0$ and $k_1:= k+1$ and their measure $\sigma$ is our 
$(1\wedge \|u\|^{2k+2})^{-1}\chi(du)$.  Observe also that Relation (4.3) in Theorem 4.2 of \cite{Sasvari:2009} 
appears to be missing the non-ignorable degree $2k+2$ polynomial component in ${\mathcal Q}$ of \eqref{e:K-h-rep}.  
This omission can be attributed to the fact that the spectral measure of a stationary process in the Bochner theorem 
could have an atom at $\{0\}$, while $\sigma$ and $\chi$ do not. See Section \ref{suppsec:thm4.3_proof}
% Section S.4.2 of \cite{shen:stoev:hsing:2020_extended}
for more details.
\end{remark}

The measure $\chi$ and polynomial $\cal Q$ in \eqref{e:K-h-rep} will be referred to as the {\em spectral characteristics} of an IRF$_k$ with 
generalized covariance ${\cal K}$.  Note that the {\em spectral characteristics} pair $(\chi,\cal Q)$ is unique modulo an additive polynomial 
of degree $2k+1$ in the component $\cal Q$.  That is, the generalized covariance in \eqref{e:K-h-rep} is unique up to an additive polynomial of order $2k+1$.  
This implies that  ${\cal K}(\lambda*\wt \mu)$ is uniquely determined for $\lambda,\mu\in \Lambda_k$, where 
${\cal K}(\nu):=\int {\cal K}(h)\nu(dh),\ \nu\in\Lambda$. Notice that $\lambda*\wt \mu \in \Lambda_{2k+1}$, for $\lambda,\mu\in \Lambda_k$.  
Thus, in view of \eqref{e:def:GC}, the covariance structure of an IRF$_k$ process is completely determined by the linear 
measure-indexed $\bbT$-valued function ${\cal K}(\nu),\ \nu\in\Lambda_{2k+1}$. By integrating \eqref{e:K-h-rep} with respect to $\nu \in \Lambda_{2k+1}$, we obtain
\begin{align}\label{e:K-of_nu}
	\mathcal{K}(\nu)=\int_{\bbR^d} \frac{\widehat\nu(u)}{1\wedge\|u\|^{2k+2}}\chi(du) + \mathcal{Q}(\nu),
	\ \nu\in\Lambda_{2k+1},
\end{align}
where $\widehat\nu(u) =\int e^{\ii u^\top x} \nu(dx)$ is the Fourier transform of $\nu$. 
Since $\widehat {\lambda *\wt \mu} = \widehat \lambda \overline{\widehat\mu}$, the cross covariance operator $\mathcal{C}_Y(\lambda,\mu)$
of $Y$ can be uniquely expressed as
\begin{align}\label{e:IRF-cov-spectral}
\mathcal{C}_Y(\lambda,\mu) = {\cal K}(\lambda * \wt \mu) = \int_{\bbR^d} \frac{ \widehat \lambda (u) 
 \overline{\widehat\mu(u) }}{1\wedge\|u\|^{2k+2}}\chi(du) + \mathcal{Q}(\lambda * \wt \mu).
\end{align}

%The fact that ${\cal K}(\lambda*\wt\mu),\ \lambda,\mu\in\Lambda_k$ is a valid cross-covariance, shows that 
%the generalized covariance $\{{\cal K}(h)\}$ in \eqref{e:K-h-rep} satisfies also the following stronger notion 
%of {\em complete} conditional positive definiteness.

Now, consider the following counterpart to Definition \ref{def:complely-pos-def-function}.

\begin{definition} \label{def:complete-positive-definiteness}
A  collection of trace-class operators $\{{\cal K}(h), h\in\bbR^d\}\subset\mathbb{T}$ is said to be 
conditionally complete positive definite of degree $k$, $k=-1, 0,1,\ldots$, if
\begin{align}\label{eq:cond-complete-psd}
\sum_{j=1}^n\sum_{j'=1}^n \langle f_j, \mathcal{K}(\mu_j*\wt\mu_{j'}) f_{j'}\rangle \ge 0,
\end{align}
for all $f_j\in \V, \mu_j\in\Lambda_k,\ j=1,\cdots,n$ and $n\in \bbN$. 
\end{definition}

Since \eqref{eq:cond-psd} is the special case of \eqref{eq:cond-complete-psd} with $n=1$, 
conditional complete positive definiteness implies conditional positive definiteness.
However, as seen from \eqref{e:IRF-cov-spectral}, ${\cal K}(\lambda*\wt\mu),\ \lambda,\mu\in\Lambda_k$ is a valid cross-covariance,
and hence the operator function $\mathcal{K}$ in \eqref{e:K-h-rep} is conditionally complete positive definite. Thus, Theorem \ref{th:IRF_k_spectral_operator_value_1} implies the following parallel of Corollary \ref{c:weak-psd-implies-complete-psd}.

\begin{corollary} If ${\cal K}:\bbR^d \to \bbT$ is continuous, then Definitions \ref{ed:cond_pos} and 
\ref{def:complete-positive-definiteness} are equivalent.  
%That is, a continuous function ${\cal K}$ is conditionally positive definite of 
%degree $k$ if and only if it is completely conditionally positive definite of degree $k$. 
\end{corollary}

We end this section with a stochastic representation result for continuous  
IRF$_k$, which parallels the Cram\'er representation in Theorem \ref{th:integral_representation_stationary}. 
The proof is given in Section \ref{supp:th:integral_representation_IRF}. %Section S.4.4 in \cite{shen:stoev:hsing:2020_extended}

\begin{theorem}\label{th:integral_representation_IRF}
Let $k\ge -1$ and the process $Y=\{Y(\lambda), \lambda\in\Lambda_k\}$ in $S(\Lambda_k,\V)$ be 
mean-square continuous. Then $Y$ is IRF$_k$ if and only if it can be uniquely represented as 
	\begin{align} \label{e:IRF_representation}
	Y(\lambda) = \sum_{(j_1,\ldots, j_d) \in J} {\partial^{k+1}  \widehat\lambda\over\partial^{j_1}\cdots\partial^{j_d}}  (0) \cdot Z_{j_1\cdots j_d} + \int _{\mathbb{R}^d}\frac{\widehat\lambda(u)}{1\wedge \|u\|^{k+1}}\xi(du),\ \lambda\in
	\Lambda_k, \quad \mbox{a.s.}
	\end{align}
	where 
	\begin{enumerate}
		\item $J = \{(j_1,\ldots,j_d): j_1,\ldots, j_d\ge 0 \hbox{ and } j_1+\cdots+j_d=k+1\}$,
		\vskip.2cm
		\item $\xi$ is an a.s.\ unique random orthogonal measure $\xi$ on $(\R^d,\mathcal{B}(\mathbb{R}^d))$ 
		with control measure $\chi$, where $\chi$ is a finite $\T_+$-valued measure with no point mass at $0$,
		and
		\vskip.2cm
		\item the $Z_{j_1\cdots j_d}$ are uncorrelated random variables with values in 
		${\mathbb V}$ and are uncorrelated with $\xi$.
	\end{enumerate}
\end{theorem}

\begin{remark}
\cite{Berschneider:2012de} also obtains the stochastic representation of IRF$_k$ with the more abstract setting of 
locally compact Abelian domains. Our result here can be considered as an extension to the case of ${\mathbb V}$-valued 
processes connecting $\xi(dx)$ with the covariance operator functions in Theorem~\ref{th:IRF_k_spectral_operator_value_1}
explicitly. 
\end{remark}

\subsection{Real and complex IRF$_k$'s: Covariance (ir)reversibility.} \label{sec:real-complex}

The general treatment in the previous two subsections involves an abstract separable Hilbert space $\V$ 
over the field of complex numbers $\mathbb C$.  In practice, however, one often deals with Hilbert spaces of real-valued functions and it is useful to know
how our results specialize to this setting.  Furthermore, the distribution of a zero-mean Gaussian process taking values in a 
{\em complex} Hilbert space $\V$ cannot be {\em directly} characterized using their covariance structure, alone.  
To this end one needs to consider both real and complex Hilbert spaces.

%The distinction between real and complex Hilbert spaces is particularly important when 
%dealing with Gaussian processes.  Indeed, this is because the mean and the covariance structure determines the distribution 
%{\em only } when the Gaussian process takes values in a {\em real} Hilbert space.  To characterize the law of Gaussian processes in {\em complex}
%Hilbert spaces, one needs also the pseudo-covariance.  In this section, we make several remarks with a view towards applications. 

\medskip
{\em Real and imaginary parts in a complex Hilbert space.} In an abstract complex Hilbert space $\V$ the notions of a real and imaginary part 
of an element $z(\in\V)$ are not well defined unless one fixes a basis. Let ${\cal E}:=\{e_j,\ j\in\N\}$ be a {\em fixed} CONS of $\V$.  Then one can postulate that the CONS ${\cal E}$ is
{\em real} and for each $z = \sum_{j} z_j e_j\in \V$, with coordinates $z_j:=\langle z,e_j\rangle,$ we can define
\begin{align}\label{e:real-complex}
\Re(z) \equiv \Re_{\cal E}(z):= \sum_{j} \Re(z_j) e_j\ \ \ \mbox{ and }\ \ \ \Im(z) \equiv \Im_{\cal E}(z):= \sum_{j} \Im(z_j) e_j,
\end{align}
as the {\em real} and {\em imaginary} parts of $z$, relative to the CONS ${\cal E}$.  (Should one change the basis ${\cal E}$ the notions of real and imaginary part may change.)
Notice that $\V_\R :=\{ z\in\V\, :\, \Im(z)=0\}$ is invariant to addition and multiplication by real scalars and it becomes a real Hilbert space, with the inner product inherited from 
$\V$.  All elements of $\V$ that belong to $\V_\R$ will be referred to as real.

For $z\in \V$, we shall write $z = \Re(z) + \ii \Im(z)$ and naturally define the {\em complex conjugate} $\overline{z} := \Re(z) - \ii \Im(z)$.   The complex conjugate operation
as well as the real and imaginary part operators extend to $\V$-valued random elements in a straightforward manner
and we shall say that $x\in\V$ is real if $x\in \V_\R$, i.e., if its imaginary part is zero.  

The complex conjugate of a linear operator ${\cal A}:\V\to \V$ is defined as:
$
\overline \A(x):= \overline {\A(\overline x)},\ \ x\in \V.
$
This implies that $\overline {\A(x)} = \overline \A (\overline x)$, the operator $\overline \A$ is also {\em linear} and one can 
define the real and imaginary parts of $\A$ in as:
$
\Re(\A):={(\A + \overline \A)}/{2} \mbox{ and }\Im(\A):={(\A-\overline \A)}/{2\ii}. 
$
Thus, $\A = \Re(\A) + \ii \Im(\A)$ and the usual operations with complex numbers and vectors extend 
to the operator Banach algebra over the complex Hilbert space $\V$.  Note that the real and imaginary parts of ${\cal A}$ can be equivalently defined 
in terms of the real and imaginary parts of the coordinates of ${\cal A}$ in the fixed CONS ${\cal E}$.  We shall say that an operator $\A$ is {\em real} if
$\A = \Re(\A)$.\\

 {\em Real IRF$_k$'s.}  The above discussion shows how one can specialize and interpret the results in Sections \ref{sec:Bochner} and \ref{ss:spectral_notation_1} for 
 the case of real Hilbert spaces $\V_\R$.  Indeed, let $\Lambda_k(\R)$ be the set of all real $\lambda \in\Lambda_k$. It is easy to see that 
 $\Lambda_k = \Lambda_k(\R) + \ii \Lambda_k(\R)$.  %Using complexification, let also $\V := \V_\R + \ii \V_\R$ be a Hilbert space over $\C$. 
 
Suppose now that $Y$ is a $\V_\R$-valued IRF$_k$.  That is, Definition \ref{d:IRF} holds with $\Lambda_k$ replaced by $\Lambda_k(\R)$.  Then, by linearity,
$Y$ extends uniquely to a $\V$-valued IRF$_k$ as follows 
\begin{align}\label{e:Y-extension}
 Y(\lambda) := Y(\Re(\lambda)) + \ii Y(\Im(\lambda)),\ \lambda\in\Lambda_k,
\end{align}
where in fact $Y(\Re(\lambda))$ and $Y(\Im(\lambda))$ are real (belong to $\V_\R$).  This leads us to the following
 
 \begin{definition} \label{def:real-IRFk} A $\V$-valued IRF$_k$ $Y$ is said to be real if $Y(\lambda)$ is real for all $\lambda\in\Lambda_k(\R)$.
 \end{definition}
 
\noindent Thus, there is a one-to-one correspondence between the real IRF$_k$'s in $\V$ and the $\V_\R$-valued IRF$_k$'s as processes indexed by $\Lambda_k(\R)$.
 
%Consider now Theorems \ref{th:IRF_k_spectral_operator_value_1} and \ref{th:integral_representation_IRF} for the special case of a real IRF$_k$ $Y$. 
%By taking real $\lambda,\mu\in\Lambda_k(\R)$ in \eqref{e:IRF-cov-spectral} and making the change of variables $v:=-u$, we see that $\overline{{\cal C}_Y(\lambda,\mu)}$
%equals
%$$
%{\cal C}_Y(\lambda,\mu) = 
 %\int_{\R^d} \frac{\wh\lambda(v) \overline{\wh\mu(v)}}{1\wedge\|v\|^{2k+2}} \overline{\chi}(-dv) +
 % \overline{\cal Q}(\lambda*\wt\mu),\ \ \lambda,\mu\in\Lambda_k(\R).
% $$
 %In view of \eqref{e:Y-extension} and the bi-linearity of the covariance operator, it follows that the last relation continues to hold for all $\lambda,\mu\in \Lambda_k$.  
% That is, $(\overline{\chi}(-dx),\overline{\cal Q})$ are also spectral characteristics of the IRF$_k$ $Y$ and the uniqueness part of Theorem 
 %\ref{th:IRF_k_spectral_operator_value_1} entails $(\overline\chi(-dx),\overline {\cal Q}) = (\chi(dx),{\cal Q})$.  This discussion yields the following.
 
 \begin{proposition}\label{p:real-IRFk} Let $Y$ be a mean-square continuous IRF$_k$ taking values in $\V$ and having spectral characteristics $(\chi,{\cal Q})$.  
 
\begin{enumerate}
 \item
$Y$ is real if and only if in Relation \eqref{e:IRF_representation} the vectors $(\ii)^{k+1} Z_{j_1\cdots j_d}$ are real and the orthogonal measure 
  $\xi$ is Hermitian, i.e., $\xi(-A) = \overline{\xi(A)},$ almost surely, for all $A\in {\cal B}(\R^d\setminus\{0\})$.
  \vskip.2cm

\item
If $Y$ is real, then ${\cal Q}$ has real (operator) coefficients 
 (modulo polynomials of degree up to $2k+1$) and the spectral measure $\chi$ is Hermitian, i.e., 
 $
 \overline{\chi(-A)} = \chi(A),$  for all $A\in {\cal B}(\R^d\setminus\{0\}).
 $
 \vskip.2cm
 
\item Conversely, if $\chi$ is Hermitian and ${\cal Q}$ real, then there is a real IRF$_k$ $Y$with spectral characteristics $(\chi,{\cal Q})$. Let now $Y$ and $Y^\prime$ be two 
  real IRF$_k$ with the same spectral characteristics $(\chi,{\cal Q})$ and such that $\E[ Y(\lambda)\otimes Y^\prime(\mu)] = \E [Y^\prime(\lambda) \otimes Y(\mu)],\ \lambda,\mu\in\Lambda_k$.  
  Then, for any $a,b\in\R,$ with $a^2 + b^2=1$, 
  the IRF$_k$ defined as $\wt Y(\lambda) = aY(\lambda) + \ii b Y^\prime(\lambda)$ has the same spectral characteristics as $Y$ and $Y^\prime$. 
\end{enumerate}
   \end{proposition}

The proof is given in Section \ref{sec:appdix_integral_def} of the Appendix.\\

{\em Covariance (ir)reversibility.} Next, we comment on an important covariance irreversibility phenomenon, which arises in 
the case of vector valued processes.  It extends the notion of {\em time reversibility} for vector time series.

\begin{definition}\label{def:cov-reversible} We shall say that an IRF$_k$  $Y$ with generalized operator covariance ${\cal K}(\cdot)$
 is covariance-symmetric or -reversible if
\begin{align}\label{e:K-symmetric}
{\cal K}(\nu) = {\cal K}((-1)\cdot \nu),\ \ \mbox{ for all }\nu\in\Lambda_{2k+1}.
\end{align}
\end{definition}

Observe that the symmetry of the generalized covariance is equivalent to the fact that the IRF$_k$ processes
$\{Y((-1)\cdot \lambda)\}$ and $\{Y(\lambda)\}$ have the same covariance structure. Indeed, recall that
\begin{align*}
\E [ Y(\lambda) \otimes Y(\mu)] = {\cal C}_Y(\lambda,\mu) = {\cal K}(\lambda*\wt \mu),\ \lambda,\mu\in\Lambda_k
\end{align*}
and observe that $(-1)\cdot \lambda*\wt\mu = ((-1)\cdot\lambda)*\widetilde{((-1)\cdot\mu)}$. Thus, the IRF$_k$ process
$\wt Y(\lambda) := Y((-1)\cdot \lambda)$ has covariance ${\cal K}((-1)\cdot \nu)$.

In the simple case, where $Y$ takes real scalar values, all IRF$_k$'s are automatically covariance-symmetric.  This is perhaps
why symmetry is often taken for granted.  In the multivariate
and especially function-valued case, however, covariance-symmetry is an exception rather than a rule.  Naturally,
in view of \eqref{e:K-of_nu} and the uniqueness of the spectral measure, \eqref{e:K-symmetric} holds if and only if
$\chi(-A) = \chi(A)$,  for all $A\in {\cal B}(\R^d\setminus\{0\})$.  This simple observation and Proposition \ref{p:real-IRFk} yield the following  fact \citep[see also Theorem 5.1 in][]{didier:pipiras:2011}.

\begin{proposition}\label{p:supp:symmetry} A real mean-square continuous IRF$_k$ $Y$ is covariance symmetric, if and only if its spectral measure $\chi$ is real.
\end{proposition}

We end this section with a comment on the use of the results from Sections \ref{sec:Bochner} and \ref{ss:spectral_notation_1} 
in the context of Gaussian processes.  Recall that a $\V$-valued random element $Y$ is said to be Gaussian, if $\langle Y,f\rangle$ is a complex Gaussian variable, 
for each $f\in\V$.  This means that the joint distribution of $(\Re(\langle Y,f\rangle), \Im(\langle Y,f\rangle))$ is bivariate normal, for all $f\in\V$.
Equivalently, $Y$ is Gaussian in $\V$ if and only if $\wt Y:= (\Re(Y),\Im(Y))$ is a Gaussian element in the {\em real} Hilbert space $\V_\R^2 := \V_\R \times \V_\R$.  
  
\begin{remark}[Characterization of Gaussian IRF$_k$'s]  Part (iii) of Proposition \ref{p:real-IRFk} is a manifestation of the fact that the covariance 
 structure alone does not determine the distribution of  zero-mean Gaussian processes taking values in {\em complex} Hilbert spaces
  (cf Example \ref{ex:supp:complex-Gauss} in Supplement). %\citep[see also Example S.5.1 in][]{shen:stoev:hsing:2020_extended}.  
  To determine the distribution of a zero-mean Gaussian IRF$_k$ 
  $Y = \{Y(\lambda),\ \lambda\in\Lambda_k\}$, one needs to know both the {\em cross-covariance} and {\em pseudo cross-covariance} operators:
$
{\cal C}_Y(\lambda,\mu) = \E[ Y(\lambda) \otimes Y(\mu)]$ and ${\cal C}_{Y,\overline Y} (\lambda,\mu)= \E[ Y(\lambda) \otimes \overline Y(\mu)]
$
(see, e.g., Section \ref{sec:real-irreversible} and Corollary \ref{c:supp:Gaussian-characterization} in Supplement).
%\citep[see e.g., Section S.5.1 and Corollary S.5.3 in][]{shen:stoev:hsing:2020_extended}.
 
 Equivalently, the distribution of a $\V$-valued Gaussian IRF$_k$ is completely determined by the real IRF$_k$ 
 $\wt Y(\cdot):= (\Re(Y)(\cdot),\Im(Y)(\cdot))$ in the product space $\V^2$.  Since the law of the real Gaussian IRF$_k$ $\wt Y$ {\em is} determined
 by its cross-covariance, the results of  Sections \ref{sec:Bochner} and \ref{ss:spectral_notation_1} provide a complete characterization of the 
 $\V$-valued Gaussian IRF$_k$'s. 
\end{remark}

\section{Second order covariance self-similar IRF$_k$'s.}\label{s:self-similar IRFk}

In view of Theorems \ref{pro:self_similar} and \ref{IRF}, essentially all tangent fields are self-similar IRF$_k$.  This motivates a 
more in-depth study of self-similar IRF$_k$'s. In this section, we focus on second order covariance self-similar IRF$_k$'s with respect to
{\em linear} operator-scaling actions.  We establish their covariance structure and spectral representation.  
Section \ref{s:operator-ss} addresses the general case, Section \ref{sec:related-work-and-examples} discusses examples and related work, 
while Section \ref{ss:T_rrx} deals with the scalar scaling action, where the generalized covariance can be  written in closed form.

%
%
%In Section \ref{s:operator-ss}, we establish several general results about operator self-similar Gaussian IRF$_k$'s. Then, in 
%Section \ref{ss:T_rrx}, we consider the case of scalar scaling action, where the generalized covariance function can be 
%written in closed form.  In both cases, there is a vast closely related literature, which primarily focuses on the finite-dimensional 
%spaces ${\mathbb V}$ and/or stationary-increment processes $(k=0)$ (see e.g.,
%\cite{meerschaert:scheffler:2001book,bierme:meerschaert:scheffler:2007,didier:pipiras:2011,yuqiang:xiao:2011,
%baek:didier:pipiras:2014,didier:meerschaert:pipiras:2017} and the references therein).  The extension of the existing literature to
%the IRF$_k$ setting is a considerable research program.  Here, we focus on the general theory for an infinite-dimensional separable 
%Hilbert space, which to the best of our knowledge is new.

\subsection{Linear operator scaling.} \label{s:operator-ss}

Let $\H:\V \to \V$ be a bounded {\em linear operator} on the Hilbert space ${\mathbb V}$.  Consider the operator scaling actions
 $T_c:= c^{\H},\ c>0$, where $c^{\H}$ is interpreted as $\exp\{\log(c) {\H}\}$ and as usual, 
 \begin{align}\label{e:exp}
e^{\H} := \sum_{n=0}^\infty \frac{{\H}^n}{n!}.
\end{align}
The latter series converges in operator norm and $\|e^{\H}\|_{\rm op} \le e^{\|{\H}\|_{\rm op}}.$  We have moreover that if the 
bounded operators ${\H}_1$ and ${\H}_2$ {\em commute}, i.e., 
${\H}_1{\H}_2  = {\H}_2{\H}_1$, then $e^{{\H}_1} e^{{\H}_2} = e^{{\H}_1+{\H}_2} = e^{{\H}_2} e^{{\H}_1}$.  
This readily implies that  $e^{\H}$ has a bounded inverse $(e^{\H})^{-1} = e^{-{\H}}$.  Consequently,
$c^{\H},\ c>0$ is a strongly (operator) continuous and invertible group action on 
${\mathbb V}$, i.e., $c_1^{\H}c_2^{\H} = (c_1c_2)^{\H},\ c_1,c_2>0$. In fact, using the power-series representation \eqref{e:exp}, one 
can readily show that $c\mapsto c^{\H}$ is continuously Fr\'echet differentiable with derivative $c^{\H-1}\H$, i.e., 
\begin{align}\label{e:Frechet-derivative-c-to-H}
\Big\| \frac{1}{h} ( (c+h)^{\cal H} - c^{\H})  - c^{\H-1}\H \Big\|_{\rm op} \to 0,\ \ \mbox{ as }h\to 0,
\end{align}
{\clb where here and below $\H-a,\ a\in\R$ is interpreted as $\H-a {\rm I}$, so that  $c^{\H-1} = c^{-1} c^{\H}$.}

As in Definition \ref{def:self-similar}, we consider the following notion of {\em covariance} operator self-similarity.

\begin{definition}\label{def:OSS} Fix an {\em arbitrary} bounded linear operator ${\H}$ on ${\mathbb V}$.
A second order IRF$_k$ $Y$ is said to be covariance ${\H}$-self-similar, if $\{Y(c\cdot \lambda),\lambda\in\Lambda_k\}$ and $\{ c^{\H} Y(\lambda),\lambda\in\Lambda_k\}$
have the same operator cross-covariance function for all $c>0$.
\end{definition}

\begin{remark} If the IRF$_k$ process $Y$ is {\em real} and Gaussian, then  $Y$ is covariance ${\cal H}$-self-similar if and only if it is $\H$-self-similar in the
following stronger sense:
\begin{align}\label{e:Y-Hss}
\{ Y(c\cdot \lambda),\ \lambda\in \Lambda_k\} \stackrel{fdd}{=} \{ c^{\H} Y(\lambda),\ \lambda\in \Lambda_k\},\ \ \mbox{ for all }c>0.
\end{align}

We emphasize that ${\H}$ in Definition \ref{def:OSS} and \eqref{e:Y-Hss} is an arbitrary bounded linear operator and we do not require that 
$c\mapsto c^{\H}$ be a scaling action on $\V$ in the sense of Definition \ref{def:scaling_action} (see also Remark \ref{rem:operator-scaling-actions}).
If $T_c:= c^{\H}$ is a scaling action, however, then \eqref{e:Y-Hss} recovers the notion of self-similarity in Definition \ref{def:self-similar}.
\end{remark}

Let now $Y$ be a second order, mean-square continuous IRF$_k$ with operator auto-covariance ${\cal K}$ and 
spectral characteristics $(\chi,\mathcal{Q})$.  By Theorem \ref{th:integral_representation_IRF}, we have the
decomposition
\begin{align}\label{e:Y-decomposition}
Y(\lambda) = Y_{(0,\mathcal{Q})}(\lambda)  +Y_{(\chi,0)}(\lambda),\ \ \mbox{ almost surely,}
\end{align}
for all $\lambda\in \Lambda_k$, where $\{Y_{(0,\mathcal{Q})}(\lambda)\}$ and $\{Y_{(\chi,0)}(\lambda)\}$ 
are orthogonal mean-square continuous IRF$_k$'s with spectral characteristics $(0,\mathcal{Q})$ and $(\chi,0)$, respectively.  
This decomposition is second order unique.  Therefore, it follows that $Y$ is covariance ${\H}$-self-similar 
{\em if and only if}  both the components $Y_{(0,\mathcal{Q})}$ and $Y_{(\chi,0)}$ are covariance ${\H}$-self-similar. 
More precisely, we have the following general characterization of covariance ${\H}$-self-similar IRF$_k$'s.
For convenience, write
\begin{align} \label{e:chi-k}
\chi_k (dx) = {\chi(dx) \over 1\wedge \|x\|^{2k+2}}.
\end{align}

\begin{theorem} \label{thm:operator-ss} Let $Y$ be {\clb an $L^2$-continuous} IRF$_k$ with 
spectral characteristics $(\chi,\mathcal{Q})$. Let also ${\H}$ be a bounded linear operator.
\begin{enumerate}
\item
We have that $Y$ is covariance ${\H}$-self-similar if and only if 
for all $c>0$ and $\lambda,\mu\in \Lambda_k$
 \begin{align}\label{e:thm:operator-ss-2}
   \chi_k(dx) =  c^{{-\H}} \chi_k(dx/c) c^{{-\H}^*}\quad \mbox{ and }\quad  \mathcal{Q}(\lambda *\wt{\mu}) = c^{k+1{-\H}} \mathcal{Q}(\lambda*\wt\mu) c^{k+1{-\H}^*},
 \end{align}
 where $\chi_k$ is as in \eqref{e:chi-k}. \\
 
\noindent Suppose henceforth that \eqref{e:thm:operator-ss-2} holds and consider the polar coordinates 
$(r,\theta):= (\|x\|, x/\|x\|)$ in $\R^d\setminus\{0\}$. \\
 
\item
There exists a finite $\bbT_+$-valued measure $\sigma$ on the unit sphere $\mathbb S = \{ \|x\|=1\}$ such that
\begin{align}\label{e:chi-k-polar-new}
\chi_k ( D ) %= \int_0^\infty \int_{\bbS} 1_{D}(r\theta ) r^{{-\H}} \sigma(d\theta) r^{{-\H}^*} \frac{dr}{r}
 = \int_0^\infty r^{{-\H}} \Big( \int_{\bbS} 1_{D}(r\theta ) \sigma(d\theta) \Big) r^{{-\H}^*} \frac{dr}{r},
\end{align}
for all Borel sets $D \in {\cal B}(\R^d\setminus\{0\})$ that are bounded away from $0$. 
If \eqref{e:chi-k-polar-new} holds, we simply write
\begin{align}\label{e:chi-k-polar}
\chi_k(drd\theta) = r^{{-\H}} \sigma(d\theta) r^{{-\H}^*} \frac{dr}{r}%\chi_k(\{ (r,\theta) \in A \times B\}) = \int_{A} u^{{-\H}} \sigma(B) u^{{-\H}^*} \frac{du}{u}, 
\end{align}
and refer to \eqref{e:chi-k-polar} as a disintegration formula for $\chi_k$.
\vskip.2cm

\item The measure $\sigma$ in \eqref{e:chi-k-polar} is uniquely determined by the measure $\chi_k$ and it does not depend 
on the possibly non-unique operator ${\H}$ in \eqref{e:thm:operator-ss-2}.
\vskip.2cm
%
%
%The measure 
%$\chi_k$ in \eqref{e:thm:operator-ss-2}, admits the following {\em disintegration formula}:
% \begin{align}\label{e:p:operator-ss-spectrum}
 %  \chi_k(dx) \equiv \chi_k(drd\theta) = r^{{-\H}} \sigma(d\theta) r^{{-\H}^*} \frac{dr}{r},
%\end{align}
%for some uniquely determined and finite $\bbT_+$-valued measure  $\sigma(d\theta)$ on the unit sphere $\mathbb S = \{ \|x\|=1\}$.  
%This means that for all Borel sets $A\in {\cal B}(0,\infty)$ and $B\in {\cal B}(\bbS)$, we have
%$$
%\chi_k( \{ (r,\theta) \in A\times B)\}) = \int_A r^{{-\H}} \sigma(B) r^{{-\H}^*} \frac{dr}{r}.
%$$
%

\item 
The component $Y_{(\chi,0)}$ of $Y$ admits the Cram\'er-type 
stochastic integral representation
\begin{align}\label{e:p:operator-ss}
  Y_{(\chi,0)} (\lambda) = \int_0^\infty \int_{\mathbb S} \wh \lambda(r \theta) W(dr,d\theta),\ \ \mbox{ almost surely,}
\end{align}
$\lambda\in\Lambda_k$, where $W(dr,d\theta)$ is an orthogonal ${\mathbb V}$-valued random measure on 
$(0,\infty)\times \mathbb S$,  such that
\begin{align}\label{e:p:operator-ss-1}
\E [ W(dr,d\theta)\otimes W(dr,d\theta)]  = r^{-({\H}+1/2)} \sigma(d\theta) r^{-({\H}^*+1/2)} dr.
\end{align}
\end{enumerate}
\end{theorem}

The proof of this result is deferred to Section \ref{sec:proofs:operator-ss}, below.
\vskip.3cm

\begin{remark}[The support of an ${\H}$-self-similar IRF$_k$ is ${\H}$-invariant] The self-similarity exponent operator ${\H}$ can in principle 
be  arbitrary outside the support of the IRF$_k$ process 
$Y$.  The support of $Y$, denoted ${\rm supp}(Y)$, is the smallest closed linear subspace of ${\mathbb V}$, which contains all 
$Y(\lambda)$'s almost surely.  One can show that ${\H}({\rm supp}(Y))$ is a dense subset of ${\rm supp}(Y)$.  This allows one to essentially
restrict the operator ${\H}$ to ${\rm supp}(Y)$ (see Section \ref{sec:supplement-support} in Supplement for more details).
%\citep[see Section S.6 in][for more details]{shen:stoev:hsing:2020_extended}. 
\end{remark}

\begin{remark} If $Y$ is ${\H}$-self-similar, so are its components $Y_{(\chi,0)}$ and $Y_{(0,{\cal Q})}$ in \eqref{e:Y-decomposition}.   
While this decomposition is unique in law, the operator ${\H}$ need not be unique.  
See for example \cite{didier:meerschaert:pipiras:2017} and Remark \ref{rem:non-unique-H} below.
For example, the polynomial component $Y_{(0,{\cal Q})}$ is always $(k+1)\cdot \I$-self-similar. In general, however, we cannot conclude 
that $\H=(k+1)\cdot \I$, where $\I$ is the identity.  The non-uniqueness of the operator self-similarity exponent in the general setting of ${\mathbb V}$-valued IRF$_k$'s 
is an interesting problem of future research.
\end{remark}

For simplicity, in the rest of this section we suppose that $Y\stackrel{d}{=} Y_{(\chi,0)}$ has a {\em trivial} polynomial 
component $Y_{(0,{\cal Q})} =0$.   We will examine two classes of operators ${\H}$ which can serve as self-similarity exponents of $Y$.
We begin with a simple criterion.

\begin{corollary} \label{c:NSC-H-chi} A finite $\bbT_+$-valued measure $\chi$ is the spectral measure of {\clb an $L^2$-continuous,} covariance 
${\H}$-self-similar IRF$_k$ if and only if for some finite $\mathbb{T}_+$-valued measure $\sigma$ on $\mathbb{S}$, we have
\begin{align} \label{e:chi_sigma_trace-1}
\chi_k (dx)\equiv \frac{1}{1\wedge \|x\|^{2k+2}} \chi(dx) = r^{{-\H}} \sigma(d\theta)r^{{-\H}^*} r^{-1} dr
\end{align}
and
\begin{align}\label{e:chi_sigma_trace-2}
\int _0^\infty r^{-1} (1\wedge r^{2k+2}) \mbox{ {\rm trace}}\left(r^{{-\H}} \sigma(\mathbb{S})r^{{-\H}^*}\right)  dr < \infty,
\end{align}
where $(r,\theta):= (\|x\|,x/\|x\|)$ are the polar coordinates in $\R^d\setminus\{0\}$.
\end{corollary}
\begin{proof} `only if' Let $Y$ be a continuous covariance ${\H}$-self-similar IRF$_k$.  Then, by Theorem \ref{thm:operator-ss},
$\chi_k$ satisfies the disintegration formula in \eqref{e:chi_sigma_trace-1}.  We have moreover that 
\begin{align*}
\chi(\mathbb{R}^d) = \int_0^\infty r^{-1} (1\wedge r^{2k+2}) r^{{-\H}} \sigma(\mathbb{S})r^{{-\H}^*} dr
\in \mathbb{T}_+.
\end{align*}
Since $\chi(\mathbb{R}^d)$ is self-adjoint and positive definite, 
$\|\chi(\mathbb{R}^d)\|_{tr}= {\rm trace}\left(\chi(\mathbb{R}^d) \right)<\infty$, which proves \eqref{e:chi_sigma_trace-2}.

`if': Conversely, suppose that $\sigma$ is a finite $\bbT_+$-valued measure on $\bbS$ such that \eqref{e:chi_sigma_trace-2} holds.
Then, the fact that for all $B\in {\cal B}(\bbS)$ and $r>0$, $r^{{-\H}}\sigma(B) r^{{-\H}^*}\le r^{{-\H}}\sigma(\bbS) r^{{-\H}^*}$ 
as positive operators in $\bbT_+$,  implies that 
\begin{align*}
\chi(D) := \int_0^\infty  r^{-1} (1\wedge r^{2k+2}) r^{{-\H}} \int_\bbS 1_D(r\theta) \sigma(d\theta) r^{{-\H}^*} dr,\ \ D\in {\cal B}(\bbR^d\setminus\{0\})
\end{align*}
is well-defined in the sense of Bochner.  The so-defined $\chi$ is a finite $\bbT_+$-valued Borel measure on $\bbR^d\setminus\{0\}$, which can be taken as the spectral measure of an IRF$_k$ process $Y = Y_{(\chi,0)}$ with trivial polynomial component.   Clearly, $\chi_k$ defined 
as in \eqref{e:chi_sigma_trace-1} satisfies the scaling property \eqref{e:thm:operator-ss-2}, which entails the covariance
 ${\H}$-self-similarity of $Y$.
\end{proof}

$\bullet$ {\em Normal diagonalizable exponents.} Corollary \ref{c:NSC-H-chi} allows us to provide a complete characterization 
of the valid pairs $(\H,\sigma)$ of operator exponents and spectral measures in the important case where $\H$ is {\em normal} and
diagonalizable operator. Namely, suppose ${\H}$ is a normal operator with 
\begin{align}\label{e:H-normal}
 {\H} =\sum_{j=1}^\infty \lambda_j e_j\otimes e_j,
\end{align}
where $\lambda_j\in \bbC$ and where $\{e_j\}$ is a CONS of $\mathrm{Range}({\H})=\mathrm{Range}({\H}^*)$.  The convergence
of the last series is understood in the weak operator topology. 

\begin{theorem} Let ${\H}$ be a normal diagonalizable operator as in \eqref{e:H-normal} and let $\sigma$ be 
a finite $\bbT_+-$valued measure on $\bbS$.  The measure 
 $\chi (drd\theta) =(1\wedge r^{2k+2}) r^{{-\H}} \sigma(d\theta)r^{{-\H}^*} r^{-1} dr$ is the spectral measure 
of an ${\H}$-self-similar IRF$_k$, if and only if  

\begin{align} \label{e:R(lambda)}
0 <  \mathfrak{R}(\lambda_j) < k+1, \ \mbox{ whenever $\langle \sigma(\bbS)e_j, e_j\rangle >0$ } 
\end{align}
and 
\begin{align} \label{e:sigma}
\sum_j \left({1\over k+1-\mathfrak{R}(\lambda_j)} + {1\over \mathfrak{R}(\lambda_j)}\right) 
\langle \sigma(\mathbb{S}) e_j, e_j\rangle
< \infty.
\end{align}
\end{theorem}
\begin{proof}
First, we prove the `if' part. Let $\sigma\in\mathbb{T}_+$ satisfy \eqref{e:R(lambda)}
and define (in polar coordinates)
$
\chi (drd\theta) = (1\wedge r^{2k+2}) r^{{-\H}} \sigma(d\theta)r^{{-\H}^*} r^{-1} dr.
$
By the simple fact
\begin{align*}
\mbox{ trace}\left(r^{{-\H}} \sigma(\mathbb{S}) r^{{-\H}^*}\right) 
= \sum_j r^{-2 \mathfrak{R}(\lambda_j)} \langle \sigma(\mathbb{S})  e_j, e_j\rangle,
\end{align*}
we obtain
\begin{align} \label{e:integral_trace}
\begin{split}
& \int _0^\infty r^{-1} (1\wedge r^{2k+2}) \mbox{ trace}\left(r^{{-\H}} \sigma(\mathbb{S}) r^{{-\H}^*}\right)  dr  \\
&= \sum_j \langle \sigma(\mathbb{S}) e_j, e_j\rangle \int _0^\infty (1\wedge r^{2k+2})
r^{-2 \mathfrak{R}(\lambda_j)-1}  dr \\
& = {1\over 2} \sum_j \left({1\over k+1-\mathfrak{R}(\lambda_j)} 
+ {1\over \mathfrak{R}(\lambda_j)}\right)\langle \sigma(\mathbb{S}) e_j, e_j\rangle < \infty
\end{split}
\end{align}
where the integration is justified by \eqref{e:R(lambda)} {\clb and \eqref{e:sigma}.}
By Corollary \ref{c:NSC-H-chi}, $\chi$ is the spectral measure of a covariance ${\H}$-self-similar 
IRF$_k$.

Conversely, suppose $\chi$ is the spectral measure of a covariance ${\H}$-self-similar IRF$_k$. 
In order for the integral in \eqref{e:chi_sigma_trace-2} to be finite, the calculations in
\eqref{e:integral_trace} show that both
\eqref{e:R(lambda)} and \eqref{e:sigma} must hold.
\end{proof}

$\bullet$ {\em General bounded operator exponents.} Suppose now that $\H$ is a {\em general} bounded operator, which need not 
be normal nor diagonalizable.  In this case, we cannot provide a complete characterization of the covariance
 ${\H}$-self-similar IRF$_k$'s, but still furnish a general sufficient condition using Riesz functional  calculus \citep[see e.g., Ch.\ VII.4 in][]{conway2019course}.  
 Recall that the spectrum ${\rm sp}({\H})$ of a bounded operator consists of all $z\in \bbC$ such that
$({\H}-z\cdot \I)$ has no bounded inverse.  The spectrum ${\rm sp}({\H})$ is always a non-empty  compact subset of $\bbC$ 
and ${\rm sp}({\H}^*) = \{ \overline z\, :\, z\in {\rm sp}({\H})\}$ consists of the complex conjugates of the elements in the spectrum 
of ${\H}$.   If $\Gamma$ is a rectifiable curve containing ${\rm sp}({\H})$ in its interior then for every holomorphic function 
 $f$ on an open set containing the curve $\Gamma$ along with its interior, 
we define
\begin{align}\label{e:Riesz-calculus}
f({\H}) := \frac{1}{2\pi\ii} \oint_{\Gamma} \frac{f(z)}{z{-\H}} dz,
\end{align}
where the latter integral over $\Gamma$ is considered in the positive direction and $1/(z-\H) := (z\cdot \I-\H)^{-1}$ is a bounded operator since $z\in \Gamma \subset \bbC \setminus {\rm sp}({\H})$.
Since $f(z) = \exp\{ -\log(r) z\}$ is analytic for all $r>0$, we can use the above Riesz functional calculus tool to study the operator $r^{{-\H}}$.  

\begin{proposition}\label{p:general-bounded-H} 
Let ${\H}$ be a bounded operator and let $\Re({\rm sp}({\H}))$ denote the set of real parts of its spectrum.
{\clb If 
\begin{align}\label{e:H-bounded-sufficient}
 \replace{\|\H\|_{\rm op} < k+1\ \ \mbox{ and }\ \ \Re({\rm sp}({\H})) \subset (0 ,\infty)}{\Re({\rm sp}({\H})) \subset (0 ,k+1)},
\end{align}
for some $k\ge 0,\ k\in\mathbb Z,$} then for all finite $\bbT_+$-valued measures $\sigma$ on $\bbS$, we have that 
\begin{align*}
\chi(drd\theta) := (1\wedge r^{2k+2}) r^{{-\H}} \sigma(d\theta) r^{{-\H}^*} r^{-1} dr,
\end{align*}
is the spectral measure of a covariance ${\H}$-self-similar IRF$_k$.
\end{proposition}
\begin{proof} We will show first that, for some $\epsilon>0$ and $C_{\H}>0$,
\begin{align}\label{e:cH-epsilon}
\| r^{{-\H}}\|_{\rm op} \le C_{\H}\cdot \Big(r^{-(k+1)+\epsilon} 1_{(0,1)}(r) + r^{-\epsilon} 1_{[1,\infty)}(r) \Big).
\end{align}
Recall that ${\rm sp}({\H})$ is a compact subset of $\bbC$.  This fact and the assumption \eqref{e:H-bounded-sufficient} 
on the spectrum of ${\H}$ implies that $\Re({\rm sp}({\H})) \subset \replace{(\epsilon,\infty)}{(\epsilon, k+1-\epsilon)}$, for some $\epsilon>0$. Since ${\rm sp}({\H})$ is compact,
one can take a closed  curve $\Gamma$ containing  ${\rm sp}({\H})$ in its interior, such that \replace{$\Re(z) \ge \epsilon$}{$\epsilon \le \Re(z) \le k+1 -\epsilon$} for all $z\in \Gamma$. 
Note that $z\mapsto \|(z \cdot \I -\H)^{-1}\|_{\rm op}$ is a continuous function of $z$ for all $z\in \Gamma 
\subset\bbC\setminus {\rm sp}({\H})$. Thus, since $\Gamma$ is a compact set, we have that 
\begin{align*}
\max_{z\in \Gamma} \| (z\cdot \I - \H)^{-1}\|_{\rm op} =:C_{\H}(\Gamma)<\infty.
\end{align*}

Now, by applying \eqref{e:Riesz-calculus} to $f(z):= r^{{-\H}} = e^{-\log(r) \H}$, {\clb for all $r>0$, 
we obtain
\begin{align}\label{e:cH-epsilon-new}
\| r^{{-\H}}\|_{\rm op} &\le \frac{1}{2\pi}  \int_\Gamma |e^{-\log(r) z}| \| (z\cdot \I -\H)^{-1}\|_{\rm op} |dz| \nonumber \\
& \le \frac{C_\H(\Gamma)}{2\pi} {\rm Len}(\Gamma) \sup_{z\in \Gamma} | e^{-\log(r) z} |, 
\end{align}
where ${\rm Len}(\Gamma) = \int_\Gamma |dz|$ is the length of $\Gamma$. 

Observe now that $|e^{-\log (r) z}| = r^{-\Re(z)}$ and recall that $\epsilon\le \Re(z)\le k+1 - \epsilon$, for all $z\in \Gamma$.
This,  implies that
$$
|e^{-\log (r) z}| \le r^{-(k+1)+\epsilon}1_{(0,1)}(r) + r^{-\epsilon} 1_{[1,\infty)}(r),
$$
which in view of \eqref{e:cH-epsilon-new} yields \eqref{e:cH-epsilon}.}

Now, by Corollary \ref{c:NSC-H-chi}, the measure $\chi$ in \eqref{e:chi_sigma_trace-1} is the spectral measure of
a covariance ${\H}$-self-similar IRF$_k$, provided \eqref{e:chi_sigma_trace-2} holds.  This, however, readily follows from 
\eqref{e:cH-epsilon}.  Indeed, by \eqref{e:BAC-inequality},
the integral in \eqref{e:chi_sigma_trace-2} is bounded above by
\begin{align}\label{e:H-bounded-1}
\| \sigma(\bbS)\|_{\rm tr} \int_0^1 r^{2k+2} \|r^{{-\H}}\|_{\rm op}^2 r^{-1} dr + 
\| \sigma(\bbS)\|_{\rm tr} \int_1^\infty \|r^{{-\H}}\|_{\rm op}^2 r^{-1} dr,
\end{align}
where we used the fact that $\|r^{{-\H}}\|_{\rm op} = \|r^{{-\H}^*}\|_{\rm op}$.  By \eqref{e:cH-epsilon}, the integrals in \eqref{e:H-bounded-1} are finite
and the proof is complete.
\end{proof}

\begin{remark} The sufficient condition in \eqref{e:H-bounded-sufficient} may appear restrictive. In particular, it implies that ${\H}$ has a bounded inverse 
(since $0\not \in {\rm sp}({\H})$). This condition is not all that restrictive when the Hilbert space ${\mathbb V}$ is finite-dimensional and our
sufficient conditions are precisely the same as the existing literature in the special case $k=0$ 
\citep[see e.g.,][] {didier:pipiras:2011,didier:meerschaert:pipiras:2018}.
\end{remark}

\subsection{Related work and examples.}\label{sec:related-work-and-examples}

Here, we first specialize the results from the previous section and discuss
existing related work when $\V$ is finite-dimensional. Then, we consider a class of stationary infinite-dimensional processes, 
which admit higher-order tangent fields under operator scaling.

\begin{example}[IRF$_0$ or operator fractional Brownian motions]  
\label{rem:OFBM}
When $k=0, d=1$, and $\V = \bbR^m$, the IRF$_k$ processes can be identified with the well-studied
class of vector-valued stationary increment processes.  The seminal paper of  the \cite{didier:pipiras:2011} established the spectral 
representation and stochastic integral representations for essentially all Gaussian operator self-similar processes with stationary increments
taking values in $\R^m$.  We demonstrate next how these processes, known as operator fractional Brownian motions (OFBM), can be 
recovered from our Theorem \ref{thm:operator-ss}. {\clb In this setting the operator ${\cal H}$ is a real $m\times m$ matrix with
eigenvalues $\lambda_i \in\bbC,\ i=1,\cdots,m$ such that 
\begin{equation}\label{e:OFBM-condition}
 0 < \Re(\lambda_i) <1,
\end{equation}
\citep[see e.g., (1.4) in ][]{didier:pipiras:2011}.  Observe that the last condition coincides with \eqref{e:H-bounded-sufficient} of Proposition \ref{p:general-bounded-H} 
for $k=0$. } 

Let $\{Y(\lambda)\}$ be a zero-mean Gaussian ${\H}$-self-similar IRF$_0$.  Then, if one considers
\begin{align*}
 \lambda_t(du):=  \delta_{t}(du) - \delta_0(du),\ \ t\in \R,
\end{align*}
the process $B(t):= Y(\lambda_t),\ t\in \replace{\R^d}{\R}$ has stationary increments.  The ${\H}$-self-similarity of $\{B(t)\}$ follows 
readily from the self-similarity of $\{Y(\lambda)\}$ and the fact that $c\cdot \lambda_t = \lambda_{ct},\ c>0,\ t\in \R$.  Conversely,
every ${\H}$-self-similar stationary increment process $\{B(t),\ t\in \R\}$ can be taken as a representer of an ${\H}$-self-similar IRF$_0$
process.

Since $\wh{\lambda_t}(x) = e^{\ii tx} - 1$ and $\bbS=\{-1,1\}$, Relation \eqref{e:p:operator-ss} yields
\begin{align}\label{e:rem:OFBM}
B(t) \equiv Y(\lambda_t) &= \int_0^\infty \int_{\{-1,1\}} (e^{\ii tr \theta } - 1) W(dr,d\theta)\nonumber\\
& = \int_{0}^\infty  (e^{\ii x t } - 1) W(dx,\{1\}) + \int_{0}^\infty (e^{-\ii xt} -1) W(dx,\{-1\}).
\end{align}
Now, following the notation in Theorem 3.1 of  \cite{didier:pipiras:2011}, let 
$\wt B(dx) = \wt B_1 (dx) + \ii \wt B_2(dx)$, where $\wt B_i,\ i=1,2$ are independent zero-mean
Gaussian $\bbR^m$-valued measures such that $\wt B_1(dx) = \wt B_1(-dx)$, $\wt B_2(dx) = -\wt B_2(-dx)$, and
\begin{align}\label{e:Didier}
\E [\wt B(dx) {\wt B(dx)}^*]  \equiv \E \Big[\wt B(dx) \overline{\wt B(dx)}^\top\Big] = \I_m dx.
\end{align}
Observe that, by \eqref{e:p:operator-ss-1},
\begin{align*}
\Big\{W(dx,\{\pm 1\}),\ dx\in (0,\infty) \Big\} \stackrel{d}{=} \Big\{ x^{-({\H}+1/2)} A_{\pm 1}\wt B(\pm dx),\ dx\in (0,\infty)\Big\},
\end{align*}
where $A_{\pm 1} A_{\pm 1}^* = \sigma(\{\pm 1\})$.  Therefore, Relation \eqref{e:rem:OFBM} yields
\begin{align*}
\{B(t)\} \stackrel{d}{=}
 \left\{  \int_{-\infty}^\infty \frac{(e^{\ii x t } - 1)}{\ii x} \Big( x_+^{-({\H}-1/2)} A_{1} + x_-^{-({\H}-1/2)}A_{-1} \Big) \wt B(dx) \right\}.
\end{align*}
This is precisely the representation established in Theorem 3.1 of \cite{didier:pipiras:2011}, wherein $A_{-1} = \overline{A_1}$ is
the complex conjugate of $A_1$ since they consider real-valued processes.  Indeed, the last stochastic integral is real-valued if 
and only if the integrand $f_t(x)$ is a Hermitian function of $x$, i.e., $f_t(-x) = \overline{ f_t(x)}$.  This is the case, if and only if
 $A_1 = \overline A_{-1}$.
 
{\clb  \begin{remark}  Note that the condition \eqref{e:OFBM-condition} on the eigenvalues of the matrix $\H$ does not imply in general 
that $T_c:=c^{\H},\ c>0$ are scaling actions in the Euclidean norm of $\V\equiv \R^m$.  This is because the monotonicity of
the function $c\mapsto \|c^{\H}\|$ may be violated except when the matrix $\H$ is normal (i.e., diagonalizable in an orthonormal basis). In particular, Lemma 
\ref{l:operator-scaling} is not applicable.  Nevertheless, as shown in \citep[Lemma 6.1.5 in ][]{meerschaert:scheffler:2001book}, there is a 
suitable norm in $\V$,  with respect to which the latter are monotone increasing and in this new (equivalent norm) $\{c^{\H},\ c>0\}$ is a scaling 
action in the sense of Definition \ref{def:scaling_action}. See also \cite{jurek:1984} for the case where $\V$ is a Banach space.
\end{remark}
\begin{remark} \label{rem:OFBM:existence}  By Proposition \ref{p:general-bounded-H} (with $k=0$), Condition \eqref{e:OFBM-condition} implies that the 
stochastic integrals in \eqref{e:rem:OFBM} are well-defined. \end{remark}}
\end{example}

\begin{example}[Operator fractional Brownian fields] \label{rem:OSS-RF}
Stationary increment vector-valued random fields (IRF$_k$ with $k=0$) where $d\ge 2$
have been actively studied by many authors  
\citep[see e.g.,][among others.]
{bierme:meerschaert:scheffler:2007,yuqiang:xiao:2011,baek:didier:pipiras:2014,didier:meerschaert:pipiras:2018} 
In the latter references, self-similarity is considered under operator rescaling of {\em both} the range and
the domain of the process.  Here, we consider only scalar rescaling of the domain. In this setting, in the special case of 
processes taking values in $\R^m$ and $k=0$, Theorem \ref{thm:operator-ss} recovers Proposition 3.1 of \cite{didier:meerschaert:pipiras:2018}.

Interestingly, using Fr\'echet differentiability in Theorem \ref{thm:operator-ss}, we extend the 
{\em disintegration formula} in Relation (3.10) of \cite{didier:meerschaert:pipiras:2018} to the case of processes taking values in
a separable Hilbert space as well as to the general case of intrinsic random functions of order $k$.
We anticipate that a version of our Theorem \ref{thm:operator-ss} holds under operator scaling of both the range and the domain of $Y$.
\end{example}

\begin{remark}[The non-uniqueness of the operator exponent ${\H}$]\label{rem:non-unique-H} Suppose that
$Y$ is an operator ${\H}$-self-similar zero-mean Gaussian continuous IRF$_k$ taking values in the {\em real} Hilbert 
space $\V_\R$ (recall Section \ref{sec:real-complex}). Then, the distribution of $Y$ is determined by its covariance structure, i.e., 
by the unique pair of  its spectral characteristics $(\chi,{\mathcal Q})$ or equivalently $(\sigma,{\mathcal Q})$. The operator 
exponent ${\H}$, however, is not necessarily unique even when $\H$ is restricted to the support of the process $Y$.  For the notion 
of a support of $Y$ and its relation to the operator exponent ${\H}$, see Section \ref{sec:supplement-support} in Supplement.
%Section S.6 in \cite{shen:stoev:hsing:2020_extended}. 
To gain some intuition, suppose that for an operator ${\mathcal A}$ on $\V:={\rm supp}(Y)$, we have that 
$
\{ c^{\mathcal A} Y(\lambda)\} \stackrel{d}{=} \{Y(\lambda)\},
$
for all $c>0$.  If ${\H}$ and ${\mathcal A}$ commute, then $c^{{\H}+{\mathcal A}} = c^{\H}c^{\mathcal A},\ c>0$ and hence $Y$ is also $({\H}+{\mathcal A})$-self-similar.  

As shown in \cite{didier:meerschaert:pipiras:2017} such non-uniqueness can arise even in the finite-dimensional case with $k=0$,  where a wealth of interesting 
phenomena emerge.  Specifically, Theorem 2.4 therein characterizes all possible operator exponents  and shows that one can always choose a {\em commuting exponent} 
$\H_0$ such that ${\H}_0 {\mathcal A} = {\mathcal A} {\H}_0$. In their terminology, the operator $\A$ belongs  
to the tangent space of the group of symmetries of the process.  Notice that ${\mathcal A}$ can indeed be viewed as a tangent since it is the 
Fr\'echet derivative of $f(c) = c^{\mathcal A}$ at $c=1$. 

 Understanding the non-uniqueness of the operator self-similarity exponent in the general infinite-dimensional 
case is a challenging problem.  We anticipate that the extension of the important characterization results of 
\cite{didier:meerschaert:pipiras:2017} to the infinite-dimensional case is possible but considerably beyond the scope of this paper.
\end{remark}

We end this section with an example of stationary Gaussian $\V$-valued processes, which admit a large class of
tangent fields.  

\begin{example}[Higher-order tangent fields in infinite dimensions] \label{ex:functional-Matern-type} {\clb In this example, we shall assume that
$\V_\R$ is a real Hilbert space and through the method of complexification define $\V = \V_\R + \ii \V_\R$, with the natural inclusion $\V_\R \subset \V$.  

Consider polar coordinates in $\R^d\setminus\{0\}$, where $u = \|x\|,\ \theta:=x/\|x\|,$ are the radial and angular components of 
$x\in \R^d\setminus\{0\}$ and let $\mu(d\theta)$ be a finite, real, $\T_+$-valued measure on the unit sphere
 $\mathbb S = \{ \theta \in\R^d\, :\, \|\theta\|=1\}$. Define the real, $\sigma$-finite $\T_+$-valued 
measure 
$$
 \nu(dx) =\nu(du,d\theta) = du \mu(d\theta),\ \ \ (u,\theta) \in (0,\infty)\times \mathbb S.
$$

Let $W_{\R} = \{W_\R(A),\ A \in {\cal B}(\R^d\setminus\{0\}) \}$ and $W_{\mathbb I}= \{W_{\mathbb I}(A),\ A \in {\cal B}(\R^d\setminus\{0\})\}$ 
be two independent, real (i.e., $\V_\R$-valued) orthogonal Gaussian measures with the same control measure $2^{-1} \nu$ (in the sense of Definition \ref{def:orthogonal-measure}).  Construct
\begin{equation}\label{e:W-WR-WI}
 W(A) =  W_{\R}(A) + \ii W_{\mathbb I}(A).
\end{equation}

It is straightforward to see that $W = \{W(A)\}$ is an orthogonal Gaussian $\V$-valued random measure with 
control measure $\nu$, i.e., for all bounded Borel $A, B\in {\cal B}(\R^d\setminus\{0\})$, 
the random vectors $W(A)$ and $W(B)$ are such that 
\begin{align}\label{e:W_AB}
\E [ W(A)\otimes W(B)] = \nu(A\cap B) = \int_0^\infty \int_{\bbS} 1_{A\cap B} (u\theta) du\mu(d\theta).
\end{align}
Note, moreover, that $W$ is also independently scattered, i.e., $W(A_i),\ i=1,\cdots,n$ are independent for all disjoint bounded Borel 
sets $A_i \subset \R^d\setminus\{0\},\ i=1,\cdots,n$, which is not necessarily true for all orthogonal Gaussian random measures taking values in 
a complex Hilbert space.

By analogy with the scalar $\C$-valued case the Gaussian random measure $W$ in \eqref{e:W-WR-WI} will be referred to as standard. Since
its real and imaginary components are iid, the distribution of the process $W=\{W(A)\}$ is completely determined by its 
cross-covariance operators in \eqref{e:W_AB}.  Moreover, $W$ has circular symmetry and self-similarity properties:
\begin{equation}\label{e:circular-symmetry}
\{ e^{\ii \varphi } W(du,d\theta) \} \stackrel{fdd}{=} \{W(du,d\theta)\}\ \ \mbox{ and } \ \ \{ r^{1/2} W(du,d\theta)\} \stackrel{fdd}{=} \{W(d(r\cdot u), d\theta)
\},
\end{equation}
for all $\varphi\in \R$ and $r>0$ .

} 

Let $\H$ be a bounded linear operator on $\V$ such that 
\begin{align}\label{e:MRFk-H-condition}
\replace{\|\H\| < k+1\ \ \ \mbox{ and }\ \ \ \Re({\rm sp}(\H)) \subset (\epsilon,\infty)}{\Re({\rm sp}(\H)) \subset (\epsilon,k+1)},
\end{align}
for some $\epsilon>0$ {\clb and $k\ge 0,\ k\in\mathbb Z$}.
Suppose also that $\A(\theta),\ \theta\in\mathbb S$ is a collection of bounded linear 
operators such that $\theta\mapsto \A(\theta)$ is Borel measurable in $\theta$ and such that
\begin{align}\label{e:MRFk-A-condition}
 \int_\bbS \| \A(\theta)\|_{\rm op}^2 \|\mu\|_{\rm tr} (d\theta) <\infty,
\end{align}
where  $\|\mu\|_{\rm tr}$ denoted the (finite) trace measure $\|\mu\|_{\rm tr}(A) := \|\mu(A)\|_{\rm tr}$.
\end{example}

\begin{proposition} \label{p:stationary-tangent-example} 
Suppose that \eqref{e:MRFk-H-condition} and \eqref{e:MRFk-A-condition} hold, {\clb for some $k\ge 0,\ k\in\Z$.}

\begin{enumerate}
\item
For all $s\in\R^d$, the stochastic integral 
\begin{align}\label{e:MRFk}
X(s) := \int_0^\infty \int_{\bbS} f_s(u,\theta) W(du,d\theta),\ \ \mbox{ where } 
f_s(u,\theta) := e^{\ii u s^\top \theta } (1\wedge u)^{k+1}  u^{-(\H+1/2)} \A(\theta)
\end{align}
exists and defines a stationary $\V-$valued Gaussian random field.
\vskip.2cm
\item
The process $X = \{X(s),\ s\in \R^d\} $ has a version with $\gamma$-H\"older continuous paths for all 
$\gamma\in (0,1\wedge \epsilon)$, where $\epsilon$ is as in \eqref{e:MRFk-H-condition}.
\vskip.2cm
\item
The {\clb continuous-path version of the} process $\{X(s)\}$ has a $k$-th order tangent field at each (any) $s_0$. More precisely,
\begin{align}\label{e:tangent-field-example}
\Big\{r^{-\H} X(s_0+r\cdot \lambda),\ \lambda\in\Lambda_k\Big\} \stackrel{d}{\longrightarrow} Y= \{Y(\lambda),\ 
\lambda\in\Lambda_k\},\ \ \mbox{ as }r\downarrow 0,
\end{align}
where the tangent process is an $\H$-self-similar IRF$_k$ given by 
\begin{align*}
Y(\lambda) = \int_{\R^d} \widehat \lambda(u\theta) u^{-\H-1/2} \A(\theta) W(du,d\theta).
\end{align*}
\end{enumerate}
\end{proposition}

The proof of this result is given in Section \ref{sec:proofs:related-work-and-examples}, below.

\begin{remark} Notice that when $\V$ is infinite-dimensional in \eqref{e:MRFk} one {\em cannot} consider Gaussian measures 
$W$ with the control measure equal to the Lebesgue measure times the identity operator $\I_m$ as in \eqref{e:Didier}.  
Indeed, for $W(A)$ to be a bona fide random element in $\V$ the control measure of $W$ must take values in $\bbT_+$.   
This is the key reason why we consider control measures of the type $du\mu(d\theta)$. 
In the {\em finite-dimensional case}, one can obtain more familiar, but ultimately equivalent stochastic integral 
representations, in terms of Gaussian $\bbC^m$-valued Gaussian random measures with the Lebesgue control measure times 
the identity by considering $\E [ W_{\rm Leb}(dr,d\theta)\otimes W_{\rm Leb} (dr,d\theta)]= v_d r^{d-1} d\theta \times \I_m$, where 
$v_d:= \pi^{d/2}/\Gamma(1+d/2)$ is the volume of the unit sphere in $\R^d$. In this case, the stochastic integral in 
\eqref{e:MRFk} can be equivalently written in Cartesian coordinates as follows:
\begin{align*}
X(s) := v_d^{-1/2}  \int_{\R^d} e^{\ii s^\top x} (1\wedge \|x\|)^{k+1} \|x\| ^{-(\H+d/2)} \A(x/\|x\| ) W_{\rm Leb} (dx),\ \ \ s\in \R^d.
\end{align*}
\end{remark}

\subsection{Scalar actions.} \label{ss:T_rrx} In this section, we characterize the spectral measure of covariance self-similar IRF$_k$'s 
with respect to the usual scalar scaling action.  In this special but important case we obtain a more complete picture of the 
$H$-self-similar IRF$_k$'s, where now $H$ is a scalar exponent.

\begin{proposition} \label{proposition:GC} Let $Y$ be a non-constant continuous
IRF$_k,\ k\ge 0$ with operator auto-covariance function ${\cal K}$ and
spectral characteristics $(\chi,{\cal Q})$. If $Y$ is covariance self-similar with exponent ${H}$, then 
$H \in (0,k+1]$ and we have the following dichotomy:
\begin{enumerate}
\item
If $H=k+1$, then $\chi \equiv 0$ and if $0<H<k+1$, then ${\cal Q}$ is trivial, i.e., ${\cal Q}(\nu)=0$, for all $\nu\in\Lambda_{2k+1}$.
\vskip.2cm
\item
In the case $0<H< k+1$, the measure $\chi_k$ in \eqref{e:chi-k} satisfies the scaling property in
\eqref{e:thm:operator-ss-2} and consequently, the disintegration formula in  \eqref{e:chi-k-polar} reads:
\begin{align}\label{e:chi_k-disintegration-scalar}
\chi_k(dr d\theta) = r^{-2H-1} dr \sigma(d\theta),\ \ (r,\theta)\in (0,\infty)\times \bbS,
\end{align} 
for some finite $\bbT_+$-valued measure $\sigma$ on $\bbS$. 
\vskip.2cm
\item
Conversely, for every $0<H<k+1$ and any finite $\bbT_+$-valued measure $\sigma$ on $\bbS$, 
there exists a covariance $H$-self-similar IRF$_k$ with spectral measure $\chi$ such that \eqref{e:chi_k-disintegration-scalar} 
holds, which can be written as in \eqref{e:p:operator-ss}.
\end{enumerate}
\end{proposition}

The proof is given in Section \ref{sec:proofs:scalar-actions}, below.\\

In view of \eqref{e:K-of_nu} and \eqref{e:chi_k-disintegration-scalar}, one can obtain explicit formulae for
 the generalized covariance ${\cal K}$ of all covariance $H$-self-similar IRF$_k$'s.  This is done next.
 
 \begin{theorem} \label{p:K_nu_polar}
 Let ${\cal K}(\cdot)$ be the generalized covariance of a covariance self-similar IRF$_k,\ k\ge 0$ with exponent $H\in (0,k+1)$.  Then,
 with $\sigma$ as in \eqref{e:chi_k-disintegration-scalar}, we have:
 \begin{enumerate}
\item If $2H\not \in\{1,\dots,k\}$ is non-integer, then for all $\nu\in\Lambda_{2k+1}$,
 \begin{align}\label{e:p:K_nu_polar}
{\cal K}(\nu) = I({H}) \int_{\mathbb S^{d-1}} | ( \theta, \cdot ) |^{2H} (\nu) \sigma(d \theta) +
\ii J({H}) \int_{\mathbb S^{d-1}} (\theta, \cdot )^{<2H>} (\nu) \sigma(d\theta),
\end{align}
where $(\theta,t)= \theta^\top t$ denotes the Euclidean inner product,  $x^{<H>}:= {\rm sign}(x) |x|^{H}$, and 
$f(\cdot)(\nu) := \int f(t) \nu(dt)$. Here the real functions $I(H)$ and $J(H)$ are such that
\begin{align}\label{e:I_and_J}
I({H}) + \ii J({H}) := \int_{0}^\infty {\Big(} e^{\ii r} -\sum_{j=0}^{\lfloor 2H\rfloor} \frac{(\ii r)^j}{j!} {\Big)} \frac{dr}{r^{2H+1}}.
\end{align}
\item
If $2H \in \{1,\dots,k\}$ is  integer, then 
\begin{align*}
{\cal K}(\nu) = 
 \int_{\mathbb S^{d-1}} \Big[ | ( \theta, \cdot ) | ^{2H} \Big(\frac{(-1)^{H+1}}{(2H)!}  \log |( \theta,\cdot )| 
 + \ii\, J({H}) {\rm sign} ( \theta, \cdot ) {\Big)} \Big] (\nu) \sigma(d\theta),\ \ \mbox{ if $2H$ is even, }
\end{align*}
and 
\begin{align*}
 {\cal K}(\nu) = 
 \int_{\mathbb S^{d-1}} \Big[ |( \theta, \cdot )| ^{2H}  \Big( I({H}) + \ii\, \frac{(-1)^{{H}+1/2}}{(2H)!}   {\rm sign}( \theta,\cdot) 
 \log |( \theta,\cdot)|  \Big) \Big] (\nu) \sigma(d\theta),\ \ \mbox{ if $2H$ is odd.}
\end{align*}
\end{enumerate}
\end{theorem}

The proof of this result is given in Section \ref{sec:proofs:scalar-actions}, below.

\begin{remark} \cite{gelfand:vilenkin:1964d} provide spectral theory for generalized random fields taking values in the dual of the 
Schwartz space on $\R^d$ with homogeneous $(1+k)$th-order increments (denoted as G-IRF$_k$ here).  \cite{Dobrushin:1979vi} then studied 
the self-similar G-IRF$_k$ and obtained results similar to Proposition~\ref{proposition:GC} where the self-similarity parameter $H$ can take any
value in $(-\infty,k+1]$. The G-IRF$_k$ class of processes  is more general than the IRF$_k$'s studied by \cite{matheron:1973} and they do 
not always have a representation on $\R^d$. Specifically, it can be shown that a Gaussian self-similar G-IRF$_k$ has a {\clb representation} as in \eqref{e:representer} on $\R^d$ only if $H>0$ 
\citep[see][]{shen:phd_thesis:2019}.
\end{remark}

\begin{remark}[On symmetry and covariance (ir)reversibility]\label{rem:real-imaginary}
Observe that for $\sigma$ in \eqref{e:p:K_nu_polar}, one can write $\sigma = \sigma_{\rm s} + \sigma_{\rm a}$, 
where $\sigma_{\rm s}(A):= (\sigma(A) + \sigma(-A))/2$ and $\sigma_{\rm a}(A) := (\sigma(A) - \sigma(-A))/2,$
are the symmetric and anti-symmetric components of $\sigma$.  Thus, 
\begin{align*}
\int_{\mathbb S^{d-1}} |(\theta,\cdot)|^{2H} (\nu) \sigma_{\rm a}(dt) 
= \int_{\mathbb S^{d-1}} ( \theta,\cdot)^{<2H>} (\nu)\sigma_{\rm s}(dt) = 0,
\end{align*}
and \eqref{e:p:K_nu_polar} can be equivalently written as:
\begin{align}\label{e:K-sym-asym}
{\cal K}(\nu) = I({H}) \int_{\mathbb S^{d-1}} |( \theta, \cdot) |^{2H}(\nu) \sigma_{\rm s}(d \theta) +
\ii J({H}) \int_{\mathbb S^{d-1}} ( \theta, \cdot)^{<2H>}(\nu) \sigma_{\rm a}(d\theta).
\end{align}
This shows that unless $\sigma_{\rm a} \equiv 0$, we have that ${\cal K}(\nu) \not = {\cal K}(-\nu),$ for
some $\nu\in \Lambda_{2k+1}$.  Recall that by $-\nu$ we understand $((-1)\cdot \nu)(dx):=\nu(- dx)$.

{\clb Recall Definition \ref{def:cov-reversible}; $Y$ is covariance reversible, i.e.,
 $\{Y(-\lambda)\}$ and $\{Y(\lambda)\}$ have the same covariance structure, if and only if }
${\cal K}(\nu) = {\cal K}(-\nu), \forall \nu\in\Lambda_{2k+1}$ 
or equivalently if and only if $\sigma_{\rm a} \equiv 0$
\citep[see also Proposition \ref{p:supp:symmetry} above as well as 
Theorem 5.1 in][for a related result]{didier:pipiras:2011}.
\end{remark}

\begin{remark}[Real $H$-self-similar IRF$_k$'s]  Recall that $\V = \V_\R + \ii \V_\R$ (cf Section \ref{sec:real-complex}).  Thus
for an $H$-self-similar IRF$_k$  $Y$, we have
\begin{align*}
Y(\lambda) = \Re Y (\lambda) + \ii \Im Y(\lambda),
\end{align*}
where the real and imaginary parts $\Re Y$ and $\Im Y$ are real, i.e., $\V_\R$-valued. 
Thus, in view of \eqref{e:K-sym-asym}, one can see that $Y$ is real-valued (i.e., $\Im Y\equiv 0$) if and only if  
$\sigma_s$ is real and $\sigma_a$ imaginary, i.e., if $\sigma$ is Hermitian,
$\sigma(A) = \overline \sigma(-A)$ for all $A\in {\cal B}(\bbS)$.  Observe that $Y$ need not be covariance-reversible 
for it to be real (see Section \ref{sec:real-complex}.)

Since the covariance structure characterizes completely the zero-mean Gaussian processes taking values in {\em real} Hilbert spaces,
Theorem \ref{p:K_nu_polar} with Hermitian $\sigma$ provides a complete characterization of all $H$-self-similar $\V_\R$-valued IRF$_k$'s.
\end{remark}

\begin{remark}[$n$-th order fractional Brownian motion]  \cite{perrin2001nth} have studied the so-called $n$-th order fractional 
Brownian motion defined (in Remark 2 therein) as
\begin{align*}
B_H^{(n)} (t) := \frac{1}{2\pi} \int_{-\infty}^\infty \frac{1}{(\ii\omega)^{H+1/2} } \Big(e^{\ii t\omega} - \sum_{\ell =0}^{n-1} \frac{(\ii t\omega)^\ell}{\ell!} 
\Big) W(d\omega),
\end{align*}
with $n-1<H<n$, $n\in\N$, where $W(d\omega)$ is a zero mean complex Gaussian measure on $\R$ with the 
Lebesgue control measure and such that
$W(-d\omega) = \overline{W(d\omega)}$.  Notice, however,
that the above integral representation is well defined only if $H\in (n-1,n)$.  While one can always put $n:= \lceil H\rceil$, the integer values of
$H$ have to be dealt with separately.  Using our abstract approach, we can handle {\em all} 
values of $H \in (0,k+1),\ k:=n-1$, in a unified manner.  

Indeed, observe that with $k=n-1$, for any $\lambda\in\Lambda_k$, we have
\begin{align}  \label{e:nFBM}
\begin{split}
B_H^{(n)}(\lambda) &= \frac{1}{2\pi} \int_{-\infty}^\infty \wh \lambda(\omega)(\ii \omega)^{-(H+1/2)} W(d\omega) \\
&= \frac{1}{2\pi} \int_0^\infty \int_{\bbS} \wh\lambda ( u\theta) u^{-(H+1/2)}  e^{- \ii \theta \frac{\pi}{2} (H+1/2)} W(du,d\theta) \\
&=: \frac{1}{2\pi} \int_0^\infty \int_{\bbS} \wh\lambda ( u\theta) u^{-(H+1/2)}  \wt W(du,d\theta),
\end{split}
\end{align} 
where we used the change to polar coordinates $(u,\theta):= (|\omega|,{\rm sign}(\omega))$, with $\bbS = \{\pm 1\}$ and the fact that
$\lambda$ annihilates all polynomials of degree up to $k$.  This latter integral is defined for all $0<H<k+1$.  Notice that
$\wt W(du,d\theta) := e^{- \ii \theta \frac{\pi}{2} (H+1/2)} W(du,d\theta)$ is equal in law to $W(du,d\theta)$
 and \eqref{e:nFBM} is a particular case of the stochastic representation of the self-similar IRF$_k$'s characterized in Proposition \ref{proposition:GC}.
\end{remark}

We end with an example outlining the general form of the  $\R$-valued $H$-self-similar Gaussian IRF$_k$'s in $\R^d$,
which may be viewed as generalized fractional Brownian fields with $(1+k)$-th order stationary increments. 

\begin{example}[Real $n$-th order fractional Brownian fields] Fix an integer $k:=n-1\ge 0$, 
and let $\lambda_t$ be as in \eqref{e:eval1}. Let also $\sigma(d\theta)$ be a finite symmetric measure 
on the unit sphere $\bbS\subset \R^d$.  For all $H \in (0,k+1)$, it can be shown that 
$f_t(x,\theta):= (\langle \cdot,\theta\rangle - x)_+^{H-1/2}(\lambda_t)$ belongs to  ${\cal L}^2(dx,\sigma(d\theta))$, 
where $dx$ is the Lebesgue measure on $\R$.  Thus, one can define the $\R$-valued Gaussian 
random field 
\begin{align*}
Y(t):=  \int_{-\infty}^\infty \int_{\bbS} \Big(\langle \cdot,\theta\rangle - x\Big)_+^{H-1/2}(\lambda_t) W(dx,d\theta),\ \ \ t\in \R^d,
\end{align*}
where $W(dx,d\theta)$ is a zero-mean Gaussian real-valued random measure on $\R\times \bbS$ 
with control measure $dx\sigma(d\theta)$.
Then, it is easy to see with a simple change of variables 
that $Y(\lambda):=\int Y d\lambda,\  \lambda \in\Lambda_k$ is an $H$-self-similar IRF$_k$ with
{\em real} and symmetric generalized covariance $K$. In this case Theorem \ref{p:K_nu_polar} yields:
 \begin{align*}
 K(\nu) = C_H \int_{\bbS} \Big[ |\langle \cdot,\theta\rangle|^{2H} \Big ( I(H)+ 
  \frac{(-1)^{H+1}}{(2H)!} 1_{\bbN}(H) \log |\langle \cdot,\theta\rangle| \Big) \Big] (\nu) \sigma(d\theta).
 \end{align*}
 For more examples and further insights, see the PhD thesis of \cite{shen:phd_thesis:2019}.
\end{example}

\appendix

\section{Proofs and auxiliary results.}\label{app}

\subsection{Proofs and tools for Section~\ref{sec:section2}.} \label{sec:proofs_sec:section2}

	\begin{proof2}{Lemma}{\ref{l:S_is_csms}} Notice that $S_c(\breve \Lambda_k, \V)$ is a closed set in the metric space 
	$(C(\R^d,\V),\rho)$ of continuous ${\mathbb V}$-valued functions on $\R^d$, equipped with the metric $\rho$ in \eqref{eq:metric}.  
	Thus, it is enough to show that $(C(\R^d,\V),\rho)$ is a complete separable metric space 
	\citep[cf.\! Theorems 1 and 2 in Chapter XIV.2 in][] {Kuratowski:1977ty}.
	
	It is known that the metric $\rho$ generates the compact-open topology \citep[see, e.g., Theorem 46.8 in][]{Munkres:2000wg}. Therefore, to prove separability it is enough to demonstrate that this topology has a countable base. Recall that the compact-open topology on $C(\R^d,\V)$ has a sub-base comprising all sets 
	$V(K,U) = \{ f\in C(\R^d,\V)\, :\, f(K)\subset U\}$, where $K\subset \R^d$ is compact and $U\subset \V$ is open.
	Since $(\V,\dist)$ is separable and $\R^d$ is locally compact, the compact-open topology on $C(\R^d,\V)$ is second countable \citep[cf.\! Theorem 5.2, page 265 in][]{Dugundji:1966wg}.  This entails the separability of $(C(\R^d,\V),\rho)$.
	
        Completeness is established in a standard manner. Let $\{f_n\}$ be a Cauchy sequence in 
        $(C(\R^d,\V),\rho)$.  In view of \eqref{eq:metric}, for each $t\in \R^d$, 
        $\{f_n(t)\}$ is a Cauchy sequence in the complete metric space $(\V,\dist)$.  
        Thus, $f_n(t) \to f(t)\in \V$.  It remains to show that $f$ is continuous and
        $\rho(f_n,f)\to 0$.  Fix an arbitrary compact $K\subset \R^d$ and  $\epsilon >0$.  
        Since $\{f_n\}$ is Cauchy in $\rho$, there exists an $N_\epsilon$ such that
        \begin{align*}
        \sup_{\tau\in K} \dist(f_n(\tau),f_m(\tau))\le \epsilon/3,\quad \mbox{ for all } m,n\ge N_\epsilon. 
        \end{align*}
        On the other hand, for every fixed $t\in K$, we have
         \begin{align*}
         \dist(f_n(t),f(t)) = \lim_{m\to\infty} \dist(f_n(t),f_m(t)) \le \epsilon/3,\quad n\ge N_\epsilon.
         \end{align*}
        Since the latter bound is uniform in $t$, we also obtain $\sup_{t\in K} \dist(f_n(t),f(t))\le \epsilon/3,\ \ n\ge N_\epsilon$.  That is, $f_n$ converge to $f$ uniformly 
        on all compact $K$.  It remains to establish that $f$ is continuous.
        For all $s,t\in K$, we have by the triangle inequality that
         \begin{align*}
        \dist(f(t),f(s))\le 2 \sup_{\tau\in K} \dist(f_n(\tau),f(\tau)) + \dist(f_n(t),f_n(s)) \le \epsilon,
         \end{align*}
        provided that $\|t-s\|<\delta$ for some sufficiently small $\delta>0$.  Here, we used the uniform continuity of $f_n$ on $K$.
        Since $\epsilon>0$ was arbitrary, this completes the proof of the (uniform) continuity of $f$ on $K$.
        \end{proof2}

\begin{proof2}{Lemma}{\ref{l:T-a-widetilde}} \label{proof:T-a-widetilde} Properties (i) and (ii) in Definition \ref{def:scaling_action} are immediate. 
We now verify (iii). Consider the coordinate-wise action on $\breve{S}_c(\Lambda_k,\V)$, also denoted as $\{\widetilde T_a,\ a\in \bbR_+\}$ for convenience.  
 One can easily verify that $\mathcal{J}(\widetilde T_a (f)) = \widetilde T_a(\mathcal{J}(f))$. Let $f_n\to f$ in $S_c( \Lambda_k, \V)$ and $a_n\to a>0$. 
	To show that $\rho(\widetilde T_{a_n}(f_n), \widetilde T_{a}(f))\to 0$, it is enough to verify that for every $K>0$, we have
	\begin{align*}
	\sup_{\|t\| \leq K} \dist( \widetilde T_{a_n}(\breve f_n)(t), \widetilde T_{a}(\breve f)(t) ) \to 0,\ \ \mbox{ as }n\to\infty,
	\end{align*}
	where $\breve f_n = \mathcal{J}(f_n)$ and $\breve f = \mathcal{J}(f)$.
	
	In view of Lemma \ref{lem:folklore-uc}, below, it is enough to show that $ \widetilde T_{a_n}(\breve f_n)(t_n) \to \widetilde T_{a}(\breve f)(t)$, whenever $t_n \to t$ in $\overline{B}_K:=\{t:\|t\|\leq K, t\in\R^d \}$.
	Notice, however, that $ \widetilde T_{a_n}(\breve f_n)(t_n) = T_{a_n}(y_n)$ and $\widetilde T_{a}(\breve f)(t)=T_a(y)$, where $y_n: = \breve f_n(t_n)$ and $y:= \breve f(t)$.  
	By applying Lemma \ref{lem:folklore-uc} again, but now to the locally converging functions $\breve f_n$ and $\breve f$, we have that $y_n = \breve f_n(t_n) \to y = \breve f(t)$ in ${\mathbb V}$, 
	whenever $t_n \to t$ in $\overline{B}_K$.  Hence, the continuity of the scaling action $\{T_a\}$, yields $T_{a_n}(y_n) \to T_a(y)$ in ${\mathbb V}$, which completes the proof of property (iii).
	
	Let now $0 \not = f \in {S}_c(\Lambda_k,\V)$.  Proving property (iv) of Definition \ref{def:scaling_action} amounts to showing 
	that $\rho(\widetilde T_{a_1}(f), 0) < \rho(\widetilde T_{a_2}(f), 0)$, for all $0<a_1<a_2$.  Observe that by property (iv) for $\{T_a\}$,
	we have 
	\begin{align*}
	\dist(\widetilde T_{a_1}(\breve f)(t),0)  \equiv \dist(T_{a_1}(\breve f(t)), 0) \le \dist(T_{a_2}(\breve f(t)), 0)  \equiv \dist(\widetilde T_{a_2}(\breve f)(t),0).
	\end{align*}
	This implies that $\rho(\widetilde T_{a_1}(f), 0) \le \rho(\widetilde T_{a_2}(f), 0)$.  We next argue that the inequality is strict. Since $f\not =0$, we have 
	$0 \not = \breve f(t) \in \V$ for some $t\in \bbR^d$.
	Let $t \in \overline B_j$ for some large enough $j$, where $B_j$ is as defined in \eqref{e:def:scaling_action-v}.  Since the suprema therein are attained,
	it is enough to show that
	\begin{align} \label{e:l:T-a-widetilde-1}
	\begin{split}
	\dist(T_{a_1}(\breve f(t_1)), 0) &:= \max_{t \in \overline{B_j}} \dist(T_{a_1}(\breve f(t)),0) \\
	&< \max_{t \in \overline{B}_j} \dist(T_{a_2}(\breve f(t)),0)=: \dist(T_{a_2}(\breve f(t_2)), 0).
	\end{split}
	\end{align}
	Observe that, $0< \dist(T_{a_1}(\breve f(t),0) \le \dist(T_{a_1}(\breve f(t_1)), 0)$ and hence $\breve f(t_1)\not =0$.  Thus, by the radial monotonicity of the
	action $\{T_a\}$, we have
	\begin{align*}
	\dist(T_{a_1}(\breve f(t_1)), 0)  < \dist(T_{a_2}(\breve f(t_1)),0) \le \dist(T_{a_2}(\breve f(t_2)),0),
	\end{align*}
	which yields \eqref{e:l:T-a-widetilde-1} and completes the proof of (iv).
	
	We now verify property (v).  In view of Remark \ref{rem:property-v-continuity}, it is equivalent to show that for all $f\in {S}_c(\Lambda_k,\V)$, we have $\widetilde T_{1/n}(f) \to 0$ in ${S}_c(\Lambda_k,\V)$.
	Suppose that this is not the case.  Then, for some compact $K\subset \R^d$, some $\epsilon_0>0$ and a sequence $t_n \in K$, we have
	\begin{align*}
	\dist(T_{1/n}(\breve f(t_n)), 0) \ge \epsilon_0>0.
	\end{align*}
	Since $K$ is compact, for some $n'\to\infty$, we have $t_{n'}\to$ some $t_*\in K$, and by the continuity of $\breve f$, we have $\breve f(t_n) \to \breve f(t_*)$ in ${\mathbb V}$.  For all $\delta>0$, fixed,
	the radial monotonicity implies that 
	\begin{align*}
	0<\epsilon_0 \le \limsup_{n'\to\infty} \dist(T_{1/n'}(\breve f(t_{n'})), 0) \le \lim_{n'\to\infty} \dist(T_\delta(\breve f(t_{n'})),0)  = \dist(T_\delta(\breve f(t_*)),0).
	\end{align*}
	Property (v), for the scaling action $\{T_a\}$, however, entails that $T_\delta(\breve f(t_*))\to 0$ in ${\mathbb V}$, 
	as $\delta\downarrow 0$, which yields a contradiction with the above inequality and completes the proof.
	\end{proof2}
	
\noindent The following result shows that the linear operator actions considered in Remark \ref{rem:operator-scaling-actions} are in fact 
actions under mild natural conditions on the operator exponent $\H$.
	
	\begin{lemma} \label{l:operator-scaling} Let $\V$ be a Hilbert space and ${\cal H}:\V\to\V$ a bounded linear operator such that
	\begin{align}\label{e:l:operator-scaling}
	\Re ( {\rm sp}({\cal H}) ) \subset (0,\infty) \ \ \ \mbox{ and }\ \ \ 2\Re \langle {\cal H}x,x\rangle_{\V} 
	 = \langle ({\cal H}+{\cal H}^*) x, x \rangle_\V >0,
	\end{align}
	for all $x\not = 0$.  Then, $T_c:= c^{\cal H},\ c>0$ is a scaling action in the sense of Definition \ref{def:scaling_action}.
	Here $\Re( {\rm sp}({\cal H}))$ denotes the set of real parts of the elements in the spectrum of ${\cal H}$. 
	\end{lemma} 
	\begin{proof} Properties (i)--(iii) of Definition \ref{def:scaling_action} are immediate.  Indeed,
	we have $(c_1 c_2)^{\cal H} = c_1^{\cal H} c_2^{\cal H},$ $c_1,c_2>0$ and $1^{\cal H} = \I$, while Property (iii) follows 
	from the strong continuity (in operator norm) of $c\mapsto c^{\cal H}$.
	To establish the radial monotonicity (Property (iv) in Definition \ref{def:scaling_action}), it is enough to show that for all 
	$x\not = 0$, the function $\varphi(c):= \langle c^{\cal H} x, c^{{\cal H}} x\rangle_\V$ is strictly increasing in $c>0$.  To this end, we will show that
	$\varphi'(c)>0$ for all $c>0$ and $x\not=0$.  By \eqref{e:Frechet-derivative-c-to-H}, $h^{-1}((c+h)^{\cal H} - c^{\H}) \to {\cal H}c^{\H-1}$, in operator norm, 
	as $h\to 0$.  Thus, for $c>0$,
	\begin{align*}
	\varphi'(c) &= \lim_{h\to 0} \frac{1}{h}\Big( \langle (c+h)^{\cal H}x,(c+h)^{\cal H}x\rangle_{\V} - \langle c^{\cal H}x,(c+h)^{\cal H}x\rangle_{\V} \\
	&\quad\quad \quad\quad + \langle c^{\cal H}x,(c+h)^{\cal H}x\rangle_{\V}- \langle c^{\cal H}x,c^{\cal H}x\rangle_{\V} \Big) \\
	&=  \langle {\cal H}c^{\H-1}x,c^{H}x\rangle_\V +\langle c^{\H},{\cal H}c^{\H-1}x\rangle_\V = 2 c^{-1} \Re \langle {\cal H}c^{\H}x,c^{\H}x\rangle_\V,
        \end{align*}
        which is strictly positive, by assumption. Finally, property (v) follows from the first condition in \eqref{e:l:operator-scaling} as 
        in the proof of Relation \eqref{e:cH-epsilon}.
        \end{proof}

%\comment{
\subsection{Proof of Proposition \ref{P4}.}
\label{Proposition_proof}

Proposition \ref{P4} is the key to establishing the a.e.\ intrinsic stationarity of the tangent fields in Theorem \ref{IRF}. 
This section outlines its proof, which is based on the following lemma and the Egorov and Lusin Theorems.

\begin{lemma}\label{lem:folklore-uc} Let $(K,\rho_K)$ be a compact metric space and   $(E,\rho_E)$ be a 
	metric space. Suppose that $f_n$ and $f : K \to E$ are Borel measurable functions.  
	
	If the function $f$ is continuous, then
	\begin{align}\label{e:lem:folklore-uc-uniform}
	\sup_{x\in K} \rho_{E}(f_n(x),f(x)) \to 0,\ \ \mbox{ as } n\to\infty,
	\end{align}
	if and only if
	\begin{align}\label{e:lem:folklore-uc-sequence}
	{\clb \rho_E(f_n(x_n),f(x))\to 0,\  \ \mbox{ whenever } \rho_K(x_n,x)\to 0.}
	\end{align}
\end{lemma}

\begin{proof} (`if') Suppose that \eqref{e:lem:folklore-uc-sequence} holds and assume that \eqref{e:lem:folklore-uc-uniform} fails.
	Then, for some $\epsilon_0>0$, there exist an infinite sequence $n_k\in\mathbb N$  and  $x_{n_k}\in K$, such that 
	$\rho_E(f_{n_k}(x_{n_k}),f(x_{n_k})) \ge \epsilon_0>0$. 
	It is easy to 
	see that $\{x_{n_k}\}$ is also an infinite sequence, since for every $k_0\in\mathbb N$, by \eqref{e:lem:folklore-uc-sequence},
	we have $f_{n_k}(x_{n_{k_0}}) \to f(x_{n_{k_0}}),$ as $n_k\to\infty$. Indeed, had $\{x_{n_k}\}$ been a finite set, for some infinite
	subsequence $\{n_{k}'\}\subset\{n_k\}$, we would have $x_{n_k'} = x_{n_{k_0}}$ and hence $\rho_E( f_{n_k'}(x_{n_k'}), f(x_{n_{k}'})) =
	\rho_E( f_{n_k'}(x_{n_{k_0}}), f(x_{n_{k_0}}))  \to 0$, contradicting 
	the construction of the $x_{n_k}$'s.
	
	The infinite sequence $\{x_{n_k}\}$ is included in the compact $K$, and hence it has a converging subsequence $x_{n_k(m)} \to x$.  
	This, in view of \eqref{e:lem:folklore-uc-sequence}, implies that $f_{n_{k}(m)}(x_{n_k(m)})\to f(x)$ in ${\mathbb V}$.  Since $f$ is continuous 
	at $x$, however,  $f(x_{n_k(m)}) \to f(x).$ This, by the triangle inequality, implies
	\begin{align*}
		& \rho_E(f_{n_k(m)}(x_{n_k(m)}),f(x_{n_k(m)}))\\ &\le \rho_E(f_{n_k(m)}(x_{n_k(m)}),f(x)) + \rho_E(f(x),f(x_{n_k(m)}))
		 \to 0, \text{ as }n_k(m)\to\infty.
	\end{align*}
	This contradicts the assumption that $\rho_E(f_{n_k(m)}(x_{n_k(m)}),f(x_{n_k(m)})) \ge \epsilon_0$.
	
	\medskip
	(`only if') Let $x_n\to x$. By the triangle inequality, we have that
	\begin{align*}
		\rho_E(f_n(x_n),f(x))& \le \rho_E (f_n(x_n),f(x_n)) + \rho_E(f(x_n),f(x))\\
		& \le \sup_{x'\in K} \rho_E (f_n(x'),f(x')) + \rho_E(f(x_n),f(x)),
	\end{align*}
	which converges to zero by \eqref{e:lem:folklore-uc-uniform} and the continuity of $f$.
\end{proof}

The next result is a restatement of Theorem 7.5.1 in \cite{dudley:1989}.

\begin{theorem}[Egorov]
	\label{thm:Egorov}
	Let $(B,{\cal B}, \mu)$ be a finite measure space and $(Y,\rho_Y)$ be a separable metric space. Suppose that $f_n:B \to Y,\ n=1,\cdots$ are
	measurable functions such that, for $\mu$-almost all $x\in B$, 
	\begin{align*}
	f_n(x) \to f(x),\ \mbox{ as }n\to\infty.
	\end{align*}
	Then, for all $\epsilon>0$, there exists a measurable set $B_\epsilon \subset B$, such that
	\begin{align*}
	\mu(B\setminus B_\epsilon) <\epsilon\  \ \mbox{ and } \ \ \sup_{x\in B_\epsilon} \rho_Y(f_n(x),f(x)) \to 0,\ \mbox{ as } n\to\infty.
	\end{align*}
\end{theorem}

We present next a relatively general form of the classic Lusin's theorem stating that every Borel function is nearly continuous.
The proof follows the elegant 3-line argument given in Theorem 1 on page 56 in \cite{loeb:talvila:2003}. We 
provide a bit more detail and tailor the result to the case of metric spaces.

\begin{theorem}[Lusin] 
	\label{thm:Luzin}
	Let $(X,\rho_X)$ be a metric space and $(Y,\rho_Y)$ be a separable metric space.  Let also $f:X\to Y$ be a Borel measurable function and
	$\mu$ be a finite Borel measure $\mu$ on $X$.  
	
	For every $\epsilon>0$, there exists a closed set $F\subset X$, such that $\mu(X\setminus F)<\epsilon$ and $f:F\to Y$ is continuous.
	If $(X,\rho_X)$ is separable and complete, then the set $F$ can be taken to be compact.
\end{theorem}

\begin{proof} 

We will essentially unpack the argument on page 56 of \cite{loeb:talvila:2003} with small modifications.
	
	By Theorem 7.1.3 on page 175 in \cite{dudley:1989} every finite Borel measure $\mu$ on $(X,\rho_X)$
	is {\em closed regular},  that is, for every Borel set $A$ in $X$, we have
	\begin{align}\label{e:thm:Lizin-1}
	\mu(A) = \sup\{ \mu(F)\, :\, F\subset A, \mbox{ $F$ is closed}\}.
	\end{align}
	Recall that $\mu$ is called regular if the sets $F$ above can be taken to be compact. Ulam's Theorem implies that
	if $(X,\rho_X)$ is separable and complete, then $\mu$ is regular \citep[cf.  Theorem 7.1.4 in][]{dudley:1989}. 
	
	We now fix an $\epsilon>0$ and construct the closed set $F$. Since $(Y,\rho_Y)$ is separable, it is second countable, i.e., its topology has a countable base. Namely, 
	there exists a countable collection of open sets $\{V_n,\ n\in \mathbb{N}\}$ in $Y$ such that every open set $V \subset Y$ can be represented 
	as a union of $V_n$'s, i.e., $V = \cup\{V_n\, :\, V_n\subset V,\ n\in\mathbb{N}\}$.   
	
	Following \cite{loeb:talvila:2003}, by \eqref{e:thm:Lizin-1} since $\mu$ is finite, we can find closed sets
	$F_n \subset f^{-1}(V_n)$ and $F_n' \subset X\setminus f^{-1}(V_n)$ in $X$ (compact if $\mu$ is regular), such that
	\begin{align*}
	\mu( f^{-1}(V_n)\setminus F_n) < \frac{\epsilon}{2^{n+1}}\ \quad \mbox{ and }\ \quad \mu( [X\setminus f^{-1}(V_n)] \setminus F_n') < \frac{\epsilon}{2^{n+1}}.
	\end{align*}
	Observe that 
	\begin{align*}
	\mu(X\setminus (F_n \cup F_n')) = \mu( f^{-1}(V_n) \setminus F_n) + \mu( [X\setminus f^{-1}(V_n)] \setminus F_n') < \frac{\epsilon}{2^n}.
	\end{align*}
	Define $F:= \cap_{n\in\mathbb{N}} (F_n \cup F_n')$ and notice that $F$ is closed and in fact compact if $(X,\rho_X)$ is separable and complete. The above relation implies moreover that
	\begin{align*}
	\mu(X\setminus F) \le \sum_{n\in\mathbb{N}} \mu(X\setminus(F_n \cup F_n')) < \sum_{n\in\mathbb{N}} \frac{\epsilon}{2^n} = \epsilon.
	\end{align*}
	To complete the proof, it remains to show that $f:F\to Y$ is continuous.  To this end, it is enough to show that for every $x\in F$ and every $V_n$ such that $f(x)\in V_n$, 
	there is an open set $U\ni x$ such that $f(U\cap F) \subset f(V_n)$.  Suppose $f(x) \in V_n$ and consider the open set $U:= X\setminus F_{n}'$.
	Since $F\subset F_n\cap F_n'$ and $U \cap F_{n}' = \emptyset$, we have
	\begin{align*}
	U\cap F \subset U\cap F_{n} \subset f^{-1}(V_{n}),
	\end{align*}
	which implies  $f(U\cap F) \subset V_n$.  We have thus established the desired continuity of $f$ on $F$.
\end{proof}

\begin{remark} Loeb and Talvila's proof of Lusin's Theorem \ref{thm:Luzin} is not constructive and it does not use approximation 
	arguments based on the Tietze--Uryson Lemma and  Egorov's theorem as many other proofs in the literature \cite[see, e.g., Theorem 7.5.2  in][]{dudley:1989}. This makes it possible to extend Lusin's theorem to functions taking values in an arbitrary separable metric 
	space.
\end{remark}

%Now, we present the proof of Proposition~\ref{P4}.
We conclude this section with the proof of Proposition~\ref{P4}. 

\begin{proof2}{Proposition}{\ref{P4}}
	By Egorov's Theorem (see Theorem \ref{thm:Egorov}, above), there is a Borel set $B_\epsilon \subset B$ such that 
	$\mu(B\setminus B_\epsilon)<\epsilon/2$ and
	\begin{align}\label{e:unif-conv}
	\sup_{s\in B_\epsilon} \rho_E(F_n(s),G(s)) \to 0,
	\end{align}
	as $n\to\infty$.   Observe that since the Lebesgue measure is closed regular (recall \eqref{e:thm:Lizin-1}), one can choose
	the set $B_\epsilon$ to be closed.  Therefore, $B_\epsilon$ with the usual metric in $\mathbb R^d$ is a complete and 
	separable metric space. Hence, we can apply Lusin's Theorem \ref{thm:Luzin} to $X:=B_\epsilon \subset \R^d$ 
	and $Y:=E$ to conclude that there is a further compact set $K_\epsilon \subset B_\epsilon$, such that 
	$\mu(B_\epsilon \setminus K_\epsilon)<\epsilon/2$ and the function $G:K_\epsilon \to E$ is {\em continuous}.
	
	Observe that 
	\begin{align*}
	\mu(B\setminus K_\epsilon) = \mu(B\setminus B_\epsilon) + \mu(B_\epsilon\setminus K_\epsilon) < \epsilon.
	\end{align*}
	By Lemma \ref{lem:folklore-uc}, the continuity of $G$ on $K_\epsilon$ and the uniform convergence 
	\eqref{e:unif-conv} imply \eqref{e:p:EL}.   
\end{proof2}

%}

\subsection{Supplementary results and proofs for Section~\ref{s:tangent_fields}.}
  \label{ss:Supplementary_proof_section3}
  
The following convergence to types lemma is rather useful. 

\begin{lemma}\label{l:Slutsky} Let $\{T_a,\ a>0\}$ be a scaling action on some complete separable {\clb (not necessarily linear)}
metric space $(\mathbb X,d_{\mathbb X})$. Let also $\xi, \widetilde \xi$ and $\xi_n$ be random elements taking 
values in $\mathbb X$.  Then the following hold.  
\begin{enumerate}
    \item If $\xi$ is non-zero, then $T_{a'}(\xi) \stackrel{d}{=} T_{a''}(\xi)$ implies $a' = a''$.
    \item Suppose that $\xi_n\cid \xi$ and $\wt \xi_n:= T_{a_n}( \xi_n) \cid
    \widetilde \xi$, for some sequence $a_n>0$, where both $\xi$ and $\widetilde \xi$ 
    are non-zero. Then $a_n\to a$ for some $a>0$ and 
	$\widetilde \xi \stackrel{d}{=} T_a (\xi)$.
	\item If $\xi$ is non-zero and $T_{a_n}(\xi)\stackrel{d}{\to} 0$, then $a_n\to 0$.
\end{enumerate}
\end{lemma}
\begin{proof}%[Lemma~\ref{l:Slutsky}] 
\label{proof:Slutsky}
 {\em Part (i).} Suppose that $a'<a''$.  Then,
	\begin{align*}
	\xi \stackrel{d}{=} T_{a''}^{-1} \circ T_{a'}(\xi) \equiv T_{a'/a''}(\xi),
	\end{align*}
	which implies $\xi \stackrel{d}{=} T_{c_n} (\xi)$, for all $n\in \N$, where $c_n = (a'/a'')^n\downarrow 0$.  
	Thus, for every $B_r$, we obtain
	\begin{align*}
	\mathbb{P}(\xi \in B_r) = \mathbb{P}( T_{c_n}(\xi) \in B_r) = \mathbb{P}(\xi \in T_{1/c_n} (B_r)).
	\end{align*}
	Since $1/c_n\to \infty$, applying Property \eqref{e:def:scaling_action-v} {\clb applied with $\V$ replaced by $\mathbb X$}
	(recall \eqref{e:action-nested-balls}), we see that $\mathbb{P}(\xi\in B_r) = 1$, for all $r>0$, which contradicts the assumption that $\xi$ is non-zero. \\
	
	{\em Part (ii).}  We will first show that $\{a_n\}$ is bounded away from $0$ and $\infty$. 
	Indeed, suppose that $a_{n'}\to \infty$ for some $n'\to\infty$.   Consider the balls 
	\begin{align*}
	B_r:= \{x\in \mathbb X\, :\, d_{\mathbb X} (x,0)<r\}
	\end{align*} 
	and observe 
	that all but countably many of them are continuity sets for the distribution of $\xi$.  Indeed, the sets 
	$\partial B_r := \overline B_r \setminus B_r,\ r>0$ are pairwise disjoint in $r$ and for each $\epsilon>0$, there are at 
	most $1/\epsilon$ distinct values for $r$, such that $\mathbb{P}(\xi \in \partial B_r)>\epsilon$.
	
	For every $r>0$ such that $\mathbb{P}(\xi \in \partial B_r) =0$, since 
	$\xi_{n'}\stackrel{d}{\to} \xi$, we have
	\begin{align} \label{e:Slutsky-2}
	\begin{split}
	\mathbb{P}(\xi \in B_r) &= \lim_{n'\to\infty} \mathbb{P}(\xi_{n'}\in B_r) = \lim_{n'\to\infty} \mathbb{P}(T_{a_{n'}} (\xi_{n'}) \in T_{a_{n'}}(B_r)) \\ 
	&\ge \limsup_{n'\to\infty} \mathbb{P}(\widetilde \xi_{n'} \in T_m(B_r)), 
	\end{split}
	\end{align}
	where $\widetilde \xi_{n'} := T_{a_{n'}} (\xi_{n'})$ and $m$ is an arbitrary fixed integer.  
	Here, we used the fact that $T_m(B_r) \subset T_{a_{n'}}(B_r)$, for all large 
	enough $n'$, by \eqref{e:def:scaling_action-v} and \eqref{e:action-nested-balls}.
	
	Now, since $T_m$ is a homeomorphism, we have $\partial T_m( B_r) = T_m(\partial B_r)$ are disjoint in $r>0$,
	and by the above argument, for all but countably many $r$'s, we have $\mathbb{P}(\widetilde \xi \in \partial T_m(B_r)) =0$ and hence
	$ \mathbb{P}(\widetilde \xi_{n'} \in T_m(B_r)) \to \mathbb{P}(\widetilde \xi \in T_m(B_r))$, as $n'\to\infty$.  Therefore, in 
	view of \eqref{e:Slutsky-2}, we obtain
	\begin{align*}
	\mathbb{P}(\xi \in B_r) \ge \mathbb{P}(\widetilde \xi \in T_m(B_r)),\ \ \mbox{ for all $m$ and all but countably many $r>0$. }
	\end{align*}
	Relation \eqref{e:def:scaling_action-v}, however, implies that $T_m(B_r)\uparrow E$ as $m\to\infty$, which implies 
	\begin{align*}
	\mathbb{P}(\xi \in B_r) = 1 = \lim_{m\to\infty} \mathbb{P}(\widetilde \xi \in T_m(B_r)),
	\end{align*}
	for all but countably many $r$.  This implies that $\mathbb{P}(\xi = 0)=1$, which is a contradiction.
	
	We have thus shown that the sequence $\{a_n\}$ is bounded above.  One can similarly show that $\{a_n\}$ is bounded away 
	from $0$.  Indeed, by defining $\widetilde a_n:= 1/a_n$, we see that $\xi_n = T_{\widetilde a_n} (\widetilde \xi_n) \stackrel{d}{\to} \xi$. 
	Therefore, repeating the above argument with $a_n, \xi_n$ and $\xi$ replaced by $\widetilde a_n, \widetilde \xi_n$ and $\widetilde \xi$,
	respectively, we see that $\{\widetilde a_n \equiv 1/a_n\}$ is bounded.
	
	We have thus shown that $\{a_n\}$ can only have positive cluster points.  Suppose that $a_{n'} \to a'>0$ and $a_{n''}\to a''>0$, for
	some sub-sequences $n', n''\to\infty$.  Since the space $(\mathbb X,d_{\mathbb X})$ is separable, by the 
	Skorokhod-Dudley representation \citep[cf.\! Theorem 3.30 of][] 
	{kallenberg:1997}, on a suitable probability space we can define $\xi^*$ and $\xi_n^*$ such that 
	\begin{align*}
	\xi_n^*\stackrel{d}{=}\xi_n, \quad \xi^* \stackrel{d}{=} \xi,\quad \mbox{ and }\quad \xi_n^* \to \xi^*,\ \mbox{ almost surely.}
	\end{align*}
	Thus, the continuity property (iii) in Definition \ref{def:scaling_action}, implies that 
	\begin{align*}
	T_{a_{n'}} (\xi_{n'}^*) \to T_{a'}(\xi^*)\quad\mbox{ and }\quad T_{a_{n''}} (\xi_{n''}^*) \to T_{a''}(\xi^*),
	\end{align*}
	almost surely.  Since also $T_{a_n}(\xi_n^*)\stackrel{d}{=} T_{a_n}(\xi_n)\stackrel{d}{\to}\widetilde \xi$, 
	and $\xi \stackrel{d}{=}\xi^*$, we obtain
	\begin{align}\label{e:l:Slutsky-2}
	T_{a'} (\xi) \stackrel{d}{=} \widetilde \xi \stackrel{d}{=} T_{a''}(\xi).
	\end{align}
	By part (i), this is only possible if $a'=a''$.  We have thus shown that the sequence $\{a_n\}$ has a unique cluster point $a=a'=a''>0$ and 
	in view of \eqref{e:l:Slutsky-2}, that $T_a(\xi)\stackrel{d}{=}\widetilde \xi$.\\
	
	{\em Part (iii).} Suppose that $\limsup_{n\to\infty} a_n >0$, i.e., for some subsequence $n'\to\infty$, we have $a_{n'}\ge \epsilon_0>0$, for all $n'$.
	Then, in view of \eqref{e:action-nested-balls}, for all $r>0$, we have
	\begin{align*}
	\mathbb{P}(T_{\epsilon_0}(\xi) \in B_r^c) \le \mathbb{P}( T_{a_{n'}} (\xi) \in B_{r}^c). 
	\end{align*}
	Since $T_{a_{n'}}(\xi)\stackrel{d}{\to} 0$,  the right-hand side vanishes, as $n'\to\infty$.  On the other hand, since $\xi$ is nonzero, so is $T_{\epsilon_0}(\xi)$ and
	the left-hand side is positive for sufficiently small $r>0$.  This contradiction yields $\limsup_{n\to\infty} a_n =0$.
\end{proof}

The next result is used in the proof of Corollary \ref{cor:IRF}, given below.

\begin{lemma}\label{l:continuity} Let $X:=\{X(\lambda),\ \lambda\in\Lambda_k\}$
and $X_n:=\{X_n(\lambda), \lambda\in\Lambda_k\}$ be random fields in 
$S_c(\Lambda_k,\V)$ such that $X_n\cid X$. Then for any sequences $v_n\to 0$ and $r_n\to 1$, we have 
\begin{align*}
 \{X_n(v_n+r_n\cdot\lambda),\ \lambda\in\Lambda_k\} 
    \Cid \{X(\lambda),\ \lambda\in \Lambda_k\}.
\end{align*}
\end{lemma}

\begin{proof} \label{proof:continuity}
 Let $Y_n(\lambda) = X_n(v_n+r_n\cdot\lambda)$ and $\breve Y_n = Y_n(\lambda_t)$.
 %$\breve Y_n = \mathcal{J}(Y_n),$  where $\mathcal{J}$ is defined in \eqref{e:iso}.  
 By Proposition \ref{l:tightness}, conditions (i) and (ii) hold for
 $\breve X_n$. We need to show that they also hold for $\breve Y_n$. The proof for (i) is an easy 
 application of the Skorokhod-Dudley representation Theorem 3.30 on page 56 in 
 \cite{kallenberg:1997}. We focus on proving (ii).
	
	First consider the relationship between $\breve X_n$ and $\breve Y_n$. It follows that
	\begin{align*}
	\breve Y_n(t) =  X_n({v_n}+r_n\cdot\lambda_t) = \int \breve X_n(v_n+r_nu)\lambda_t(du),
	\end{align*}
	where one can write 
	\begin{align*}
	\lambda_t(du) = \sum_{i=1}^{M_k+1} c_i(t) \delta_{s_i(t)}(du),
	\end{align*}
	for some continuous functions $c_i(t)\in \mathbb C$ and $s_i(t)\in \R^d,\ i=1,\dots,M_k+1$ {\clb (recall \eqref{e:eval1})}.
	
	Since $\breve X_n (t) = X_n(\lambda_t)$, for any $K$ and $n$ large enough,
	\begin{align}\label{eq:boundness1}
	\begin{split}
		\sup_{\|t\|_{2}\leq K}\dist(\breve Y_n(t),\breve X_n(t)) &\le \sup_{\|t\|_{2}\leq K} \sum_{i=1}^{M_k+1} d_\V \Big(c_i(t)\breve X_n(v_n+r_ns_i(t)),\ c_i(t)\breve X_n(s_i(t))\Big)\\
		&\leq  (M_k+1)\sup_{\substack{\|s \|,\|t\|\leq 2K\\ |c|\leq {C}_K, \|s-t\|\leq \delta_n}}\dist(c\breve X_n(s),c\breve X_n(t)),  
	\end{split}
	\end{align}
	where $\delta_n := |r_n-1|S_K+|v_n| \to 0$ with 
	\begin{align*}
	S_K := \sup_{\substack{\|t\|\leq K\\ i=1,\dots,M_k+1}} \|s_i(t)\| \quad \mbox{ and }\quad {C}_K := \sup_{\substack{\|t\|\leq K\\ i=1,\dots,M_k+1}} |c_i(t)|.
	\end{align*}
	%\tcr{(We also require the extra condition for $d$ here as well.)}
	Therefore, for any $\eta, \epsilon >0$,  by (\ref{item:condition2}) in Proposition~\ref{l:tightness} there exists $n$ large enough such that,
	$\mathbb P(A_n < \eta ) > 1-\epsilon$, where
	\begin{align*}
	A_n:=\sup_{\substack{\|s\|,\|t\|\leq 2K \\ \|s-t\|\leq \delta_n}}\dist(\breve X_n(s),\breve X_n(t)).
	\end{align*}
	Then according to \eqref{e:translation}, \eqref{eq:extra_condition} and \eqref{eq:boundness1}, we have on the event $\{ A_n <\eta\}$, that
	\begin{align*}
	\sup_{\|t\|\leq K}\dist(\breve {Y}_n(t),\breve X_n(t))\leq (M_k+1) f_{{C}_K}(\eta),
	\end{align*}
	where for any ${C}_K>0$, $f_{{C}_K}(\eta) := \sup_{\dist(x,y)<\eta,|c|<{C}_K}\dist(cx,cy)\to 0$ as $\eta\to 0$.
	Thus, on $\{ A_n <\eta\}$, we have
	\begin{align*}
	 \sup_{\substack{\|s\|,\|t\|\leq K \\ \|s-t\|\leq \delta_n}}\dist(\breve {Y}_n(s),\breve Y_n(t)) & \leq
	 2  \sup_{\substack{\|t\|\leq K}}\dist(\breve{Y}_n(t),\breve X_n(t)) + \sup_{\substack{\|s\|,\|t\|\leq K \\ \|s-t\|\leq \delta_n}}\dist(\breve X_n(s),\breve X_n(t))\\
	 & \leq 2(M_k+1) f_{{C}_K}(\eta)+\eta.
	\end{align*}
	%where $C$ depends on both $B$ and the sup total variation norm of $|\lambda_t|$ for $t\in \overline{B}_K$.
	Thus, the second condition of Proposition \ref{l:tightness} for $\breve Y_n$ follows.
	\end{proof}

\begin{proof2}{Corollary}{\ref{cor:IRF}} Consider the context of the proof of Theorem \ref{IRF}. 
When $T_{s,c} = T_c,\ c>0$, for a {\em fixed} scaling action $T_c$. Relation \eqref{eq:re2} becomes
\begin{align}\label{cor:IRF-1} 
\begin{split}
\xi_n &= \Big\{T_{c_n(s_n)} \circ T_{c_n(s)}^{-1} \circ T_{c_n(s)} X(s,(1/n)\cdot(w_n+\lambda)), \ \lambda\in\Lambda_k\Big\} \\
&=: T_{c_n(s_n)/c_n(s)} \wt\xi_n.
\end{split}
\end{align}
where $c_n(s) := c(s,1/n)$, and we used the fact that $T_{c_n(s)}^{-1} =  T_{1/c_n(s)}.$

Relation \eqref{eq:re1} implies that 
\begin{align*}
\xi_n  \Cid \xi:=\{Y_s(\lambda),\ \lambda\in\Lambda_k\}.
\end{align*}
On the other hand, by Lemma \ref{l:continuity}, we have
\begin{align*}
\wt\xi_n= \{ T_{c_n(s)} X(s,(1/n)\cdot(w_n + \lambda)),\ \lambda\in\Lambda_k\} \Cid \wt\xi:=\{Y_s(w+\lambda),\ \lambda\in\Lambda_k\}.
\end{align*}
Thus, in view \eqref{cor:IRF-1}, we have 
$\wt\xi_n=T_{c_n(s)/ c_n(s_n)} \xi_n  \cid \wt\xi$ and since both $\xi$  and $\wt\xi$ are non-zero, Lemma \ref{l:Slutsky}
 implies that
\begin{align*}
\frac{c_n(s)}{c_n(s_n)} \equiv \frac{c_n(s)}{c_n(s+w_n/n)} \to a_s(\{w_n\})>0.
\end{align*}
One can verify that the limit $a_s(\{w_n\})$ is independent of the choice of the sequence $\{w_n\}$.  Indeed, if there exists another 
$w'_n\to w$ and $s + w'_n/n\in K_{\epsilon}$ then we will have	
	\begin{align*}
		\{Y_s(w+\lambda),\ \lambda\in\Lambda_k\} \overset{d}{=}\{T_{a_s (\{w_n\})}(Y_s(\lambda)),\ \lambda\in\Lambda_k \}
		\overset{d}{=}\{T_{a_s (\{w'_n\})}(Y_s( \lambda)),\lambda\in\Lambda_k\},
	\end{align*}
	which shows $a_s (\{w_n\}) = a_s (\{w'_n\})$ by (i) of Lemma~\ref{l:Slutsky}. 
	Thus, we can just use the notation $a_s(w)$ and we have
	\begin{align}\label{eq:equality_limit}
	\{Y_s(w + \lambda),\ \lambda\in\Lambda_k\} \stackrel{d}{=} \{T_{a_s(w)} Y_s(\lambda),\ \lambda \in\Lambda_k\}.
	\end{align}
	
	To prove that \eqref{e:T_regularity} holds, or equivalently $T_{c_n(s)/c_n(s_n)} \to T_1\equiv \I$, we only need to 
	verify $a_{s}(w) = 1$, which we do next. 	
	By (\ref{eq:equality_limit}), it is easy to see that $a_s({0})=1$ and $a_s(w+u)=a_s(w)a_s(u)$. By Theorem~\ref{pro:self_similar}, there is 
	a positive scalar $\alpha = \alpha(s)>0$, such that 
	\begin{align} \label{e:SS1}
	\{Y_s(r\cdot\lambda),\ \lambda\in\Lambda_k\}\overset{d}{=}\{T_{r^{\alpha}}(Y_s(\lambda)),\ \lambda\in\Lambda_k \},\ \mbox{ for all }r>0.
	\end{align}
	Consider $Y_s(r\cdot(w+\lambda))$ with $r\in \mathbb{N}$. 
	On one hand, \eqref{eq:equality_limit} and \eqref{e:SS1} imply that
	\begin{align}\label{e:Y-w+lambda-1}
	\begin{split}
	\{Y_s(r\cdot(w+\lambda)),\ \lambda\in\Lambda_k\}
	 \overset{d}{=}\{T_{r^\alpha }(Y_s(w+\lambda)),\ \lambda\in\Lambda_k\} 
	\overset{d}{=} \{T_{r^\alpha a_s(w)}(Y_s(\lambda)),\ \lambda\in\Lambda_k\}.
	\end{split}
	\end{align}
	On the other hand, since $a_s(rw) = a_s^r(w)$, for all $r\in \mathbb N$, viewing $Y_s(r\cdot(w+\lambda))$ as
	$Y_s(rw +r\cdot\lambda)$, by \eqref{eq:equality_limit}, we have
	\begin{align} \label{e:Y-w+lambda-2}
	\begin{split}
	\{Y_s(r\cdot(w+\lambda)),\ \lambda\in\Lambda_k\}
	 \overset{d}{=}\{T_{a_s(rw)}(Y_s(r\cdot\lambda)),\ \lambda\in\Lambda_k\} 
	 \overset{d}{=}\{T_{a^r_s(w)r^\alpha}( Y_s(\lambda)),\ \lambda\in\Lambda_k\}.
	\end{split}
	\end{align} 
	Thus, by Lemma \ref{l:Slutsky} applied to \eqref{e:Y-w+lambda-1} and \eqref{e:Y-w+lambda-2}, we obtain 
	$a_s(w)=a^r_s(w)$ for all $r\in\mathbb{N}$, which, since $a_s(w)>0$, implies $a_s(w)=1$.
\end{proof2}

\subsection{Supplemental background and some proofs for Section~\ref{sec:c2s5}.} \label{sec:appdix_integral_def}

In Bochner's Theorem, we need the notion of integration of a $\bbC$-valued function
on $\bbR^d$ with respect to a {\em finite} $\mathbb{T}_+$-valued measure.  
A finite $\mathbb{T}_+$-valued measure is a mapping $\mu:{\cal B}(\R^d)\to \bbT_+$ from the 
class of Borel sets in $\R^d$ to $\bbT_+$ that is {\em countably additive}.  Notice that this readily implies that $\mu(\emptyset) = 0$;
that $\mu$ is monotone, i.e., $\mu(A)\le \mu(B)$ as positive operators for $A\subset B$; and that $\mu$ is finite since 
$\mu(\R^d)\in \bbT_+$.  This notion can be developed along the line of ordinary Lebesgue integration, making use of the 
completeness of $\bbT$. We will provide a brief outline below and leave the details to 
Section \ref{sec:integration} in Supplement. %Section S.3 in \cite{shen:stoev:hsing:2020_extended}. 
(The construction naturally extends to the case of $\sigma$-finite $\bbT_+$-valued measures.)\\

%For a $\T_+$-valued measure $\mu$, define the	positive and negative parts of $\mu$, denoted by $\mu_+,\mu_-$, respectively, so that, for any Borel set $A$, $\mu(A) = \mu_+(A) - \mu_-(A)$ where $\mu\pm(A)\ge 0$.
	
	Let $\mu$ be a finite $\mathbb{T}_+$-valued measure. 
	We follow the development of ordinary Lebesgue integration:
	\begin{enumerate}%[label=(\roman*)]
		\item [(i)]
		For any real nonnegative simple function $f =\sum_{i=1}^k c_i I_{A_i}$, define 
		$\int fd\mu = \sum_{i=1}^k c_i\mu(A_i)$.
		\vskip.2cm
		\item [(ii)]
		For nonnegative measurable functions $f$, let
		\begin{align}\label{e:mono_conv}
		\int fd\mu = \lim_{n\to\infty} \int f_nd\mu
		\end{align}
		in $(\bbT,\|\cdot,\|_{\tr})$ where $\{f_n\}$ is any sequence of simple functions such that 
		\begin{enumerate}%[label=(\alph*)]
			\item [(a)]
			$f_n\le f_{n+1}$,
			\item [(b)]
			$f_n(x)\uparrow f(x)$ for all $x$,
			\item [(c)]
			$\int f_nd\mu \le \CB$ for all $n$ and some fixed $\CB\in\bbT_+$. 
		\end{enumerate}
		The existence of $\{f_n\}$ satisfying (a) and (b) for any given 
		nonnonegative measurable $f$ follows from standard measure theory. 
		However, we need the extra condition (c) (along with the completeness of $\bbT_+$)
		to ensure that the limit on rhs of (\ref{e:mono_conv}) exists and does not depend on 
		the choice of $\{f_n\}$. Clearly, (c) is automatically fulfilled if $f$ is bounded.
		\vskip.2cm
		\item [(iii)]
		For a general real measurable $f$, let
		\begin{align*}
		\int fd\mu = \int f_+d\mu - \int f_-d\mu 
		\end{align*}
		provided both terms on the right is finite. For a general complex $f$,
		let
		\begin{align*}
		\int fd\mu = \int f_{\mathrm{re}}d\mu +\ii \int f_{\mathrm{im}}d\mu 
		\end{align*}
		where $f_{\mathrm{re}},f_{\mathrm{im}}$ be the real and imaginary parts, respectively.
		\end{enumerate}
		
{\clb It is immediate that $\|\mu\|_{\rm tr}(A):= \|\mu(A)\|_{\rm tr},\ A\in {\cal B}(\R^d)$ defines a finite Borel measure referred to as to the
{\em trace measure} of $\mu$.  The following useful integrability criterion is straightforward.
\begin{proposition}\label{p:mu-integrability} The integral $\int_{\R^d} f d\mu \in \T$ is well-defined, for all $f\in L^1(\R^d,\|\mu\|_{\rm tr})$, and 
$$
\int_{\R^d} f (x)  \|\mu\|_{\rm tr}(dx)  = {\rm trace}\Big( \int_{\R^d} f(x) \mu(dx) \Big).
$$
\end{proposition}

The integral can be readily extended to {\clb $\sigma$-finite} $\mathbb{T}$-valued signed measures.  One can also naturally consider finite signed $\T$-valued
measures, namely $\mu=\mu_+-\mu_-$ where $\mu_+,\mu_-$ are both finite $\mathbb{T}_+$-valued measures, by $\int f d\mu := \int f d\mu_+ - \int f d\mu_-$.  In this case, the 
trace measure of $\mu$ is the (scalar) signed measure defined as $\|\mu\|_{\rm tr}(\cdot) := \|\mu_+\|_{\rm tr} (\cdot) - \|\mu_-\|_{\rm tr} (\cdot)$.}

\vskip.4cm
Next we turn to the stochastic integral in Cram\'er's representation. 
Recall that $\bbL^2(\Omega)$ is the $L^2$-space of all ${\mathbb V}$-valued random 
elements $\eta$ on the probability space $(\Omega,\CF,\pr)$ with
$\E\|\eta\|^2 < \infty$, equipped with the inner product
\begin{align*}
\langle \eta_1,\eta_2\rangle_{\Omega} := \E\langle\eta_1,\eta_2\rangle.
\end{align*}
{\clb 
We briefly discuss the properties of the orthogonal random measure
$\xi = \{\xi(A),\ A\in {\cal B}(\R^d)\}$ in Definition \ref{def:orthogonal-measure}.  Condition (ii) therein readily implies the orthogonality 
as well as finite additivity of the measure $\xi$.  Indeed, 
\begin{equation}\label{e:strong-orthogonality}
 \E [ \xi(A) \otimes \xi(B) ] = \mu(\emptyset) = 0,\ \ \ \mbox{ for all disjoint $A, B \in {\cal B}(\R^d)$,}
\end{equation}
where $0$ stands for the zero operator in $\T$.   This implies in particular that $\langle \xi(A),\xi(B)\rangle_\Omega =0$ (cf Remark \ref{rem:strong-orthogonality}). 
We have, moreover, that for all (orthogonal) projection operators $\Pi:\V \to \V$, the random measure $\eta(\cdot):= \Pi \xi(\cdot)$ is also an orthogonal measure with 
control measure $\Pi \mu(\cdot) \Pi$.

Now, to show finite additivity, observe that for disjoint Borel sets $A$ and $B$,
\begin{align*}
& \E \|\xi(A\cup B) -\xi(A) - \xi(B)\|^2  \\
& \quad\quad \quad  = \E {\rm trace}\Big ( (\xi(A\cup B) -\xi(A) - \xi(B)) \otimes ( \xi(A\cup B) -\xi(A) - \xi(B)) \Big )\\
& \quad\quad \quad   = {\rm trace} \Big( \mu(A\cup B) - \mu(A) - \mu(B) \Big)= 0,
\end{align*}
by the fact that the $\E$ and ${\rm trace}$ operators can be exchanged and the finite additivity of $\mu$. This shows $\xi(A\cup B) = \xi(A) + \xi(B)$, almost surely.

Finally, Condition (i) in Definition \ref{def:orthogonal-measure} along with the established finite additivity implies the $\sigma$-additivity of $\xi$, in the sense
that for any sequence of pairwise disjoint $A_n\in {\cal B}(\R^d),\ n=1,2,\cdots$, we have
$$
\xi\Big(\bigcup_{n=1}^\infty A_n\Big) = \sum_{n=1}^\infty \xi(A_n),\ \ \ \mbox{ almost surely,}
$$ 
where the latter series converges in $\bbL^2(\Omega)$.

\begin{remark}\label{rem:strong-orthogonality} Observe that for random vectors $X, Y\in  \bbL^2(\Omega)$, we have that
$$
\E [ X\otimes Y ] =0 \ \ \ \mbox{ implies } \ \ \ \langle X,Y\rangle_{\Omega} = 0,
$$
but the converse implication is not always true.  Thus, the orthogonality condition in \eqref{e:strong-orthogonality} is stronger than requiring 
 simply $\langle \xi(A), \xi(B)\rangle_\Omega=0$. 
\end{remark}
}

Introduce the finite Borel scalar measure 
\begin{align*}
\|\mu\|_{\rm tr}(A) := {\rm trace}(\mu(A))\equiv\|\mu(A)\|_{\tr},\ \ \ A \in {\cal B}(\R^d).
\end{align*}
It is easy to verify (cf Lemma \ref{l:cross-cov} in Supplement)
% \citep[cf. Lemma S.2.2 in][]{shen:stoev:hsing:2020_extended} 
that, for all $A\in\mathcal{B}(\bbR^d)$, 
\begin{align*}
\|\xi(A)\|_\Omega^2 = \|\mu(A)\|_{\rm tr} = \|\mu\|_{\rm tr}(A).
\end{align*}

Next, for  an orthogonal random measure $\xi$ with control measure $\mu$, 
we sketch the construction of the stochastic integral $\CI_\xi f := \int_{\R^d} f(t) \xi(dt)$ defined 
for all functions $f\in \bbL^2(\bbR^d,\|\mu\|_{\rm tr})$,\ i.e., all measurable $f:\bbR^d\to \C$ with 
$\int_{\R^d} | f(t)|^2 \|\mu\|_{\rm tr} (dt)<\infty$.
For any simple function $\psi(x) = \sum_{j=1}^n c_jI_{A_j}(x)$, 
where $A_j\in\mathcal{B}(\bbR^d), A_j\cap A_{j'}=\emptyset$  when $j\neq j'$, define the integral
\begin{align*}
\CI_\xi(\psi) = \sum_{j=1}^n c_j\xi(A_j). 
\end{align*}
{\clb The integral $\CI_\xi(\psi)$ takes values in $\bbL_\xi^2(\Omega)$, defined as the closure of 
\begin{align*}
\mathrm{span}(\xi):=\Big\{\sum_{j=1}^n c_i \xi(A_j), c_j\in\bbC, A_j\in\mathcal{B}(\bbR^d), n=1,2,\ldots\Big\}
\end{align*} 
in $\bbL^2(\Omega)$.}
Property (ii) in Definition \ref{def:orthogonal-measure} entails
\begin{align*}
\|\CI_\xi(\psi)\|_\Omega^2 = \int_{\R^d} |\psi(x)|^2 \|\mu\|_{\rm tr}(dx). 
\end{align*}
Therefore, $\CI_\xi$ is an isometric linear mapping between the class of simple functions in the
$L^2$-space $\bbL^2(\bbR^d,\|\mu\|_{\rm tr})$ {\clb and the Hilbert space $\mathbb L^2_\xi(\Omega)$}.

As the class of simple functions is dense in $\bbL^2(\bbR^d,\|\mu\|_{\rm tr})$ 
and the integrals $\CI_\xi(\psi)$ are dense in $\bbL_\xi^2(\Omega)$, 
the linear operator $\CI_\xi$ can be uniquely extended to an isometric linear 
mapping between $\bbL^2(\bbR^d,\|\mu\|_{\rm tr})$ and $\bbL_\xi^2(\Omega)$.  This
completes the construction of the stochastic integral
\begin{align*}
\int_{\R^d} f(x) \xi(dx):=\CI_\xi(f) , \ \ \ f\in\bbL^2(\bbR^d,\|\mu\|_{\rm tr}).
\end{align*}
{\clb Observe, moreover that for all $f,g\in  \bbL^2(\bbR^d,\|\mu\|_{\rm tr})$, we have
$$
\E [ \CI_\xi(f) \otimes \CI_\xi(g)] = \int_{\R^d} f(x)\overline{g}(x) \mu(dx),
$$
where the last integral is well-defined in view of Proposition \ref{p:mu-integrability}. 
}

 \begin{proof2}{Proposition}{\ref{p:real-IRFk}} 
 Part (i): If $Y$ is a real IRF$_k$ then $Y$ is also a real IRF$_{k+1}$ with trivial polynomial spectral characteristic. Indeed, since for all $\lambda \in\Lambda_{k+1}$,
 we have $\partial_{j_1,\cdots,j_d}^{k+1} \wh\lambda(0)=0$, Relation \eqref{e:IRF_representation} becomes
 \begin{align*}
 Y(\lambda) = \int_{\R^d} \frac{\wh \lambda(u)}{1\wedge\|u\|^{k+1}} \xi(du) =: \int_{\R^d} \frac{\wh \lambda(u)}{1\wedge\|u\|^{k+2}} \eta(du),
 \end{align*}
 where $\eta (du):= (1\wedge \|u\|^{k+2}) (1\wedge \|u\|^{k+1})^{-1} \xi(du)$.  By taking $\lambda\in\Lambda_{k+1}(\R)$, since $Y(\lambda) = \overline{Y(\lambda)}$ and
 $\overline{\wh \lambda(u)} = \wh \lambda(-u)$,  we obtain that 
 \begin{align*}
  Y(\lambda) = \int_{\R^d} \frac{\wh \lambda(u)}{1\wedge\|u\|^{k+2}} \overline{\eta(-du)},\ \ \lambda\in\Lambda_{k+1}(\R).
 \end{align*}
 By \eqref{e:Y-extension}, the last relation continues to hold for all complex $\lambda\in\Lambda_{k+1}$.  Hence, appealing to the uniqueness of the representation 
 \eqref{e:IRF_representation} of $Y$ viewed as an IRF$_{k+1}$ (with $k$ replaced by $k+1$ and $\xi$ by $\eta$),  we obtain
 \begin{align*}  
 \eta(du) = \overline{\eta(-du)}, \ \ \mbox{ almost surely},
 \end{align*}
 or equivalently $\xi(du) = \overline{\xi(-du)}$, a.s., which shows that the orthogonal measure $\xi$ is Hermitian.
 
 The fact that $\xi$ is Hermitian and $Y$ a real IRF$_k$, imply that 
 for all $\lambda\in\Lambda_k(\R)$,
 \begin{align*}
 Y(\lambda) - \int_{\R^d} \frac{\wh \lambda(u)}{1\wedge\|u\|^{k+1}} \xi(du) = \sum_{j_1,\cdots,j_d} \partial_{j_1,\cdots,j_d}^{k+1} \wh\lambda(0) Z_{j_1,\cdots,j_d}\ \  \mbox{ is real.}
 \end{align*}
 By taking suitable {\em real} $\lambda$'s for which $ \partial_{j_1,\cdots,j_d}^{k+1} \wh\lambda(0)$ vanish for all but each 
 one term in the sum, we obtain that all $(\ii)^{k+1} Z_{j_1,\cdots,j_d}$ must be real.
 
 Conversely, if $\xi$ is Hermitian and the $(\ii)^{k+1} Z_{j_1,\cdots,j_d}$'s are real, it is immediate that $Y(\lambda)$ is real for all 
 $\lambda\in\Lambda_k(\R)$.  This completes the proof of part (i). Part (ii) is an immediate consequence of (i).
 
 Part (iii): One can define $Y$ as in \eqref{e:Y-extension}, where all $(\ii)^{k+1} Z_{j_1,\cdots,j_d}$'s are real and $\xi$ Hermitian.  By part (i), this entails that
 $Y$ is a real IRF$_k$.  Suppose now that $\wt Y(\lambda) = a Y(\lambda) + \ii b Y^\prime(\lambda)$ and observe that
 \begin{align*}
 {\cal C}_{\wt Y}(\lambda,\mu) &= a^2\E[ Y(\lambda)\otimes Y(\mu) ] + b^2\E[ Y^\prime(\lambda)\otimes Y^\prime(\mu) ] 
 + ab \ii \Big[ \E  Y^\prime(\lambda) \otimes Y(\mu) - \E  Y(\lambda) \otimes Y^\prime(\mu)  \Big] \\
 & =  \E[ Y(\lambda)\otimes Y(\mu) ] = {\cal C}_{Y}(\lambda,\mu),
 \end{align*}
 since $a^2 + b^2 = 1$ and the cross-covariance terms cancel.   This shows that $\wt Y$ and $Y$ have the same covariance structure.
  \end{proof2}

\begin{lemma}\label{l:identification_operator}
	Let $T(x,y)$ be a bivariate mapping from $\mathbb{D}\times \mathbb{D}$ to $\mathbb{F}$, 
	where both $\mathbb{D}$ and $\mathbb{F}$ are linear spaces
	over $\mathbb{C}$. Assume that $T$ is sesquilinear form with $T(cx,y) = cT(x,y)$ 
	and $T(x,cy) = \bar c T(x,y), x, y\in D, c\in\mathbb{C}$. Then, for $x,y\in \mathbb{D}$,
	\begin{align*}
	T(x,y) = \frac{\ii-1}{2} (T(x,x)+T(y,y)) + {1\over 2} T(x+y,x+y)-\frac{\ii}{2} T(\ii x+y, \ii x+y).
	\end{align*}
	\end{lemma}
	\begin{proof} The proof is trivial by noticing the identities:
	\begin{align*}
	T(x+y,x+y) & = T(x,x) + T(y,y) + T(x,y) + T(y,x), \\
	T(\ii x+y,\ii x+y) & = T(x,x) + T(y,y) + \ii T(x,y) - \ii T(y,x).
	\end{align*}
	\end{proof}

\subsection{Proofs for Section \ref{s:operator-ss}.}

We start with an auxiliary result needed for the proof of Theorem \ref{thm:operator-ss} below.

\begin{lemma}\label{l:Frechet-derivative} For any ${\cal A}\in \bbT$ and any bounded linear operator ${\H}$, define 
$f(c):= c^{{-\H}} {\cal A} c^{{-\H}^*},\ c>0$.  The function $f:(0,\infty)\to \bbT$ is continuously Fr\'echet differentiable in 
$(\bbT,\|\cdot\|_{\rm tr})$ with derivative
\begin{align*}
f'(c) = -c^{-1} (\H f(c) + f(c)\H^*) =   -c^{-1} c^{{-\H}} (\H {\cal A} + {\cal A} \H^*) c^{{-\H}^*}, c>0.
\end{align*}
That is, 
\begin{align}\label{e:l:Frechet-derivative}
\| f(c+h) - f(c) - f'(c) h\|_{\rm tr} = o(h), \ \ \mbox{ as $h\to 0$.}
\end{align}
\end{lemma}
\begin{proof}{\clb  Observe that  $(c+h)^{-{\cal H}} -c^{-{\cal H}} = \sum_{n=1}^\infty (- {\cal H})^n [ (\log(c+h))^n - (\log c)^n]/n!$. 
Applying the mean value theorem to the terms $(\log(c+h))^n - (\log c)^n$, and using the triangle inequality for $\|\cdot\|_{\rm op}$,
one can } show that for all $c>0$, 
\begin{align}\label{e:Freceht-diff}
\Big\| (c+h)^{{-\H}} - c^{{-\H}} - c^{{-\H}} \ ({-\H}) c^{-1}h\Big\|_{\rm op} = o(h),\ \ \mbox{ as }h\to 0. 
\end{align}
That is,  $c\mapsto c^{{-\H}}$  and similarly $c\mapsto c^{{-\H}^*}$ are Fr\'echet differentiable in the Hilbert space ${\mathbb V}$.

On the other hand, by Proposition IV.5.4 on page 62 of \cite{Traces_and_Determinants_book}, for any two bounded operators ${\cal B}$ and ${\cal C}$ and a trace class operator $\A\in \bbT$, we have that ${\cal BAC}\in \bbT$ and moreover
\begin{align}\label{e:BAC-inequality}
\|{\cal BAC} \|_{\rm tr} \le \|{\cal B}\|_{\rm op} \|{\cal A}\|_{\rm tr} \|{\cal C}\|_{\rm op}.
\end{align}
This inequality can be used to show that the Fr\'echet differentiability of $c\mapsto c^{{-\H}}$ and $c\mapsto c^{{-\H}^*}$ 
in the operator norm induced by the Hilbert space norm in ${\mathbb V}$ entails the Fr\'echet differentiability in trace-norm.  

Indeed, for all $c>0$ and $h>-c$, we have
\begin{align*}
f(c+h) - f(c) &= (c+h)^{{-\H}} {\cal A} (c+h)^{{-\H}^*} - c^{{-\H}} {\cal A} c^{{-\H}^*} \\
& = ((c+h)^{{-\H}} - c^{{-\H}}) {\cal A} (c+h)^{{-\H}^*} + c^{{-\H}} {\cal A} ((c+h)^{{-\H}^*} - c^{{-\H}^*}).
\end{align*}
Now, in view of \eqref{e:Freceht-diff} and using the inequality \eqref{e:BAC-inequality}, we obtain
\begin{align*}
\| f(c+h) - f(c) - h c^{{-\H}-1}({-\H}) {\cal A} (c+h)^{{-\H}^*} - h(c+h)^{{-\H}} {\cal A} ({-\H}^*) c^{{-\H}^*-1}\|_{\rm tr} = o(h),
\end{align*}
as $h\to 0$.  The continuity of $c\mapsto c^{{-\H}}$ and $c\mapsto c^{{-\H}^*}$ in the operator norm and another application of
\eqref{e:BAC-inequality} entails \eqref{e:l:Frechet-derivative}.
\end{proof}

\label{sec:proofs:operator-ss}
\begin{proof2}{Theorem}{\ref{thm:operator-ss}}
  Fix $c>0$ and define the rescaled IRF$_k$ process $\wt Y(\lambda):= c^{{-\H}} Y(c\cdot \lambda),\ \lambda\in\Lambda_k$.
Observe that $Y$ is covariance ${\H}$-self-similar if and only if $\wt Y$ and $Y$ have the same covariance structure or, equivalently,
the same spectral characteristics $(\chi,{\cal Q})$. 

In view of  \eqref{e:IRF-cov-spectral}, for all $\lambda,\mu\in \Lambda_k$, we have
 \begin{align} \label{e:thm:operator-ss-1}
 \begin{split}
 \mathcal{C}_{\wt Y}(\lambda,\mu) & = \E [ \wt Y (\lambda) \otimes \wt Y(\mu)]  
 = c^{{-\H}} \E [ Y(c\cdot \lambda)\otimes Y(c\cdot \mu)] c^{{-\H}^*} \\
 & = \int_{\bbR^d} \wh{c\cdot\lambda}(u) \ol{\wh{c\cdot\mu}}(u) c^{{-\H}} \chi_k(du) c^{{-\H}^*} + c^{{-\H}} \mathcal{Q}(c\cdot\lambda*\wt{c\cdot\mu}) c^{{-\H}^*} \\
 & =  \int_{\bbR^d} \wh{\lambda}(x) \ol{\wh{\mu}}(x) c^{{-\H}} \chi_k(dx/c) c^{{-\H}^*} + c^{{-\H}} c^{2k+2} \mathcal{Q}( \lambda*\wt\mu) c^{{-\H}^*}, 
\end{split}
 \end{align}
 where in the last relation we used that $\wh{c\cdot\lambda}(u) = \wh{\lambda}(cu)$, the change of variables $u:= x/c$, and the
 fact that $\mathcal{Q}(c\cdot \nu)= c^{2k+2} \mathcal{Q}(\nu)$ for all $\nu\in \Lambda_{2k+1}$.
 
Relation \eqref{e:thm:operator-ss-1} shows that the spectral characteristics of 
$\wt Y$ are 
\begin{align*}
(c^{{-\H}} \chi_k(dx/c) c^{{-\H}^*}, c^{k+1{-\H}} \mathcal{Q}(\cdot) c^{k+1{-\H}^*}).
\end{align*}  
Hence, by the uniqueness of the spectral representation in \eqref{e:IRF-cov-spectral}, the IRF$_k$ process
 $Y$ and $\wt Y$ have the same covariance structure, if and only if Relation \eqref{e:thm:operator-ss-2} holds. 
 This completes the proof of part {\em (i)}. \\
 
 Part {\em (iv)} is an immediate consequence of part {\em (ii)} and Theorem \ref{th:integral_representation_IRF}, where the random 
 measure $\xi(dx)/(1\wedge |x|^{k+1})$ therein, is now written in polar coordinates as $W(dr d\theta)$.  
 Thus, in the remainder of the proof we focus on establishing the {\em disintegration formula} \eqref{e:chi-k-polar} (part {\em (ii)}) and
 the uniqueness of the measure $\sigma$ (part {\em (iii)}).\\
 
Define the $\bbT$-valued set-mapping
\begin{align}\label{e:sigma-formula}
\sigma(B):= {\H} \chi_k( (1,\infty)\times B) +  \chi_k( (1,\infty)\times B) {\H}^{*},\ \ B\in {\cal B}(\mathbb S),
\end{align}
where in short, we write $\chi_k(A\times B)$ for $\chi_k(\{(r,\theta)\in A\times B\})$, with
$A\subset (0,\infty)$ and $B\subset \bbS$.  The fact that $\chi_k$ is a $\bbT_+$-valued measure readily implies  
$\sigma(\emptyset)=0$ and the countable additivity of $\sigma$. Note also that $\sigma$ is finite, since 
for all $B\in {\cal B}(\bbS)$,  by \eqref{e:BAC-inequality},
\begin{align*}
\|\sigma(B) \|_{\rm tr} \le (\|{\H}\|_{\rm op}  + \|{\H}^*\|_{\rm op} ) \|  \chi_{k}((1,\infty)\times \bbS)\|_{\rm tr} <\infty.
\end{align*}

We will argue next that $\sigma(B)$ is positive and hence it defines a $\bbT_+$-valued measure on $\bbS$.  We will also 
show that for all $c>0$ and $B\in {\cal B}(\bbS)$, we have
\begin{align}\label{e:polar-decomposition}
\chi_k( (c,\infty)\times B ) = \int_{c}^\infty u^{{-\H}} \sigma(B) u^{{-\H}^*} \frac{du}{u}.
\end{align}
This fact and a standard $\pi-\lambda$ argument then entail that \eqref{e:chi-k-polar-new} holds.  Indeed, the right-hand side of
\eqref{e:chi-k-polar-new} defines a $\sigma$-finite $\bbT_+$-valued measure, say $\wt \chi_k$, on ${\cal B}(\R^d\setminus\{0\})$.  By 
Relation \eqref{e:polar-decomposition}, the measures $\chi_k$ and $\wt\chi_k$ agree on the semi-ring of rectangle sets 
$(c,\infty)\times B$. Since the latter generates the $\sigma$-field ${\cal B}(\R^d\setminus\{0\})$, by considering projections on
a fixed CONS, it can be seen that the two $\bbT_+$-valued measures coincide.

We now prove that $\sigma$ is a finite $\bbT_+$-valued measure and show \eqref{e:polar-decomposition}.  In view of 
Lemma \ref{l:Frechet-derivative} (above), it follows that the function
\begin{align*}
f(u) := u^{{-\H}} \chi_k((1,\infty)\times B) u^{{-\H}^*}. 
\end{align*}
is Fr\'echet continuously differentiable in $(\mathbb T,\|\cdot\|_{\rm tr})$ with derivative:
\begin{align} \label{e:polar-decomposition-1}
\begin{split}
f'(u)& = - u^{-1} u^{{-\H}} \Big({\H} \chi_k((1,\infty)\times B) +  \chi_k((1,\infty)\times B) {\H}^*\Big) u^{{-\H}^*} \\
&= -u^{-1} u^{{-\H}} \sigma(B)u^{{-\H}^*},\ \ u>0. 
\end{split}
\end{align}
Observe that by the operator-scaling property for $\chi_k$ in \eqref{e:thm:operator-ss-2}, we have
\begin{align*}
f(u) = \chi_k( (u,\infty)\times B),\ u>0.
\end{align*}
The monotonicity of the $\bbT_+$-valued measure $\chi_k$ then implies that
\begin{align*}
\chi_k((u,u+h]\times B) = f(u) - f(u+h) \in \bbT_+,
\end{align*}
for all $h>0$.  This shows that $-f'(u) \in \bbT_+$ and by setting $u=1$, we obtain $- f'(1) = \sigma(B) \in \bbT_+$, which shows that
$\sigma(B)\in \bbT_+$, completing the proof that the so-defined set-mapping in \eqref{e:sigma-formula} is a finite $\bbT_+$-valued measure. 

Now, using \eqref{e:polar-decomposition-1} and a straightforward extension of the fundamental theorem of calculus to 
Bochner integrals, we obtain
\begin{align*}
\int_{c}^\infty u^{{-\H}} \sigma(B) u^{{-\H}^*} \frac{du}{u} &= - \int_c^\infty f'(u)du = f(c) \\
& = c^{{-\H}} \chi_k((1,\infty)\times B) c^{{-\H}^*}.
\end{align*}
(Note that $f(u) = \chi_k((u,\infty)\times B) \downarrow 0$ as $u\uparrow \infty$.)
The latter, in view of the scaling property of $\chi_k$, equals $\chi_k((c,\infty)\times B)$, completing the
proof of \eqref{e:polar-decomposition} and part {\em (ii)}.\\

{\em Part (iii).} We now show that  $\sigma$ is uniquely determined by $\chi_k$, alone. 
By definition, we have 
\begin{align*}
\sigma(B) = -\frac{d}{dc}{\Big\vert}_{c=1} \chi_k( (c,\infty)\times B),\ \  B\in {\cal B}(\bbS),
\end{align*}
where the latter is interpreted as the Fr\'echet derivative in $(\bbT,\|\cdot\|)$ of the function $c\mapsto  \chi_k( (c,\infty)\times B)$, 
evaluated at $c=1$.  This shows that $\sigma$ is uniquely determined in terms of the measure $\chi_k$ and it does not depend on 
the choice of the exponent operator ${\H}$, which need not be unique (see e.g., Remark \ref{rem:non-unique-H}).
 \end{proof2}

\subsection{Proofs for Section \ref{sec:related-work-and-examples}.} \label{sec:proofs:related-work-and-examples}

\begin{proof2}{Proposition}{\ref{p:stationary-tangent-example}}
Stationarity is immediate, provided that the stochastic integral in \eqref{e:MRFk} is well-defined.  
To this end, it suffices to show that
\begin{align}\label{e:MRFk-1}
\begin{split}
\int_0^\infty\int_{\bbS} \|f_s(u,\theta)\|_{\rm op}^2 du\|\mu\|_{\rm tr} (d\theta) 
&:=   \int_0^\infty\int_{\bbS} (1\wedge u)^{2(k+1)}  \Big\| u^{-\H} \A(\theta)  \Big\|_{\rm op}^2 u^{-1} du \|\mu\|_{\rm tr}(d\theta) \\
&\le \int_{\bbS} \| \A(\theta)\|_{\rm op}^2 \|\mu\|_{\rm tr}(d\theta) \int_0^\infty  (1\wedge u)^{2(k+1)}  \| u^{-\H}\|_{\rm op}^2 u^{-1} du.
\end{split}
\end{align}

By the assumption \eqref{e:MRFk-A-condition}, the first integral in \eqref{e:MRFk-1} is finite.  It remains to show that the second one therein is
also finite.  {\clb This, however, readily follows from the inequality
\begin{align}\label{e:uniform-bound-H}
(1\wedge u)^{2(k+1)}  u^{-1}  \| u^{-\H}\|_{\rm op}^2  \le C\times \Big( 1_{(0,1)}(u) u^{2\delta-1} + 1_{[1,\infty)}(u) u^{-2\epsilon-1} \Big),
\end{align}
where $\epsilon>0$ is as in \eqref{e:MRFk-H-condition} and $\delta>0$ is such that $\Re({\rm sp}(\H)) \subset (\epsilon, k+1-\delta)$. 
Relation  \eqref{e:uniform-bound-H} can be established exactly as in the proof of \eqref{e:cH-epsilon} using \eqref{e:MRFk-H-condition} 
and Riesz functional calculus. This completes the proof of \eqref{e:uniform-bound-H} and part {\em (i)}.}

\medskip
To prove part {\em (ii)}, it suffices to establish that, for all $s,t\in B(0,M),\ M>0$,
\begin{align}\label{e:diff-bound}
\E[ \|X(s) - X(t)\|^2] \le \int_0^\infty \int_{\bbS}  \| f_s(u,\theta) - f_{t}(u,\theta)\|_{\rm op}^2 du \|\mu\|_{\rm tr}(d\theta)  \le C \| s-t\|^{2\zeta}.
\end{align}
Indeed, consider the Gaussian $\V$-valued variables $\xi_{s,t} := (X(s)-X(t))/\sigma_{s,t}$, where $\sigma_{s,t}^2:= \E [ \| X(s)-X(t)\|^2]$,
and where by convention $\xi_{s,t}:=0$ if $\sigma_{s,t} =0$.  By 
Corollary  \ref{c:uniform-tail-boubnds-for-Gaussian-norms-in-a-Hilnbert-space}  
%Corollary S.5.4 % \ref{c:uniform-tail-boubnds-for-Gaussian-norms-in-a-Hilnbert-space}  in \cite{shen:stoev:hsing:2020_extended} 
(with $\theta:=1/4$ therein) we have that
\begin{align*}
 \E [\|\xi_{s,t} \|^{p}] \le c_p \E \exp\{\|\xi_{s,t}\|^2/4\} \le c_p \sqrt{2} 
\end{align*}
for all $p>0$ and some finite universal constant $c_p$.  The last bound and Relation \eqref{e:diff-bound}
can be equivalently written as
\begin{align*}
\E [\|X(s)-X(t)\|^p] \le C_p\sqrt{2} \|s-t\|^{p\zeta}.
\end{align*}
This, in view of Proposition \ref{p:tightness-via-moments}, implies the existence of a $\gamma$-H\"older continuous version of  
$\{X(s),\ s\in B(0,M)\}$ for all $\gamma \in (0,\zeta-d/p)$. Taking $p$ large, we see that every $\gamma\in (0,\zeta)$ is a possible H\"older 
exponent.  We shall continue to denote this continuous-path version of the process by $\{X(s)\}$. \\

Now, we turn to proving \eqref{e:diff-bound}. We have
\begin{align}\label{e:diff-f-bound}
\|f_s(u,\theta) - f_{t}(u,\theta)\|_{\rm op}^2& \le  2 |e^{\ii s^\top u\theta} - e^{\ii t^\top u\theta}|^2  
(1\wedge u)^{2(k+1)} u^{-1}  \| \A(\theta)\|_{\rm op}^2   \| u^{-\H}\|_{\rm op}^2 
\end{align}
By Relations \eqref{e:MRFk-A-condition} and \eqref{e:uniform-bound-H}, we have
\begin{align*}
I_{s,t}&:= \int_0^\infty \int_{\bbS} \|f_s(u,\theta) - f_{t}(u,\theta)\|_{\rm op}^2 du\|\mu\|_{\rm tr}(d\theta) \\
& \le C  \int_0^\infty \sup_{\theta\in \bbS} |e^{\ii (s-t)^\top u\theta} -1|^2 \Big (u^{2(\replace{k+1-\|\H\|_{\rm op}}{\delta} )-1} 1_{[0,1]}(u) + u^{-2\epsilon -1} 
1_{[1,\infty)}(u) \Big) du.
\end{align*}
Thus, Lemma \ref{l:term-I1} applied with $\Delta:= (s-t)^\top\theta$, $\gamma:= 2\replace{(k+1-\|\H\|_{\rm op})}{\delta}>0$ and $\epsilon:= 2\epsilon$, yields
\begin{align*}
I_{s,t} \le C  \|s-t\|^{2 (\epsilon\wedge 1)} (  1+ |\log(\|s-t\|)|^2
 1_{\{ \epsilon=1\}}).
\end{align*}
This implies that \eqref{e:diff-bound} holds with any $\zeta < 1\wedge \epsilon$.\\

{\em Part (iii).} Consider the measures in \eqref{e:eval1} and observe that 
\begin{align}\label{e:lambda-t-recall}
\lambda_t = \delta_t - \sum_j c_j(t) \delta_{t_j},
\end{align} 
where the $c_j(t)$'s are polynomials in $t$ of degrees up to $k$ and the $t_j$'s are some fixed points in $\R^d$.   In view of 
Proposition \ref{p:tightness-via-moments}, to prove \eqref{e:tangent-field-example} it is enough to work with 
\begin{align*}
\breve X_r(t) &:= r^{-\H} X(s_0+r\cdot \lambda_t) \\
& = r^{-\H} \int_0^\infty\int_{\bbS} e^{\ii s_0^\top u\theta} \Big( e^{\ii rs^\top u\theta} - \sum_j c_j(s) e^{\ii rt_j^\top u\theta} \Big) (1\wedge u)^{k+1} u^{-\H-1/2} \A(\theta) W(du,d\theta)\\
& =  \int_0^\infty\int_{\bbS} \wh \lambda_t(ru\theta) (1\wedge u)^{k+1} (ru)^{-\H-1/2} \A(\theta) \widetilde W(du,d\theta),
\end{align*}
where $\widetilde W(du,d\theta) := e^{\ii u s_0^\top \theta} r^{1/2} W(du,d\theta)$.   {\clb We will show that since the $\V$-valued Gaussian 
random measure $W$ is circularly symmetric and self-similar (recall \eqref{e:circular-symmetry}), we have 
\begin{equation}\label{e:tilde-W=Wr}
 \{\widetilde W(du,d\theta)\} \stackrel{fdd}{=} \{W(d(r\cdot u),d\theta)\}.
\end{equation}  
Indeed, since $W$ is standard Gaussian, it is not only orthogonal but also an independently scattered measure.  This readily
implies that $\wt W$ is also independently scattered.  It is straightforward that $\wt W$ has
control measure $r du \mu(d\theta)$.  It remains to show that $\wt W$ is standard, i.e., its real and imaginary components are independent and
identically distributed. Indeed, for all bounded Borel
$A\subset \R^d\setminus\{0\}$, in view of \eqref{e:W-WR-WI}, we have
\begin{align*}
\wt W(A) &= \int_A \cos(\varphi_{u,\theta}) r^{1/2} W_\R(du,d\theta) - \int_A \sin(\varphi_{u,\theta})r^{1/2} W_{\mathbb I}(du,d\theta) \\
& \ \ \ + \ii \Big( \int_A \sin(\varphi_{u,\theta}) r^{1/2} W_\R(du,d\theta) + \int_A \cos(\varphi_{u,\theta}) r^{1/2} W_{\mathbb I}(du,d\theta)\Big)\\
& =: \Re(\wt W(A)) + \ii \Im(\wt W(A)),
\end{align*} 
where $\varphi_{u,\theta} := u s_0^\top \theta$.
Using the fact that $W_\R$ and $W_{\mathbb I}$ are independent and identically distributed real orthogonal measures, we obtain that 
$\Re(\wt W(A))$ and $\Im(\wt W(A))$ are independent and identically distributed.  This, since $\wt W$ is independently scattered and Gaussian, 
completes the proof of \eqref{e:tilde-W=Wr}. } 

Therefore, by \eqref{e:tilde-W=Wr},
\begin{align} \label{e:star}
\begin{split}
\{\breve X_r(t),\ t\in \R^d\} &\stackrel{d}{=}\Big\{\int_0^\infty\int_{\bbS} \wh \lambda_t(v\theta) (1\wedge (v/r))^{k+1} v^{-\H-1/2} \A(\theta) 
W(dv,d\theta),\ t\in\R^d \Big\} \\ 
&=: \{ X_r(t),\ t\in \R^d\},  
\end{split}
\end{align}
where we made the change of variables $v:= ru$.  

This relation readily implies that the convergence in \eqref{e:tangent-field-example} holds in the sense of finite-dimensional distributions.
Indeed, writing $\breve Y(t) := \int_0^\infty \int_{\bbS} \wh \lambda_t(v\theta) v^{-\H-1/2} \A(\theta) 
W(dv,d\theta)$, with the {\em same} measure defining the $X_r(t)$'s in \eqref{e:star}, it is enough to show that 
\begin{align}\label{e:Y-Xr}
\begin{split}
& \E \| X_r(t) - \breve Y(t)\|^2 \\
& \le \int_0^\infty \int_{\bbS} |\wh\lambda_t(v\theta)|^2 |\cdot  |(1\wedge (v/r))^{k+1} - 1 |^2\cdot  \Big\| v^{-\H-1/2} \A(\theta) \Big\|_{\rm op}^2 \|\mu\|_{\rm tr}(d\theta) dv \to 0,
\end{split}
\end{align}
as $r\to 0$.  It is easy to see that since $\lambda_t\in \Lambda_k$, we have
$
| \wh\lambda_t(x)| \le C (1\wedge \|x\|)^{k+1},\ \ x\in \R^d
$
Therefore, in view of Relations \eqref{e:MRFk-A-condition} and \eqref{e:uniform-bound-H}, 
the fact that $|(1\wedge (v/r))^{k+1} - 1 |\to 0,\ r\to 0$ and the Dominated Convergence Theorem imply \eqref{e:Y-Xr},
proving the convergence of the finite-dimensional distributions.  To complete the proof of \eqref{e:tangent-field-example}, we will establish the
tightness of $\{\breve X_r(\cdot),\ r\in (0,1)\}$.  Since the $\breve X_r$'s are Gaussian, as argued above, in view of Proposition 
\ref{p:tightness-via-moments} it is enough to show that for all $M>0$, there exist $C>0$ and $\zeta>0$ such that
\begin{align*}
\sup_{r\in (0,1)} \E \| \breve X_r(s) - \breve X_r(t)\|^2 \le C \|s-t\|^{2\zeta},\ \ \mbox{ for all } s,t\in B(0,M).
\end{align*}
{\clb This is because for $k\ge 0$, the tightness condition on $\{\breve X_r(s_0),\ r\in(0,1)\}$ is automatically fulfilled if one takes $s_0=t_1$ (recall \eqref{e:tilde_f2} and
 Remark \ref{rem:tightness_at-a-point})}.

To this end, we begin with some key observations about the measures $\lambda_t$.   
Since the $c_j(t)$'s in \eqref{e:lambda-t-recall} are fixed polynomials, we have
\begin{align}\label{e:lambda-hat-1}
\begin{split}
| \wh \lambda_s(v\theta) - \wh\lambda_t(v\theta)|^2 &\le | e^{\ii s^\top v\theta} - e^{\ii t^\top v\theta } |^2 + \sum_j |c_j(s) - c_j(t)|^2 \\
&\le C v^2\| s-t\|^2 + C_M \| s-t\|^2  \le C_M (1\vee v^2) \|s-t\|^2,
\end{split}
\end{align}
for all $v>0$ and $s,t\in B(0,M)$, where the constant $C_M$ does not depend on $v,s,t$ and $\theta\in\bbS$.

On the other hand, in view of \eqref{e:lambda-t-recall}, the $\hat\lambda_t$'s are uniformly bounded for $t\in B(0,M)$ and since 
$\lambda_t\in\Lambda_k$ they annihilate polynomials of degree up to $k$ and hence
$\wh \lambda_t(v\theta) = {\cal O}(v^{k+1})$. This implies that
\begin{align}\label{e:lambda-hat-2}
\sup_{t,s\in B(0,M), \ \theta \in \bbS} |  \wh \lambda_s(v\theta) - \wh\lambda_t(v\theta)|^2\le C_M(1\wedge v^{2(k+1)}).
\end{align}

We are now ready to estimate the difference moments.    For all $s,t\in B(0,M)$, we have
\begin{align*}
\E \| \breve X_r(s) - \breve X_r(t)\|^2 &\le \int_0^\infty\int_{\bbS} | \wh \lambda_s(v\theta) - \wh\lambda_t(v\theta)|^2  \Big\| v^{-\H-1/2} \A(\theta) \Big\|_{\rm op}^2 \|\mu\|_{\rm tr}(d\theta) dv\\
&=:\int_{0}^\infty \int_{\bbS} g_{s,t}(v,\theta) \|\mu\|_{\rm tr} (d\theta) dv.
\end{align*}

Observe first that by \eqref{e:lambda-hat-2} the bound in  \eqref{e:uniform-bound-H} applies and by  \eqref{e:MRFk-A-condition}, we have
\begin{align}\label{e:h-bound}
\int_{\bbS} g_{s,t}(v,\theta) \|\mu\|_{\rm tr} (d\theta) \le C \Big ( v^{2\replace{(k+1-\|\H\|_{\rm op})}{\delta}-1} 1_{(0,1)}(v) + v^{-2\epsilon -1} 1_{[1,\infty)}(v)\Big),
\end{align}
where the latter function is integrable in $v$ over $(0,\infty)$.  

The rest of the strategy is as follows.  We will consider the integral of $g_{s,t}$ over three regions $(v,\theta)\in (0,\alpha]\times \bbS$, 
$(\alpha,1/\alpha)\times \bbS$ and $[1/\alpha,\infty)\times \bbS$.  We will choose $\alpha = 1\wedge \|s-t\|^\kappa$ for some $\kappa>0$,
such that each of the three integrals can be dominated by $C \alpha^\eta$, for some positive $\eta$.  Namely, let
\begin{align*}
I_1(s,t) + I_2(s,t) + I_3(s,t) := \Big(\int_{(0,\alpha]\times \bbS} + \int_{(\alpha,1/\alpha)\times\bbS} + \int_{[1/\alpha,\infty)\times \bbS} \Big) 
g_{s,t}(v,\theta) \|\mu\|_{\rm tr}(d\theta) dv.
\end{align*}
In view of \eqref{e:h-bound}, we have
\begin{align*}
I_1(s,t) \le C \int_0^\alpha  v^{2\replace{(k+1-\|\H\|_{\rm op})}{\delta}-1} dv \le  C \alpha^{ 2\replace{(k+1-\|\H\|_{\rm op})}{\delta}}
\end{align*}
and
\begin{align*}
I_3(s,t) \le C\int_{1/\alpha}^\infty v^{-2\epsilon-1} dv \le C \alpha^{2\epsilon}.
\end{align*}
Now, for the middle piece, using the bound in \eqref{e:lambda-hat-1}, we obtain
\begin{align*}
I_2(s,t) \le C \|s-t\|^2  \int_{\alpha}^{1/\alpha} (1\vee v^2) dv \le C \|s-t\|^2 \alpha^{-3}.
\end{align*}
Setting $\|s-t\|^2 \alpha^{-3} = \alpha^{2(\epsilon \wedge \delta)}$, we see that
$
\alpha = 1 \wedge \|s-t\|^{2/(3+2(\epsilon\wedge \delta ))}
$
yields the desired bound
$
 \E \| \breve X_r(s) - \breve X_r(t)\|^2\le \sum_{i=1}^3 I_i(s,t) \le C \|s-t\|^{2\zeta},
$ 
uniformly in $r\in (0,1)$, where  $\zeta = 2(\epsilon\wedge \delta)/(3+2(\epsilon\wedge \delta))>0$.
\end{proof2}

\begin{lemma} \label{l:term-I1} {\em (i)} There exists a constant $C$ such that 
\begin{align*}
\int_1^\infty | e^{\ii \Delta u} - 1 |^2 u^{-1-\epsilon} du 
\le C \left\{\begin{array}{ll}
|\Delta|^{2} &,\ \mbox{ if }\epsilon>2\\
|\Delta|^2 |\log|\Delta|| &,\ \mbox{ if }\epsilon=2\\
|\Delta|^{\epsilon} &,\ \mbox{ if }0<\epsilon<2.
\end{array} \right.
\end{align*}

{\em (ii)} Also, for any $\gamma>0$, there is a constant $C_\gamma>0$, such that
$
\int_0^1  | e^{\ii \Delta u} - 1 |^2 u^{\gamma-1} du \le C_\gamma |\Delta|^2.
$
\end{lemma}
\begin{proof} Noting that $| e^{\ii \Delta u} - 1 |^2 = 2(1 - \cos(\Delta u)) = {\cal O}(\Delta^2),$ as $\Delta\to 0$, 
the claim in part {\em (ii)} is immediate.  Part {\em (i)} follows by straightforward calculus by considering the change of variables 
$x:= |\Delta| u$.
\end{proof}

\subsection{Proofs for Section \ref{ss:T_rrx}.}\label{sec:proofs:scalar-actions}

\begin{proof2}{Proposition}{\ref{proposition:GC}}  Theorem \ref{thm:operator-ss} applied to ${\cal H}:= H\cdot \I$ implies
the scaling property of $\chi_k$ and the disintegration formula \eqref{e:chi_k-disintegration-scalar} in part {\em (ii)}.  Part {\em (iii)}
follows from Proposition \ref{p:general-bounded-H}.  We need only prove that $0<H\le k+1$ and the dichotomy claim {\em (i)}.

Recall the decomposition in \eqref{e:Y-decomposition} and observe that both components
$Y_{(\chi,0)}$ and $Y_{(0,{\cal Q})}$ are covariance $H$-self-similar.  Notice, however, that $Y_{(0,{\cal Q})}$ is always 
either $(k+1)$-self-similar or zero. In particular, for all $c>0$,
\begin{align*}
c^H Y_{(0,{\cal Q})}(\lambda) \stackrel{d}{=}  Y_{(0,{\cal Q})}(c\cdot \lambda) \stackrel{d}{=} 
 c^{k+1} Y_{(0,{\cal Q})}(\lambda),\ \ \  \lambda\in \Lambda_k
\end{align*}
This, since $H$ is scalar, implies $H= k+1$ unless $Y_{(0,{\cal Q})} \equiv 0$ (recall Lemma \ref{l:Slutsky}).  

Thus, the polynomial spectral component ${\cal Q}$ is non-trivial only if $H=k+1$.  To complete the proof of {\em (i)}, it remains to show
that if $Y_{(\chi,0)}$ is non-zero, then its self-similarity exponent is in the range  $0<H < k+1$.  Firstly, note that 
$Y_{(\chi,0)}(c\cdot \lambda) \stackrel{d}{=} c^{H} Y_{(\chi,0)} (\lambda) \to 0$, in probability, as $c\downarrow 0$, by 
the continuity of $Y$.  This implies that $H>0$.  

Now, observe that by the scaling property of $\chi_k$, we have, for all $r\in (0,1)$, 
\begin{align*}
r^{-2H} \chi_k( (1,\infty)\times \bbS) &= \chi_k((r,\infty)\times \bbS) = \int_{\{r<\|u\| \} } \frac{1}{1\wedge \|u\|^{2k+2}} \chi(du) \\
&=
\int_{\{ r< \|u\|\le 1\}} \frac{1}{\|u\|^{2k+2}} \chi(du) + \chi(\{ \|u\|>1\} )\\
&= r^{-(2k+2)} \int_{\{r< \| u\|\le 1\}} \Big(\frac{r}{\|u\|} \Big)^{2k+2} \chi(du) + \chi(\{ \|u\|>1\}).
\end{align*}
Since $\chi$ is a finite $\bbT_+$-valued measure, however, the last integral vanishes as $r\downarrow 0$, by the Dominated Convergence Theorem.  Thus, by multiplying the last expression by $r^{2k+2}$, we obtain
\begin{align*}
r^{2( k+1 -H)}  \chi_k( (1,\infty)\times \bbS)  \to 0,
\end{align*}
as $r\to 0$, which means that $H<k+1$, since by the scaling property $\chi_k((r,\infty)\times \bbS)\not = 0$, for any (all) $r>0$.  
This completes the proof.
\end{proof2}

\begin{proof2}{Theorem} {\ref{p:K_nu_polar}}  

The proof of {\em Part (i)} is as follows. Since $H<k+1$, the polynomial ${\mathcal Q}$ in the spectral characteristic 
of ${\cal K}$ is zero and $\chi$ satisfies \eqref{e:chi_k-disintegration-scalar} in polar coordinates.  
Let $\nu\in\Lambda_{2k+1}$.  Using that
$
\wh \nu (r\theta ) = \int_{\mathbb R^d}  e^{i r \theta^\top t }\nu(dt), 
$
Relation \eqref{e:K-of_nu} in polar coordinates becomes
\begin{align}\label{e:K(nu)-polar}
\begin{split}
{\cal K}(\nu) &= \int_{\mathbb S^{d-1}}\int_0^\infty \int_{\R^d} e^{\ii r \theta^\top t} \nu(dt) r^{-(2H+1)} dr \sigma(d\theta) \\
& = \int_{\mathbb S^{d-1}} \int_0^\infty \int_{\R^d} 
\left(  {\Big(} e^{\ii r  \theta^\top t } - \sum_{j=0}^{\lfloor 2H\rfloor}
        \frac{(\ii r \theta^\top t)^j}{j!} {\Big)}  \nu(dt) \right) r^{-(2H+1)} dr \sigma(d\theta) \\
& = \int_{\mathbb S^{d-1}}  \int_{\R^d} \int_0^\infty
\left(  {\Big(} e^{\ii r   \theta^\top t } - \sum_{j=0}^{\lfloor 2H\rfloor}
        \frac{(\ii r  \theta^\top t )^j}{j!} {\Big)}  r^{-(2H+1)} dr  \right) \nu(dt) \sigma(d\theta)
\end{split}
\end{align}
where the second relation follows from the fact that $\nu\in \Lambda_{2k+1}$ and therefore we could add a polynomial in $t$ of degree 
$\lfloor 2H\rfloor \le 2k+1$ without changing the integral.  We will justify next the interchange of the inner two integrals leading to
\eqref{e:K(nu)-polar} and compute the inner integral, therein. 

For all integer $m\ge 0$, one can show that
\begin{align}\label{e:exp-expansion}
\Big| e^{\ii z} - \sum_{j=0}^m \frac{(\ii z)^j}{j!} \Big| \le  \min\Big\{ \frac{|z|^{m+1}}{(m+1)!}, \frac{2 |z|^{m}}{m!} \Big\},\quad z\in\R
\end{align}
\citep[see e.g., page 298 in][] {resnick:1999book}. Thus, for $0\le \lfloor 2H\rfloor<  2H+1 < \lfloor 2H\rfloor +1 $ (recall $2H$ is not integer),
we obtain
\begin{align*}
\int_0^\infty \Big| e^{\ii r \theta^\top t } - \sum_{j=0}^{\lfloor 2H\rfloor} \frac{(\ii r \theta^\top t)^j}{j!} \Big| r^{-(2H+1)}dr  <\infty.
\end{align*}
This allows us to interchange the order of the two inner integrals in \eqref{e:K(nu)-polar}. 
 Doing so and making the change of variables  $z:= r |  \theta^\top t |$, we obtain
\begin{align*}
& \int_0^\infty {\Big(} e^{\ii r \theta^\top t } - \sum_{j=0}^{\lfloor 2H \rfloor} \frac{(\ii r \theta^\top t)^j}{j!} \Big) r^{-(2H+1)}dr \\
\quad & =  |\theta^\top t  |^{2H} \int_0^\infty {\Big(} e^{\ii z\,  {\rm sign}(\theta^\top t) }   
- \sum_{j=0}^{\lfloor 2H \rfloor} \frac{(\ii z\, {\rm sign} (\theta^\top t) )^j}{j!} \Big) z^{-(2H+1)}dz\\
\quad & = | \theta^\top t  |^{2H} \int_0^\infty \Big(\cos(z) - \sum_{0\le 2j\le \lfloor 2H \rfloor,\ j\in\mathbb N } \frac{(-1)^{j} z^{2j}}{2j!}\Big) z^{-(2H+1)} dz\\
\quad & + { \ii }\, {\rm sign} (\theta^\top t)  | \theta^\top t |^{2H}  \int_0^\infty \Big(\sin(z) 
     - \sum_{1\le 2j+1\le \lfloor 2H \rfloor,\ j\in\mathbb N} \frac{(-1)^{j} z^{2j+1}}{(2j+1)!}\Big) z^{-(2H+1)} dz\\
\quad & =  | \theta^\top t |^{2H}  I({H}) + \ii  (\theta^\top t)^{<2H>} J({H}),
\end{align*}
where $I({H})$ and $J({H})$ are in \eqref{e:I_and_J}.   Note that this argument also demonstrates
that $I({H})$ and $J({H})$ are well-defined.

By substituting the last expression in the right-hand side of \eqref{e:K(nu)-polar}, we obtain
\begin{align*}
{\cal K}(\nu) &= I({H}) \int_{\mathbb S^{d-1}} \int_{\mathbb R^d} |  \theta^\top t |^{2H} \nu(dt) \sigma(d\theta) \\
&\quad\quad\quad + \ii J({H}) \int_{\mathbb S^{d-1}} \int_{\mathbb R^d} ( \theta^\top t)^{<2H>} \nu(dt) \sigma(d\theta)\\
&= I({H}) \int_{\mathbb S^{d-1}} |(\theta,\cdot)|^{2H} (\nu) \sigma(dt) + \ii J({H})  \int_{\mathbb S^{d-1}} (\theta,\cdot)^{<2H>} (\nu) \sigma(dt). 
\end{align*}
This completes the proof of \eqref{e:p:K_nu_polar}.

\medskip
{\em Part (ii):} Suppose now that $2H \in \Z$, where $1\le 2H \le  2k+1$ and observe that  $2H$ may be either even or odd.
Using the fact that $\nu\in\Lambda_{2k+1}$, we get
\begin{align*}
\widehat \nu(r\theta) &= \int_{\bbR^d}{\Big(}  e^{\ii r \theta^\top t} - \sum_{j=0}^{2H-1} \frac{(\ii r \theta^\top t )^j}{j!}
  - 1_{[0,1]}(r) \frac{(\ii r \theta^\top t)^{2H}}{(2H)!}  \Big) \nu(dt) \\
  &=: \int_{\bbR^d} f(r\theta^\top t,\, r) \nu(dt).
\end{align*}
Relation \eqref{e:exp-expansion} implies that $|f(r\theta^\top t,\, r)| = {\cal O}( |r|^{2H+1})$ as $r\to 0$ and
because of the presence of the indicator function $1_{[0,1]}(r)$, we have $|f(r\theta^\top t,\, r)| = {\cal O}(|r|^{2H-1})$ as $r\to\infty$.
Therefore, 
\begin{align*}
\int_0^\infty | f(r\theta^\top t,\, r)| r^{-(2H+1)} dr <\infty,
\end{align*}
and by Fubini, we obtain
\begin{align}\label{e:p:K_nu_polar-integer-case}
 & \int_0^\infty \widehat \nu(r\theta) r^{-(2H+1)} dr  = \int_{\bbR^d} \int_0^\infty f(r\theta^\top t,\, r) r^{-(2H+1)} dr  \nu(dt) \nonumber \\
 & = \int_{\bbR^d} \Big[ |\theta^\top t|^{2H}
 \int_0^\infty \Big( e^{\ii z \kappa } - \sum_{j=0}^{2H-1} \frac{ (\ii z\kappa)^{j}}{j!} 
    - 1_{[0,1]}(z/|\theta^\top t|) \frac{(\ii z\kappa)^{2H}}{(2H)!} \Big) z^{-(2H+1)} dz \Big] \nu(dt)\\
    &=: \int_{\bbR^d} I_{H}(t)\, \nu(dt) \nonumber
\end{align}
where we made the change of variables $z:= r|\theta^\top t|$ and where $ \kappa := {\rm sign}(\theta^\top t)$. 

{\em We now consider two cases.}  Suppose fist that $2H$ is even. Upon separating the real and imaginary parts
the integral in the r.h.s.\ of \eqref{e:p:K_nu_polar-integer-case} becomes:
\begin{align}\label{e:p:K_nu_polar-integer-case-1}
I_{H}(t)  & =  | \theta^\top t |^{2H} \int_0^\infty \Big( \cos(z) - \sum_{0\le 2j\le 2H-1,\ j\in\Z} \frac{ (-1)^j z^{2j}}{(2j)!} 
- 1_{[0,1]}(z/|\theta^\top t|) \frac{(-1)^{H} z^{2H}}{(2H)!} \Big) z^{-(2H+1)} dz  \nonumber\\
& +\ii\,  \kappa |\theta^\top t|^{2H}\int_0^\infty \Big( \sin(z) - \sum_{1\le 2j+1\le 2H-1,\ j\in\Z} \frac{ (-1)^j z^{2j+1}}{(2j+1)!} 
 \Big) z^{-(2H+1)} dz\nonumber\\
 & =:  (\theta^\top t)^{2H} I_{H,\cos}(t) + \ii  (\theta^\top t)^{<2H>} J({H}).
 \end{align}
Observe that the integral in the imaginary part above equals $J({H})$ in \eqref{e:I_and_J} and where in the real part we dropped the absolute
value around $\theta^\top t$ since $2H$ is even.

Since $ (\theta^\top t)^{2H}$ is a polynomial in $t$ of degree $2H<2k+1$, we have $\int_{\bbR^d}  (\theta^\top t)^{2H} \nu(dt) =0$.
Therefore, 
\begin{align}\label{e:p:K_nu_polar-integer-case-2}
\int_{\bbR^d} (\theta^\top t)^{2H} I_{H,\cos}(t) \nu(dt) = \int_{\bbR^d} (\theta^\top t)^{2H}  ( I_{H,\cos}(t) - \mathcal{C}_{H,\cos}  ) \nu(dt), 
\end{align}
where 
\begin{align*}
\mathcal{C}_{H,\cos} :=  \int_0^\infty \Big( \cos(z) - \sum_{0\le 2j\le 2H-1,\ j\in\Z} \frac{ (-1)^j z^{2j}}{(2j)!} 
- 1_{[0,1]}(z) \frac{(-1)^{H} z^{2H}}{(2H)!} \Big) z^{-(2H+1)} dz,
\end{align*}
which is almost the same as the first integral in \eqref{e:p:K_nu_polar-integer-case-1}
except that the indicator function $1_{[0,1]}(z)$ no longer depends on $t$.  Notice again that 
$\mathcal{C}_{H,\cos}$ is well-defined.  We thus obtain
\begin{align}\label{e:p:K_nu_polar-integer-case-3}
\begin{split}
 I_{H,\cos}(t) - \mathcal{C}_{H,\cos} &= \int_{0}^\infty \Big(1_{[0,1]}(z) - 1_{[0,|\langle \theta,t\rangle|]}(z) \Big)  \frac{(-1)^{H} z^{2H}}{(2H)!} z^{-(2H+1)} dz \\
 & =\frac{(-1)^{H}}{(2H)!} \int_{|\theta^\top t |}^1  z^{-1} dz = \frac{(-1)^{H+1}}{(2H)!} \log (|\theta^\top t|).
\end{split}
\end{align}
In view of \eqref{e:p:K_nu_polar-integer-case-1}, \eqref{e:p:K_nu_polar-integer-case-2}, and \eqref{e:p:K_nu_polar-integer-case-3},
we get
\begin{align*}
{\cal K}(\nu) &= \int_{\mathbb S^{d-1}}\int_{\bbR^d} I_{H}(t)\nu(dt) \sigma(d\theta)\\
& = \frac{(-1)^{H+1}}{(2H)!}  \int_{\mathbb S^{d-1}}\int_{\bbR^d} (\theta^\top t) ^{2H} \log | \theta^\top t | \nu(dt) \sigma(d\theta) \\
&\quad\quad + \ii\, J({H}) \int_{\mathbb S^{d-1}}\int_{\bbR^d} (\theta^\top t) ^{<2H>}  \nu(dt) \sigma(d\theta)\\
& = \frac{(-1)^{H+1}}{(2H)!}  \int_{\mathbb S^{d-1}} \Big[ (\theta, \cdot ) ^{2H} \log |(\theta,\cdot) | \Big] (\nu) \sigma(d\theta)
+ \ii\, J({H}) \int_{\mathbb S^{d-1}} ( \theta,\cdot)^{<2H>}(\nu)  \sigma(d\theta),
\end{align*}
which completes the expression in the case when $2H$ is an even integer.\\

{\em Suppose now that $2H$ is odd.}  With a very similar argument, using the fact that $\nu\in\Lambda_{2k+1}$ annihilates polynomials
of degree $2k+1\ge 2H$, we obtain that the right-hand side of \eqref{e:p:K_nu_polar-integer-case} equals
\begin{align}\label{e:2H-odd}
\begin{split}
 \int_{\R^d} |\theta^\top t|^{2H} I(H) \nu(dt) +&
 \ii \int_{\R^d} \kappa |\theta^\top t|^{2H}  \int_0^\infty \Big( \sin(z) - \sum_{1\le 2j+1 \le 2H-1} \frac{(-1)^j z^{2j+1}}{(2j+1)!} \\
 &\quad\quad\quad\quad \quad\quad\quad\quad
- 1_{[0,1]}(z/|\theta^\top t|) \frac{(-1)^{H-1/2} z^{2H}}{(2H)!} \Big) z^{-(2H+1)}dz \nu(dt) \\
 = I(H) |(\theta,\cdot)|^{2H}(\nu) + &\ii   \int_0^\infty \Big[ \int_{\R^d} (\theta^\top t)^{2H} (1_{[0,1]}(z) - 1_{[0,1/|\theta^\top t|]}(z)) \frac{(-1)^{H-1/2}z^{-1}}{(2H)!} \nu(dt) \Big] dz,
\end{split}
\end{align}
where we used the fact that  $\kappa |\theta^\top t|^{2H} = (\theta^\top t)^{2H}$ and the last relation is obtained with the same 
strategy as in \eqref{e:p:K_nu_polar-integer-case-2}.  More precisely, 
applying Fubini and using that $\int_{\R^d} (\theta^\top t)^{2H} \nu(dt) = 0$, allows us to eliminate the terms involving 
$\sin(z)$ and $z^{2j+1}$.  At the same time, we add the term $(\theta^\top t)^{2H} 1_{[0,1]}(z) z^{-1}$, which is a 
polynomial in $t$ also annihilated by $\nu$.  

Now, since the inner integrand in the r.h.s.\ of  \eqref{e:2H-odd} is integrable with respect to $z$, another application of Fubini 
shows that right-hand side of \eqref{e:p:K_nu_polar-integer-case} equals
\begin{align*}
& I(H) |(\theta,\cdot)|^{2H}(\nu) + \frac{\ii (-1)^{H-1/2} }{(2H)!} \int_{\R^d} (\theta^\top t)^{2H}  \Big[ \int_{1/|\theta^\top t|}^ 1 z^{-1}dz \Big] \nu(dt) \\
& = I(H) |(\theta,\cdot)|^{2H}(\nu) + \frac{\ii (-1)^{H+1/2} }{(2H)!} \int_{\R^d} (\theta^\top t)^{2H} \log(|\theta^\top t|)\nu(dt).
\end{align*}
This leads to the desired expression
\begin{align*}
{\cal K}(\nu) = I({H}) \int_{\mathbb S^{d-1}} |(\theta,\cdot)| ^{2H}(\nu)  \sigma(d\theta)
+ \ii\, \frac{(-1)^{{H}+1/2}}{(2H)!}  \int_{\mathbb S^{d-1}} \Big[ (\theta, \cdot) ^{2H} \log |( \theta,\cdot) | \Big] (\nu) \sigma(d\theta),
\end{align*}
completing the proof. 
 \end{proof2}

\section*{Acknowledgements.}
\addcontentsline{toc}{section}{Acknowledgements}
SS and TH were partially supported by the NSF Grant DMS-1916226 
{\em The Argo Data and Functional Spatial Processes}.   
We dedicate this paper to the memory of Mark Marvin Meerschaert (1955-2020). Mark was 
a great visionary, mentor, and friend to us. His work has inspired and guided us in this paper and elsewhere.
We are very grateful to Rafail Kartsioukas for his exceptionally careful reading of the manuscript and help with fixing a number of important mathematical errors.   We thank Yimin Xiao for an inspiring discussion and pointing out important references to the available 
literature. We are also very indebted to two anonymous referees for their exceptionally detailed and insightful comments, which helped us correct a number of errors and improve the manuscript. 
%\begin{supplement} %[suppA]
%\stitle{Supplement to ``Tangent Fields, Intrinsic Stationarity and Self Similarity''}
%\slink[doi]{COMPLETED BY THE TYPESETTER}
%\sdatatype{Supplement.pdf}
%\sdescription{This supplements provides proofs and background materials not in the paper.}
%\end{supplement}

%***
%\bibliography{Lib2.bib}
%\bibliographystyle{imsart-nameyear}
%\end{document}

%\begin{acknowledgements}
%SS and TH were partially supported by the NSF Grant DMS-1916226 {\em The Argo Data and Functional Spatial Processes}.   
%To the memory of Mark M.\ Meerschaert (1955-2020) -- a great visionary, friend, and mentor...
%We are very grateful to Rafail Kartsioukas for his exceptionally careful reading of the manuscript and help with fixing a 
%number of important mathematical errors.   We also thank Yimin Xiao for an inspiring discussion and pointing out important references to the available literature. 
%\end{acknowledgements}


 
% \section*{Data availability}
  
%   The manuscript has no associated data sets.

%\small

\addcontentsline{toc}{section}{References}
\bibliographystyle{abbrvnat}      % basic style, author-year citations
%\bibliographystyle{imsart-nameyear}
%% \bibliography{Lib2.bib}   % name your BibTeX data base

%%

\newpage

\hypertarget{supp}{\centerline{\large \sc SUPPLEMENT}}
\addcontentsline{toc}{section}{S\ \ \ Supplement}
%\runtitle{Supplement to ``Tangent fields, intrinsic stationarity and self similarity"}

\renewcommand{\theequation}{S.\arabic{section}.\arabic{equation}}
\renewcommand{\thesection}{S.\arabic{section}}
%\startlocaldefs
%\numberwithin{equation}{section}
%\theoremstyle{plain}
%\newtheorem{theorem}{Theorem}[section]
%\newtheorem{lemma}[theorem]{Lemma}
%\newtheorem{corollary}[theorem]{Corollary}
%\newtheorem{proposition}[theorem]{Proposition}
%\theoremstyle{remark}
%\newtheorem{definition}[theorem]{Definition}
%\newtheorem{remark}{Remark}[section]
%\newtheorem{example}{Example}[section]
%\newtheorem{fact}[theorem]{Fact}
%\newtheorem{proposition}[theorem]{Proposition}
%\endlocaldefs

\setcounter{section}{0}
	\vspace{1cm}
	The rest of the appendix contains the supplement to the main paper.  In order to differentiate sections and results 
	in the supplement from  those in the main paper, we add a character ``S" in front of sections, lemmas, etc., in the supplement.  
	%For convenience, we include the table of contents, next.
	
	%\tableofcontents
	
\section{Notation and preliminaries on Hilbert spaces.}
	
	The purpose of this section is to fix some notation and collect basic facts on Hilbert spaces used in the main paper and 
	the proofs. The details can be found in most standard functional analysis texts such as \cite{conway2019course}.
		
	Fix a separable Hilbert space $\V$ over the field of complex numbers $\mathbb C$ 
	with inner product $\langle f,g\rangle$ and norm $\|f\| := \sqrt{\langle f,f\rangle}$,  for $f,g\in \V$.  
	We will focus on bounded linear operators ${\cal A}:\V \to \V$, namely linear operator with a bounded 
	operator norm: 
	$$
	\|{\cal A}\|_{\rm{op}} = \sup_{f:\|f\|=1} \|{\cal A} f\|.
	$$ 
	Observe that for any two bounded operators ${\cal A}$ and ${\cal B}$, we have
	$
	\| {\cal A} {\cal B}\|_{\rm op} \le \| {\cal A}\|_{\rm op}\| {\cal B}\|_{\rm op}.
	$
	
	The adjoint of ${\cal A}$ is denoted as ${\cal A}^*$, and we have
	$$
	 \langle {\cal A} f,g\rangle = \langle f,{\cal A}^* g\rangle = \overline{\langle {\cal A}^* g, f\rangle}, \quad
	\mbox{for all $f,g\in \V$}.
	$$
	${\cal A}$ is self adjoint if ${\cal A} = {\cal A}^*$, which holds if and only if 
	$ \langle {\cal A} f,f\rangle\in\bbR$ for all 
	$f\in\V$.

	For self-adjoint operators ${\cal A}$ and ${\cal B}$, 
	write ${\cal A} \le {\cal B}$ or ${\cal B} \ge {\cal A}$ whenever
	$$
	\langle f, {\cal A}f \rangle \le \langle f, {\cal B}f\rangle, \quad \mbox{for all } f\in \V.
	$$
	In particular, ${\cal A}$ is positive definite (or just positive) if ${\cal A} \ge 0$.
		
	The outer  (or tensor) product of two elements $f,g\in \V$, denoted by
	$f\otimes g$, is the operator that maps $h$ to $\langle h,g\rangle f$ on $\V$.
	Clearly, $(f\otimes g)^* = g \otimes f$.
	For a compact, self-adjoint operator ${\cal A}$ with spectral decomposition
	$$
	{\cal A} = \sum_{j=1}^\infty \lambda_j e_j \otimes e_j,
	$$
	where the $\lambda_j\in\bbR, e_j\in\V$ are the eigenvalues and eigenfunctions of ${\cal A}$, 
	let 
	\begin{align} \label{e:op_pm}
	{\cal A}^\pm =\sum_{j=1}^\infty \lambda^\pm_j e_j\otimes e_j,
	\end{align}
	where $a^\pm = \max\{0,\pm a\},\ a\in \R$, and
	$$
	|{\cal A}| := A^+ + A^-.
	$$
	
	The trace of an operator ${\cal A}:\V\to \V$ is 
	$$
	{\rm trace}(A) = \sum_{j=1}^\infty \langle A e_j, e_j\rangle
	$$
	if it is well defined, where $\{e_j\}$ is any CONS of $\V$. 
	
	An operator ${\cal A}:\V\to \V$ is trace class (or nuclear) if the self-adjoint positive operator 
	$({\cal A}^* {\cal A})^{1/2} = $ the square-root operator of ${\cal A}^* {\cal A}$ 
	has finite trace, in which case the trace norm of ${\cal A}$ is defined as
	$$
	\|{\cal A}\|_{\rm tr} := {\rm trace}(\mathcal{B}) = \sum_{j=1}^\infty \lambda_j,
	$$
	where $\{\lambda_j\}$ is the set of eigenvalues of $({\cal A}^* {\cal A})^{1/2}$ or singular values of ${\cal A}$ 
	(counting multiplicities).  The space of trace-class operators equipped with the trace norm will be 
	denoted by $\bbT$, which is a Banach space. 
	The collection of positive definite trace-class operators is denoted by
	$\mathbb{T}_+$.

	\section{Cross covariance operators.}
	
	For a fixed probability space
	 $(\Omega,{\cal F},\P)$, the class of $\V$-valued random elements will be denoted by ${\cal L}^0(\V)$.  
	 We shall also work with 
	 the class ${\cal L}^2(\V)$ of all $X\in {\cal L}^0(\V)$ such that $\E \|X\|^2 <\infty$.
	  The space ${\cal L}^2(\V)$ becomes a Hilbert space with respect to the inner product
	  $ \E [ \langle X,Y\rangle ]$.

	The expectation of $\V$-valued random elements can be defined in the sense of Bochner.  
	See, for example, \cite{diestel1977vector} or \cite{yosida2012functional}
	for details, or Section 2.5 in \cite{Hsing2015} for a brief treatment on Bochner's integral.
	
	For two zero-mean
	random elements $X,Y\in {\cal L}^2(\V)$, it is also natural to consider the cross covariance operator
	$$
	{\mathcal C}(X,Y):= \E [ X \otimes Y],
	$$
	where $X\otimes Y$ is a random element taking values in the space of trace-class operators $\bbT$.  
	The latter expectation can also be defined in the sense of Bochner in the Banach space $(\bbT,\|\cdot\|_{\rm tr})$.
	 For completeness, we briefly review next the Bochner integral in this setting.
	
	\subsection{The Bochner integral in $(\bbT,\|\cdot\|_{\rm tr})$.} \label{supp:sec:Bochner}
	
	The set of trace class operators $(\bbT,\|\cdot\|_{\rm tr})$ is a separable Banach space.
	We will need to consider integrals of measurable functions $f: (\Omega,{\cal F},\mu) \to (\bbT, {\cal B}(\bbT)),$
	for a finite measure space $(\Omega,{\cal F},\mu)$, where ${\cal B}(\bbT)$ is the Borel $\sigma$-field on $\bbT$.
	
	Such integrals can be defined in a standard way in the sense of Bochner as discussed in the references mentioned 
	above. Here, we only give a simple criterion needed for our purposes, akin to Theorem 2.6.5 in \cite{Hsing2015}.
        
	\begin{theorem}\label{thm:Bochner-integrability} Let  $f:\Omega\to \bbT$ be a measurable function. 
	 If $\int_\Omega \|f\|_{\rm tr}d \mu <\infty$, then $f$ is Bochner integrable
	 and
	 $$
	 \left\| \int_\Omega f d\mu \right \|_{\rm tr} \le \int_\Omega \|f\|_{\rm tr}d \mu.
	 $$ 
	 \end{theorem}
	 \begin{proof} 
	  Fix a CONS $\{e_j,\ j\in\mathbb N\}$ of $\V$, let $\Pi_n := \sum_{j=1}^n e_j\otimes e_j$ 
	  be the projection operator onto ${\rm span}\{e_j,\ j=1,\dots,n\}$ and define 
	  %the space $\bbT_n:= {\rm span}\{ e_i\otimes e_j\, :\, i,j=1,\dots,n\}$.  For every ${\cal A}\in \bbT$, we have 
	  %${\cal A}_n:=\Pi_n{\cal A} \Pi_n \in \bbT_n$.  
	   ${\cal A}_n:=\Pi_n{\cal A} \Pi_n$.  
	  It is easy to see that 
	  \begin{equation}\label{e:thm:Bochenr-integrability-0}
	  \| {\cal A} - {\cal A}_n\|_{\rm tr} \le 2 \|{\cal A}\|_{\rm tr} \quad 
	  \mbox{ and }\quad \|{\cal A} - {\cal A}_n\|_{\rm tr} \to 0,\ \mbox{ as }n\to\infty.
	  \end{equation}
	  Indeed, Proposition IV.2.3 on page 51 in \cite{Traces_and_Determinants_book} implies that 
	  $$\| {\cal A}_n\|_{\rm tr} = \| \Pi_n {\cal A} \Pi_n \|_{\rm tr} \le \| \Pi_n\|_{\rm op}^2 \|{\cal A}\|_{\rm tr} = \|{\cal A}\|_{\rm tr}$$ 
	  and
	  hence the inequality in \eqref{e:thm:Bochenr-integrability-0} follows from the triangle inequality.  The convergence in
	   \eqref{e:thm:Bochenr-integrability-0} is a consequence of Theorem IV.5.5 on page 63 in  \cite{Traces_and_Determinants_book}.
	  
	  In view of Theorem 2.6.4 in \cite{Hsing2015}, since $\int_{\Omega} \| f\|_{\rm tr} d\mu <\infty$, to prove the Bochner integrability of
	  $f$, it suffices to show that 
	  \begin{equation}\label{e:thm:Bochenr-integrability-1}
	   \int_{\Omega} \| f- f_n\|_{\rm tr} d\mu \to 0,\quad \mbox{ as }n\to\infty,
	  \end{equation}
	  where $f_n(t) := \Pi_n f(t) \Pi_n$.  This, however, follows from \eqref{e:thm:Bochenr-integrability-0} and the 
	  Lebesgue dominated convergence theorem.
	  \end{proof}

	\begin{remark} It is well-known that if ${\cal O}:\bbT \to \mathbb F$ is a bounded linear operator into another Banach space 
	$(\mathbb F,\|\cdot\|_{\mathbb F})$, and $f:\Omega\to \bbT$ is Bochner integrable, then ${\cal O}f$ is Bochner integrable and
        $$
          {\cal O}\int_{\R^d} fd\mu = \int_{\R^d} {\cal O}f d\mu.
        $$
        In particular, since ${\rm trace}:\bbT \to \bbR$ is a continuous linear functional, we obtain
        \begin{equation}\label{e:trace-commutes-with-integral}
        {\rm trace}{\Big(}  \int_{\R^d} f d\mu {\Big)} = \int_{\R^d} {\rm trace}(f) d\mu,
        \end{equation}
        for all Bochner integrable $f$.
       \end{remark}

	\subsection{Existence and continuity of cross covariance operators.} \label{supp:existence-cross-cov}
		
	For $X,Y\in {\cal L}^2(\V)$, using Theorem \ref{thm:Bochner-integrability}, we will show next that the 
	{\em cross covariance} operator is well-defined the sense of Bochner.

	\begin{lemma} \label{l:cross-cov} 	
	Let $X,Y\in {\cal L}^2(\V)$, then $X\otimes Y \in {\cal L}^1(\bbT)$. This implies that 	
	\begin{equation}\label{e:CXY}
	{\mathcal C}(X,Y):= \E [ X\otimes Y],
	\end{equation}
 is a well-defined element of $\bbT$.  We have moreover that 
	  \begin{equation}\label{e:l:cross-cov-bound}
	  \|{\mathcal C}(X,Y)\|_{\rm tr} \le \E \| X\otimes Y\|_{\rm tr} = \E \Big[ \|X\| \|Y\| \Big] \le (\E \|X\|^2 )^{1/2}  (\E \|Y\|^2 )^{1/2},
	\end{equation}
	and 
	 \begin{equation}\label{e:l:cross-cov-XX}
	\|{\mathcal C}(X,X)\|_{\rm tr} = \E [\|X\|^2].
         \end{equation}
\end{lemma}
\begin{proof}  Note that $X\otimes Y \in \bbT$ and since  $(X\otimes Y)^* (X\otimes Y)  = \|X\|^2 (Y\otimes Y),$
we have that $\| X\otimes Y\|_{\rm tr}^2 = {\rm trace}( \|X\|^2 (Y\otimes Y) ) = \|X\|^2\|Y\|^2.$   
This shows that
$$
\| X\otimes Y\|_{\rm tr} = \|X\|\|Y\|,
$$ 
and hence by the Cauchy-Schwartz inequality,
$$
\E \| X\otimes Y\|_{\rm tr} \le  \E\| X\|\|Y\| \le  (\E \|X\|^2 )^{1/2}  (\E \|Y\|^2 )^{1/2}  <\infty
$$
Theorem \ref{thm:Bochner-integrability} then implies that $X\otimes Y$ is  Bochner integrable,
${\mathcal C}(X,Y)$ is well-defined, and \eqref{e:l:cross-cov-bound} holds.

Finally, note $X\otimes X$ and ${\mathcal C}(X,X)$ are both positive operators and hence
\begin{equation}\label{e:l:cross-cov-1} 
\|X\otimes X\|_{\rm tr} = {\rm trace}( X\otimes X)\  \mbox{ and }\ 
 \|{\mathcal C}(X,X)\|_{\rm tr} = {\rm trace}({\mathcal C}(X,X)).
\end{equation}
Since the trace is a bounded linear functional, we have 
$$ 
{\rm trace}(\E (X\otimes X) )=\E ({\rm trace}(X\otimes X)) = \E \|X\|^2,
$$
which in view of \eqref{e:l:cross-cov-1}  yields \eqref{e:l:cross-cov-XX}.
\end{proof}

For zero-mean random variables in ${\cal L}^2(\V)$, i.e., $\E[X]=\E[Y]=0$, the operator 
${\mathcal C}(X,Y)$ defined by \eqref{e:CXY} is referred to as
the cross-covariance operator of $X,Y$. 

Continuity of the covariance is a key element in Bochner's Theorem. 
The following result is the counterpart to the classical fact that for cross covariance functions, continuity at 
the diagonal implies continuity everywhere.

\begin{proposition}\label{supp:p:continuity-cross-cov}
Let $X=\{X(t),t\in\bbR^d\}$ be a $\V$-valued process. Then $X$ is ${\cal L}^2$-continuous, namely,
$\E \|X(t') - X(s')\|^2 \to 0$ as $s',t'\to t$,  if and only if, for every $t\in\R^d$,
\begin{equation}\label{e:Cx}
 \|{\mathcal C}_X(s',t') - {\mathcal C}_X(t,t)\|_{\rm tr}\to 0,\ \ \mbox{ as }(s',t') \to (t,t).
\end{equation}
In this case, we also have $\|{\mathcal C}_X(s',t') - {\mathcal C}_X(s,t)\|_{\rm tr}\to 0$, as $(s',t') \to (s,t)$.  
\end{proposition}

\begin{proof}
 Suppose first that $\{X(t)\}$ is ${\cal L}^2$-continuous. 
By \eqref{e:l:cross-cov-bound} 
\begin{align*}
\|{\mathcal C}_X(s',t) - {\mathcal C}_X(s,t)\|_{\rm tr} \le (\E \|X(t)\|^2)^{1/2}    (\E \|X(s') - X(s)\|^2)^{1/2} \to 0,
\end{align*}
as $s'\to s$.  Thus, the triangle inequality implies
$$
\|{\mathcal C}_X(s',t') - {\mathcal C}_X(s,t)\|_{\rm tr} 
\le \|{\mathcal C}_X(s',t') - {\mathcal C}_X(s,t')\|_{{\rm tr}} + \|{\mathcal C}_X(s,t') - {\mathcal C}_X(s,t)\|_{\rm tr} \to 0, 
$$
as $(s',t') \to (s,t)$, where we also used the elementary fact that ${\cal L}^2$-continuity implies
the continuity of $t\mapsto \E [ \|X(t)\|] $.  We have thus shown the continuity of the cross covariance operator in 
the trace norm.

Conversely, assume \eqref{e:Cx}, i.e., the continuity of ${\mathcal C}_X$ on the diagonal.
In the sense of Bochner on the space $(\bbT,\|\cdot\|_{\rm tr})$, 
we have
\begin{align*}
{\mathcal D} (s',t') &:=  \E (X(t') - X(s'))\otimes ( X(t') - X(s'))  \\
 & = {\mathcal C}_X(t',t') + {\mathcal C}_X(s',s') - ({\mathcal C}_X(t',s') + {\mathcal C}_X(s',t')),
\end{align*}
which converges to zero in trace norm, by \eqref{e:Cx}.  
Relation \eqref{e:l:cross-cov-XX} in Lemma \ref{l:cross-cov}, however, yields 
$
 \|{\mathcal D}(s',t')\|_{\rm tr}  =\E \|X(t') - X(s')\|^2 \to 0
$
as $s',t'\to t$, proving the desired continuity of $\{X(t)\}$.
\end{proof}

\section{Integration Theory for $\bbT_+$-valued measures.}\label{sec:integration}

	%For a $\T_+$-valued measure $\mu$, as in \eqref{e:op_pm}, define the positive and negative parts of $\mu$, 
	%denoted by $\mu_+,\mu_-$, respectively, so that, 	for any Borel set $A$, 
	%$\mu(A) = \mu_+(A) - \mu_-(A)$ where $\mu\pm(A)\ge 0$.
	
	Let $\mu$ be a $\T_+$-valued measure on ${\cal B}(\bbR^d)$, i.e., $\mu$ is a countably additive
	function on ${\cal B}(\bbR^d)$ taking values in $\T_+$.
	Below we develop integration of a real or complex valued Borel measurable function
	with an operator-valued measure $\mu$. 
	
	We follow the development of ordinary Lebesgue integration:
	\begin{enumerate}[label=(\roman*)]
		\item[(1)]
		For any real nonnegative simple function $f =\sum_{i=1}^k c_i I_{A_i}$, define 
		$\int fd\mu = \sum_{i=1}^k c_i\mu(A_i)$.
		\item[(2)]
		For nonnegative measurable functions $f$, let
		\begin{align}\label{e:mono_conv:supp}
		\int fd\mu = \lim_{n\to\infty} \int f_nd\mu
		\end{align}
		in $(\bbT,\|\cdot,\|_{\tr})$ where $\{f_n\}$ is any sequence of simple functions such that 
		\begin{enumerate}[label=(\alph*)]
			\item
			$f_n\le f_{n+1}$,
			\item
			$f_n(x)\uparrow f(x)$ for all $x$,
			\item
			$\int f_nd\mu \le \CB$ for all $n$ and some fixed $\CB\in\bbT_+$. 
		\end{enumerate}
	
	The existence of $\{f_n\}$ satisfying (a) and (b) for any given 
		nonnonegative measurable $f$ follows from standard measure theory. 
		However, we need the extra condition (c) (along with the completeness of $\bbT_+$)
		to ensure that the limit on rhs of (\ref{e:mono_conv:supp}) exists and does not depend on 
		the choice of $\{f_n\}$.  Since $\mu(\bbR^d)\in\T_+$, (c) is automatically fulfilled if $f$ is bounded.
		\item[(3)]
		For a general real measurable $f$, let
		\begin{align*}
		\int fd\mu = \int f_+d\mu - \int f_-d\mu 
		\end{align*}
		provided both terms on the right is finite. For a general complex $f$,
		let
		\begin{align*}
		\int fd\mu = \int f_{\mathrm{Re}}d\mu +\ii \int f_{\mathrm{Im}}d\mu 
		\end{align*}
		where $f_{\mathrm{Re}}$ and $f_{\mathrm{Im}}$ are the real and imaginary parts, respectively.
		\end{enumerate}

The integral can be further extended to $\mathbb{T}$-valued signed measures, namely $\mu=\mu_+-\mu_-$
where $\mu_+,\mu_-$ are both $\mathbb{T}_+$-valued measures, by $\int f d\mu := \int f d\mu_+ - \int f d\mu_-$.	

The following propositions justify the definition of the integration described above.

		\begin{proposition} \label{pp:monotone_convergence}
		Let $\CB$ and $\CT_n, n \ge 1$, be operators in $\bbT_+$ such that 
		$\CT_n \le \CT_{n+1} \le \CB$ for all $n$. 
		Then $\CT_n$ converges to a limit in $\bbT_+$.
	\end{proposition}
	
	\begin{proof}
		For $i\ge j$,
		\be
		\|\CT_i-\CT_j\|_{\tr} = \tr(\CT_i-\CT_j) = \|\CT_i\|_{\tr}-\|\CT_j\|_{\tr}.
		%\sum_{k=1}^\infty \langle (\CT_j-\CT_i) e_k, e_k\rangle
		\ee
		Now, $\|\CT_n\|_{\tr}$ is nondecreasing and bounded by $\|\CB\|_{\tr}$ for all $n$. 
		Hence, $\|\CT_n\|_{\tr}$ converges to some finite nonnegative limit. Consequently, 
		\be
		\lim_{n\to\infty} \sup_{i\ge j\ge n} \|\CT_i-\CT_j\|_{\tr} = 0.
		\ee
		This shows that $\{\CT_n\}$ is Cauchy and has a limit by the completeness of $\bbT_+$.
	\end{proof}
	\begin{proposition} \label{pp:simple_int_ineq}
		Let $f_n$ be a sequence of nonnegative simple functions with $f_n\le f_{n+1}$ for
		each $n$ and such that $\int f_n d \mu\le\CB$ for some $\CB\in\bbT$. 
		\begin{enumerate}
			\item
			$\int f_nd\mu$ converges in the space of trace class operators. 
			\item
			If $\lim_{n\to\infty} f_n(x) \ge g(x)$ for all $x$ for some simple function $g$, 
			then $\lim_{n\to\infty} \int f_n d\mu \ge \int g d\mu$.
			\end{enumerate}
	\end{proposition}
	\begin{proof}
		Note that nonnegative simple functions $g\le f$ can be formulated as
		$f=\sum_{i=1}^k c_i I_{A_i}$, $g=\sum_{i=1}^k d_i I_{A_i}$ where $0\le c_i\le d_i$ and the
		$A_i$ are disjoint. Then it is easy to conclude that
		\be
		\int g d \mu = \sum_{i=1}^k c_i \mu(A_i) \le \sum_{i=1}^k d_i \mu(A_i) = \int f d\mu.
		\ee
		Thus, part (i) follows readily from Proposition \ref{pp:monotone_convergence}.
		
		To prove (ii),  let $g=\sum_{i=1}^k d_i I_{A_i}$, $A=\cup_{i=1}^k A_i$, and
		\be
		E_n=\{x:f_n(x) +\epsilon>g(x)\}
		\ee
		for some fixed $\epsilon > 0$.
		Since
		\be
		f_n\ge f_n I_{E_n\cap A} \ge (g-\epsilon) I_{E_n\cap A},
		\ee
		we have 
		\be
		\int f_n d\mu &\ge& \int (g-\epsilon) I_{E_n\cap A} d\mu \\
		&\ge& \int g I_{E_n\cap A} d\mu - \epsilon\mu(A) \\
		&=&\sum_{i=1}^k c_i \mu(E_n\cap A_i) - \epsilon\mu(A).
		\ee
		Letting $n\to\infty$, by the fact $E_n\cap A_i\uparrow A_i$,
		\be
		\lim_{n\to\infty} \int f_n d\mu \ge \sum_{i=1}^k c_i \mu(A_i) - \epsilon\mu(A)
		= \int g d\mu- \epsilon\mu(A)
		\ee
		Since $\epsilon>0$ is arbitrary, (ii) follows.
	\end{proof}
	\begin{proposition}\label{p:converge_integral_operator}
		Let $f$ be a nonnegative measurable function and $\{f_n\}$ be an increasing sequence
		of nonnegative simple functions satisfying (a)-(c) above.
		Then the limit $\lim_{n\to\infty} \int f_nd\mu$ does not depend on the particular sequence $\{f_n\}$.
	\end{proposition}
	
	\begin{proof}
		Suppose there are two sequences $\{f_n\}$ and $\{g_n\}$ of simple functions
		both satisfying (a)-(c). Then we have
		\be
		\lim_{n\to\infty} f_n \ge g_m, \ \lim_{n\to\infty} g_n \ge g_m
		\ee
		for all $m$, and, by Proposition \ref{pp:simple_int_ineq},
		\be
		\lim_{n\to\infty} \int f_n d\mu \ge \int g_m d\mu,
		\ \lim_{n\to\infty} \int g_n d\mu \ge \int f_m d\mu.
		\ee
		The result follows by letting $m\to\infty$.
	\end{proof}

   \section{Proofs and auxiliary results for Section \ref{sec:c2s5}.}

	\subsection{Proof of Theorem~\ref{th:operator_Bochner} (Bochner).}
	\label{suppsec:thm4.2_proof}
		
	We begin by defining the notion of tightness for $\mathbb{T}_+$-valued measures as follows.
	
	\begin{definition}\label{def:tightness}
A sequence of finite $\T_+$-valued measures $\{\mu_n\}$  is said to be tight if the following two conditions hold:
\begin{enumerate} [label=(\roman*)]
	\item
	There exists some  ${\cal B} \in \T_+$ such that $\mu_n(\R^d)\le {\cal B}$ for all $n$.
	\item
	For any $\epsilon>0$, there is a compact set $K_\epsilon \subset \R^d$ such that
	$$
	\sup_n \| \mu_n(\R^d\setminus K_\epsilon)\|_{\tr} <\epsilon.
	$$
\end{enumerate}
	\end{definition}
	
	%\vskip.3cm
	%	\noindent
	%	\textsc{Proof} for Theorem~\ref{th:operator_Bochner}:
	%	\vskip.2cm\noindent
	% For simplicity of notation, we only illustrate proofs  on $\R$ and the extension to $\R^d$ should be straightforward. 

	We now proceed to prove Theorem \ref{th:operator_Bochner}.
	By Lemma~\ref{l:continuity_CK} below,  $\CK(t)$ is continuous in trace norm. 
	We first consider the case
	\ben\label{e:traceK}
	\int_{\bbR^d} \|\CK(t)\|_{\tr} \, dt < \infty,
	\een
	which will be relaxed in the second part of the proof.
	Let $\Pi_n = \sum_{i=1}^n e_i\otimes e_i$ be the projection operator onto 
	$\mathrm{span}\{e_1,\ldots,e_n\}$ where $\{e_i\}$ is CONS of $\V$. 
	Define $\CK_{n}(t) = \Pi_n\CK(t)\Pi_n$. 
	By Proposition IV.2.3 on page 51 in 
	\cite{Traces_and_Determinants_book}, $\|\CK_n(t)-\CK_n(s)\|_{\tr}\leq \|\CK(t)-\CK(s)\|_{\tr}$ 
	and therefore $\CK_n(t)$ is also continuous in trace norm. 
	The same argument entails that
	\begin{align} \label{e:KnKbound}
	\int_{\bbR^d} \|\CK_n(t)\|_{\tr} \, dt \le \int_{\bbR^d} \|\CK(t)\|_{\tr}\, dt < \infty.
	\end{align}
	By Theorem \ref{thm:Bochner-integrability}, 
	the following Bochner's integrals in $\bbT$ are well defined: 
	\ben\label{e:Kn_hatKn}
	\widehat \CK_n(x) := \frac1{(2\pi)^{d}}\int_{\bbR^d} e^{-\ii t^\top x} \CK_n(t)dt,\quad
	\widehat \CK(x) :=\frac1{(2\pi)^{d}}\int_{\bbR^d} e^{-\ii t^\top x} \CK(t)dt.
	\een
	%As $\CK(t)$ is positive-definite, we have $\CK(0)\geq 0$ and hence self-adjoint. 
	%Taking $c_1=c_2=1$, $t_1=0$, $t_2=t$ leads to $2\CK(0)+\CK(t)+\CK(-t)\geq 0$, which 
	%implies that $\CK(t)+\CK(-t)$ is self-adjoint. Finally, taking $c_1=1, c_2=i, t_1=t,t_2= 0$
	%gives $-i\CK(t)+i\CK(-t)\geq 0$ and self-adjoint.
	%Therefore, we have the following two equations
	%\be
	%\CK^*(t)+\CK^*(-t)=\CK(t)+\CK(-t)
	%\ee
	%and
	%\be
	%i\CK^*(t)-i\CK^*(-t) = -i\CK(t)+i\CK(-t),
	%\ee
	%from which we conclude
	%\tcr{
	%\ben\label{eq:symmetric_CK}
	%\CK^*(t)=\CK(-t).
	%\een
	%By \eqref{eq:symmetric_CK}, 
	%\be
	%&&\left\langle f,\widehat\CK(x) g\right\rangle 
	%= \left\langle f,(2\pi)^{-1}\int_{\bbR^d} e^{-\ii t^\top x} \CK(t)dt \hspace{.05cm} g\right\rangle 
	%= \left\langle (2\pi)^{-1}\int_{\bbR^d} e^{\ii t^\top x} \CK(-t)dt \hspace{.05cm} f, g\right\rangle \\
	%&& = \left\langle (2\pi)^{-1}\int_{\bbR^d} e^{-\ii t^\top x} \CK(t)dt \hspace{.05cm}f, g\right\rangle 
	%= \left\langle \widehat\CK(x)f, g\right\rangle,
	%\ee
	%which shows that $\widehat{\CK}(x)$ is self-adjoint. 
	It follows from Theorem IV.5.5 on page 63 in  \cite{Traces_and_Determinants_book},
	Theorem \ref{thm:Bochner-integrability}, \eqref{e:KnKbound}, and the Dominated Convergence Theorem, that
	\ben\label{e:conv_tr}
	\|\CK_n(t)-\CK(t)\|_{\tr} \to 0 \quad\hbox{and}\quad \|\widehat\CK_n(x)-\widehat\CK(x)\|_{\tr} \to 0
	\een
	for all $t,x$.
	For each $f,g\in\V$, define the complex-valued functions
	\be
	{\mathcal K}_{f,g}(t) = \langle \CK(t) g, f\rangle \quad\hbox{and}\quad {\mathcal K}_f (t) = {\mathcal K}_{f,f}(t).
	\ee
	Since $\CK(\cdot)$ is positive definite and continuous at $0$ in operator 
	norm (entailed by continuity in trace norm), ${\mathcal K}_f(\cdot)$ is a
	positive-definite function and continuous at $0$. It is also integrable by (\ref{e:traceK}). 
	We obtain, by the classical Bochner Theorem, that 
	\ben\label{e:F_inversion}
	{\mathcal K}_f(t) = \int_{\bbR^d} e^{\ii t^\top x} \widehat {\mathcal K}_f(x)dx \quad\hbox{and}\quad
	\widehat {\mathcal K}_f(x) := {\frac{1}{(2\pi)^d}} \int_{\bbR^d} e^{-\ii t^\top x}{\mathcal K}_f(t)dt, 
	\een
	where $\widehat {\mathcal K}_f(x)\ge 0$ for all $x$.
	%\ben\label{e:K_f(0)}
	%\int_{\bbR^d} \widehat {\mathcal K}_f(x)dx = {\mathcal K}_f(0).
	%\een
	Since $\langle \widehat \CK(x) f, f\rangle = \widehat {\mathcal K}_f(x) \ge 0$, 
	for all $f\in \V$, it follows that $\widehat{\CK}(x)$ is a positive operator for all $x$.
	Notice that, by (\ref{e:F_inversion}) %\eqref{eq:symmetric_CK} 
	and Lemma~\ref{l:identification_operator},
	\begin{align*}
	{\mathcal K}_{n}(t) &= \sum_{i=1}^n\sum_{j=1}^n {\mathcal K}_{e_i,e_j}(t)\, e_i\otimes e_j \\
	&= \sum_{i=1}^n\sum_{j=1}^n \left(\frac{\ii-1}{2}({\mathcal K}_{e_i}(t)+{\mathcal K}_{e_j}(t))-\frac{\ii}{2}{\mathcal K}_{\ii e_j+ e_i}(t)+\frac{1}{2}{\mathcal K}_{e_i+e_j}(t)\right) \, e_i\otimes e_j \\
	&= \sum_{i=1}^n\sum_{j=1}^n \int_x e^{\ii t^\top x} 
	\left(\frac{\ii-1}{2}(\widehat {\mathcal K}_{e_i}(x)+\widehat {\mathcal K}_{e_j}(x))-\frac{\ii}{2}\widehat {\mathcal K}_{\ii e_j+e_i}(x)+\frac{1}{2}\widehat {\mathcal K}_{e_i+e_j}(x)\right) dx \, e_i\otimes e_j\\
	&=\int _{\bbR^d} e^{\ii t^\top x} \widehat {\mathcal K}_n(x)dx,
	\end{align*}
	which, by (\ref{e:conv_tr}), implies
	\ben\label{e:F_inversion_1}
	\CK(t) = \int_{\bbR^d} e^{\ii t^\top x} \widehat\CK(x) dx.
	\een
	To summarize, we have shown under (\ref{e:traceK}) that
	\ben\label{e:FT}
	\CK(t) = \int_{\bbR^d} e^{\ii t^\top x}\widehat\CK(x)dx \quad\hbox{and}\quad 
	\widehat\CK(x) =  \frac{1}{(2\pi)^d} \int_{\bbR^d} e^{-\ii t^\top x} \CK(t)dt \in \T_+.
	\een
	%\tcr{As $\widehat\CK(x)$, $\CK(t)$ are both self-adjoint, by Lemma~\ref{l:integral_projection}, we have
	%\be
	%\CK(t)=\CK^*(t) = \int_x e^{-itx}\widehat{\CK}(x)dx = \CK(-t),
	%\ee
	%and
	%\be
	%\widehat{\CK}(x) = \frac{1}{2\pi}\int_t e^{itx}\CK(t)dt=\frac{1}{2\pi}\int_t e^{-itx}\CK(t)dt = \widehat{\CK}(-x).
	%\ee}
	Next, relax (\ref{e:traceK}) and define
	\ben\label{e:smoothing}
	\CK_\sigma(t) = e^{-\sigma^2\|t\|^2/2} \CK(t)
	= \frac{1}{(\sqrt{2\pi}\sigma)^{d}} \int_{\bbR^d}  \CK(t)e^{\ii t^\top y}e^{-\|y\|^2/(2\sigma^2)}dy,  \ \sigma\ge 0.
	\een
	Note that $\CK(t)e^{\ii t^\top y}$ is positive definite in $t$. 
	As a convex combination of positive definite functions, $\CK_\sigma(\cdot)$ is also positive definite. 
	Since $\int_{\bbR^d} \|\CK_\sigma (t)\|_{\tr} dt < \infty$, it follows from \eqref{e:FT} that
	\ben\label{e:smoothing1}
	\CK_\sigma(t) = \int_{\bbR^d} e^{\ii t^\top x}\widehat\CK_\sigma(x)dx \quad\hbox{and}\quad 
	\widehat\CK_\sigma(x) = \frac{1}{(2\pi)^d} \int_{\bbR^d} e^{-\ii t^\top x} \CK_\sigma(t)dt \in \T_+
	\een
	for any $\sigma\ge 0$. Define $\mu_\sigma(\cdot) := \int_.\widehat\CK_\sigma(x)dx$
	and suppose now that $\{\mu_\sigma\}$ is tight, which will be established in the last step of the proof. 
	For any ${\mathbb T}_+$-valued measure $\mu$ and $e_i,e_j$ in CONS, 
	write $\mu^{(i,j)}(A) = \langle \mu(A)e_j, e_i\rangle$ for any Borel set $A$.
	By Lemma~\ref{l:weak_convergence_operator}, for any sequence $\sigma_n\to 0$, 
	the tightness of $\{\mu_{\sigma_n}\}$ implies there exist a  $\T_+$-valued measure 
	$\mu$ and a subsequence sequence $n'$ of $n$ 
	such that
	\begin{equation*}
	\int_{\bbR^d} h(x) \mu_{\sigma_{n'}}^{(i,j)}(dx) \to \int_{\bbR^d} h(x) \mu^{(i,j)}(dx)
	\end{equation*}
	for all $i,j$ and all bounded and continuous functions $h$. 
	%As each $\widehat{\CK}_{\sigma_n}$ 	is symmetric, it follows that $\mu$ is a symmetric measure. 
	Since $\|\CK_\sigma(t)-\CK(t)\|_{\tr}\to 0$ by the definition of $\CK_\sigma$, 
	we have, for all $i,j$,
	\be
	\langle \CK(t) e_j, e_i \rangle &=& \lim_{n'\to\infty}\langle  \CK_{\sigma_{n'}}(t)  e_j, e_i \rangle
	=\lim_{n'\to\infty}\left\langle \int_x e^{\ii t^\top x}\mu_{\sigma_{n'}}(dx) e_j, e_i \right\rangle\\
	&=&\lim_{n'\to\infty}\int_{\bbR^d} e^{\ii t^\top x}\left\langle \mu_{\sigma_{n'}}(dx) e_j, e_i \right\rangle
	=\lim_{n'\to\infty}\int_{\bbR^d} e^{-\ii t^\top x}\mu_{\sigma_{n'}}^{(i,j)}(dx)\\
	&=&\int_{\bbR^d} e^{\ii t^\top x}\mu^{(i,j)}(dx)=\left\langle \int_{\bbR^d} e^{\ii t^\top x}\mu(dx) e_j, e_i \right\rangle,
	\ee
	where the interchange of inner product and integration can be easily justified by the properties of 
	Bochner's integral. 	Thus, for any $f,g\in \V$, 
	\be
	\langle \CK(t) g, f \rangle=\left\langle \int_{\bbR^d} e^{\ii t^\top x}\mu(dx) g, f \right\rangle,
	\ee
	which entails that $\CK(t) = \int_{\bbR^d} e^{\ii t^\top x}\mu(dx)$. Finally, we have 
	$$
	 {\mathcal K}_f(t) = \int_{\bbR^d} e^{\ii t^\top x}\langle \mu(dx)f, f \rangle.
	$$ Since this is a Fourier transform
	which does not depend on the sequence $\sigma_n$ or $\sigma_{n'}$, we conclude that $\mu$ is unique.

	It remains to show that $\{\mu_\sigma\}$ is tight.
	By (\ref{e:smoothing}) and \eqref{e:smoothing1}, we have 
	\be
	\mu_\sigma(\bbR^d) = \CK_\sigma(0)
	=  \CK(0) \in \T_+,
	\ee
	which implies (i) of the tightness definition. 
	Next,
	\begin{align} \label{e:tightness_1}
	\begin{split}
	& {1\over (2T)^d} \int_{t\in [-T,T]^d} \CK_\sigma(t)dt \\
	&= {1\over (2T)^d} \int_{t\in [-T,T]^d} \ \int_{x\in\bbR^d} e^{\ii t^{\top}x}\widehat\CK_\sigma(x)dxdt \\
	&= \int_{x\in\bbR^d} \prod _{j=1}^d {\sin Tx_j\over Tx_j} \widehat\CK_\sigma(x)dx \\
	&\le \int_{x\in [-b,b]^d} \prod _{j=1}^d \left|{\sin Tx_j\over Tx_j}\right|  \widehat\CK_\sigma(x)dx
	+ \int_{x\not\in[-b,b]^d} \prod _{j=1}^d \left|{\sin Tx_j\over Tx_j}\right|  \widehat\CK_\sigma(x)dx \\
	&\le \mu_\sigma([-b,b]^d) + {1\over (Tb)^d} \mu_\sigma(([-b,b]^d)^c) \\
	&= \mu_\sigma(\bbR^d) - (1-1/(Tb)^d) \mu_\sigma(([-b,b]^d)^c),
	\end{split}
	\end{align}
	where ``$\leq$" here is the operator inequality. 
	By (\ref{e:tightness_1}), with $b=T^{-1} (1-2^{-d})^{-1/d}$ so that $1-1/(Tb)^d=2^{-d}$, we have
	\be
	 \mu_\sigma(([-b,b]^d)^c) &\le& 2^d\left(\mu_\sigma(\bbR)-
	{1\over (2T)^d} \int_{[-T,T]^d} \CK_\sigma(t)dt  \right) \\
	&=& {1\over T^d} \int_{[-T,T]^d} (\CK(0)-\CK_\sigma(t))dt.
	\ee
	Thus,
	\be
	\left\|\mu_\sigma(([-b,b]^d)^c)\right\|_{\tr} 
	\le {1\over T^d} \int_{[-T,T]^d} \|\CK(0)-\CK_\sigma(t)\|_{\tr}\, dt.
	\ee
	By the triangle inequality,
	\be
	\|\CK(0)-\CK_\sigma(t)\|_{\tr}
	\le (1-e^{-\sigma^2t^2/2})\|\CK(0)\|_{\tr} + e^{-\sigma^2t^2/2}\|\CK(0)-\CK(t)\|_{\tr} \to 0,
	\ee
	by the continuity of $\CK(t)$ at $0$.  This show that
	$\left\|\mu_\sigma(([-b,b]^d)^c)\right\|_{\tr}  \to 0$ as $T\to 0$ and
	establishes the tightness of $\{\mu_\sigma\}$.
	\ep

	\begin{lemma}\label{l:continuity_CK}
		Assume that $\{\CK(t),t\in\bbR\}$ is a collection of  operators satisfying the
		assumptions of Theorem~\ref{th:operator_Bochner}. Then we have 
		\be
		\lim_{\delta\to 0}\sup_{|t-s|<\delta}\|\CK(t)-\CK(s)\|_{\tr}\to 0.
		\ee
	\end{lemma}
	\begin{proof}
		For any element $f\in\V$, denote $\CK_f = \langle \CK f, f  \rangle$. It follows that $\CK_f(t)$ is positive definite. By taking $c_1=c_2 =1$ and $t_1=0$, $t_2 = t$, one will have ${\rm Im}(\CK_f(t)) = - {\rm Im}(\CK_f(-t))$. Similarly, by taking $c_1=1$, $c_2=i$, $t_1=0$, $t_2=t$, we will have ${\rm Re}(\CK_f(t)) = {\rm Re}(\CK_f(-t))$. Therefore $\CK_f(t) = \overline{\CK_f(-t)}$. By taking $t_1=t,t_2=s,t_3=0$, we will have the following matrix to be nonnegative definite
		\begin{eqnarray*}
			\begin{pmatrix}
				\CK_f(0)	& \CK_f(t-s)  & \CK_f(t)\\ 
				\overline{\CK_f(t-s)}	& \CK_f(0) &\CK_f(s) \\ 
				\overline{\CK_f(t)}	& \overline{\CK_f(s)} & \CK_f(0).
			\end{pmatrix}
		\end{eqnarray*}
		As a result, its determinant will be nonnegative, that is
		\be
		0&\leq& \CK_f^3(0) - \CK_f(0)[|\CK_f(t)-\CK_f(s)|^2+|\CK_f(t-s)|^2] \\&&- 2 \ {\rm Re}([\CK_f(s)\CK_f(t)(\CK_f(0)-\CK_f(t-s))])\\
		&\leq&\CK_f^3(0) - \CK_f(0)\Big [ |\CK_f(t)-\CK_f(s)|^2+|\CK_f(t-s)|^2 \Big] \\
		& & \quad\quad\quad- 2 \CK_f^2(0)|\CK_f(0)-\CK_f(t-s)|,
		\ee
		where in the last line we applied the fact $|\CK_f(t)|\leq \CK_f(0)$ for all $t$. Rearranging terms then gives 
		\be
		|\CK_f(t)-\CK_f(s)|^2 &\leq& \CK_f^2(0) - |\CK_f(t-s)|^2 + 2\CK_f(0)|\CK_f(0)-\CK_f(t-s)|\\
		&\leq&4\CK_f(0)|\CK_f(0)-\CK_f(t-s)|,
		\ee
		where we again use the fact that $|\CK_f(t-s)|\leq \CK_f(0)$. 
		
		Then, for CONS $\{e_i\}$ that is the system of eigenfunction of $\CK(t)-\CK(s)$, we have
		\be
		\| \CK(t)-\CK(s) \|_{\tr} &=& \sum_{i} |\CK_{e_i}(t)-\CK_{e_i}(s)|\\
		&\leq& \sum_{i} 2\sqrt{\CK_{e_i}(0)|\CK_{e_i}(0)-\CK_{e_i}(t-s)|}\\
		&\leq&2\sqrt{\sum_{i}\CK_{e_i}(0)\sum_{i}|\CK_{e_i}(0)-\CK_{e_i}(t-s)|}\\
		&\leq&2\sqrt{\|\CK(0) \|_{\tr}\|\CK(0)-\CK(t-s) \|_{\tr}},
		\ee
		where the last expression converges to $0$ uniformly as $|t-s|\to 0$. 
	\end{proof}

\begin{lemma} \label{l:weak_convergence_operator}
Let $\{\mu_n\}$  be be a class of finite $\T_+$-valued measures. If $\{\mu_n\}$ is tight in the sense
of Definition \ref{def:tightness}, then 
		there exists a finite $\T_+$-valued measure $\mu$ and an infinite subsequence $n'$ such that 
		\begin{equation}\label{e:muij}
		\int_{\R^d} h(x) \mu_{n'}^{(i,j)}(dx) \to \int_{\R^d} h(x) \mu^{(i,j)}(dx),
		\end{equation}
		for all $i,j$ and all bounded and continuous functions $h$, where $\mu^{(i,j)}(A) 
		= \langle \mu(A) e_j, e_i\rangle$.
\end{lemma}

\begin{proof} In this proof, the term measure without any qualification refers to the usual $\R_+$-valued measure.
Let $\{e_j\}$ be a CONS which we fix as the eigenfunctions of the operator ${\cal B}$ in Definition \ref{def:tightness}.  
First consider the signed measures $\mu_n^{(i,j)}(A) = \langle \mu_n (A) e_j, e_i\rangle$.
The tightness of $\{\mu_n\}$ implies that, for any $f$, the sequence of finite measures 
$\{\mu_{n,f}(A) := \langle \mu_n(A) f, f \rangle\}$ is tight. By Prokhorov's Theorem,
		there exist a finite measure $\mu_f$ and an infinite subsequence $n'$ such that
		$$
		\int_{\R^d} h(x) \mu_{n',f}(dx) \to \int_{\R^d} h(x) \mu_f(dx)
		$$
		for all bounded and continuous functions $h$. In particular, for each $i$
		there exist a finite measure $\mu^{[i,i]}$ and an infinite subsequence $n'$ such that
		$$
		\int_{\R^d} h(x) \mu_{n'}^{(i,i)}(dx) \to \int_{\R^d} h(x) \mu^{[i,i]}(dx)
		$$
		for all bounded and continuous functions $h$.
		Applying the identity (Lemma S.4.2 of supplement)
		\begin{align} \label{e:mu_ij}
		\begin{split}
		\mu_n^{(i,j)} (A) = & {\ii-1\over 2} (\mu_n^{(i,i)}(A) + \mu_n^{(j,j)}(A))
		+ {1\over 2} \langle e_i+e_j,  \mu_n (A) (e_i+e_j) \rangle \\
		& -{\ii\over 2} \langle e_i+\ii e_j, \mu_n(A) (e_i+ \ii e_j) \rangle,
		\end{split}
		\end{align}
		we also conclude that, for each $i,j$, there is an infinite subsequence $n'$ such that
		$$
		\int_{\R^d} h(x) \mu_{n'}^{(i,j)}(dx) \to \int_{\R^d} h(x) \mu^{[i,j]}(dx)
		$$
		for all bounded and continuous functions $h$, where $\mu^{[i,j]}$ is defined as in \eqref{e:mu_ij} 
		with $\mu_n^{(i,i)}$ replaced by $\mu^{[i,i]}$.
		Note that $\mu^{[i,j]}$ is in general a complex-valued set-function for $i\not=j$. 
		Since the set of pairs $(i,j)$ is countable, a straightforward argument shows that the 
		convergence holds along a common subsequence $n'$ for all $i,j\in \mathbb N$.
		
		Construct, for a fixed $m$,
		$$
		\mu (A;m) := \sum_{i=1}^m \sum_{j=1}^m \mu^{[i,j]} (A) e_i\otimes e_j, \ A\in{\cal B}(\R^d).
		$$
		We consider the properties of $\mu (A;m)$ for any $m$ and Borel set $A$.
		
		\begin{enumerate} [label=(\roman*)]
		\item
		$\mu(\cdot; m)$ is $\sigma$-additive as a finite sum of countably additive set-functions $\mu^{[i,j]}$.
		
		\item $\mu(A;m)\ge 0$.
		
		{\em Proof.} We know that $\mu(A;m)\ge 0$ for $A\in {\cal F},$ where  
		\begin{align} \label{e:calE}
		\mathcal{F} = \{ A\in {\cal B}(\R^d): \mu^{[i,i]}(\partial A) = 0 \mbox{ for all $i$}\}.  
				\end{align}
		Indeed, the weak convergence $\mu_{n'}^{(i,i)}\stackrel{w}{\to} \mu^{[i,i]}$ implies that $\mu_{n'}^{(i,i)}(A)\to \mu^{[i,i]}(A)$ 
		for all $A\in {\cal F},\ i\in\mathbb N$, and hence 
		$$
		  \Pi_m \mu_{n'}(A) \Pi_m = \sum_{i,j=1}^m \mu_{n'}^{(i,j)}(A) e_i\otimes e_j\to \mu(A;m),
		 $$ 
		as finite rank operators, where $\Pi_m:= \sum_{i=1}^m e_i\otimes e_i$.  This shows that $\mu(A;m)\ge 0$.
		 Note that ${\cal F}$ is a field (i.e., a nonempty collection of sets containing the whole space and
		closed under finite unions and complements).
		Consider the measure $\mu_f (\cdot;m) := \langle f, \mu(\cdot;m) f\rangle$ on ${\cal F}$
		for any fixed $f\in\V$. 
		By the measure extension theorem, there exists a unique measure 
		$\wt{\mu_f}(\cdot;m)$ on ${\cal B}(\R^d)$ that agrees with $\mu_f(\cdot;m)$ 
		on ${\cal F}$. Since both $\wt{\mu_f}(\cdot;m)$ and $\mu_f(\cdot;m)$ 
		are countably additive on ${\cal B}(\R^d)=\sigma({\cal F})$, Proposition \ref{pp:meas_equiv} implies 
		that $\wt{\mu_f} (\cdot;m)\equiv \mu_f(\cdot;m)$ on ${\cal B}(\R^d)$.
		Thus, $\mu_f(A;m)\ge 0$ for all Borel sets $A$ and all $f\in\V$.  
		This together with (i) prove that $\mu(\cdot;m)$ is a $\mathbb T_+$-valued 
		measure.
		
		\item $\mu(A;m) \le {\cal B}$.
		
		{\em Proof.} By (ii), weak convergence and the tightness assumption, for all $m$, Borel set $A$, and $f\in\V$,
		\begin{align*}
		\langle f, \mu(A;m) f\rangle 
		\le \langle f, \mu(\R^d;m)  f\rangle 
		= \lim_{n'\to\infty} \langle f, \mu_{n'}(\R^d;m) f\rangle 	
		\le \langle f, {\cal B} f\rangle.
		\end{align*}
		
		\item
		$\mu(A;m) = \Pi_m \mu(A;m')\Pi_m$ for all $m'>m$.
		\end{enumerate}
		
		Fix any Borel set $A$. By Lemma \ref{l:trace}, for any infinite sequence $m$ there exists a subsequence $m'$ such that
		$\|\mu(A;m') - \mu(A)\|_{\rm HS}\to 0$ for some $\mu(A)\in\T_+$ along $m'$. It follows that 
		$$
		\langle \mu(A) e_j, e_i\rangle = \langle \mu(A;\ell) e_j, e_i \rangle 
		= \mu^{[i,j]}(A)\mbox{ for $\ell\ge i \vee j$}.
		$$
		Thus, $\mu(A)$ does not depend on $m'$ and we have
		\begin{align} \label{e:mu_conv}
		\lim_{m\to\infty} \|\mu(A;m) - \mu(A)\|_{\rm HS}=0.
		\end{align}
		This implies, in particular, that $\mu(A)\ge 0$.
		
		It remains to show $\mu(A)$ is countably additive.
		Suppose $B_k\downarrow \emptyset$. It follows that
		\begin{align*}
		\sup_m \|\mu(B_k;m)\|_{\tr} & =  \sup_m\sum_{j=1}^\infty \langle \mu(B_k;m) e_j, e_j\rangle\\
		&\le \sum_{j=1}^\infty  \sup_m \langle \mu(B_k;m) e_j, e_j\rangle 
		= \sum_{j=1}^\infty \langle \mu(B_k;j) e_j, e_j\rangle,
 		\end{align*}
		since, by Property (iv), $\langle \mu(B_k;m) e_j, e_j\rangle = 0$ 
		or $\langle \mu(B_k;j) e_j, e_j\rangle$ depending on $m < j$ or $\ge j$. 		
		Note that $\mu(B_k;j)\le {\cal B} \in \T_+$ for all $j$ by Property (iii) above. Therefore,
		$$
		\langle \mu(B_k;j)e_j, e_j\rangle \le  B_j:= \langle {\cal B} e_j, e_j\rangle,
		$$ 
		where $\|{\cal B}\|_{\rm tr } = \sum_{j} B_j<\infty$. We also have
		$\langle \mu(B_k;j)e_j, e_j\rangle\to 0$ as $k\to\infty$ for fixed $j$ by the continuity property of 
		measure. Thus, by the DCT, 
		\begin{align} \label{e:mu_cont1}
		\sup_m \|\mu(B_k;m)\|_{\tr} \to 0 	\mbox{ as $k\to\infty$}.
		\end{align}
		By \eqref{e:mu_conv}, \eqref{e:mu_cont1} and the triangle inequality, 
		\begin{align} \label{e:mu_cont2}
		\|\mu(B_k)\|_{\rm HS} = \lim_{m\to\infty}\|\mu(B_k;m)\|_{\rm HS} \le \sup_m \|\mu(B_k;m)\|_{\rm tr} \to 0
		\mbox{ as $k\to\infty$}.
		\end{align}
		Let $A_i, i\ge 1$, be disjoint Borel sets.
		Since $\mu(\cdot)$ is finitely additive,
		\begin{align}\label{e:HS-to-0}
		\Big\|\mu(\cup_{i=1}^\infty A_i) - \sum_{i=1}^k \mu(A_i)\Big\|_{\rm HS} 
		= \left\|\mu(\cup_{i=k+1}^\infty A_i)\right\|_{\rm HS} \to 0,
		%\ \ \ \tcr{\mbox{we can use $\|\cdot\|_{\rm tr}={\rm trace}(\cdot)$ here, too}}
		\end{align}
		by \eqref{e:mu_cont2}.  Letting ${\cal T}_k:= \sum_{i=1}^k \mu(A_k) \le {\cal B} \in \bbT_+$, since 
		${\cal T}_k\le {\cal T}_{k+1},\ k\in \mathbb N$, 
		Proposition \ref{pp:monotone_convergence} implies that there exists an operator ${\cal T}\in \bbT_+$, 
		such that $\|{\cal T}_k - {\cal T}\|_{\rm tr} \to 0$.   This also implies that $\|{\cal T}_k -{\cal T}\|_{\rm HS}\to 0$ and hence by 
		\eqref{e:HS-to-0}, ${\cal T} = \mu(\cup_{i=1}^\infty A_i)$ and we have
		$$
		\Big\| \mu(\cup_{i=1}^\infty A_i) - \sum_{i=1}^k \mu(A_i)\Big\|_{\rm tr} \to 0,\ \ \mbox{ as }k\to\infty,
		$$
		which proves the $\sigma$-additivity of $\mu$ in $(\bbT_+,\|\cdot\|_{\rm tr})$.
		 
		%\tcr{Is the "trace" functional continuous in operator norm?  Specifically, if $A_n\uparrow A$ (in operator norm) 
		%where $A_n,A$ are positive and trace class, then does it follow that ${\rm trace}(A_n) \uparrow {\rm trace}(A)$?} 
		\end{proof}

\begin{proposition} \label{pp:meas_equiv}
Let ${\cal F}$ be a field of subsets of $\R^d$ 
and $\mu_1, \mu_2$ be finite countably additive set functions on $\sigma({\cal F})$ that agree on ${\cal F}$.
Then $\mu_1\equiv\mu_2$.
\end{proposition}

\begin{proof}
Let
$$
{\cal L} = \{B \in \sigma({\cal F}): \mu_1(B) = \mu_2(B)\}.
$$
First, ${\cal F} \subset {\cal L}$ by assumption.
For $B\in {\cal L}$, since $\R^d\in{\cal F} \subset {\cal L}$,
$$
\mu_1(B^c) = \mu_1(\R^d) - \mu_1(B) = \mu_2(\R^d) - \mu_2(B)
= \mu_2(B^c).
$$
So ${\cal L}$ is closed under the operation of complement.
We can similarly show that ${\cal L}$ is closed under countable disjoint unions.
Thus, ${\cal L}$ is a $\lambda$-system containing ${\cal F}$, and we conclude 
${\cal L}= \sigma({\cal F})$ by the $\pi-\lambda$ Theorem. 
\end{proof}

As usual, we say that a sequence $\{x_n\} \subset \V$ is relatively compact if there there is an $x\in\V$ such 
that $\|x_{n'}-x\|\to 0$ along some subsequence. The following results is straightforward;  
see e.g., Proposition 30 in \cite{melrose:2013_notes}.

\begin{lemma} \label{l:r-compact} Let $\V$ be a separable Hilbert space. 
A sequence $\{x_n\} \subset \V$ is relatively compact if and only if {\em (i)} $\sup_n \|x_n\| <\infty$ and {\em (ii)}
for some (any) CONS $\{e_j\}$, we have 
$$
\lim_{N\to\infty} \sup_{n} \sum_{i\ge N} |\langle x_n,e_i\rangle|^2 = 0.
$$
\end{lemma}

\begin{lemma}\label{l:trace} Let $0\le \A_n \le \B$, where ${\rm trace}(\B)<\infty$.  Then, there exists a subsequence $n'\to\infty$ and an operator $\A\in\mathbb T_+$, such that 
$$
\|\A_{n'} - \A\|_{\rm HS} \to 0
$$
where $\|\cdot\|_{\rm HS}$ stands for the Hilbert-Schmidt norm.
\end{lemma} \begin{proof} Fix a CONS $\{e_i\}$ and
consider the positive, self-adjoint Hilbert-Schmidt operators $\CC_n$ such that  $\A_n = \CC_n^2$.  
We have 
\begin{equation}\label{e:Cn-le-B}
\langle \A_n e_i,e_i\rangle = \| \CC_n e_i\|^2 \le \langle \B e_i, e_i\rangle =:B_i,\ \ \mbox{ where }\|\B\|_{\rm tr} = \sum_{i} B_i <\infty.
\end{equation}
We will first show that for all $e_j$ the sequence $\{\CC_n e_j\}$ is relatively compact in $\V$.  Indeed, 
$\sup_n \| \CC_n e_j\| <\infty$ and by the Cauchy-Schwartz inequality and \eqref{e:Cn-le-B}, we have
$$
\sup_n \sum_{i\ge N}  | \langle \CC_n e_j, e_i\rangle|^2  = \sup_n \sum_{i\ge N}  | \langle e_j, \CC_n e_i\rangle|^2 \le \sup_n \sum_{i\ge N}  \|\CC_n e_i\|^2 \le \sum_{i\ge N} B_i.
$$
Lemma \ref{l:r-compact} implies the desired relative compactness.  Thus, for some $n'\to\infty$, we have
$\|\CC_{n'} e_j - \CC(e_j)\|\to 0$ for some $\CC(e_j)\in \V$.  By passing to a further subsequence, we can assume that the latter
convergence holds for all $j$.  In view of \eqref{e:Cn-le-B}, we have 
\begin{equation}\label{e:C2}
\|\CC(e_j)\|^2 \le B_j \ \ \mbox{ and hence }\ \ \sum_j \| \CC(e_j)\|^2 \le \|\B\|_{\rm tr} <\infty.
\end{equation}

We shall argue next that $\CC$ can be extended by linearity to a bounded linear Hilbert-Schmidt operator and show 
that $\| \CC_n - \CC\|_{\rm HS}\to 0$. For $x =\sum_{i} x_j e_j \in \V$, consider
\begin{equation}\label{e:Cej}
\CC(x) := \sum_{j} x_j \CC(e_j).
\end{equation}
The latter series converges in $\V$.  Indeed, by the Cauchy-Schwartz inequality and \eqref{e:C2}, for any $J\subset \N$,
we have 
$$
\Big(\sum_{j\in J} \| x_j \CC(e_j)\| \Big)^{2}  \le \Big(\sum_{j\in J} |x_j|^2\Big)  \Big( \sum_{j\in J} \|\CC(e_j)\|^2\Big)\le \|\B\|_{\rm tr} \sum_{j\in J} |x_j|^2.
$$
This proves that \eqref{e:Cej} converges in norm and in fact implies that $\CC(x)$ is a bounded linear operator with 
$\|\CC\|_{\rm op} \le \|\B\|_{\rm tr}^{1/2}$.  Furthermore, clearly $\CC$ is self-adjoint and by Relation \eqref{e:C2} we have that 
$\CC$ is a Hilbert-Schmidt operator and $\CC^2$ is trace class.  

Observe now that
\begin{equation}\label{e:Cn'-C}
\|\CC_{n'} - \CC\|_{\rm HS}^2 = \sum_i \|(\CC_{n'} - \CC)e_i\|^2 \le \sum_{i=1}^{N} \| \CC_{n'} e_i - \CC e_i\|^2 + 4 \sum_{i>N} B_i,
\end{equation}
where we used \eqref{e:Cn-le-B}, \eqref{e:C2}, the triangle inequality for the norm in $\V$ and the simple bound 
$(a+b)^2 \le 2a^2 +  2b^2,\ a,b\ge 0$.  For every $\epsilon>0$, one can pick $N$ large enough such that $\sum_{i>N}B_i< \epsilon/8$. 
Holding $N$ fixed and appealing to the convergence $\|\CC_{n'} e_i - \CC e_i\|\to 0$, as $n'\to\infty$,
 one can make the first term in the right-hand side of
\eqref{e:Cn'-C} smaller than $\epsilon/2$, for all $n'$ large enough. This argument shows that 
$$
\|\CC_{n'} - \CC\|_{\rm HS} \to 0.
$$ 
Now,
$$
\CC_{n'}^2 - \CC^2 = (\CC_{n'}-\CC)(\CC_{n'}+\CC) + \CC(\CC_{n'}-\CC) + (\CC-\CC_{n'}) \CC.
$$  
Thus, $\|\A_{n'} -\A\|_{\rm HS} = \|\CC_{n'}^2 - \CC^2\|_{\rm HS} \to 0$.
% \fbox{DOES it? WHY?} \tcr{Here, we may need to use 
%Tailen's DCT result in the attached PDF.} {\color{blue} Here is an argument based on the}

\end{proof}

	\subsection{Proof of Theorem~\ref{th:IRF_k_spectral_operator_value_1}.}\label{suppsec:thm4.3_proof}

	We first prove (ii). For simplicity, we only illustrate proofs for $d=1$. The extension to the general case of $d$ is straightforward. Also, since the ``if\," part is obvious, it suffices to focus on the ``only if\," part by assuming that
	$\mathcal{K}: \bbR^d\mapsto \mathbb{T}$ is continuous and conditionally positive definite of degree $k$.
	
	Define 
	\begin{align} \label{e:C_mu}
	C_\mu(h) = \mathcal{K}(\mu*\wt\mu+h), \quad h\in\bbR, \mu\in\Lambda_k.
	\end{align}
	We first observe that $C_\mu$ is positive definite, i.e., 
	\begin{align*}
	\sum_{i=1}^n\sum_{j=1}^n c_i\bar c_j C_\mu(t_i-t_j) \ge 0 \quad
	\mbox{for all $c_i\in\bbC, t_i\in\bbR, i =1,\ldots, n$}.
	\end{align*}
	This follows simply from 
	\begin{align*}
	\sum_{i=1}^n\sum_{j=1}^n c_i\bar c_j C_\mu(t_i-t_j) 
	= \mathcal{K}\left((\mu*\lambda)*\wt{\mu*\lambda}\right) \ge 0,
	\end{align*}
	where $\lambda:=\sum_{i=1}^n c_i\delta_{t_i}$, since $\mu*\lambda\in\Lambda_k$
	and $\mathcal{K}$ is conditionally positive definite. 
	Thus, $C_\mu$  admits a spectral representation by Theorem \ref{th:operator_Bochner} 
	(Bochner's Theorem) where we denote by $\tau_\mu$ the spectral measure.
	
	Note also that
	\begin{align} \label{e:C_mu*nu}
	C_{\mu*\nu}(h) =  C_{\mu}(\nu*\wt\nu+h) = C_{\nu}(\mu*\wt\mu+h), \quad h\in\bbR, \mu, \nu\in\Lambda_k.
	\end{align} 
	If $\mu=\sum_{k=1}^m c_k \delta_{x_k}$ then
	\be
	C_\nu(\mu*\widetilde\mu+h) = \int \sum_{k=1}^m\sum_{l=1}^m c_k\bar c_le^{\ii (h+x_k-x_l)u}\tau_\nu(du)
	= \int e^{\ii hu} |\widehat\mu(u)|^2\tau_\nu(du)
	\ee
	where $\widehat\mu(u) = \int e^{\ii ux} \mu(dx)$. By \eqref{e:C_mu*nu} and the uniqueness of spectral measure, 
	\ben\label{e:equiv1}
	|\widehat\mu(u)|^2\tau_\nu(du) \equiv |\widehat\nu(u)|^2\tau_\mu(du), \quad \mu, \nu\in\Lambda_k.
	\een
	
	Let $\{\mu_n\}\in \Lambda_k$ and $\mu_0$ be as in Lemma~\ref{l:reference_measure} where 
	 $\mu_n\to\mu_0$ weakly and $\widehat\mu_0(u) \not = 0$ for $u\not=0$. Since $\mathcal{K}$ 
	 is continuous, we define a positive definite $C_\mu$ by taking a limit in \eqref{e:C_mu}, and conclude that
	 \eqref{e:equiv1} holds for any $\mu \in \Lambda_k$ and $\nu=\mu_0$.
	Define the measure
	\ben\label{e:sigma:supp}
	\chi(du) = \begin{cases} {\tau_{\mu_0}(du)\over |\widehat\mu_0(u)|^2}(1\wedge |u|^{2k+2}) & u\neq 0,\\
		0 & u=0.
	\end{cases}
	\een
	We will show $\chi$ is a $\T_+$-valued measure in Proposition~\ref{pp:K}. 
	For $\mu\in\Lambda_k$, let $S_\mu := \{ u: |\widehat{\mu}(u)|^2= 0 \}$, which contains $0$. 
	By Lemma~\ref{l:reference_measure}, $\widehat\mu_0(u)\not=0$ for $u\not=0$.
	Thus, for $u \not\in S_\mu$, both $\widehat\mu(u)$ and $\widehat\mu_0(u)$ are nonzero and hence 
	\eqref{e:equiv1} entails that
	\begin{align*}
	{\tau_{\mu_0}(du) \over |\widehat\mu_0(u)|^2} = {\tau_{\mu}(du) \over |\widehat\mu(u)|^2}.
		\end{align*}
	This implies that
	\begin{align}\label{e:RHS}
	\int e^{\ii hu} \tau_\mu(du) = \int  \frac{e^{\ii hu}|\widehat\mu(u)|^2}{1\wedge |u|^{2k+2}} \chi(du)
	+ \int_{u\in S_\mu} e^{\ii hu}\tau_\mu(du).
	\end{align}
	Since $\widehat\mu(z), z\in\mathbb{C}$, is an entire function, the set of zeros is countable by the 
	the Identity Theorem. Also, for $u\in S_\mu\setminus\{0\}$, we have $\widehat\mu(u) = 0$ and 
	$\widehat\mu_0(u) \not= 0$, which, by \eqref{e:equiv1} with $\nu=\mu_0$, entails $\tau_\mu(\{u\})=0$.
	%Notice that by similar construction as we used for $\nu_s$, for any fixed $\mu$ and location $u_0\in S_\mu\setminus\{0\}$, we will have
	%\begin{eqnarray}\label{eq:zero}
	%\begin{split}
	%\tau_\mu(\{u_0\}) &=\lim_{\epsilon\to 0}\int_{(u_0-\epsilon,u_0+\epsilon)}\tau_{\mu}(du)= \lim_{\epsilon\to 0}\int_{(u_0-\epsilon,u_0+\epsilon)}\frac{|\widehat\mu(u)|^2}{|\widehat\mu_0(u)|^2}\tau_{\mu_0}(du)\\& = \frac{|\widehat\mu(u_0)|^2\tau_{\nu_{s_{u_0}}}(\{u_0\})}{|\widehat \mu_0(u_0)|^2}=0.
	%\end{split}
	%\end{eqnarray}
	Therefore, \eqref{e:RHS} becomes
	\ben\label{e:RHS1}
	\int e^{\ii hu} \tau_\mu(du) = \int \frac{e^{\ii hu}|\widehat\mu(u)|^2}{1\wedge |u|^{2k+2}} \chi(du)+\tau_{\mu}(\{0\}).
	\een
	Define
	\begin{align} \label{e:K0_1}
	\mathcal{K}_0(h)=\int \frac{e^{\ii hu} - I_B(u) P(uh)}{1\wedge |u|^{2k+2}} \chi(du),
	\end{align}
	where $B$ is a bounded interval containing $0$ and
	\be
	P(x)=\sum_{j=0}^{2k+1}  (\ii x)^{j}/j!.
	\ee
	We will establish in Proposition~\ref{pp:K} that $\mathcal{K}_0(h)$ is well defined for every $h$. 
	Notice that in \eqref{e:RHS1}, LHS $= C_\mu(h)$ and hence RHS is equal to
	\ben\label{eq:K_02C_mu}
	\begin{split}
		C_\mu(h) &=  \int \frac{\sum_{i=1}^m\sum_{j=1}^m c_i\bar c_j e^{\ii(h+x_i-x_j)u}}{1\wedge |u|^{2k+2}} \chi(du) +\tau_{\mu}(\{0\})\\
		&= \sum_{i=1}^m\sum_{j=1}^m c_i\bar c_j \int \frac{e^{\ii(h+x_i-x_j)u}-I_B(u) P(u(h+x_i-x_j))}{{1\wedge |u|^{2k+2}}}
		\chi(du) +\tau_{\mu}(\{0\})\\
		%&= \sum_{i=1}^m\sum_{j=1}^m c_i\bar c_j \mathcal{K}_0(h+t_i-t_k)+\tau_{\mu}(\{0\}),
		&= {\mathcal K}_0(\mu*\widetilde{\mu}+h) + \tau_{\mu}(\{0\}),
	\end{split}
	\een
	since $P(\mu*\widetilde\mu) = 0$. The rest of the proof focuses on the property of $\tau_{\mu}(\{0\})$. 
	
	For any $\mu,\nu\in\Lambda_k$, define
	\be
	C_{\mu,\nu}(h) = \mathcal{K}(\mu*\wt\nu+h)
	\ee
	By Lemma~\ref{l:identification_operator},
	\begin{align} \label{e:mu_nu_1}
	C_{\mu,\nu}(h) = \frac{\ii-1}{2}(C_\mu(h)+C_\nu(h))+\frac{1}{2}C_{\mu+\nu}(h)-\frac{\ii}{2}C_{\ii\mu+\nu}(h),
	\end{align}
	and
	\begin{align} \label{e:mu_nu_2}
	\begin{split}
	\mathcal{K}_0(\mu*\widetilde{\nu}+h) & = \frac{\ii-1}{2}(\mathcal{K}_0(\mu*\widetilde{\mu}+h)+\mathcal{K}_0(\nu*\widetilde{\nu}+h))\\
	& \hspace{1cm} +\frac{1}{2}\mathcal{K}_0((\mu+\nu)*\widetilde{(\mu+\nu)}+h)
	-\frac{\ii}{2}\mathcal{K}_0((\ii\mu+\nu)*\widetilde{(i\mu+\nu)}+h).
	\end{split}
	\end{align}
	Thus, it follows from \eqref{eq:K_02C_mu} that
	\begin{align} \label{e:mu_nu_3}
	C_{\mu,\nu}(h) - \mathcal{K}_0(\mu*\widetilde{\nu}+h) 
	=  \frac{\ii-1}{2}(\tau_\mu(\{0\})+\tau_\nu(\{0\}))+\frac{1}{2}\tau_{\mu+\nu}(\{0\})-\frac{\ii}{2}\tau_{\ii\mu+\nu}(\{0\}).
	\end{align}
	%It follows that
	%\be
	%C_{\mu_1,\mu_2}(h_1-h_2) = \sum_{i,j}c^{(1)}_i\bar{c}^{(2)}_j{\mathcal C}(s_0+x^{(1)}_i+h_1, s_0+x^{(2)}_j+h_2), \
	%h_1,h_2\in \R,
	%\ee
	%where $\mu_j = \sum_i c^{(j)}_i \delta_{x^{(j)}_i}, j = 1,2$. 
	Define 
	\be
	F(h_1,h_2) = {\mathcal K}(h_1 - h_2) - \mathcal{K}_0(h_1-h_2),
	\ee
	and
	$$
	F(\mu, \nu) = \iint F(h_1,h_2) d\mu(h_1)d\nu(h_2), \quad \mu,\nu \in \Lambda_k.
	$$ 
	Applying \eqref{e:mu_nu_3}, we obtain
	\begin{align} \label{e:F_1}
	\begin{split}
	F(\mu+h_1,\nu+h_2) &= C_{\mu,\nu}(h_1-h_2)-\mathcal{K}_0(\mu*\widetilde{\nu}+h_1-h_2)\\
	&=\frac{\ii-1}{2}(\tau_\mu(\{0\})+\tau_\nu(\{0\}))+\frac{1}{2}\tau_{\mu+\nu}(\{0\})
	-\frac{\ii}{2}\tau_{\ii\mu+\nu}(\{0\}).
	\end{split}
	\end{align}
	Since this expression does not depend on $h_1,h_2$ for all $\mu, \nu\in \Lambda_k$,
	Lemma~\ref{l:bivariate_polynomial} can be invoked to give
	\begin{align} \label{e:F_2}
	F(h_1,h_2) = \sum_{l=1}^k G_l^{(1)}(h_1)(h_2)^l + \sum_{l=1}^k G_l^{(2)}(h_2)(h_1)^l + C_0(h_1-h_2)^{2k+2},
	\end{align}
	where the $G_l^{(i)}$ are arbitrary functions. Then, \eqref{e:F_1} and \eqref{e:F_2} entail that
	\begin{align*}
	F(\mu,\mu) = \tau_{\mu}(\{0\}) = C_0 w(\mu*\widetilde{\mu})
	\end{align*}
	where $w(h) = h^{2k+2}$ and $\mu$ is any arbitrary measure in $\Lambda_k$ such that 
	$w(\mu*\widetilde{\mu})\not = 0$. Thus, we conclude that
	\begin{align*}
	C_0 = {\tau_{\mu}(\{0\}) \over w\left(\mu*\widetilde{\mu}\right)},
	\end{align*}
	which does not depend on $\mu\in\Lambda_k$.
	It follows that	
	\begin{align*}
	\mathcal{K}_1(h) := \int \frac{e^{\ii hu} - I_B(u) P(uh)}{1\wedge \|u\|^{2k+2}} \chi(du) + C_0h^{2k+2}
	\end{align*}
	is an equivalent version of $\mathcal{K}$ in the sense that $\mathcal{K}_1(\mu*\wt\nu) = 
	\mathcal{K}(\mu*\wt\nu), \mu,\nu\in\Lambda_k$.
	To prove uniqueness of $\chi$, assume there is another $\chi'$ that can be used in the representation
	of the generalized covariance operator. Then, for any $\mu\in\Lambda_k$,  we have 
	\begin{align} \label{e:chi=chi'}
	\int \frac{e^{\ii hu}|\widehat{\mu}(u)|^2}{{1\wedge |u|^{2k+2}}}\chi(du) 
	\equiv \int \frac{e^{\ii hu}|\widehat{\mu}(u)|^2}{{1\wedge |u|^{2k+2}}}\chi'(du).
	\end{align}
	By Lemma~\ref{l:reference_measure} and weak convergence, \eqref{e:chi=chi'} holds for $\mu=\mu_0$
	as well. Since $\widehat\mu_0(u)\not=0$ for $u\not=0$, it follows immediately that $\chi=\chi'$. 
	This completes the proof of (ii).
	
	We next prove (i). The first statement is obviously true. The proof of the second
	statement follows the exact same line of arguments as the proof of (ii), except we let 
	\begin{align*}
	C_\mu(h) = C_Y(\mu+h,\mu) \quad\mbox{and}\quad  C_{\mu,\nu}(h) = C_Y(\mu+h, \nu).
	\end{align*}
	
	\ep
	
	\begin{proposition} $\chi(du)$ defined in \eqref{e:sigma:supp} is a finite $\T_+$-valued measure and \label{pp:K} $\mathcal{K}_0$  
	defined in \eqref{e:K0_1} is well defined and $\bbT$-valued function whenever $X$ is mean-square continuous.
	\end{proposition}
	\begin{proof}
		It suffices to show that there exists a neighborhood $B$ containing $0$ such that
		\ben\label{e:target1_sigma}
		\|\chi(B)\|_{\tr} = \left\|\int_B \frac{1\wedge |u|^{2k+2}}{|\widehat \mu_0(u)|^2} 
		\tau_{\mu_0}(du) \right\|_{\tr}<\infty,
		\een
		\ben\label{e:target2_sigma}
		\|\chi(B^c)\|_{\tr} = \left\|\int_{B^c} \frac{1\wedge |u|^{2k+2}}{|\widehat \mu_0(u)|^2} 
		\tau_{\mu_0}(du) \right\|_{\tr}<\infty
		\een
		\ben\label{e:target1}
		\left\|\int_{B}\frac{\left|e^{\ii hu} - P(uh)\right|}{1\wedge|u|^{2k+2}} \chi(du)\right\|_{\tr}< \infty,
		\een
		and
		\ben\label{e:target2}
		\left\|\int_{B^c}\frac{1}{1\wedge|u|^{2k+2}} \chi(du)\right\|_{\tr}< \infty.
		\een
		By Taylor's expansion,
		\be
		\left\|\int_B |e^{\ii hu}-P(uh)|\chi(du)\right\|_{\tr}\leq \left\| {|h|^{2k+2}\over (2k+2)!} \int_B |u|^{2k+2}\chi(du)\right\|_{\tr}
		\ee
		and
		\be
		2(1-\cos u) \ge u^2-u^4/12 > u^2/2, \ |u| < \sqrt{6}.
		\ee
		Thus, taking $B=(-\sqrt{6},\sqrt{6})$, $\mu=(\delta_1-\delta_0)^{*(k+1)}$, the $(k+1)$-convolution power of the measure 
		$\delta_1-\delta_0$, we have $|\wh \mu(u)|= | 1- e^{\ii u}|^{k+1}$, and since $\tau_{\mu_0}(\{0\}) =0$, we obtain
			\be
			&&\left\|\int_B \frac{1\wedge |u|^{2k+2}}{|\widehat \mu_0(u)|^2} \tau_{\mu_0}(du) \right\|_{\tr}
			\leq \left\|\int_B \frac{ |u|^{2k+2}}{|\widehat \mu_0(u)|^2}\tau_{\mu_0}(du) \right\|_{\tr}\\
			&& \leq \left\|\int_B \frac{ |u|^{2k+2}}{|\widehat \mu(u)|^2} \tau_{\mu}(du) \right\|_{\tr}\leq 
			\left\|\int_B \tau_{\mu}(du) \right\|_{\tr}<\infty,
			\ee
			and
			\be
			&&\left\|\int_B \frac{|e^{\ii hu}-P(uh)|}{1\wedge |u|^{2k+2}}\chi(du)\right\|_{\tr}\leq \left\| {|h|^{2k+2}\over (2k+2)!} \int_B \frac{|u|^{2k+2}}{1\wedge |u|^{2k+2}}\chi(du)\right\|_{\tr}\\
			&& \leq {|h|^{2k+2}\over 2(2k+2)!} \left\|\int_B \frac{|\widehat{\mu}(u)|^{2}}{1\wedge |u|^{2k+2}}\chi(du)\right\|_{\tr}<\infty,
			\ee
			where the last inequality follows from the facts that both sides of \eqref{e:RHS} are finite for any $h$ (and in particular for $h=0$) and $\mu\in\Lambda_k$.  
			This establishes \eqref{e:target1_sigma} and \eqref{e:target1}.

		To prove \eqref{e:target2}, take $\mu_0$  as the measure in Lemma~\ref{l:reference_measure}
		and in the proof of Theorem~\ref{th:IRF_k_spectral_operator_value_1}.	
		%\fbox{Is there any relationship between $M_c$ and the set of Radom measure annihilating $k$th order polynomial? }
		%\fbox{It is out of the scope of this document though.}
		By Lemma~\ref{l:reference_measure}, we can pick
		$\delta$ such that $\inf_{u\in B^c}|\widehat{\mu}_0(u)|>\delta > 0$. Therefore, 
		\be
		\left\|\int_{B^c}\frac{(1\wedge |u|^{2k+2})}{|\widehat\mu_0(u)|^2} \tau_{\mu_0}(du)\right\|_{\tr}\leq \left\|\delta^{-2}\int_{B^c}\tau_{\mu_0}(du)\right\|_{\tr}<\infty,
		\ee
		\be
		\left\|\int_{B^c} \frac{1}{1\wedge|u|^{2k+2}} \chi(du)\right\|_{\tr}\leq \left\|\delta^{-2} \int \frac{|\widehat\mu_{0}(u)|^2}{1\wedge |u|^{2k+2}} \chi(du)\right\|_{\tr}<\infty
		\ee
		which proves \eqref{e:target2_sigma} and \eqref{e:target2}.
		
	\end{proof}

Let 
$$
M_f(k) = \{\mu=\mu_1*\cdots*\mu_{k+1}: \mu_i\in\Lambda_0\}.
$$
Also, define the larger class
\begin{align*}
M_c(k) &= \Big\{\mu_1*\cdots*\mu_{k+1}: 1(\mu_i) = 0, \\
 & \quad\quad\quad \text{ $\mu_i$'s are finite signed measures with 
	compact supports} \Big\}.
\end{align*}
Note that $M_c(k)$ contains discrete as well as diffuse measures. 
It is easy to see that any $\mu\in M_c(k)$ annihilates polynomials of degree $k$. 

\begin{lemma}\label{l:M_f2M_c+bound2} 
	\begin{enumerate}[label=(\roman*)]
	\item \label{l:M_f2M_c}	
	For any $k\in \mathbb{N}$, $M_f(k)$ is a dense set in $M_c(k)$ with respect to the weak topology.
	\item \label{l:bound2} 
	For any open set $O$ containing $0$, there exists $\mu\in M_c(k)$ and $\delta>0$ such that $\mu$ is symmetric, $\|\mu\|_{\rm TV}<\infty$, where $\|\mu\|_{\rm TV}$ stands for the total variation of $\mu$, and for any $u\in O^c$,
	$\widehat{\mu}(u)>\delta$.
	\item \label{l:completeness_M_f}
	Let $f(h):\R\mapsto \mathbb{T}$ be a continuous function in trace norm. If for some $k\in \mathbb{N}$, $f(\mu)=0$ for any $\mu\in M_f(k)$, then $f(\lambda)=0$ for any $\lambda\in\Lambda_{k}$.
	\end{enumerate}
	\end{lemma}
	
	\begin{proof}
	Part (i) follows readily from Lemma 2.5 of \cite{Sasvari:2009}. The proof of (ii) is almost exactly the same 
	as the proof for Lemma 3.2 of \cite{Sasvari:2009}.  One only needs to notice that $\mu$ in the proof of their original lemma can be made to be symmetric under our scenario.  Part (iii) is a trivial extension of Lemma C.9.4 of \cite{Sasvari:2013wn}.
	\end{proof}

\begin{lemma}\label{l:reference_measure}
	There exists a measure $\mu$ on $\R$ that satisfies
	\begin{enumerate}
		\item\label{item:property1} $\|\mu\|_{\rm TV}<\infty$.
		\item For any open set $O$ containing $0$, there exists $\delta>0$ such that for any $u\in O^c$, $\widehat\mu(u)>\delta.$
		\item There exists a sequence of $\{\mu_n\}\subset M_f(k)$ converging to $\mu$ weakly.
	\end{enumerate}
\end{lemma}
\begin{proof}
	For $V_n = [-1/n,1/n]$, by \ref{l:bound2} of Lemma~\ref{l:M_f2M_c+bound2}, there exists $\{\mu_n\}\subset M_c(k)$ be a set of symmetric measures satisfying $\|{\mu}_n\|_{\rm TV}\leq 1/2^n$ and $\mu_n(u)>0$ for $u\notin V_n$. Then define $\mu = \sum_n \mu_n$, which obviously satisfies properties (i) and (ii). Property (iii) is guaranteed by (i) of Lemma~\ref{l:M_f2M_c+bound2} and the fact $\sum_{n=1}^M \mu_n\to \mu$ weakly. 
\end{proof}

The following result can be easily obtained from Proposition 3 in \cite{Berschneider:2012de} and its proof.
\begin{lemma}\label{l:coef_poly}
	Let $f(h):\R\mapsto \mathbb{T}$. If for some $k\in \mathbb{N}$, $f(\lambda)=0$ for any $\lambda\in\Lambda_{k}$, then  $f$ has the following decomposition
	\be
	f(h) = \sum_{i=0}^{k}a_ih^i,
	\ee
	where $a_i = \sum_{j=0}^{M_{k}}b_{ij}f(x_j)$ for some constants $b_{ij}, x_j, i = 0,\dots, k+1, j = 0,\dots, M_{k}$.
	\end{lemma}
\begin{lemma}\label{l:bivariate_polynomial}
	Assume that the bivariate function $F(h_1,h_2)$ from $\bbR^2$ to $\mathbb{T}$ is continuous in trace norm.   If $F(h_1+\mu,h_2+\nu)$
	does not depend on $h_1,h_2$ for any two measures $\mu,\nu\in\Lambda_k$, then $F(h_1,h_2)$ has the following decomposition
	\be
	F(h_1,h_2) = \sum_{l=1}^k G_l^{(1)}(h_1)(h_2)^l + \sum_{l=1}^k G_l^{(2)}(h_2)(h_1)^l + C_0(h_1-h_2)^{2k+2},
	\ee
	where $C_0$ is an arbitrary operator in $\T$ and the $G_l^{(i)}$'s are arbitrary $\T$-valued functions. 
\end{lemma}

\begin{proof}
	
	Fix $h_1,\mu\in\Lambda_k$. By ~\ref{l:completeness_M_f} of Lemma~\ref{l:M_f2M_c+bound2}, 
	we obtain for any $\nu\in\Lambda_{k+1}$, $F(h_1+\mu,h_2+\nu) =0$. Thus, by Lemma~\ref{l:coef_poly},
	\be
	F(h_1+\mu,h_2) =\sum_{i=0}^{k+1}G_{1,i}(h_1+\mu)h_2^i  \quad\mbox{where}\quad
	G_{1,i}(h) = \sum_{j=0}^{M_{k+1}}b_{ij}F(h,x_j).
	\ee
	Therefore, again by Lemma~\ref{l:coef_poly},  
	\be
	F(h_1,h_2) = \sum_{i=0}^{k+1}G_{1,i}(h_1)h_2^i + \sum_{i=0}^{k}G_{2,i}(h_2)h_1^i.
	\ee
	Notice that for any $\mu,\nu\in\Lambda_k$, $F(h_1+\mu,h_2+\nu) = G_{1,k+1}(h_1+\mu)C(\nu)$,
	where $C(\nu) = \int h^{k+1}\nu(dx)$ is some constant. Since, by assumption $G_{1,k+1}(h_1+\mu)$
	does not depend on $h_1$, we conclude that $G_{1,k+1}(\mu)=0$ for all $\mu\in M_f({k+1})$. Applying
	\ref{l:completeness_M_f} of Lemma~\ref{l:M_f2M_c+bound2} and Lemma~\ref{l:coef_poly} 
	again, we get $G_{1,k+1}(h_1) = \sum_{i=0}^{k+1}a'_ih_1^i$. 
	The result follows after reorganizing the terms.
	\end{proof}
	
	\subsection{Proof of Theorem~\ref{th:integral_representation_stationary}.} \label{supp:proof:Cramer-stationary}
	
	%For simplicity, we only illustrate the proof on $\R$. The extension to $\R^d$ should be straightforward. 
	Recall the notation in Section~\ref{sec:appdix_integral_def}.
	Since
	\begin{align*}
	 {\rm trace} (X(s)\otimes X(t)) &=  \sum_j \langle X(s)\otimes X(t) e_j, e_j\rangle\\
	 &= \sum_j \langle X(s), e_j\rangle \overline{\langle X(t), e_j \rangle}
	 = \langle X(s), X(t)\rangle,
	 \end{align*}
	we have, by Bochner's Theorem, 
	\begin{align} \label{e:Bochner_fact}
	\langle X(s), X(t) \rangle_\Omega = \mathbb{E}\langle X(s),X(t)\rangle 
	= {\rm trace} (C(t-s)) = \int e^{\ii (t-s)^\top x} \|\mu\|_\tr (dx),
	\end{align}
	where $\mu$ the spectral measure of the stationary covariance operator, and $ \|\mu\|_\tr(A) :=  {\rm trace}(\mu(A))$.
	
	Let $\bbL$ be the linear span of the functions $\{e^{\ii t^\top\cdot}, t\in \R^d\}$. 
	First, note that $\bbL$ is dense in $\bbL^2(\R^d, \|\mu\|_\tr)$. To see this, assume 
	$g\in \bbL^2(\R^d, \|\mu\|_\tr)$ and 
	\be
	\int e^{\ii t^\top x}\overline{g(x)} \|\mu\|_\tr(dx) = 0 \mbox{ for all $t$}.
	\ee
	Then $\nu(dx):= \overline{g(x)}\|\mu\|_\tr (dx)$ is a finite measure and thus the above implies
	that $g$ corresponds to the zero element in $\bbL^2(\R^d, \|\mu\|_\tr)$.

	Define ${\mathbb H}(X) = \overline{{\rm span}(X(t), t \in \bbR^d)}$, the closed linear span in
	$\bbL^2(\Omega)$, and the linear mapping
	\be
	\CJ_X:  \sum_{k = 1}^{n}c_ke^{\ii t_k^\top x} \mapsto  
	\sum_{k = 1}^{n} c_k X(t_k), \ \bbL \mapsto \bbH(X).
	\ee
	By \eqref{e:Bochner_fact},
	\be 
	\|\CJ_X(f)\|_{\Omega} = \|f\|_{\bbL^2(\R^d, \|\mu\|_\tr)}, \ f\in\bbL.
	\ee
	By the denseness of $\bbL$, $\CJ_X$ can be readily extended to an isometric isomorphism, 
	still denoted as $\CJ_X$, from $\bbL^2(\R^d, \|\mu\|_\tr)$ to ${\mathbb H}(X)$.

	%by Theorem D.3.6 in \cite{Sasvari:2013wn}, 
	
	%For any $A,B,A_n\in \V(X)$, satisfying $\lim_{n\to\infty}\|A_n-A\|_{\Omega} =0$, we have
	%\be
	%|\langle A_n-A, B\rangle_\Omega|\leq \sqrt{\langle A_n-A, A_n-A\rangle_{\Omega}\langle B, B\rangle_{\Omega}}\to 0,
	%\ee
	%and thus $\mathbb{E}[A_n\otimes B]$ converges in operator norm 
	%(relative to $\|\cdot\|_\Omega$) to $\mathbb{E}[A\otimes B]$. 
	%As $X(t)$ is dense in $\V(X)$ and ${\mathcal C}(s,t)$ is self-adjoint, we will have $\mathbb{E}[A\otimes B]$ is self-adjoint for any $A,B\in \V(X)$. 
	
	For  $f,g\in \bbL^2(\R^d, \|\mu\|_\tr)$ and $e\in \V$, we have
	\be
	\langle \mathbb{E}[\CJ_X(f)\otimes \CJ_X(g)] e, e \rangle = \int f(x)\overline{g(x)}\langle \mu(dx) e, e\rangle.
	\ee
	Indeed, this obviously holds or $f,g\in \bbL$, and, by the denseness of $\bbL$, the extension to
	$\bbL^2(\R^d, \|\mu\|_\tr)$ is straightforward.
	Then, applying Lemma~\ref{l:identification_operator}, it follows that
	\begin{align}\label{eq:equity_operator}
	\mathbb{E}[\CJ_X(f)\otimes \CJ_X(g)] = \int f(x)\overline{g(x)}\mu(dx)
	\hbox{ for all $f,g\in \bbL^2(\R^d, \|\mu\|_\tr)$}.
	\end{align}
	Define $\xi$ by 
	\be
	\xi(A) = \CJ_X(\mathbf{1}_{A}), \ A\in\mathcal{B}(\bbR^d).
	\ee
	By isometry, for any sequence of Borel sets $A_n$ satisfying $A_n\to \emptyset$, 
	\be
	\|\xi(A_n)\|_{\Omega} = \|\mathbf{1}_{A_n}\|_{\bbL^2(\R^d, \|\mu\|_\tr)} =  \|\mu(A_n)\|_\tr
	\to 0.
	\ee
	For any disjoint Borel sets $A$ and $B$, we have by \eqref{eq:equity_operator} that
	\be
	\mathbb{E}[\xi(A)\otimes \xi(B)] =0. 
	\ee
	Finally, by \eqref{eq:equity_operator}, for any Borel set $A$, 
	\be\label{eq:structure}
	\mathbb{E}[\xi(A)\otimes \xi(A)] =\mu(A).
	\ee
	Thus, we have proved that $\xi$ is a random orthogonal measure with control measure 
	$\mu$ (cf Definition \ref{def:orthogonal-measure})
	 and therefore $\int e^{itx}d\xi(x)$ is well-defined by the construction in Section~\ref{sec:appdix_integral_def}.

	 It remains to show 
	$$
	X(t) = \int_{\R^d} e^{\ii t^\top x} \xi(dx).
	$$
	By \eqref{eq:equity_operator}, for any Borel set $A$, 
	\be
	\mathbb{E}[X(t)\otimes \xi(A)] = \int e^{it^\top x} \mathbf{1}_{A}(x)\mu(dx). 
	\ee
	Next, observe that, for any fixed $s$, $\int_{\R^d} e^{\ii s^\top x} \xi(dx)\in \bbH(X)$ since it is the limit of
	$\CJ_X(g_n)$ where $g_n(x) = \sum_k c_k^n\mathbf{1}_{A_k^n}(x),n\in\N, A_{k}\cap A_{l}=\emptyset$ when $k\neq l$, are step functions converging to $e^{is^\top x}$ in $\bbL^2(\R^d, \|\mu\|_\tr)$ as $n\to \infty$.
	Consequently, 
	\be
	\mathbb{E}\left[X(t)\otimes  \int_{\R^d} e^{\ii s^\top x} \xi(dx)\right] 
	= \int e^{\ii (t-s)^\top x} \mu(dx) = \mathbb{E}[X(t)\otimes X(s)].
	\ee
	Thus, we have 
	\be
	\mathbb{E}\left[\sum_{i=1}^m c_iX(t_i)\otimes Y(s)\right]=0
	\ee
	for $Y(s):= \int_{\R^d} e^{\ii s^\top x} \xi(dx)-X(s)$, which is in $\bbH(X)$.	
	Since linear combinations $\sum_{i=1}^m c_iX(t_i)$ are dense in $\bbH(X)$, 
	we conclude that $\E (Y\otimes Y) = 0$ which implies that $Y=0$ in $\bbH(X)$.
	
	To prove the uniqueness of $\xi$, let $\eta$ be a random orthogonal measure on $\R^d$ 
	both with structure measure $\chi$. Then we will have
	\be
	\int e^{\ii t^\top x}\eta(dx) = \int e^{\ii t^\top x}\xi(dx),
	\ee
	and, applying again the denseness of $\bbL$, we obtain
	\be
	\int h(x)\eta(dx) = \int h(x)\xi(dx), \ \ \mbox{ almost surely},\ \forall h\in \bbL^2(\R^d, \|\mu\|_\tr).
	\ee
	Taking $h = \mathbf{1}_B$ for any Borel set $B$ immediately leads to the conclusion $\eta \equiv \xi$. 
	\ep
	
	\subsection{Proof of Theorem~\ref{th:integral_representation_IRF}.}
	\label{supp:th:integral_representation_IRF}
	%For simplicity, we only illustrate proofs  on $\R$ and the extension to $\R^d$ should be straightforward. 
	
	%\vskip.3cm
	%\noindent
	%\textsc{Proof} for Theorem~\ref{th:integral_representation_IRF}:
	%\vskip.2cm\noindent
	We will focus on the ``only if" part since the ``if" part is obvious.
	By Theorem~\ref{th:integral_representation_stationary}, for any $\lambda\in\Lambda_k$, 
	\be
	Y(\lambda+t) = \int e^{\ii t^\top x} \xi_{\lambda}(dx)
	\ee
	for an orthogonal random measure $\xi_{\lambda}$ with control measure $\tau_{\lambda}$.
	Let $\mu_0$ be the measure defined in the proof of Theorem~\ref{th:IRF_k_spectral_operator_value_1}.
	By the proof of Theorem~\ref{th:IRF_k_spectral_operator_value_1}, we also have
	\be
	Y(\mu_0+t) = \int e^{\ii t^\top x} \xi_{\mu_0}(dx)
	\ee
	for an orthogonal random measure
	$\xi_{\mu_0}$ with control measure $\tau_{\mu_0}$. Thus,
	\be
	Y(\lambda*\mu_0+t) = \int e^{\ii t^\top x} \hat\mu_0(x)\xi_{\lambda}(dx)
	= \int e^{\ii t^\top x} \hat\lambda(x) \xi_{\mu_0}(dx).
	\ee
	By the uniqueness of the orthogonal random measure in the integral representation of a stationary
	process, we have
	\begin{align} \label{e:equiv_rm}
	\hat\mu_0(x)\xi_{\lambda}(dx) \equiv \hat\lambda(x) \xi_{\mu_0}(dx).
	\end{align}	
	By \eqref{e:RHS1}, for any $\mu=\mu_0$ and $\lambda$, we have
	\be
	\int \frac{e^{\ii t^\top x}|\widehat{\mu}(x)|^2}{1\wedge|x|^{2k+2}} \chi(dx) 
	= \int e^{\ii t^\top x}\tau_\mu(dx)-\tau_\mu(\{0\}). 
	\ee
	Now, define $\xi(\{0\}) = 0$, and for any Borel set $B$,
	\be
	\xi(B) = \int_{\R\backslash \{0\}} \mathbf{1}_B(x)\frac{1\wedge |x|^{k+1}}{\hat\mu_0(x)}\xi_{\mu_0}(dx). 
	\ee
	Then $\xi$ is a well-defined orthogonal random measure with control measure $\chi$. Set
	\be
	U(t) := Y(t) - \int \frac{e^{\ii t^\top x} - P(t^\top x)}{1\wedge|x|^{k+1}} \xi(dx),
	\ee
	where $P(x) = \sum_{j=0}^{k} (\ii x)^j/j!,\ x\in \R$.
	The integral exists by a similar argument to that of Proposition~\ref{pp:K}. Therefore, for any $\lambda\in \Lambda_k$, 
	\begin{align}\label{eq:polynomial_Yt}
	\begin{split}
	U(\lambda +t) & = \int e^{\ii t^\top x}\xi_{\lambda}(dx) - \int_{\R^d\backslash\{0\}}e^{\ii t^\top x}
	\frac{\hat\lambda(x)}{1\wedge|x|^{k+1}} \xi(dx) \\ 
	& = \int e^{\ii t^\top x}\xi_{\lambda}(dx) - \int_{\R^d\backslash\{0\}}e^{\ii t^\top x}
	\frac{\hat\lambda(x)}{\hat\mu_0(x)} \xi_{\mu_0}(dx) \\
	& = \int e^{\ii t^\top x}\xi_{\lambda}(dx) - \int_{\R^d\backslash\{0\}}e^{\ii t^\top x} \xi_{\lambda}(dx) \\
	& = \xi_\lambda(\{0\})
	\end{split}
	\end{align}
	where we applied \eqref{e:equiv_rm}.
	By a trivial extension of the proof of Lemma C.9.4 in \cite{Sasvari:2013wn}, it follows that 
	$U$ is a random polynomial of degree no greater than $k+1$. The uniqueness of $\xi$ and Properties (iii) of the theorem 
	are also easily established. \ep

\section{Gaussian elements in a Hilbert space.} \label{sec:supp:Gaussian}

Gaussian distributions in general Banach spaces have been studied extensively (see e.g.\ the monographs of \cite{ledoux:talagrand:1991} and 
\cite{kwapien:woyczynski:1992}, and the references therein.)
 In this section, we provide a simple self-contained review of some basic results in the special case when the Banach space 
 is a {\em separable Hilbert space} $\V$. The treatment may be of independent interest since it does not use advanced tools.

 \subsection{Real and complex Hilbert spaces.} \label{sec:real-irreversible}
 
 It is well-known that the mean and covariance structure determine the distribution of {\em real} Gaussian vectors.  Perhaps surprisingly, the
 mean and covariance alone are not enough to determine the distribution of complex Gaussian vectors. To clarify the issue, we start with a simple 
 example in the case $\V = \C^m$. 

\begin{example}\label{ex:supp:complex-Gauss}
 Let $Z = X+ \ii Y$, where $X$ and $Y$ are independent and Normally distributed random vectors in $\R^m$ having zero means and 
 variance-covariance matrices $\Sigma_X = \E [ X X^\top]$ and $\Sigma_Y = \E [ Y Y^\top]$.  Observe that the covariance operator of $Z$ in $\C^m$,
 in the standard basis, can be expressed as the matrix:
 $$
{\cal C}_Z =  \E [ Z\otimes Z] \equiv \E [ Z \overline Z^\top] = \Sigma_X  + \Sigma_Y.
 $$
 On the other hand, the real vector $\wt Z:= X+Y$ has the same covariance operator in $\C^m$ as $Z$.  This clearly shows that the distribution of 
 zero-mean Gaussian random elements in complex Hilbert spaces are {\em not determined} by their covariance operators alone.  To determine their distribution,
 one needs the additional information about the cross-covaraince between the real and imaginary parts, that is, the pseudo-covariance
 $
 {\cal C}_{Z,\overline Z} = \E [Z\otimes \overline Z].
 $
 \end{example}
 
 The purpose of this section is to clarify how one can deal with this issue in an abstract complex Hilbert space.  We start by recalling and expanding on the 
 notion of a real and imaginary part in a complex Hilbert space outlined in Section \ref{sec:real-complex} of the main paper.\\

{\em Real and imaginary parts in a complex Hilbert space.} Note that in an abstract complex Hilbert space (over the field $\bbC$) the notion of a real and imaginary part 
is not well-defined unless one fixes a basis. Let ${\cal E}:=\{e_j,\ j\in\N\}$ be a {\em fixed} CONS of $\V$.  Then one can postulate that the CONS ${\cal E}$ is
{\em real} and for each $z = \sum_{j} z_j e_j\in \V$, with coordinates $z_j:=\langle z,e_j\rangle,$ we can define
\begin{equation}\label{e:supp:real-complex}
\Re(z) \equiv \Re_{\cal E}(z):= \sum_{j} \Re(z_j) e_j\ \ \ \mbox{ and }\ \ \ \Im(z) \equiv \Im_{\cal E}(z):= \sum_{j} \Im(z_j) e_j,
\end{equation}
as the {\em real} and {\em imaginary} parts of $z$, relative to the CONS ${\cal E}$.  (Should one change the basis ${\cal E}$ the notions of real and imaginary part will change.)
Notice that $\V_\R :=\{ z\in\V\, :\, \Im(z)=0\}$ is invariant to addition and multiplication by real scalars and it becomes a real Hilbert space, with the inner product inherited from 
$\V$.  All elements of $\V$ that belong to $\V_\R$ will be referred to as real.

For $z\in \V$, we shall write $z = \Re(z) + \ii \Im(z)$ and naturally define the {\em complex conjugate} $\overline{z} := \Re(z) - \ii \Im(z)$.   The complex conjugate operation
as well as the real and imaginary part operators extend to $\V$-valued random elements in a straightforward manner
and we shall say that $x\in\V$ is real if $x\in \V_\R$, i.e., if its imaginary part is zero.  

The complex conjugate of a linear operator ${\cal A}:\V\to \V$ is defined as:
$
\overline \A(x):= \overline {\A(\overline x)},\ \ x\in \V.
$
This implies that $\overline {\A(x)} = \overline \A (\overline x)$, the operator $\overline \A$ is also {\em linear} and one can 
define the real and imaginary parts of $\A$ in as:
\begin{equation}\label{e:re-im-operator}
\Re(\A):=\frac{\A + \overline \A}{2}\ \ \ \mbox{ and }\ \ \ \Im(\A):= \frac{\A-\overline \A}{2\ii}. 
\end{equation}
Thus, $\A = \Re(\A) + \ii \Im(\A)$ and the usual operations with complex numbers and vectors extend 
to the operator Banach algebra over the complex Hilbert space $\V$.  We shall say that an operator $\A$ is {\em real} if
$\A = \Re(\A)$ (i.e., $\Im(\A) = 0$).  Observe that this is the case if and only if $\A(\V_{\R}) \subset \V_\R$.\\	

{\em Complexification of a real Hilbert space. } On the other hand, suppose that one starts with a {\em real} Hilbert space $\V_\R$ with inner product 
$\langle\cdot,\cdot\rangle_{\V_\R}$.  Then to be able to apply the results in Sections \ref{sec:Bochner} and \ref{ss:spectral_notation_1}, one needs to extend $\V_\R$ to a 
Hilbert space over $\C$.  This can be done with the standard method of {\em complexification}.  Namely, consider the set
 $\V:= \V_\R + \ii \V_\R$ of all pairs of $(x,y)\in \V_\R\times \V_\R$ 
written as $z:= x+\ii y,\ x,y\in\V_\R$, where by definition 
 \begin{equation}\label{e:real-complex-constructive}
 \Re(z):=x\ \ \ \mbox{ and }\ \ \ \Im(z):= y
 \end{equation}
 are the real and imaginary parts of $z$.  The complex conjugate operation is $\overline{z} := \Re(z) -\ii \Im(z)$ and the scalar multiplication is 
 $(\alpha + \ii \beta)\cdot z := (\alpha\cdot x - \beta\cdot y)+ \ii (\alpha\cdot y + \beta\cdot x),\ \ \alpha,\beta\in\R,\ x,y\in\V_\R$.
 The inner product in $\V$ is defined as  $\langle z,z'\rangle := (\langle x,x'\rangle_{\V_\R}+\langle y,y'\rangle_{\V_\R}) + \ii (\langle y,x'\rangle_{\V_\R}- \langle x,y'\rangle_{\V_\R})$, 
 where $z'=x'+\ii y',\ x',y'\in\V_\R$.   Thus, it is easy to see that $\V$ becomes a Hilbert space over $\C$ and $\V_\R =\{ z\, :\, \Im(z) = 0\}$ is trivially 
 embedded in $\V$.  The complex conjugate operation
as well as the real and imaginary part operators extend to $\V$-valued random elements in a straightforward manner
and we shall say that $x\in\V$ is real if $x\in \V_\R$, i.e., if its imaginary part is zero. 

If one fixes a CONS ${\cal E}:=\{e_j\}$ in $\V_\R$, then it readily follows that ${\cal E}$ this is also a CONS of $\V$
 and the definition \eqref{e:supp:real-complex} of the real an imaginary parts operators relative to ${\cal E}$ coincides with \eqref{e:real-complex-constructive}.  Thus,
 the notions of real, imaginary parts, as well as complex conjugate of $\V$-valued random elements and linear operators ${\cal A}: \V\to\V$ are exactly the
 same as outlined in \eqref{e:re-im-operator} above, for example.

 \subsection{Fundamentals.}

We shall assume that as in Section \ref{sec:real-irreversible}, $\V$ is a complex Hilbert space with a fixed real CONS
so that we can consider the complex conjugates, real, and imaginary parts of the elements and operators on $\V$.  Sometimes, it is
helpful to view $\V$ as $\V_\R + \ii \V_\R$ obtained by the method of complexification from the real Hilbert space $\V_\R$.
 
\begin{definition}\label{def:Gaussian-dist} A $\V$-valued random element $X$ is said to be Gaussian if $\langle X,f\rangle$ is a (complex)
Gaussian random variable, for all $f\in \V$.  A random variable in $\C$ is said to be complex Gaussian if its real and imaginary parts are 
jointly real Gaussian variables.
\end{definition}

We shall need the following simple but fundamental result, which entails well-known classic bounds on the tail behavior of the norm 
of a Gaussian vector (see Remark \ref{rem:Fernique} below).

\begin{proposition}\label{prop:zero-one} Let $Z_j,\ j\in\N$ be jointly Gaussian zero-mean real random variables
with ${\rm Var}(Z_j) = \sigma_j^2$ and let
$$
\xi := \sum_{j\in\N} Z_j^2\ \ \mbox{ and }\ \ \sigma^2:= \sum_{j\in\N} \sigma_j^2,
$$
which may take infinite values. Then, the following are equivalent:\\

{\em (i)} $\sigma^2 \equiv \sum_{j\in\N} \sigma_j^2 <\infty$\\

{\em (ii)} $\pr(\xi <\infty) = 1$\\

{\em (iii)} $\pr( \xi <\infty)>0$.\\

If one (and hence all) of the above conditions holds, then
\begin{equation}\label{e:MGF-bounds-for-Gauss-norm-2}
 e^{\theta \sigma^2} \le \E \Big[ e^{\theta \xi } \Big] \le  \frac{1}{\sqrt{1-2 \theta \sigma^2}},\ \ \ \mbox{ for all } -\infty< \theta<\frac{1}{2\sigma^2}.
\end{equation}
We emphasize that the above inequality is valid, in particular, for all {\em negative} $\theta$.
\end{proposition}

Before we give the proof, we make a few comments on the important consequences of this result.

\begin{remark} Proposition \ref{prop:zero-one} is a type of zero-one law for zero-mean 
Gaussian sequences stating that their squares are summable with probability zero or one. Many general
zero-one laws for Gaussian sequences exist \citep[see, e.g.,][] {kallianpur:1970}.
\end{remark}

\begin{remark}\label{rem:Fernique}
Let $\{e_j,\ j\in \N\}$ be a CONS of the Hilbert space $\V$.  Then, provided one (and hence all) of
the conditions in Proposition \ref{prop:zero-one} hold, we have that
$$
X:= \sum_{j\in\N} Z_j e_j
$$
is a well-defined random element in $\V$.  Since $\V$ is separable, all probability distributions on
$\V$ are tight (cf Ulam's tightness theorem, or more generally, Theorem 11.5.1 in \cite{dudley:1989}). Thus,
$X$ is {\em automatically} a Radon random element in the terminology in Section 2 of \cite{ledoux:talagrand:1991}.  That is,
for every $\epsilon>0$, there is a compact set ${\mathcal K}_\epsilon \subset E$, such that $\P(X\in {\mathcal K}_\epsilon) \ge 1-\epsilon$.

In this context, Proposition \ref{prop:zero-one} implies that every Gaussian random element $X$ in $\V$ is square integrable, in
the sense that $\E [\|X\|^2]<\infty$.  Moreover, the upper bound in \eqref{e:MGF-bounds-for-Gauss-norm-2} recovers the 
well-known result that $\E[ e^{\theta \|X\|^2} ] <\infty$, for all $\theta < 1/2\sigma^2$.  This result is valid in much greater generality 
for (Radon) Gaussian random variables taking values in a Banach space 
\citep[see, e.g., Corollary 3.2 on page 59-60 in][] 
{ledoux:talagrand:1991}. 
\end{remark}

\begin{proof}[Proof of Proposition \ref{prop:zero-one}] Since $\E [Z_j^2] = \sigma_j^2$, part {\em (i)} implies {\em (ii)} by 
the Tonelli-Fubini theorem.  Trivially, {\em (ii)}$\Rightarrow$ {\em (iii)}.

We now prove the implication {\em (iii)}$\Rightarrow$ {\em (i)}. Suppose that $\xi = \sum_{j=1}^\infty Z_j^2$, which is finite with some positive probability. 
For all fixed $n$, using the joint Gaussianity of the $Z_j$'s, we have that
$$
\xi_n:= \sum_{j=1}^n Z_j^2 = \sum_{j=1}^n \lambda_{n,j} Z_{n,j}^2,
$$
where $Z_{n,j},\ j=1,\cdots,n$ are independent standard normal and where $\lambda_{n,j}\ge 0$ are such that 
\begin{equation}\label{e:vn-lambda}
 v_n:= \sum_{j=1}^n \sigma_j^2 = \sum_{j=1}^n \lambda_{n,j}.
\end{equation}  
Since $0\le \xi_n\le \xi$ and $\P[\xi <\infty]>0$, for all $t>0$, we have
\begin{equation}\label{e:Xn-Laplace}
0< \E[ e^{-t \xi}]\le  \E [ e^{-t \xi_n}] = \prod_{j=1}^n \E [ e^{-t\lambda_{n,j} Z_{n,j}^2}] = \prod_{j=1}^{n} \frac{1}{\sqrt{1+2t \lambda_{n,j}}},
\end{equation}
where in the last two equalities we used the independence of the $Z_{n,j}$'s and their normality.  
Upon taking logs, and changing the sign, we obtain
\begin{equation}\label{e:log-lambda}
\frac{1}{2} \sum_{j=1}^n \log(1+2t \lambda_{n,j}) \le -\log \E [e^{-t \xi}] <\infty.
\end{equation}
The function $f(\lambda_1,\cdots,\lambda_n):= 2^{-1} \sum_{j=1}^n \log(1+2\lambda_j),\ \lambda_j\ge 0$ is concave.   
Therefore, its minimum in the simplex 
$$
\Big\{\lambda = (\lambda_i)_{i=1}^n\, :\, \lambda_i\ge 0, i=1,\cdots,n, \sum_{i=1}^n\lambda_i = t v_n\Big\}
$$
is attained at one of the extremal points $(tv_n,0,\cdots,0),\ \cdots, (0,\cdots,0,t v_n)$.  Since 
$$
f(t v_n,0,\cdots,0)=\cdots=f(0,\cdots,0,t v_n)=\frac12 \log(1+2 t v_n),
$$
from \eqref{e:log-lambda}, we obtain
\begin{equation}\label{e:log-vn}
\sup_{n\in\N} \frac{1}{2} \log(1+ 2 t v_n) \le -\log \E [e^{-t \xi}] <\infty,
\end{equation}
which in turn implies that $\sup_{n\in\N} v_n = \sum_{n=1}^\infty \sigma_n^2 <\infty$, completing the proof of (i).\\

Assume now that one and hence all of the conditions (i)--(iii) hold. Since $\E [\xi ] = \sigma^2<\infty$, the lower bound in 
\eqref{e:MGF-bounds-for-Gauss-norm-2} follows by appealing to the Jensen's inequality for the convex function $x\mapsto e^{\theta x}$.

We now prove the upper bound in \eqref{e:MGF-bounds-for-Gauss-norm-2}. Notice first that the case $\theta := -t <0$ follows 
by taking $n\to\infty$ in \eqref{e:log-vn}.  For the case $0\le \theta < 1/(2\sigma^2)$, as above let $\xi_n = \sum_{j=1}^n Z_j^2$ and as in
\eqref{e:Xn-Laplace}, we obtain
$$
 \E [ e^{\theta \xi_n}] = \prod_{j=1}^n \E [ e^{\theta \lambda_{n,j} Z_{n,j}^2}] = \prod_{j=1}^{n} \frac{1}{\sqrt{1-2\theta \lambda_{n,j}}},
$$
where now $-t$ is replaced by $\theta$ such that $\theta \in [0, 1/(2v_n))$.  As before, consider the concave function
$g(\lambda_1,\cdots,\lambda_n):= 2^{-1} \sum_{i=1}^n\log(1-2\lambda_i)$, over the simplex 
$$
\{ \lambda= (\lambda_i)_{i=1}^n \,:\, \lambda_i\ge 0,\ \sum_{i=1}^n\lambda_i = \theta v_n <1\}.
$$
Since the minimum of $g$ is attained at an extremal point of the simplex, we obtain
$$ 
-\log( \E [ e^{\theta \xi_n}]) = g(\theta \lambda_{n,1},\cdots, \theta \lambda_{n,n}) \ge g(\theta v_n,0,\cdots,0) = \frac{1}{2} \log(1-2 \theta v_n)
$$
which entails
$$
\E [  e^{\theta \xi_n}] \le \frac{1}{\sqrt{1-2\theta v_n}}. 
$$
Since $\xi_n\uparrow \xi$ and $v_n\uparrow \sigma^2 <\infty$, by letting $n\to\infty$ and appealing to the Monotone Convergence Theorem,
we obtain the upper bound in \eqref{e:MGF-bounds-for-Gauss-norm-2}.
\end{proof}

Proposition \ref{prop:zero-one} implies  the following natural result. 

\begin{corollary}\label{c:supp:Gaussian-characterization} If $X$ is a Gaussian random element in the separable Hilbert space $\V$, then:\\

{\em (i)} $\E [\|X\|^2] <\infty$ and consequently the mean vectors $\mu_X:= \E [ X ]$, the covariance $\CC_X := \E [ (X-\mu_X)\otimes (X-\mu_X)]$ and
pseudo-covariance operators $\CC_{X,\overline X} := \E [ (X-\mu_X)\otimes (\overline X-\overline{\mu_X})]$ are well-defined elements 
in $\V$ and $\bbT_+$, respectively, in the sense of Bochner. \\

{\em (ii)} The distribution of $X$ is determined by the mean vector $\mu_X$ and the pair of 
covariance and pseudo-covariance operators 
\begin{equation}\label{e:cov-pseudo-cov}
\CC_X=\E [ (X-\mu_X)\otimes (X-\mu_X)]\quad\mbox{ and }\quad 
\CC_{X,\overline X} =
\E [(X-\mu_X)\otimes (\overline X - \overline \mu_X)].
\end{equation}
Equivalently, the distribution of $X$ is determined by $\mu_X$ and 
the covariance operator of the vector $ Y:= (X,\overline X)^\top$ in the product Hilbert space $\V^2:=\V\times \V$.
\end{corollary} 
\begin{proof} Let $\{e_j,\ j\in\N\}$ be a {\em real} CONS of $\V$, i.e., $\Im(e_j) = 0,\ j\in\N$. We have that
$$
X = \sum_{j\in \N} \xi_j e_j,
$$
where $\xi_j:= \la X,e_j\ra,\ j\in \N$ are (complex) jointly Gaussian (Definition \ref{def:Gaussian-dist}).  

Suppose that $\mu_j := \E[\xi_j]$ and $\sigma_j^2 := \E [| \xi_j - \mu_j|^2],\ j\in\N$.  To prove that $\E [\|X\|^2] <\infty$, it suffices to show
that $\sum_{j\in\N} (|\mu_j|^2 + \sigma_j^2)<\infty$.

Let $\widetilde X:= \sum_{j\in\N} \wtilde \xi_j e_j$ be an independent copy of $X$ and observe that
$Y:= X - \widetilde X$ is also a Gaussian $\V$-valued random element.  The random variables $Z_j:= \xi_j - \wtilde \xi_j,\ j\in \N$ are 
(complex-valued) jointly Gaussian. Since  $\E [ Z_j] = 0$ and $\P(\|Y\|^2 <\infty) = 1$, 
by Proposition \ref{prop:zero-one} applied to the real and imaginary parts of the $Z_j$'s, we obtain 
$$
\E [\|Y\|^2 ]= \sum_{j\in\N} \E [ |Z_j|^2 ] = 2 \sum_{j\in \N}\sigma_j^2  <\infty.
$$ 
This, since $\sigma_j^2 = \E | \xi_j -\mu_j|^2$, implies that $X_0:= \sum_{j\in \N} (\xi_j - \mu_j) e_j$ takes values in $\V$, with probability one. 
But then $X-X_0 = \sum_{j\in\N} \mu_j e_j$ takes values in $\V$ and hence $\sum_{j\in\N} |\mu_j|^2 <\infty$.  This completes the proof of part {\em (i)}.

Since $\E [\|X\|^2] = \E[\|\overline X\|^2] <\infty$ it follows that $\E [\|X\|]<\infty$ and hence the expectation $\mu:= \E [X]$, the covariance, and pseudo-covariance 
operators in \eqref{e:cov-pseudo-cov} are well-defined in the
sense of Bochner \citep[cf.\ Theorem 2.6.5 in] [and Theorem \ref{thm:Bochner-integrability}]
{Hsing2015} in the spaces $(\V,\|\cdot\|)$ and 
$(\bbT,\|\cdot\|_{\rm tr})$, respectively.  From the properties of the Bochner integral, we readily obtain that
$$
\mu = \sum_{j\in \N} \mu_j e_j,\ \ \ \CC_X = \sum_{i,j\in \N} \E \Big[ (\xi_i-\mu_i)\overline{(\xi_j-\mu_j)} \Big] e_i\otimes e_j,
$$
and 
$$
\CC_{X,\overline X}  = \sum_{i,j\in \N} \E \Big[ (\xi_i-\mu_i)(\xi_j-\mu_j) \Big] e_i\otimes e_j.
$$
The distribution of $X$ on $\V$ is determined by the finite-dimensional distributions of the real and imaginary parts of its coordinates 
$\xi_j,\ j\in \N$ in any fixed CONS \citep[see, e.g., Theorem 7.1.2 in][]{Hsing2015}. The latter are, in turn, determined by the coordinates of 
$\mu$, $\CC_X$, and $\CC_{X,\overline X}$, completing the proof of {\em (ii)}.
\end{proof}

The next result provides uniform tail bounds for Gaussian vectors in a Hilbert space under the minimal condition that their 
norms are tight.  Note that this does not mean in general that the Gaussian vectors are (uniformly) tight.

\begin{corollary}\label{c:uniform-tail-boubnds-for-Gaussian-norms-in-a-Hilnbert-space} 
Suppose that $\{X_n,\ n\in \N\}$ is a collection of zero-mean Gaussian random elements taking values in the separable Hilbert space $\V$.
If the set of real random variables $\{\|X_n\|,\ n\in \N\}$ has uniformly tight distributions, then 
$$
\sigma^2 :=\sup_{n\in \N} \E [ \|X_n\|^2] <\infty
$$ 
and 
\begin{equation}
\sup_{n\in \N} \E\Big[ \exp\{ \theta \|X_n\|^2 \} \Big] \le  \frac{1}{\sqrt{1-2 \theta \sigma^2}}<\infty, \ \ \ \mbox{ for all }\theta < \frac{1}{2\sigma^2}.
\end{equation}
\end{corollary}
\begin{proof} Fix a CONS $\{e_j,\ j\in \N\}$ of $\V$. Let 
$$
U_{n,i} := \Re(\la X_n, e_i\ra) \ \ \mbox{ and }\ \ V_{n,i} := \Im(\la X_n,e_i\ra),
$$
be the real and imaginary parts of $\la X_n, e_i\ra$.  Since $X_n$ is a valid $\V$-valued random variable, we have that
\begin{equation}\label{e:X_n-square-norm-U-V}
\| X_n\|^2 = \sum_{i\in\N} (U_{n,i}^2 + V_{n,i}^2) <\infty,
\end{equation}
and hence by Proposition \ref{prop:zero-one} applied to the jointly Gaussian zero-mean random variables $\{U_{n,i}, V_{n,i},\ i\in\N\}$, 
we obtain $\sigma_n^2:=\E [ \|X_n\|^2 ] <\infty$.

Since $\{\|X_n\|,\ n\in\N\}$ are uniformly tight, there is an $M>0$ such that 
$$
\inf_{n\in\N} \P (\| X_n\|^2 \le M) \ge 1/2>0.
$$
This implies that 
\begin{equation}\label{e:exp-M-bound}
0<\frac{1}{2} e^{-M} \le \inf_{n\in\N}\E [ \exp\{ -\|X_n\|^2 \} ] \le \inf_{n\in\N} \frac{1}{\sqrt{1+2\sigma_n^2}},
\end{equation}
where the last inequality follows by applying the upper bound in \eqref{e:MGF-bounds-for-Gauss-norm-2} with $\theta:=-1$ to each
$\|X_n\|^2$ in \eqref{e:X_n-square-norm-U-V}.

Relation \eqref{e:exp-M-bound} shows that $\sigma^2:=\sup_{n\in\N} \sigma_n^2 <\infty$.  The inequality in 
\eqref{c:uniform-tail-boubnds-for-Gaussian-norms-in-a-Hilnbert-space} is a simple application of the upper bound in \eqref{e:MGF-bounds-for-Gauss-norm-2}.
\end{proof}

\subsection{Convergence in distribution.}

We start with the following natural criterion for convergence in distribution in the space ${\cal L}^2(\V)$ of all $\V$-valued 
random elements with $\E [\|X\|^2] <\infty$.

For a sequence of $X_n \in {\cal L}^2(\V),$ let 
$$
\mu_n:= \E [ X_n]\quad \mbox{ and }\quad {\cal C}_n:= \E [ (X_n-\mu_n)\otimes (X_n-\mu_n)]
$$ 
be the mean vectors and covariance operators, respectively.  As we know (cf Lemma \ref{l:cross-cov}), 
${\cal C}_n\in \bbT_+$ are necessarily positive trace-class operators.

\begin{theorem} \label{t:weak-conv-in-L2E} Fix some (any) CONS $\{e_j,\ j\in\N\}$ of $\V$ and let $\Pi_k:= \sum_{j=1}^k  e_j \otimes e_j$.
Let $\{X_n\}\subset {\cal L}^2(\V)$.

\noindent (i) We have $X_n\cid X_\infty$, as $n\to\infty$, if 
\begin{enumerate}
\item[(1)] $\langle X_n,f\rangle \cid \langle X_\infty,f\rangle$, for all $f\in\V$
\item[(2)]
$\|\mu_n-\mu_\infty\|\to 0$, and 
\item[(3)]
$\|\mathcal{C}_n-\mathcal{C}_\infty\|_{\rm op} \to 0$ and
\begin{align*}\limsup_{n\to\infty} \big(\tr(\mathcal{C}_n) - \tr(\mathcal{C}_{n,k}) \big)
\to 0 \mbox{ as $k\to\infty$},
\end{align*}
where ${\cal C}_{n,k} := \Pi_k {\cal C}_n \Pi_k$. 
\end{enumerate}

\noindent (ii) Conversely, if $\{ \|X_n\|^2,\ n\in\N\}$ are uniformly integrable, then the convergence  $X_n\cid X_\infty$, as $n\to\infty$ implies (1), (2) and 
(3) in part (i), above.
\end{theorem}

\begin{remark} Notice that the above result applies to general not necessarily Gaussian $\V$-valued random elements.
\end{remark}

\begin{proof} Part (ii) is immediate since uniform integrability and convergence in distribution imply convergence of the moments.

Part (i): If $\|\mu_n-\mu_\infty\|\to 0$, then by Slutsky's theorem,
the convergence $X_n - \mu_n \cid X_\infty -\mu_\infty$ implies $X_n\cid X_\infty$.  Thus, without loss of 
generality, we will assume that $\mu_n=\mu_\infty=0$.
 
 Condition (1), in view of the Wold device, is equivalent to the finite-dimensional convergence
\begin{align} \label{e:fdc}
(\langle X_n, f_1\rangle, \ldots, \langle X_n, f_k\rangle)
\cid (\langle X_\infty, f_1\rangle, \ldots, \langle X_\infty, f_k\rangle).
\end{align}
To prove $X_n\cid X_\infty$, it remains to establish tightness. For convenience of notation, assume that the eigenspaces
of $\mathcal{C}_n, \mathcal{C}_\infty$ are all one dimensional. If this is not true, we need to
work with eigenspaces and the notation becomes more complicated.
Let $e_{n,j}$ and $e_{\infty,j}$ be the eigenfunctions of $\mathcal{C}_n$ and $\mathcal{C}_\infty$, 
respectively, that correspond to the $j$-th descending eigenvalues.
Let $S_{n,k}=\mathrm{span}(e_{n,j}, j\le k)$ and $S_{\infty,k}=\mathrm{span}(e_{\infty,j}, j\le k)$.
For any $S\subset\V$, define
$$
S^\epsilon = \{f\in\V: \|f-g\| < \epsilon \mbox{ for some $g\in S$}\}.
$$
Write
$$
X_n = X_{n,k} + \wt X_{n,k}.
$$
where $X_{n,k}=\Pi_{S_{n,k}}X_n$.
Clearly, if $X_{n,k} \in S_{\infty,k}^{\epsilon}$ and $\|\wt X_{n,k}\| \le \epsilon$, then
$X_{n} \in S_{\infty,k}^{2\epsilon}$.
Thus,
\begin{align} \label{e:eq1}
\begin{split}
\P(X_n\in S_{\infty,k}^{2\epsilon}) & \ge \P(X_{n,k} \in S_{\infty,k}^{\epsilon}, \|\wt X_{n,k}\| \le \epsilon)  \\
& = \P(X_{n,k} \in S_{\infty,k}^{\epsilon}) - \P(X_{n,k} \in S_{\infty,k}^{\epsilon},\|\wt X_{n,k}\| > \epsilon).
\end{split}
\end{align}
By \eqref{e:fdc} and the fact that the eigenvalues and eigenfunctions of $\mathcal{C}_n$
converge to those of $\mathcal{C}_\infty$ under $\|\mathcal{C}_n-\mathcal{C}_\infty\|_{op} \to 0$
\citep[cf.\ Theorems 5.1.6 and 5.1.8 in][] {Hsing2015},
we have $X_{n,k}=\Pi_{S_{n,k}}X_n\cid \Pi_{S_{\infty,k}}X_\infty$.
Since $S_{\infty,k}^{\epsilon}$ is open, by the Portmanteau Theorem of weak convergence, 
\begin{align*}
\liminf_{n\to\infty} \P(X_{n,k} \in S_{\infty,k}^{\epsilon}) \ge \P(\Pi_{S_{\infty,k}} X_\infty \in 
S_{\infty, k}^{\epsilon}) = 1 \
\mbox{for each $k$}.
\end{align*}
It follows that, for any $k$,
\begin{align} \label{e:eq3}
\begin{split}
& \limsup_{n\to\infty} \P(X_{n,k} \in S_{\infty,k}^{\epsilon},\|\wt X_{n,k}\| > \epsilon) \\
& \le \limsup_{n\to\infty} \P(\|\wt X_{n,k}\| > \epsilon) \\
& \le  \epsilon^{-2} \limsup_{n\to\infty}\big(\tr(\mathcal{C}_n)  - \tr(\mathcal{C}_{n,k})\big).
\end{split}
\end{align}
By \eqref{e:eq1}-\eqref{e:eq3} and assumption (ii), we conclude that for any $\epsilon, \delta > 0$, there exists
$k=k(\epsilon,\delta)$ and $n(\epsilon,\delta)$, such that for all $n>n(\epsilon,\delta)$, we have
\begin{align*}
\P(X_n\in S_{\infty,k}^{2\epsilon}) \ge 1-\delta.
\end{align*}
For $n\le n(\epsilon,k)$, we can create another finite dimentionsal set $S$ by enlarging $S_{\infty,k}$ 
(e.g., adding eigenfunctions of $\mathcal{C}_n$'s to the spanning set) so that
\begin{align*}
\inf_n \P(X_n\in S^{2\epsilon}) \ge 1-\delta.
\end{align*}
Thus, the {\em flat concentration} condition \citep[cf.][Theorem 7.7.5]{Hsing2015} is fulfilled. 
\end{proof}

The following is a necessary and sufficient condition for the weak convergence of Gaussian distributions 
on a {\em real} separable Hilbert space.  Applying this result to the concatenation of the real and imaginary parts, one
can extend this criterion to Gaussian distributions in complex Hilbert spaces.

\begin{theorem}\label{t:Gauss-Hilbert-weak-convergence} 
Let $X_n,\ n\in \N$ be Gaussian elements taking values in the real separable Hilbert space  $\V_\R$.  Suppose that the
 $X_n$'s have means $\mu_n$ and covariance operators ${\cal C}_n$.
  
We have $X_n\cid X_\infty$ as $n\to\infty$ in $\V_\R$ if and only if $X_\infty$ is Gaussian  and
\begin{equation}\label{e:mu_n-C_n}
\|\mu_n - \mu_\infty\|\to 0\quad \mbox{ and } \quad \| \mathcal{C}_n - \mathcal{C}_\infty\|_{\rm tr}  \to 0,
\end{equation}
where $\mu_\infty$ and ${\cal C}_\infty$ are the mean vector and covariance operator of $X_\infty$.
\end{theorem}
\begin{proof} The `if' part follows from Theorem \ref{t:weak-conv-in-L2E}.  Indeed, by \eqref{e:mu_n-C_n} it follows that
$\langle X_n, f\rangle \cid \langle X_\infty, f\rangle$ for all $f\in \V$.  Relation \eqref{e:mu_n-C_n} implies also that  
$\|{\cal C}_n\|_{\rm tr} \to \|{\cal C}_\infty\|_{\rm tr}$ and for each $k$, $\|{\cal C}_{n,k}\|_{\rm tr} \to \|{\cal C}_{\infty,k}\|_{\rm tr}$. Thus,
 $$
 {\rm limsup }_{n\to\infty}  \left(\|{\cal C}_n\|_{\rm tr} - \|{\cal C}_{n,k}\|_{\rm tr}\right) = \|{\cal C}_\infty \|_{\rm tr} - \|{\cal C}_{\infty,k}\|_{\rm tr},
 $$
 vanishes as $k\to\infty$, which shows that the conditions of Theorem \ref{t:weak-conv-in-L2E} are fulfilled.
 
 We now prove the `only if' part.  The weak convergence $X_n\cid X_\infty$ implies that for all $f\in \V$, we have
 $\langle X_n, f\rangle \cid \langle X_\infty,f\rangle$, where the limit is a (complex) Gaussian random variable.
 This proves that $X_\infty$ is a Gaussian element in $\V$ with some mean vector $\mu_\infty\in \V$ and covariance operator ${\cal C}_\infty$.  
 
 We will first prove that $\|\mu_n-\mu_\infty\|\to 0$.  To this end, it is enough to show that $\{\mu_n\}$ is relatively compact, since the convergence 
 of the finite-dimensional distributions implies $\langle \mu_n,f\rangle \to \langle \mu_\infty,f\rangle$.   
 
 Let the $\wtilde X_n$'s be independent copies  of the $X_n$'s and observe that 
 \begin{equation}\label{e:t:Gauss-Hilbert-weak-convergence-1} 
 (Y_n, Z_n) := (X_n - \wtilde X_n, X_n + \wtilde X_n) \cid (Y_\infty, Z_\infty).
 \end{equation}
 Notice that by using the symmetry of the Gaussian distribution
\begin{equation}\label{e:t:Gauss-Hilbert-weak-convergence-2} 
 Z_n \stackrel{d}{=} 2\mu_n + Y_n.
 \end{equation}
 By Prokhorov's theorem, the convergence in distribution in \eqref{e:t:Gauss-Hilbert-weak-convergence-1} entails the tightness in distribution
 of both the sequences $\{Z_n\}$ and $\{Y_n\}$.  On the other hand \eqref{e:t:Gauss-Hilbert-weak-convergence-2} shows that 
 $\{2\mu_n + Y_n\}$ is also tight.  This implies that $\{\mu_n\}$ is relatively compact.  Indeed, by tightness, there exist compact sets ${\mathcal K}_1$
 and ${\mathcal K}_2$, such that $\P(Y_n\in {\mathcal K}_1) \ge 2/3$ and $\P(2\mu_n+Y_n \in {\mathcal K}_2)\ge 2/3$.  This implies that
 \begin{align*}
\frac{1}{3} &\le \P(Y_n\in {\mathcal K}_1) + \P(2\mu_n + Y_n\in {\mathcal K}_2) - 1 \\
& \le \P( Y_n \in {\mathcal K}_1, 2\mu_n + Y_n\in {\mathcal K}_2)  \le \P( 2\mu_n \in {\mathcal K}_2 - {\mathcal K}_1),
 \end{align*}
which means the deterministic sequence $\{\mu_n\}$ belongs to the compact set $2^{-1} ({\mathcal K}_2-{\mathcal K}_1)$ with positive probability. Hence,
 $\{\mu_n\} \subset 2^{-1} ({\mathcal K}_2-{\mathcal K}_1)$ is a relatively compact sequence completing the proof of the convergence $\|\mu_n - \mu_\infty\| \to 0$.
 
 We next show that $\|{\cal C}_n - {\cal C}_\infty\|_{\rm tr} \to 0$ and without loss of generality suppose that $\mu_n = \mu_\infty=0$ so
 that ${\cal C}_n = \E [ X_n \otimes X_n]$. Fix a CONS $\{e_i\}$ of $\V$ and define the projection operators 
$$
 \Pi_m := \sum_{i=1}^m e_i\otimes e_i\ \ \mbox{ and }\ \ \widetilde \Pi_m:= \sum_{i=m+1}^\infty e_i \otimes e_i.
$$
Since $X_n\cid X_\infty$, $\|X_n\|\cid\|X_\infty\|$  and since $\V$ is separable it follows (Prokhorov) that $\{X_n\}$ is tight.  Hence, by 
Corollary  \ref{c:uniform-tail-boubnds-for-Gaussian-norms-in-a-Hilnbert-space}, $\{\|X_n\|^2\}$ is uniformly integrable and
the convergence $\|X_n\|\cid \|X_\infty\|$ implies
\begin{equation}\label{e:tr_Cn-to-tr_C}
\|\mathcal{C}_n\|_{\rm tr} = \E [\|X_n\|^2] \to \|\mathcal{C}_\infty\|_{\rm tr} = \E [\|X_\infty\|^2],\ \ \ \mbox{ as }n\to\infty.
\end{equation}

Fix an arbitrary $\epsilon>0$ and pick $m = m(\epsilon)$ large enough so that 
$$
\|\widetilde \Pi_m\mathcal{C}\widetilde \Pi_m\|_{\rm tr} < \epsilon.
$$

Recall that every self-adjoint trace-class operator ${\cal A}$ can be decomposed as follows 
$$
{\cal A} = {\cal A}^+ - {\cal A}^-,
$$
where ${\cal A}^\pm$ are self-adjoint, positive and such that ${\cal A}^+ {\cal A}^- = {\cal A}^-{\cal A}^+=0$.  
Thus,  $({\cal A}{\cal A}^*)^{1/2} = {\cal A}^+ + {\cal A}^-$ and consequently,
$$
\|{\cal A}\|_{\rm tr} = {\rm trace}({\cal A}^+) + {\rm trace}({\cal A}^-) \equiv \| {\cal A}^+\|_{\rm tr} + \| {\cal A}^-\|_{\rm tr}.
$$
Therefore, for the self-adjoint trace-class operator ${\cal C}_n - {\cal C}_\infty$, we have
\begin{align}\label{e:Cn-C-trace}
\| \mathcal{C}_n - \mathcal{C}_\infty\|_{\rm tr} &= {\rm trace}((\mathcal{C}_n - \mathcal{C}_\infty)^+ ) 
+ {\rm trace}( (\mathcal{C}_n - \mathcal{C}_\infty)^-).
\end{align}
Notice that for every positive self-adjoint operator ${\cal A}$, we have
\begin{align}\label{e:trace-norm-proj}
\|{\cal A}\|_{\rm tr} &= {\rm trace} ({\cal A}) \nonumber \\
& = {\rm trace} (\Pi_m {\cal A} \Pi_m) +  {\rm trace} (\widetilde \Pi_m {\cal A} \widetilde \Pi_m)\\
& = \|\Pi_m {\cal A} \Pi_m\|_{\rm tr} + \|\widetilde \Pi_m {\cal A} \widetilde \Pi_m\|_{\rm tr}.\nonumber
\end{align}
Thus, by writing ${\cal A}_n^{\pm}:=  (\mathcal{C}_n - \mathcal{C}_\infty)^\pm$, in view of \eqref{e:Cn-C-trace}, we obtain
\begin{align}\label{e:trace-norm-proj-identity}
\| \mathcal{C}_n - \mathcal{C}_\infty\|_{\rm tr} &= {\rm trace} ( \Pi_m{\cal A}_n^{+} \Pi_m ) +  {\rm trace} ( \widetilde \Pi_m{\cal A}_n^{+} \widetilde \Pi_m ) \nonumber\\
& + {\rm trace} (\Pi_m{\cal A}_n^{-} \Pi_m)+ {\rm trace} ( \widetilde\Pi_m{\cal A}_n^{-} \widetilde \Pi_m). 
\end{align}
Since $X_n\cid X_\infty$, by Lemma \ref{lem:uniform-var}, we have that $\| {\cal C}_n - {\cal C}_\infty\|_{\rm op}  \to 0$.  
On the other hand, since ${\rm span}\{e_1,\cdots,e_m\}$ is finite-dimensional, Lemma \ref{lem:tr-norm-bound} applied to 
${\cal A}:= {\cal C}_n - {\cal C}_\infty$  yields
\begin{equation}\label{e:Pi_m-tr-conv}
\|\Pi_m ({\cal C}_n - {\cal C}_\infty)^{\pm} \Pi_m\|_{\rm tr} \le m \| {\cal C}_n - {\cal C}_\infty\|_{\rm op} 
 \to 0,
\end{equation}
as $n\to\infty$.

On the other hand, since $0\le ({\cal C}_n-{\cal C}_\infty)^+ \le {\cal C}_n$, for the positive operators $({\cal C}_n-{\cal C}_\infty)^+$ and 
${\cal C}_n$, we obtain
\begin{equation}\label{e:Pi-tilde-Cn+}
\| \wtilde \Pi_m ({\cal C}_n - {\cal C}_\infty)^+ \wtilde \Pi_m\|_{\rm tr} \le \| \wtilde \Pi_m {\cal C}_n  \wtilde \Pi_m\|_{\rm tr} 
= {\rm trace}(\wtilde  \Pi_m {\cal C}_n \wtilde \Pi_m).
\end{equation}
Similarly, $0\le ({\cal C}_n-{\cal C}_\infty)^- \le {\cal C}$, and hence
\begin{equation}\label{e:Pi-tilde-Cn-}
\| \wtilde \Pi_m ({\cal C}_n - {\cal C}_\infty)^- \wtilde \Pi_m\|_{\rm tr} \le \| \wtilde \Pi_m {\cal C}_\infty  \wtilde \Pi_m\|_{\rm tr} 
={\rm trace}( \wtilde \Pi_m {\cal C}_\infty \wtilde \Pi_m).
\end{equation}

The last two bounds imply that
\begin{equation}\label{e:tilde_Pi_m-tr-conv}
 \| \widetilde \Pi_m(\mathcal{C}_n - \mathcal{C}_\infty)^\pm \widetilde \Pi_m\|_{\rm tr} \le  \| \widetilde \Pi_m\mathcal{C}_n \widetilde \Pi_m\|_{\rm tr} +\|\widetilde \Pi_m\mathcal{C}_\infty\widetilde \Pi_m\|_{\rm tr} \le 3\epsilon,
\end{equation}
for all sufficiently large $n$.  Indeed, $\|\widetilde \Pi_m\mathcal{C}_\infty\widetilde \Pi_m\|_{\rm tr}<\epsilon$, by the choice of $m$.  Whereas,
by \eqref{e:tr_Cn-to-tr_C} and \eqref{e:trace-norm-proj}, we have
\begin{align*}
\| \widetilde \Pi_m\mathcal{C}_n \widetilde \Pi_m\|_{\rm tr} &=  \|\mathcal{C}_n\|_{\rm tr} - \| \Pi_m  \mathcal{C}_n\Pi_m\|_{\rm tr}\\
&\quad \to \|\mathcal{C}_\infty\|_{\rm tr} - \| \Pi_m  \mathcal{C}_\infty\Pi_m\|_{\rm tr} = \| \widetilde \Pi_m\mathcal{C}_\infty\widetilde \Pi_m\|_{\rm rt} \le \epsilon.
\end{align*}
This ensures that for all $n$ large $\| \widetilde \Pi_m\mathcal{C}_n \widetilde \Pi_m\|_{\rm tr} \le 2\epsilon$ and hence \eqref{e:tilde_Pi_m-tr-conv} holds.

Since $\epsilon$ was arbitrary, in view of \eqref{e:trace-norm-proj-identity}, Relations \eqref{e:Pi_m-tr-conv} and \eqref{e:tilde_Pi_m-tr-conv} yield
$\| \mathcal{C}_n - \mathcal{C}_\infty\|_{\rm tr}\to 0$, as $n\to\infty$.
\end{proof}

\begin{lemma}\label{lem:tr-norm-bound} Let $e_i\in \V,\ i=1,\cdots,m$ be orthonormal and $\Pi_m = \sum_{i=1}^m e_i\otimes e_i$.

For every bounded self-adjoint linear operator ${\cal A}:\V\to \V$, we have
$$
\|\Pi_m  {\cal A}^\pm \Pi_m\|_{\rm tr} \le m \sup_{\|f\|=1} | \langle f, {\cal A}f\rangle |. 
$$
\end{lemma}
\begin{proof} This crude bound follows by observing that the self-adjoint 
operator ${\cal B}:= \Pi_m  {\cal A}^\pm \Pi_m$ has at most $m$ non-zero eigenvalues, which 
are all bounded above in absolute value by the spectral norm of ${\cal A}$.  
\end{proof}

\begin{lemma} \label{lem:uniform-var} Let $X_n$ and $X$ be zero-mean Gaussian $\V$-valued vectors with covariance
operators ${\cal C}_n$ and ${\cal C}$, respectively.  If $X_n\cid X$, then
$$
\| {\cal C}_n - {\cal C}\|_{\rm op} \equiv \sup_{\|f\|=1} | \langle f, ({\cal C}_n-{\cal C}) f \rangle |  \to 0,
$$
as $n\to\infty$.
\end{lemma}
\begin{proof}
 By Skorokhod's representation, we can define $X_n, X$ on the same probability space 
 where $X_n\cas X$.  Notice that for $a,b\in \mathbb C$, we have
 $$
  \Big| |a|^2 - |b|^2 \Big| = \Big | \Re (a-b)\overline{(a+b)} \Big| \le |a-b| |a+b|.
 $$ 
By applying this inequality to $a:= \langle X_n,f\rangle$ and $b:=  \langle X,f\rangle$, for all $f\in E,\ \|f\|=1$,
we obtain
 \begin{align}\label{e:lem:uniform-var-1}
 \Big| \langle f,{\cal C}_n f\rangle -  \langle f,{\cal C} f\rangle \Big | & =  \Big|\E |\langle X_n,f\rangle|^2 - \E |\langle X,f\rangle|^2  \Big| \nonumber\\
&\le \E  \Big| |\langle X_n,f\rangle|^2 -  |\langle X,f\rangle|^2  \Big| \nonumber\\
& \le \E | \langle X_n-X,f\rangle \langle X_n+X,f\rangle| \nonumber\\
  &\le \left(\E\|X_n-X\|^2 \E\|X_n+X\|^2\right)^{1/2},
 \end{align}
 where the last bound follows from the Cauchy-Schwartz inequality.
 
 The fact that $\{\|X_n -X\|\}$ and $\{\|X_n+X\|\}$ are tight, in view of Corollary 
 \ref{c:uniform-tail-boubnds-for-Gaussian-norms-in-a-Hilnbert-space}, implies the uniform integrability of both 
 $\{\|X_n -X\|^2\}$ and $\{\|X_n+X\|^2\}$. This, since $\|X_n -X\| \cas 0$, shows that the right-hand side of \eqref{e:lem:uniform-var-1}
 vanishes, which completes the proof.
 \end{proof}

\subsection{Hilbert space valued Gaussian processes.}

Let $\{X(t),\ t\in T\}\subset {\cal L}^2(\V)$ be a second-order stochastic process taking values in the
Hilbert space $\V$, i.e., $\E \|X(t)\|^2 <\infty$.  Then, as seen above (cf Lemma \ref{l:cross-cov}),
its cross covariance operator 
\begin{equation}\label{e:C_X}
{\cal C}_X(t,s) := \E [ X(t)\otimes X(s)],\ \ t,s\in T,
\end{equation}
is well-defined and takes values in $\bbT$. Recall that the latter expectation is understood in the sense of Bochner in the Banach 
space $\bbT:=\bbT(\V)$ of trace-class operators on $\V$ equipped with the trace norm $\|\cdot\|_{\rm tr}$.
The next definition is a natural extension of a classic concept (see also Definition \ref{def:complete-positive-definiteness}). 

\begin{definition}\label{def:sup:PSD-function} Fix an arbitrary index set $T$. A function $\CC:T\times T \mapsto \bbT(\V)$,
is said to be (completely) positive definite, if for all $f_i \in E,\ t_i\in T,\ i=1,\cdots,n$ and $n\in\N$,
\begin{equation}\label{e:psd-E-valued-function} 
 \sum_{i,j} \langle f_i, {\mathcal C}(t_i,t_j) f_j\rangle \ge 0.
\end{equation}
\end{definition}
  
\begin{remark} By taking $(t_1,t_2):=(t,s)$, $(f_1,f_2) = (f, \ii g),$ for some (any) $f,g\in \V$,
 and $(f_1,f_2) =(f,g)$ and applying Relation \eqref{e:psd-E-valued-function} twice, one obtains that 
$$
\langle f, {\mathcal C}(t,s) g \rangle =  \langle f, {\mathcal C}(s,t)^* g\rangle ,\ \  t,s\in T,
$$
and hence ${\mathcal C}(t,s) = {\mathcal C}(s,t)^*$.  That is, positive definite functions are necessarily Hermitian.
\end{remark}

It is straightforward to see that every cross covariance function as in \eqref{e:C_X} is positive definite in the sense of the
above definition.  Naturally, the converse is also true and formally shown next using Gaussian processes.
Recall Definition \ref{def:Gaussian-dist}.

\begin{definition}\label{def:Gaussian-proc} A $\V$-valued stochastic process $\{X(t),\ t\in T\}$ is said to be Gaussian, if 
$\sum_{i=1}^n \langle f_i, X(t_i)\rangle$ is a (complex) Gaussian random variable, for all $f_i\in \V, t_i\in T,\ i=1,\cdots,n$.
 \end{definition}

\begin{proposition}\label{supp:prop:cross-covariance-existence} A function ${\cal C}:T\times T\to \bbT(\V)$ is positive definite in the sense of Definition \ref{def:sup:PSD-function},
if and only if 
$$
 {\mathcal C}(t,s) = \E [ X(t)\otimes X(s)],\ \ t,s\in T
$$
for some second-order $\V$-valued stochastic process $X = \{X(t),\ t\in T\}$.  In this case, the process $X$ can be taken to 
be zero-mean Gaussian.
\end{proposition}		
\begin{proof} The `if' part is immediate, by observing that 
$$\E \left |\sum_{i=1}^n \langle f_i, X(t_i) \rangle \right|^2 = \sum_{i,j=1}^n \langle f_i, {\mathcal C}(t_i,t_j) f_j\rangle,
$$
where ${\mathcal C}(t,s) = \E [ X(t)\otimes X(s)],\ t,s\in T$.

To prove the `only if' part, consider the direct-sum Hilbert space $\V^{\oplus n} = \{ (f_i)_{i=1}^n,\ f_i\in \V\}$, which is equipped 
with the inner product
$$
\langle f,g\rangle_n := \sum_{i=1}^n \langle f_i,g_i\rangle.
$$
For any $\{t_i\}_{i=1}^n \subset \V$, define the operator $C_n := ({\mathcal C}(t_i,t_j))_{n\times n}$ acting on $\V^{\oplus n}$ as follows:
$$
\langle f, C_n g\rangle_n := \sum_{i,j=1}^n \langle f_i, {\mathcal C}(t_i,t_j)g_j \rangle,
$$
for all $f = (f_i)_{i=1}^n$ and $g = (g_i)_{i=1}^n$ with $f_i,\ g_i\in \V$.
Relation \eqref{e:psd-E-valued-function}  implies that $C_n$ is self-adjoint and positive. Therefore, to establish that 
$C_n\in \bbT_+(\V^{\oplus n})$, it is enough to show that its trace is finite.  Let $\{e_{j}\}$ be a CONS of $\V$ and define 
the elements in $\V^{\oplus n}$:
$$ 
\psi_{1,j} = (e_j,0,\cdots,0),\ \psi_{2,j} = (0,e_j,0,\cdots,0),\cdots, \psi_{n,j} := (0,\cdots,0,e_j),
$$
for $j\in\N$.  Then, it is easy to see that $\{\psi_{i,j},\ i=1,\cdots,n,\ j\in\N\}$ is a CONS of the direct-sum Hilbert space $\V^{\oplus n}$.
Since $C_n$ is self-adjoint and positive, we have
\begin{align*}
\| C_n\|_{\rm tr} &= {\rm trace}(C_n) = \sum_{i=1}^n\sum_{j\in \N} \langle \psi_{i,j}, C_n \psi_{i,j} \rangle_n\\
& = \sum_{i=1}^n  \sum_{j\in\N} \langle e_j, {\mathcal C}(t_i,t_i) e_j\rangle = \sum_{i=1}^n {\rm trace}({\mathcal C}(t_i,t_i))\\
& = \sum_{i=1}^n \| {\mathcal C}(t_i,t_i)\|_{\rm tr} <\infty, 
\end{align*}
where the terms in the last sum are finite since ${\mathcal C}(t_i,t_i)\in \bbT_+(\V)$ are positive trace-class operators.

The spectral theorem for the positive trace-class operator $C_n$ yields the decomposition
$$
C_n = \sum_{j\in\N} \sigma_{j} \varphi_j \otimes \varphi_j,
$$
where $\sigma_j \ge 0$ and $\{\varphi_j\}$ is an orthonormal set of eigenvectors of the operator $C_n$.
Now, taking independent standard normal random variables $Z_j,\ j\in \N$, we define
\begin{equation}\label{e:X-n-star}
X_n:= \sum_{j} \sqrt{\sigma_j} Z_j \varphi_j.
\end{equation}
Since $\sum_{j} \sigma_j = \|C_n\|_{\rm tr} <\infty$, by the It\^{o}-Nisio Theorem, the above series converges in norm 
with probability one and defines a zero-mean Gaussian vector in $\V^{\oplus n}$.   Noting that $X_n = (X_n(i))_{i=1}^n$, 
where $X_n(i)\in \V$, and using the fact that 
$$
e\otimes f = (e_i\otimes f_j)_{i,j=1}^n,
$$
for all $e = (e_{i})_{i=1}^n$ and $f=(f_i)_{i=1}^n$ in $E^{\oplus n}$, it readily follows that the collection of Gaussian vectors 
$\{X_n(i),\ i=1,\cdots,n\}$ has cross covariance  operators
$$
 {\mathcal C}(t_i,t_j) = \E [ X_n(i)\otimes X_n(j) ],\ \ i,j=1,\cdots,n. 
$$
The joint distribution of the components of the vector $X_n$ defines a probability distribution
$F_{t_1,\cdots,t_n}$ on the {\em product} Hilbert space $\V^n$.  Since the law of the Gaussian vectors $X_n$ in \eqref{e:X-n-star} is (by definition) 
completely determined by $\{(\sigma_j,\varphi_j),\ j=1,\cdots,n\}$, it is easy to see that the family $\{F_{t_1,\cdots,t_n},\ t_i\in T\}$ is consistent and hence the 
generalized version of the Kolmogorov existence theorem applies (see, e.g., Theorem 5.16 in \cite{kallenberg:1997}).  Thus, on a suitable probability space 
one can have an $\V$-valued stochastic process $\{X(t),\ t\in T\}$ with finite-dimensional distributions $F_{t_1,\cdots,t_n}$, and in particular such that 
$
\E [ X(t)\otimes X(s)] = {\mathcal C}(t,s),\ \ t,s\in T.
$
\end{proof}

{\bf Caution.} Suppose that $C: T\times T\to \bbT(\V)$ is such that 
$\sum_{i,j=1}^n  a_i\overline a_j {\mathcal C}(t_i,t_j)$ is a positive self-adjoint operator, for all $a_i\in \bbC$ and $t_i\in T,\ i=1,\cdots,n$.
One may be tempted to conclude that $\{{\mathcal C}(t,s)\}$ is then positive definite in the sense of Definition \ref{def:sup:PSD-function}.
This is not the case in general according to the following classic example borrowed from the literature of $C^*$-algebras.  (See, however,  
Definitions \ref{def:pos-def-function}, \ref{def:complely-pos-def-function}, and 
Corollary \ref{c:weak-psd-implies-complete-psd} in the main paper.)

\begin{example} \label{supp:ex:strong-psd-is-needed}   
Consider the simple case $T =\{1,2\}$ and $\V = \R^2$, where ${\mathcal C}(i,j) = C_{i,j},\ i,j\in T$
are as follows
$$
C:= \left (\begin{array}{ll} C_{1,1} & C_{1,2}\\
C_{2,1} & C_{2,2} \end{array}\right)  := \left( \begin{array}{ll | ll}
 1 & 0 & 0 & 0\\
 0 & 0 & 1 & 0\\
 \hline
 0 & 1 & 0 &0 \\
 0 & 0 & 0 & 1\\
\end{array} \right).
$$
It is easy to see that
$$
\sum_{i,j=1}^2 a_i \overline a_j C_{i,j} = \left( \begin{array}{cc}
  |a_1|^2  &  \overline{a}_1a_2 \\
  a_1 \overline{a}_2  & |a_2|^2
  \end{array}
  \right),
$$
which is positive semidefinite.  However, since ${\rm det}(C)=-1$ 
the matrix $C$ is not positive semi-definite.  
\end{example}

\section{On the support of operator self-similar IRF$_k$'s.}\label{sec:supplement-support}
In this section, we focus on zero-mean second-order operator self-similar IRF$_k$'s taking values in a separable Hilbert space $\V$.
We provide miscellaneous results about their supports and operator exponents complementing Section \ref{s:self-similar IRFk} of
 the main paper. 

\begin{definition}\label{def:supp-Y} For a second-order IRF$_k$ process $Y$, we write
$$
{\rm supp}(Y) := \overline{\rm span} \Big(\bigcup_{\lambda\in \Lambda_k} {\rm Im}(\mathcal{C}_Y(\lambda,\lambda))\Big),
$$ 
where ${\rm Im}(\mathcal{C}_Y(\lambda,\lambda))$ stands for the range of the cross covariance operator
$\mathcal{C}_Y(\lambda,\mu) = \E [Y(\lambda)\otimes Y(\mu)]$ and $\overline{\rm span}$ denotes the closure of the linear span.
The closed subspace ${\rm supp}(Y)$ of ${\mathbb V}$ will be referred to as the {\em support} of $Y$.  If ${\rm supp}(Y) = \V$, then
$Y$ will be referred to as {\em proper}.
\end{definition}

Since for the self-adjoint positive operator $\mathcal{C}_Y(\lambda,\lambda)$, we have 
${\rm Im}(\mathcal{C}_Y(\lambda,\lambda)) = {\rm Ker}(\mathcal{C}_Y(\lambda,\lambda))^\perp$, one can equivalently write
\begin{align*}
{\rm supp}(Y) &= \Big( \cap_{\lambda\in\Lambda_k} {\rm Ker}(\mathcal{C}_Y(\lambda,\lambda) \Big)^\perp \\
& =  \Big \{ e\in \V\, :\, \E | \langle e, Y(\lambda)\rangle|^2 = 0,\ \mbox{ for all }\lambda\in \Lambda_k\Big\}^\perp.
\end{align*}
The last relation follows from the fact that $ \E | \langle e, Y(\lambda)\rangle|^2 = \langle {\cal C}_Y e, e\rangle = \langle {\cal A} e,{\cal A}e\rangle$, where ${\cal A}$ is a self-adjoint positive operator such that ${\cal C}_Y = {\cal A}^2$.

Since for all $\lambda_i\in \Lambda_k$ and $c_i\in \bbC,\ i=1,\cdots,n$, we have
$
Y(\lambda) = \sum_{i=1}^n c_i Y(\lambda_i),
$
for $\lambda:= \sum_{i=1}^n c_i\cdot \lambda_i$, it follows ${\rm supp}(Y)$ is the smallest closed linear subspace $W$ of ${\mathbb V}$ such that
$Y(\lambda) \in W,$ almost surely, for all $\lambda$.

By analogy with the notion of a support of an IRF, we shall introduce the notion of a range of a $\bbT_+$-valued measure.  

\begin{definition}  Let $\chi$ be a $\sigma$-finite $\bbT_+$-valued measure on ${\cal B}(\R^d\setminus\{0\})$.  For
$a\in \V$, consider the $\sigma$-finite measure $\mu_a(du):= \langle a,\chi(du)a\rangle$.  The {\em range} of $\chi$, denoted
${\rm Range}(\chi)$ is defined as:
$$
{\rm Range}(\chi) := \Big\{ a \in \V\, :\, \mu_a(du) = 0 \Big\}^\perp.
$$
\end{definition}
It is easy to see that ${\rm Range}(\chi) = \overline{{\rm span}} \cup_{A \, :\, \|\chi(A)\|_{\rm tr} <\infty} {\rm Im}( \chi(A)).$
The intuition behind the above definition is that ${\rm Range}(\chi)$ is the minimal subspace of ${\mathbb V}$ such that the projected measures
$\mu_a$ are non-trivial.  The support of an IRF with trivial polynomial component is precisely the range of its spectral 
measure. 

\begin{proposition}\label{p:suppY=rangeChi} 
Let $Y$ be a continuous second-order IRF$_k$ with spectral representation $(\chi_k,{\cal Q})$ as in \eqref{e:IRF-cov-spectral} 
(in the main paper), with trivial polynomial component ${\mathcal Q}$.  Then, with $\chi_k$ as in \eqref{e:chi-k}, 
$$
{\rm supp}(Y) = {\rm Range}(\chi_k) \equiv {\rm Range}(\chi).
$$
\end{proposition}
\begin{proof} By \eqref{e:IRF-cov-spectral}, we can write
$
{\mathcal C}(\lambda,\lambda) = \int_{\R^d} |\wh \lambda(u)|^2 \chi_k(du),
$
for all $\lambda\in\Lambda_k$, since ${\cal Q}(\lambda*\wt \lambda) = 0$ by assumption.  Thus, for all $a\in \V$, 
\begin{equation}\label{e:support=range}
\langle a, {\mathcal C}(\lambda,\lambda)a\rangle = \int_{\R^d}  |\wh \lambda(u)|^2 \mu_a(du),
\end{equation}
where $\mu_a(du) = \langle a,\chi_k(du)a\rangle$.

Recall that ${\rm supp}(Y) = {\mathbb W}^\perp$, where  ${\mathbb W} 
= \cap_{\lambda\in\Lambda_k} {\rm Ker}({\mathcal C}(\lambda,\lambda))$.
Thus, \eqref{e:support=range} implies that $a\in {\mathbb W}$ if and only if  $\mu_a(du) =0$.  Indeed, if
$\mu_a$ is the zero measure, then $a\in {\rm Ker}({\mathcal C}(\lambda,\lambda)$ for all $\lambda\in\Lambda_k$.  Conversely, if 
$\int_{\R^d} |\wh \lambda(u)|^2 \mu_a(du) = 0$ for all $\lambda \in \Lambda_k$, then it follows that $\mu_a$ is the zero measure.   
Indeed, this follows from the fact that one can choose a judicious sequence of measures
$\lambda_n\in \Lambda_k$, such that $|\wh \lambda_n(u)|^2 \to g(u)$, where $g(u)>0$ for all $u\not=0$ 
(see e.g.\ Lemma \ref{l:reference_measure}).  The Fatou Lemma then entails $\int g(u)\mu_a(du) = 0$, and hence $\mu_a =0$, since 
$\mu_a(\{0\})=0$. This completes the proof since $\mu_a(du) = (1\wedge \|u\|)^{-2k-2} \langle a,\chi(du)a\rangle$  and the facts  that $\chi_k$ and $\chi$ put no mass at $\{0\}$ 
entail ${\rm Range}(\chi_k) = {\rm Range}(\chi)$.
\end{proof}

Let now ${\H}:\V\to \V$ be a bounded linear operator on the Hilbert space ${\mathbb V}$. Recall that a zero-mean
IRF$_k$ $Y$ is said to be second-order ${\H}$-self-similar if $\{c^{\H} Y(\lambda)\}$ and $\{Y(c\cdot \lambda)\}$ have the same
cross covariance operators, for all $c>0$.   In particular, if $Y$ is Gaussian, this entails the ${\H}$-self-similarity of $Y$.

The operator ${\H}$ can in principle be rather arbitrary outside the support $Y$.  This is perhaps why a {\em standard} assumption 
adopted in the literature on operator self-similarity in the finite-dimensional setting is that  $Y$ be {\em proper}, i.e., 
${\rm supp}(Y)=\V$ so that its support cannot be confined to a proper linear subspace of ${\mathbb V}$ (see e.g., 
\cite{didier:meerschaert:pipiras:2017} and the references therein). The following result allows us to restrict the operator ${\H}$
 to ${\rm supp}(Y)$.

\begin{proposition}\label{p:supplement:support} If $Y$ is a continuous second-order ${\H}$-self-similar IRF$_k$, then:
\begin{enumerate}
\item 
${\H} ({\rm supp}(Y))\subset {\rm supp}(Y)$,
\item
$\overline{{\H}({\rm supp}(Y))} = {\rm supp}(Y)$. 
\end{enumerate}

\end{proposition}
\begin{proof} Let ${\mathbb W} := \cap_{\lambda\in \Lambda_k} {\rm Ker}( \mathcal{C}_Y(\lambda,\lambda)),$ where 
$\mathcal{C}_Y(\lambda,\mu) = \E[Y(\lambda)\otimes Y(\mu)]$, $\lambda,\mu\in\Lambda_k$.  For every $r>0$, and 
$\lambda\in \Lambda_k$, the operator self-similarity Relation \eqref{e:Y-Hss} implies that
\begin{equation}\label{e:p:support--1}
r^{\H} \mathcal{C}_Y(r^{-1}\lambda,r^{-1}\lambda) r^{\H^*} = \mathcal{C}_Y(\lambda, \lambda).
\end{equation}
This implies that 
$$
r^{{\H}^*} {\rm Ker}( \mathcal{C}_Y(\lambda,\lambda)) \subset {\rm Ker}( \mathcal{C}_Y(r^{-1}\lambda,r^{-1}\lambda)) \subset  {\mathbb W}.
$$
Since $\lambda$ was arbitrary, it follows that $r^{{\H}^*} ({\mathbb W}) \subset {\mathbb W}$.  Note however that
$r^{{\H}^*}$ is invertible with bounded inverse $r^{{-\H}^*}$ such that $r^{{-\H}^*} r^{{\H}^*} = r^{{\H}^*} r^{{-\H}^*} = {\rm I},\ r>0$.  Therefore,
$$
{\mathbb W} = r^{{-\H}^*} ( r^{{\H}^*} {\mathbb W}) \subset r^{{-\H}^*}({\mathbb W}),\ \ \mbox{ for all $r>0$}.
$$
By replacing $r$ with $r^{-1}$ above we obtain ${\mathbb W} \subset r^{\H^*} {\mathbb W}$ and hence
\begin{equation}\label{e:p:support}
r^{\H^*}({\mathbb W}) = {\mathbb W},\ \ \ \mbox{ for all }r>0.
\end{equation}
Observe that ${\rm supp}(Y) = {\mathbb W}^\perp$.  Therefore, Relation \eqref{e:p:support} implies $r^{\H} {\rm supp}(Y) ={\rm supp}(Y)$, for all $r>0$. Indeed,
this can be seen by writing
\begin{equation}\label{e:p:support-1}
\langle r^{\H} x, y\rangle = \langle x, r^{{\H}^*}y\rangle
\end{equation}
Taking $y\in  {\mathbb W}$ and $x\in {\rm supp}(Y) \equiv  {\mathbb W}^\perp$, we see that  $r^{{\H}^*}y\in  {\mathbb W}$ by \eqref{e:p:support} and hence $\langle r^{\H} x,y\rangle =0$,
for all $x\in {\rm supp}(Y)$ and $y\in  {\mathbb W}$.  This shows that $r^{\H}({\rm supp}(Y)) \subset {\rm supp}(Y),$ for all $r>0$.  
As argued above, by applying the inverse $r^{{-\H}} = (r^{\H})^{-1}$, we get  
$$
r^{\H} ({\rm supp}(Y)) = {\rm supp}(Y),\ \ \mbox{ for all }r>0.
$$
Now, we argue that this relation entails ${\H}({\rm supp}(Y)) \subset {\rm supp}(Y)$.  Observe that $r\mapsto r^{\H}$ is Fr\'echet differentiable in the Banach
space of bounded operators on ${\mathbb V}$ equipped with the operator norm.  That is, for all $r>0$, we have
$$
\left \|\frac{(r+h)^{\H} - r^{\H}}{h} - r^{{\H}-1} {\H} \right \| \to 0,\ \ \mbox{ as }h\to 0.  
$$
By \eqref{e:p:support-1}, for every $x\in {\rm supp}(Y)$, we have $y_h:= h^{-1}((r+h)^{\H} - r^{\H})(x) \in {\rm supp}(Y)$, thus the Fr\'echet differentiability
relation above implies that
$\| y_h - r^{{\H}-1} {\H} x\|\to 0$, as $h\to 0$, which, since ${\rm supp}(Y)$ is closed entails $r^{{\H}-1}{\H} x\in{\rm supp}(Y)$. Taking $r=1$ and 
considering that $x\in {\rm supp}(Y)$ was arbitrary, we obtain
$\H({\rm supp}(Y))\subset {\rm supp}(Y),$ which completes the proof of {\em (i)}.  Note that unless ${\H}$ is invertible with bounded inverse,
we cannot readily conclude that ${\H}({\rm supp}(Y)) = {\rm supp}(Y)$.

Since ${\rm supp}(Y)$ is a closed linear subspace of ${\mathbb V}$, it is itself a Hilbert space.  Part {\em (i)} allows us 
to consider the restriction ${\H}_Y:= {\H}\vert_{{\rm supp}(Y)}$ of the operator ${\H}$ to ${\rm supp}(Y)$.  To emphasize that the domain is 
now restricted to ${\rm supp}(Y)$, we denote by $\H_Y^*$ the adjoint of ${\H}_Y$ in ${\rm supp}(Y)$.

Note that ${\rm Im}({\H}_Y)^\perp ={\rm Ker}({\H}_Y^*)$ and since ${\rm Im}({\H}_Y)$ is a linear space, it follows that
 $\overline{{\rm Im}({\H}_Y)} = {\rm Ker}({\H}_Y^*)^\perp$.
Thus, to prove {\em (ii)}, it is enough to show that ${\rm Ker}({\H}_Y^*) = \{0\}$.  To this end, suppose that ${\H}_Y^*x = 0$ for some 
$x\in  {\rm supp}(Y)$.  We have that $r^{{\H}_Y^*}x = x$, for all $r>0$ since
$$
r^{{\H}_Y^*}(x) = e^{\log(r) {\H}_Y^*}x = x + \sum_{n=1}^\infty \frac{\log^n(r) ({\H}_Y^*)^n x}{n!} = x.
$$
This, in view of \eqref{e:p:support--1}, implies that 
$$
\E | \langle x,Y(r\lambda)\rangle| ^2 =  \langle r^{{\H}_Y^*} x, \mathcal{C}_Y(\lambda,\lambda) r^{{\H}_Y^*} x\rangle  =  \langle x, \mathcal{C}_Y(\lambda,\lambda)x\rangle,
$$
for all $r>0$ and $\lambda \in \Lambda_k$.  Now, the continuity of the IRF $Y$ implies that the left-hand side of the last
expression vanishes as $r\downarrow 0$.  This means that $x\in W$, but ${\rm supp}(Y) = W^\perp$ (in ${\mathbb V}$) and hence $x \in W \cap W^{\perp} =\{0\}$.
This shows that ${\rm Ker}({\H}_Y^*) = \{0\}$ in ${\rm supp}(Y)$ completing the proof of {\em (ii)}.
\end{proof}

\tableofcontents

\end{document}